\documentclass[11pt]{article}

\usepackage{mathtools}
\usepackage{amssymb}

\usepackage[dvipsnames]{xcolor}
\usepackage{graphicx}
\usepackage{caption}
\usepackage{subcaption}
\usepackage{microtype}
\usepackage[T1]{fontenc}

\usepackage{tikz}
\usetikzlibrary{arrows.meta}

\usepackage[a4paper,left=3cm,right=3cm,top=3cm,bottom=3cm]{geometry}

\usepackage{hyperref}
\definecolor{colorlinks}{RGB}{0, 24, 168}
\definecolor{colorcites}{RGB}{124, 10, 2}
\hypersetup{
    colorlinks=true,
    linkcolor=colorlinks,
    citecolor=colorcites,
    urlcolor=colorlinks,
    pdfborder={0 0 0}
}

\numberwithin{equation}{section}
\usepackage{amsthm}
\theoremstyle{plain}
    \newtheorem{theorem}[equation]{Theorem}
    \newtheorem{proposition}[equation]{Proposition}
    \newtheorem{lemma}[equation]{Lemma}
    \newtheorem{corollary}[equation]{Corollary}

    \newtheorem*{convention*}{Convention}
    \newtheorem*{proposition*}{Proposition}
    \newtheorem*{theorem*}{Theorem}

\theoremstyle{definition}
    \newtheorem{definition}[equation]{Definition}
    \newtheorem*{remark*}{Remark}
    \newtheorem*{remarks*}{Remarks}

\newcommand{\R}{\mathbb{R}}  
\newcommand{\Z}{\mathbb{Z}}

\newcommand{\N}{\mathbb{N}}

\def\PB{\operatorname{PB}}
\def\FB{\operatorname{FB}}

\newcommand\numberthis{\addtocounter{equation}{1}\tag{\theequation}}
\usepackage[shortlabels]{enumitem}

\usepackage{titlesec}
\usepackage{tabto}

\titleformat{\chapter}[display]              
  {\usefont{T1}{qhv}{b}{n}\selectfont\Huge}  
  {\chaptertitlename~\thechapter}            
  {.5em}                                     
  {}                                         
  [\vspace{.25em}{\titlerule[1pt]}]          

\titleformat{\section}[hang]                 
  {\usefont{T1}{qhv}{b}{n}\selectfont\Large} 
  {\thesection}                              
  {.5em}                                     
  {\optionalindent}                          

\titleformat{\subsection}[hang]              
  {\usefont{T1}{qhv}{b}{n}\selectfont\large} 
  {\thesubsection}                           
  {.5em}                                     
  {\optionalindent}                          

\titleformat{\subsubsection}[hang]                
  {\usefont{T1}{qhv}{b}{n}\selectfont\normalsize} 
  {\thesubsubsection}                             
  {.5em}                                          
  {\optionalindent}                               

\newcommand\optionalindent{}
\begin{document}

\title{Macroscopic behavior of Lipschitz random surfaces}
\author{Piet Lammers\textsuperscript* \and Martin Tassy\textsuperscript\textdagger}
\date{}
\maketitle

\vfill

\begin{abstract}
The motivation for this article
is to derive strict convexity of the surface tension
for Lipschitz random surfaces, that is, for models of random Lipschitz functions from $\mathbb Z^d$
to $\mathbb Z$ or $\mathbb R$.
An essential innovation is that random surface models
with long- and infinite-range interactions
are included in the analysis.
More specifically, we cover at least:
uniformly random graph homomorphisms from $\mathbb Z^d$
to a $k$-regular tree for any $k\geq 2$
and Lipschitz potentials which satisfy the FKG lattice condition.
The latter includes
perturbations of dimer- and six-vertex models
and of Lipschitz simply attractive potentials
introduced by Sheffield.
The main result is that we prove strict convexity of the surface
tension---which implies uniqueness for the
limiting macroscopic profile---if
the model of interest is monotone in the boundary conditions.
This solves a conjecture of Menz and Tassy,
and answers a question posed by Sheffield.
Auxiliary to this,
we prove several results which may be of independent interest,
and which do not rely on the model being monotone.
This includes existence and topological properties of
the specific free energy,
as well as a characterization of its minimizers.
We also prove a general large deviations principle which describes both the macroscopic
profile and the local statistics of the height functions.
This work is inspired by, but independent of,
\emph{Random Surfaces} by Sheffield.

\end{abstract}

\vfill
\vfill
\vfill
\vfill

\newcommand\symbolspace{\tabto{\parindent}}

\noindent
\symbolspace
2020 \emph{Mathematics Subject Classification}.
Primary
82B20, 
82B41, 
60F10, 
82B30. 

\smallskip

\noindent
\symbolspace
\emph{Keywords.}
Surface tension,
limit shapes,
variational principle,
large deviations principle,
gradient Gibbs measures,
entropy minimizers,
ergodicity,
stochastic monotonicity,
six-vertex model,
Lipschitz functions.

\smallskip

\noindent
\textsuperscript*
\symbolspace
Statistical Laboratory, Centre for Mathematical Sciences, University of Cambridge\\
\texttt{p.g.lammers@statslab.cam.ac.uk}

\smallskip

\noindent
\textsuperscript\textdagger
\symbolspace
Department of Mathematics, Dartmouth College\\
\texttt{mtassy@math.dartmouth.edu}

\newpage

\setcounter{tocdepth}{2}
\tableofcontents
\renewcommand{\listfigurename}{List of figures}
\listoffigures

\newpage

\renewcommand\optionalindent{\tabto{1.5cm}}


\section{Introduction}
\subsection{Preface}

We study the macroscopic behavior of models
of \emph{Lipschitz random surfaces},
that is,
random Lipschitz functions from $\mathbb Z^d$
to $\mathbb Z$ or $\mathbb R$.
Examples of such models include
 height functions of dimer models
and six-vertex models
and uniformly random $K$-Lipschitz functions.
One studies in particular the local Gibbs measures,
subject to boundary conditions.
It is generally expected that the macroscopic limit of a random surface under the influence of boundary conditions is governed by a \emph{variational principle}.
This variational principle asserts that, under suitable boundary conditions on a bounded domain $D\subset\mathbb R^d$, the asymptotic macroscopic profile $f^*$ must
concentrate on any neighborhood of the set of
minimizers of
the integral
\begin{equation}
\label{eq:fundamental_integral}
\int_D \sigma(\nabla f(x)) dx
\end{equation}
over all those functions $f$ that match these boundary conditions.

The convex function $\sigma$, which is called the \emph{surface tension}, is specific to the model and encodes the free energy density of gradient Gibbs measures which are constrained to a certain slope.
Sheffield proves
in his seminal work \emph{Random Surfaces}~\cite{S05}
that this variational principle can be generalized into a large deviation principle that governs not only the macroscopic profile, but also the local statistics of a random surface over
macroscopic regions.
These results apply to a significant number of models.
The fundamental integral in~\eqref{eq:fundamental_integral} connects the large deviations principle
and the variational principle:
it appears as the rate function in the large deviations principle,
which implies the asserted concentration.
When $\sigma$ is strictly convex,
the rate function of the large deviations principle
has a unique minimizer $f^*$ and the random functions concentrate around
this unique minimizer
(see~\cite{de2010minimizers} for a proof that
strict convexity of $\sigma$ implies uniqueness of the minimizer
of the integral).
This also implies that the model is stable under microscopic changes in the boundary conditions.
On the other hand, when $\sigma$ fails to be strictly convex,
simulations have suggested that microscopic changes to boundary
conditions might have macroscopic effects,
and (more generally) that random surfaces might be macroscopically
disordered.
To illustrate this point, we refer to Figure~\ref{fig:5vertex}
for two samples from the $5$-vertex model,
one with parameters which make $\sigma$ strictly convex,
and one with parameters for which $\sigma$
is not strictly convex.
 The difference in the macroscopic appearance of these
 two figures is striking.
 This dichotomy underlines the pivotal role played by the surface tension in the study of the asymptotic behavior of random surfaces.

In the last thirty years, there have been various models in statistical physics for which strict convexity
of the surface tension has been derived.
The two most famous are probably the dimer model~\cite{CKP01} for $\mathbb Z$-valued random surfaces and the Ginzburg-Landau $\nabla\phi$-interface under suitable conditions~\cite{funaki1997motion,MR2227242} for $\mathbb R$-valued random surfaces.
In either case, the strategy employed to demonstrate strict convexity of the surface tension relies heavily on particular properties of the model under consideration.
For dimer models,
one is able to calculate $\sigma$ due
to exact integrability of the model~\cite{CKP01};
for the Ginzburg-Landau $\nabla\phi$-interface,
the strategy relies on the fact that the potentials considered are almost Gaussian~\cite{funaki1997motion}.
A decisive breakthrough was made in~\cite{S05} in the pursuit of a more general approach.
In this work, Sheffield proves that statistical physics models associated with simply attractive potentials---that is, convex potentials for which the interactions are exclusively between pairs of points---must have a strictly convex surface tension.
Beyond the surprising generality of the result,
this work also distinguishes itself by the method that was used to prove strict convexity of the surface tension.
Rather than using direct computational arguments, the author reasons by contradiction:
if there is a line segment on which the surface tension is affine,
then the minimizing measures corresponding to either endpoint are used to construct
a new measure which minimizes the specific free energy, but is not a Gibbs measure.
This is then shown to be impossible.

\begin{figure}
	\centering
	\begin{subfigure}{.5\textwidth}
		\centering
		\includegraphics[width=.8\linewidth]{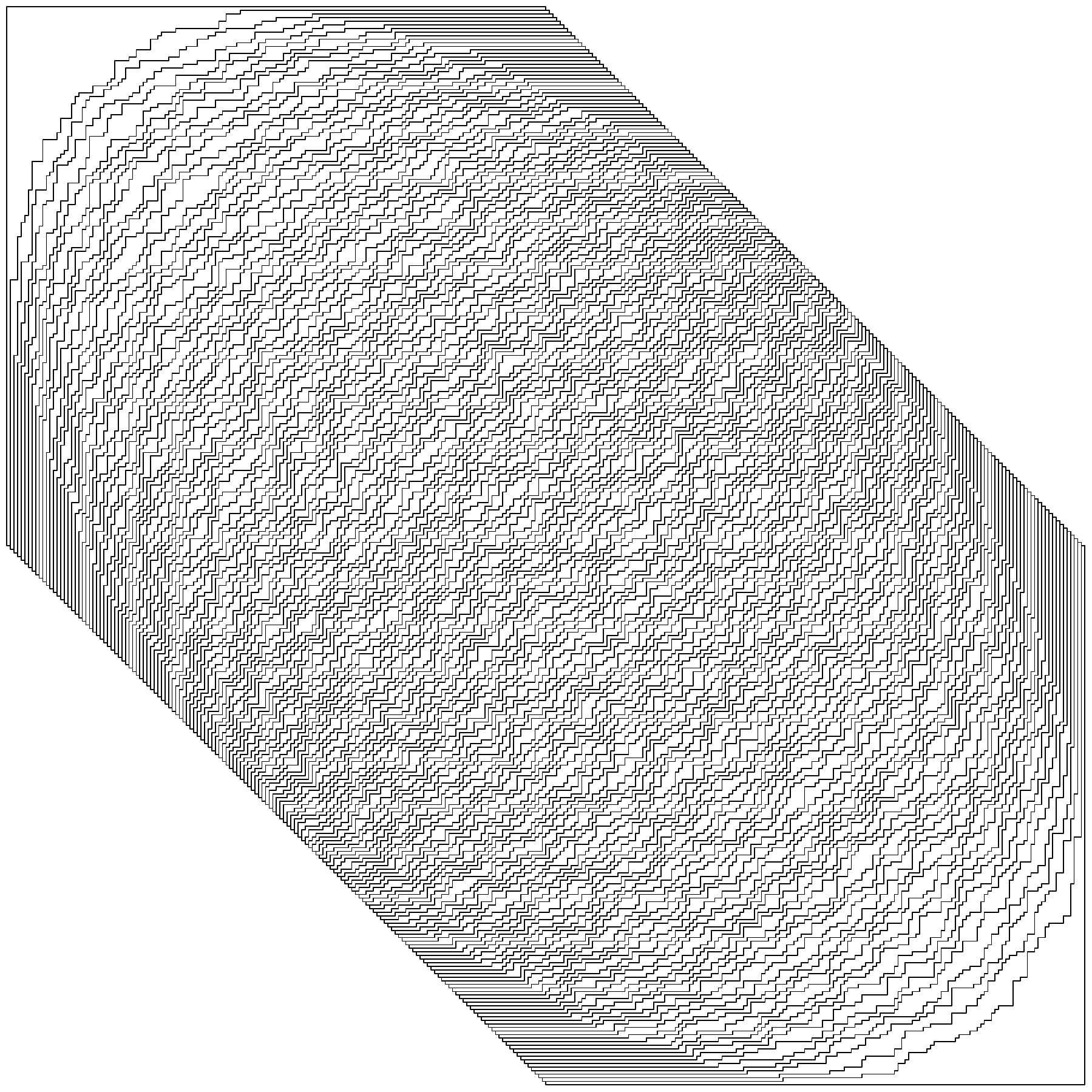}
		\caption*{Monotone parameters}

	\end{subfigure}%
	\begin{subfigure}{.5\textwidth}
		\centering
		\includegraphics[width=.8\linewidth]{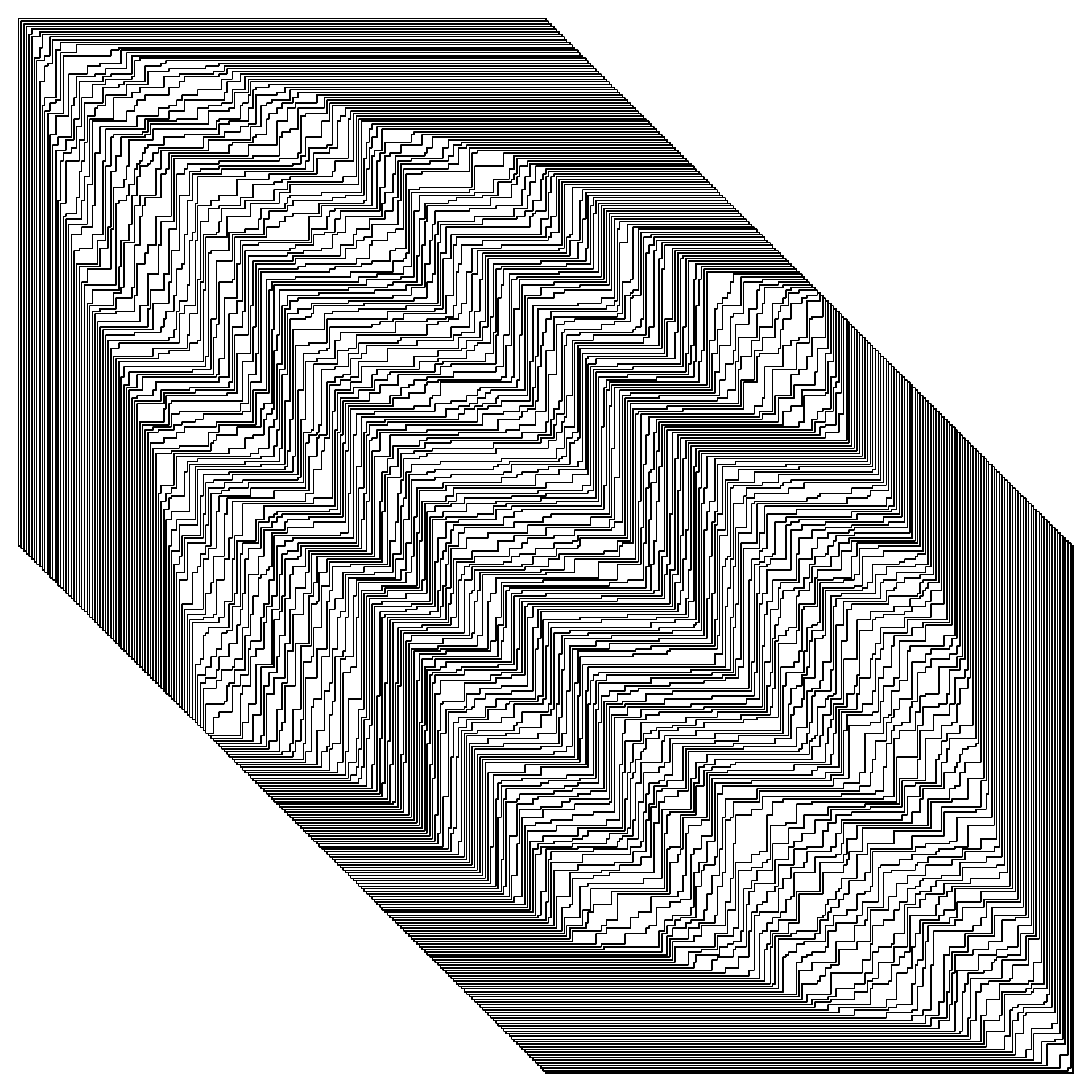}
		\caption*{Non-monotone parameters}
	\end{subfigure}
	\caption[Limiting behavior of the $5$-vertex model]{Limiting behavior of the $5$-vertex model for different parameters}\label{fig:5vertex}
\end{figure}

Despite this significant progress, the techniques used in~\cite{S05} rely heavily on the interactions being between pairs of points only---they cannot capture what happens for models with
interactions involving larger clusters of points.
The purpose of this article is to dramatically increase the class of models for which
strict convexity of the surface tension can be derived.
We do so by providing a new approach which does not rely on a particular formalism of the model in terms of a potential,
but instead on \emph{stochastic monotonicity}.
Notably, the new class includes all Lipschitz models
for which the interaction potential satisfies the Fortuin–Kasteleyn–Ginibre (FKG) lattice condition.
Such potentials are also called \emph{submodular},
and form a natural generalization of the class of
simply attractive potentials.
Moreover, the new class also covers interaction
potentials which assign a weight to each level set of the height function,
in the spirit of the random-cluster model.
Such models have infinite-range interactions,
and we use them to derive strict convexity of the surface tension for
the tree-valued graph homomorphisms studied in~\cite{MT16}.

There are several ideas which suggest that stochastic monotonicity
is a suitable starting point for studying the macroscopic
behavior of random surfaces.
First, for general percolation models, such as
independent percolation and Fortuin-Kasteleyn percolation,
the FKG inequality
is essential to the understanding of the
macroscopic behavior of the model:
most, if not all, modern techniques in percolation theory
 rely on this crucial observation.
 It appears that stochastic monotonicity is
 the most general equivalent of the FKG inequality
 in the context of random height functions.
Second, when the height functions of interest are also Lipschitz,
the Azuma-Hoeffding inequality implies immediately
that the random surface concentrates in some precise sense;
the picture on the right in Figure~\ref{fig:5vertex}
is therefore instantaneously ruled out.
Third,
it turns out that for this $5$-vertex model,
stochastic monotonicity (which depends on the choice of parameters),
is in fact \emph{equivalent} to strict convexity of $\sigma$.

Finally, stochastic monotonicity does not depend on any formalism of potentials.
This is a significant difference with the class of simply attractive
models in~\cite{S05}, which depends on a particular representation of the model
in terms of an underlying interaction potential.
Stochastic monotonicity is thus practical:
it suffices to check the Holley criterion.
For discrete finite-range models, this is particularly efficient, as it amounts
to evaluating a finite number of cases.
%

\subsection{Description of the main results}

Let us now broadly describe the main results of this article. Precise statements of the corresponding theorems are to be found in Section~\ref{sec:main}.
Write $\Omega$ for the set of \emph{height functions},
that is, functions $\phi$
from $\mathbb Z^d$ to $E$,
where the choice of $d$ and $E\in\{\mathbb Z,\mathbb R\}$
depends on the model of interest.
Write $\Lambda\subset\subset\mathbb Z^d$
if the former is a finite subset of the latter;
the model of interest is formalized in terms of a specification
$\gamma=(\gamma_\Lambda)_{\Lambda\subset\subset\mathbb Z^d}$ which allows one to forget about the values of $\phi$
on $\Lambda$ and resample those values according to the model.
The measure $\gamma_\Lambda(\cdot,\phi)$ is also
called the \emph{local Gibbs measure} in $\Lambda$
with boundary conditions $\phi$.
This model must be invariant by some full-rank sublattice
$\mathcal L$ of $\mathbb Z^d$ if any convergent
macroscopic behavior is to be expected.
We impose two key restrictions on $\gamma$
for the main results to apply:
that $\gamma_\Lambda(\cdot,\phi)$ is supported
on height functions which are suitably Lipschitz whenever $\phi$
is Lipschitz,
and that $\gamma_\Lambda(\cdot,\phi)\preceq
\gamma_\Lambda(\cdot,\psi)$ whenever $\phi\leq \psi$.
Models satisfying the former condition are called \emph{Lipschitz},
if they satisfy the latter then they are called \emph{stochastically monotone}.
Finally, for the thermodynamical
formalism, we require that the specification $\gamma$
is generated by some \emph{interaction potential} $\Phi$
which encodes the interactions of the values of $\phi$
at different vertices.
We shall see that the heart of the proof does
not rely on the formalism of potentials as it
is expressed directly in terms of the specification.
As a consequence, we are able to incorporate potentials
$\Phi$ belonging to a very large class
which is described in detail in Section~\ref{section:class_of_models}.
Informally, we allow any potential $\Phi$ which
decomposes as the sum of two potentials $\Psi$ and
$\Xi$,
where $\Psi$ is a potential of finite range which enforces
the Lipschitz property (by assigning infinite potential
to functions which are not Lipschitz),
and where $\Xi$ is potentially an infinite-range potential
whose intensity decays fast enough for the
specific free energy to be well-defined.

%

While the finite-range part $\Psi$ of the potential encompasses all common finite-range models in statistical physics, the infinite-range part $\Xi$ is tailored to fit long-range interaction potentials such as
those associated with the random-cluster model or the Loop $O(n)$ model.
We demonstrate in Subsection~\ref{subsection:applications_tree} that this formalism can even be used to prove a conjecture on the limiting behavior of uniformly random graph homomorphisms
from $\mathbb Z^d$ to a $k$-regular tree for $k\geq 2$.

Let us now introduce a few notions before describing
the main results.
Write $\mathcal P_\mathcal L(\Omega,\mathcal F^\nabla)$
for the collection of $\mathcal L$-invariant
gradient measures on $\Omega$.
Any measure $\mu\in\mathcal P_\mathcal L(\Omega,\mathcal F^\nabla)$
has an associated \emph{slope} $S(\mu)$
which is the unique linear functional $u\in(\mathbb R^d)^*$
such that
\[
 u(x)=\mu(\phi(x)-\phi(0))
\]
for all $x\in\mathcal L$.
The \emph{specific free energy} of $\mu$
is defined by the limit
\[
 \mathcal H(\mu|\Phi):=\lim_{n\to\infty}n^{-d}\mathcal H_{\Pi_n}(\mu|\Phi),
\]
where $\Pi_n\subset\subset\mathbb Z^d$
denotes a box of sides $n$,
and where $\mathcal H_{\Lambda}(\mu|\Phi)$
denotes the \emph{free energy}
 of $\mu$ over $\Lambda$
 with respect to the interior Hamiltonian generated by $\Phi$;
 this quantity is introduced formally in Section~\ref{section:formal_setting}.
 The \emph{surface tension} is
 the function $\sigma:(\mathbb R^d)^*\to\mathbb R\cup\{\infty\}$
 defined by
 \[
 \sigma(u):=\inf_{\text{$\mu\in\mathcal P_\mathcal L(\Omega,\mathcal F^\nabla)$ with $S(\mu)=u$}}\mathcal H(\mu|\Phi).
 \]
 This function is automatically convex as
 $S(\cdot)$ and $\mathcal H(\cdot|\Phi)$
 are affine over $\mathcal P_\mathcal L(\Omega,\mathcal F^\nabla)$---as will be shown---and
 we write $U_\Phi$ for the topological interior of the set $\{\sigma<\infty\}\subset(\mathbb R^d)^*$.
 Finally, call a shift-invariant measure $\mu\in\mathcal P_\mathcal L(\Omega,\mathcal F^\nabla)$
 a \emph{minimizer} if $\mu$ satisfies the equation
 \[
 \mathcal H(\mu|\Phi)=\sigma(S(\mu))<\infty.
 \]

Let us start with the motivating result of this article.

 \begin{theorem*}[strict convexity of the surface tension]
 Let $\Phi$ denote a potential which decomposes
 as described above, and such that
 the induced specification $\gamma^\Phi$ is monotone.
 \begin{enumerate}
 \item If $E=\mathbb R$,
 then $\sigma$ is strictly convex on $U_\Phi$.
 \item If $E=\mathbb Z$,
 then $\sigma$ is strictly convex on $U_\Phi$
 if for any affine map $h:(\mathbb R^d)^*\to\mathbb R$
 with $h\leq \sigma$,
 the set $\{h=\sigma\}\cap\partial U_\Phi$
 is convex.
 In particular, $\sigma$ is strictly convex on $U_\Phi$
 if $E=\mathbb Z$ and at least one of the following conditions is satisfied:
 \begin{enumerate}
 \item $\sigma$ is affine on $\partial U_\Phi$,
 but not on $\bar U_\Phi$,
 \item $\sigma$ is not affine on $[u_1,u_2 ]$ for any distinct $u_1,u_2 \in \partial U_\Phi$ such that $[u_1,u_2 ] \not\subset \partial U_\Phi$.
 \end{enumerate}
 \end{enumerate}
 \end{theorem*}

 See Theorem~\ref{thm_main_main}
 for the formal statement of this theorem.
 The extra condition for $E=\mathbb Z$
 is necessary to control the behavior of
 ergodic measures whose slope is extremal.
 %
 It is shown in the last part of this article
 that this condition holds true for all classical models.
 What happens in general is that measures whose slope lies in $\partial U_\Phi$
 have zero combinatorial entropy, which makes it straightforward to derive the inequalities
 required for satisfying the extra condition.
 However, it is possible to design exotic models for which it is not known if the condition holds true or not,
 and consequently we cannot rule out the existence of an affine part of the surface tension for such exotic models.

Our second main result concerns a characterization of minimizers, for potentials which decompose as described above. This generalizes the results of~\cite{lammers2019variational}
to the gradient setting.
It is valid even if $\gamma^\Phi$
fails to be monotone, and if $\sigma$ fails to be strictly convex.
However, if $\sigma$
is strictly convex,
then there exists an ergodic minimizer
of slope $u$ for any $u\in U_\Phi$.

\begin{theorem*}[minimizers of the specific free energy]
	Consider a potential $\Phi$ which decomposes as described,
 as well as a minimizer $\mu\in\mathcal P_\mathcal L(\Omega,\mathcal F^\nabla)$.
 Then $\mu$ has finite energy in the sense of Burton and Keane,
 which means that any local configuration
 that is Lipschitz, has a positive
 density (if $E=\mathbb R$) or probability (if $E=\mathbb Z$)
 of occurring.
 Moreover, if the specification $\gamma^\Phi$
 is quasilocal, then $\mu$ is a Gibbs measure,
 and if $\gamma^\Phi$ is not quasilocal
 but if $\mu$ is supported on its points of
 quasilocality,
 then $\mu$ is an almost Gibbs measure---which implies
 in particular that $\mu=\mu\gamma^\Phi_\Lambda$
 for any $\Lambda\subset\subset\mathbb Z^d$.
 Finally, if $\mu$ is not supported
 on the points of quasilocality of $\gamma^\Phi$,
 then we obtain results on the
 regular conditional probability distributions of $\mu$
 which are similar in spirit to those obtained in
\cite{lammers2019variational}.
\end{theorem*}

See Theorem~\ref{thm_mr_minimizers}
for the formal statement of this theorem.

The third main result of this paper is a large deviations principle.
This large deviations principle concerns both the macroscopic profile of a height function, as well as the local statistics of the height function within a region of macroscopic size.
Its formal description requires a significant amount of
technical constructions,
for which we refer to Sections~\ref{sec:main} and~\ref{sec:LDP}.
One can also consider the large deviations principle
on macroscopic profiles only, and the rate
function so appearing is given by~\eqref{eq:fundamental_integral}
up to an additive constant so that its minimum equals zero.
This immediately implies the classical variational principle of~\cite{CKP01}.
The formal statements are included in Theorem~\ref{thm:sec_mr_ldp},
Corollary~\ref{cor_vp}, and Theorem~\ref{p_LDP}.

\begin{theorem*}[variational principle]
 Consider a potential $\Phi$
 which decomposes as above.
	Let $(D_n, b_n)_{n\in \mathbb N}$ denote
 a sequence of pairs of discrete regions $D_n\subset\subset\mathbb Z^d$ and boundary conditions $b_n\in\Omega$ which, after rescaling, suitably approximates some continuous region
 $D\subset\mathbb R^d$ endowed with some boundary
 function $b:\partial D\to\mathbb R$.
 Then
 the random function $f_n$ obtained by sampling a configuration from $\gamma^\Phi_{D_n}(\cdot,b_n)$ and rescaling,
 is contained with high probability as $n\to\infty$
 in any neighborhood of the set of minimizers $f^*$
of the integral
\[
 \int_D\sigma(\nabla f(x))dx
\]
over all functions $f:\bar D\to\mathbb R$
which equal $b$ on $\partial D$.
If $\sigma$ is strictly convex,
then this minimizer $f^*$ is unique, in
which case $f_n\to f^*$ in probability
as $n\to\infty$.
\end{theorem*}

In the final part of this article, we provide several applications of our results.

Sheffield conjectured that similar
results to those obtained in~\cite{S05}
apply to finite-range submodular potentials,
that is, finite-range potentials which satisfy the FKG lattice condition.
We prove that our framework applies to
submodular Lipschitz potentials,
and we prove that the extra condition for
$E=\mathbb Z$ is automatically satisfied
if the model of interest is $\mathcal L$-invariant
for $\mathcal L$ equal to the full lattice $\mathbb Z^d$.
In fact, we do not even require that the
submodular potential of interest has finite range.
See Theorem~\ref{thm:main_results_submodular} for the corresponding
formal statements.

We furthermore consider the model of uniformly random
graph homomorphisms from $\mathbb Z^d$
to a $k$-regular tree.
Remark that $k$-regular trees are also Cayley
graphs of finitely generated free groups.
We confirm the conjecture in~\cite{MT16}, which asserts that the surface tension associated with this model is strictly convex:
see Theorem~\ref{thm:main_tree_valued_statement}.
This is remarkable because our theory is phrased in terms
of $\mathbb R$- or $\mathbb Z$-valued functions only.


\subsection{Ideas and strategy of the proof}

The proof of the main results splits into two parts.
The first part develops a range of thermodynamical machinery
for the class of potentials under consideration.
The line of thought motivating these results
and proofs was already present in the literature,
most notably in the work of Georgii~\cite{G11},
Sheffield~\cite{S05}, and a previous work of the authors~\cite{lammers2019variational}.
However, it requires significant effort to adapt these existing tools
to the generality of our setting.
The second part provides a proof of strict convexity
of the surface tension, if the potential of interest
furthermore induces a specification that is stochastically monotone.
This is where we break new ground.
Sheffield~\cite{S05} proves that the surface tension
is strictly convex
by employing the following general strategy:
\begin{enumerate}
 \item Suppose that $\sigma$
 is affine on a line segment $[u_1,u_2]$
 for $u_1,u_2\in U_\Phi$ distinct,
 \item Construct a shift-invariant gradient measure
 in $\mathcal P_\mathcal L(\Omega,\mathcal F^\nabla)$
 of slope $u=(u_1+u_2)/2$ with minimal specific free
 energy and which does not have finite energy,
 \item Conclude that this contradicts the characterization
 of the minimizers of the specific free energy,
 as mentioned earlier in this introduction.
\end{enumerate}
The same strategy is employed here,
but the construction of the gradient measure, as well as the heuristic that this construction is based on, are entirely original.
The remainder of this subsection gives an overview of this construction.

First, the surface tension $\sigma(u)$
at some slope $u$ can be expressed
in terms of the asymptotic behavior of the partition function
of $\gamma_{\Pi_n}(\cdot,\phi^u)$
where $\phi^u$ approximates $u$ in some precise sense:
this is a consequence of the large deviations principle.
We then consider the product measure $\mu:=\gamma_{\Pi_n}(\cdot,\phi^u)\times\gamma_{\Pi_n}(\cdot,\phi^u)$;
write $(\phi_1,\phi_2)$
for the random pair of height functions in $\mu$,
and write $f$ for the difference $\phi_1-\phi_2$.
One can use the fact that $\sigma$
is affine on the line segment $[u_1,u_2]$
to derive that the function $f$ deviates macroscopically---that is,
at scale $n$---from $0$
with log probability of order $o(n^d)$ as $n\to\infty$.
We then use monotonicity of the specification
$\gamma$ to compare the probability of a macroscopic deviation
of $f$ to the probability that the set $\{f\in [a,b]\}\subset\Pi_n\subset\subset\mathbb Z^d$
has many large connected components for fixed $0<a<b<\infty$.
This requires the development of an essential and original
geometrical construction. The connected components
of $\{f\in [a,b]\}$ of interest are called
\emph{moats}.
Finally, we randomly shift the functions $\phi_1$ and $\phi_2$
by a vector in $\Pi_n\cap\mathcal L$
and take limits to produce a shift-invariant measure
on the product space, such that each marginal has slope $u$.
The two lower bounds on probabilities imply an upper
bound on the specific free energy of this product measure.
We show that the moats---the large connected components
of
$\{f\in [a,b]\}$---grow to be distinct infinite components
in this limiting procedure.
This contradicts that for a shift-invariant measure
with finite energy, the random set $\{f\in [a,b]\}$
cannot have more than one infinite component due
to the argument of Burton and Keane:
the desired contradiction.

Let us finally elaborate briefly on the geometrical construction involving
 moats.
The goal is to find a lower bound on the
probability that $\{f\in[a,b]\}$ has many large level set,
in terms of the probability that $f$ deviates macroscopically from
$0$.
Write $c_n:=(\lfloor n/2\rfloor,\dots,\lfloor n/2\rfloor)\in\Pi_n$
for the center vertex of $\Pi_n$,
and suppose, by means of illustration,
that $f(c_n)>\varepsilon n$ for some $\varepsilon>0$.
If $\phi_1$ and $\phi_2$ are $K$-Lipschitz for some $K\in (0,\infty)$,
then $f$ is $2K$-Lipschitz.
Choose $a=4K$ and $b=8K$.
Since $f(c_n)$ is large and since $f$ equals $0$
on the complement of $\Pi_n$,
we observe that $\{f\in [a,b]\}$
must contain a connected component
which is contained in $\Pi_n$ and
surrounds the vertex $c_n$ in some precise sense.
This connected component is called a \emph{moat}.
Now fix an arbitrary connected set $M\subset \Pi_n$,
and condition on the event that $M$ is a moat,
and that $f$ is larger than $b$ directly inside $M$.
Equipped with monotonicity,
it is straightforward to demonstrate that
it is more likely (in this conditioned measure)
that $f(c_n)\leq -\varepsilon n+10K$,
than that $f(c_n)\geq \varepsilon n$.
But if $f(c_n)\leq -\varepsilon n+10K$
and if $f$ is larger than $b$ directly on the inside
of $M$,
then $\{f\in[a,b]\}$ must have another connected component
which surrounds $c_n$,
and which is in turn surrounded by the original moat $M$.
One can continue this procedure to generate a
sequence of moats of length $\lfloor\varepsilon n/10K\rfloor$,
such that each moat surrounds the moat that succeeds it.
It is important that the union of all moats occupy a
uniformly positive proportion of $\Pi_n$ as $n\to\infty$,
so that they do not disappear in the limiting procedure
after rerandomizing the position of the origin;
this is indeed the case because of the lower bound on the number
of moats.

\subsection{Open questions}

The first natural question which is left open in this work is to decide if it is possible to drop the requirement that random functions are Lipschitz.
We believe that it is indeed the case, a significant clue being that this requirement does not appear in~\cite{S05}.
Finding a way around this restriction would open the main result to a whole new class of interactions.
However, the geometrical
construction involving the moats
relies heavily on the Lipschitz
property.

Secondly, it would be interesting
to study how the requirement of stochastic
monotonicity can be relaxed.
Results on strict convexity of the surface tension have been obtained for some non-monotone models
for a class of non-convex
potentials~\cite{cotar2009strict,AIHPB_2012__48_3_819_0,adams2016strict},
and for small non-monotone perturbations of dimer models~\cite{MR3606736,
  GMT2019}.
In the simulation on the
right in Figure~\ref{fig:5vertex},
macroscopic disorder is
explained by a heuristic.
For this simulation, the parameters
of the model are chosen such that
straight lines are much preferred
over corners.
This means that the random surface
is able to build
\emph{momentum}:
deviations from the mean reinforce
each other.
This is the exact opposite of
stochastic monotonicity.
However, there are more subtle (and potentially more local) ways
in which stochastic monotonicity might fail.
A simple example would be to consider
random $1$-Lipschitz functions
from $\mathbb Z^d$
to $\mathbb Z$,
where the potential discourages
neighboring vertices
from taking the exact same value.
It is easy to show that this model is not monotone,
but there is no heuristic of momentum building which would imply
macroscopic disorder.
Perhaps it would be possible to prove that this model is stochastically
monotone in some relaxed sense, in which case the results on moats
could be adapted to fit this model.

\section{The thermodynamical formalism}
\label{section:formal_setting}

The interest is in distributions of
the random function $\phi$ which assigns a value $\phi(x)$
from $E$ to each vertex $x\in\mathbb Z^d$,
where $d\geq 2$ and---depending on the model of interest---$E$ denotes
either $\mathbb Z$ or $\mathbb R$.
Such distributions are studied in relation to an underlying model,
which encodes the interactions that exist between
the function values of $\phi$ at different vertices in $\mathbb Z^d$.
At the very least, the underlying model must give rise to a functional,
which assigns a real number---the specific free energy---to any shift-invariant distribution of $\phi$.
In the non-gradient setting there are at least three ways to characterize
the model of interest:
\begin{enumerate}
 \item Through a reference measure on $E$ and an interaction potential,
 \item Through a reference distribution of $\phi$,
 \item Directly through the specification.
\end{enumerate}
Each formulation has slightly different properties,
but they all generate a suitable entropy functional whenever
the correct conditions are imposed. See~\cite{lammers2019variational} for an overview.
In the gradient setting of this paper we must be more careful,
and it seems that only the first formulation generates
a suitable entropy functional.
The goal of this section is to efficiently describe
the standard objects for the formal framework of gradient models on $\mathbb Z^d$.

Subsection~\ref{subsection:formal_setting:gradient_formalism} introduces the
necessary objects and symmetries for the shift-invariant gradient setting.
The same subsection also introduces the key restrictions
on the model: that the specification is monotone,
and that it produces Lipschitz functions.
Subsection~\ref{subsection:formal_setting:potentials_specifications} describes the formalism of potentials.
Subsection~\ref{subsection:formal_setting:surface_tension} introduces
the specific free energy and the surface tension.
The specific free energy is well-defined
for all potentials $\Phi$ in the class $\mathcal S_\mathcal L+\mathcal W_\mathcal L$ which is introduced in Section~\ref{section:class_of_models};
we prove existence of the specific free energy
in Section~\ref{section:specific_free_energy}.
All definitions in the current section are standard.

\subsection{The gradient formalism}
\label{subsection:formal_setting:gradient_formalism}

\subsubsection{Height functions}

We are interested in distributions of the random
function $\phi$, which assigns values from the measure space
$(E,\mathcal E,\lambda)$ to
the vertices of the  square lattice
$\mathbb Z^d$.
Here $E$ refers to either $\mathbb Z$ or $\mathbb R$, depending on the context,
$\mathcal E$ is the Borel $\sigma$-algebra,
and $\lambda$ denotes the counting measure (if $E=\mathbb Z$)
or the Lebesgue measure (if $E=\mathbb R$).
The choice of $E$ is considered fixed throughout the entire work.
The set of all functions $\phi$ from $\mathbb Z^d$ to $E$ is denoted by $\Omega$.
Functions in $\Omega$ are called \emph{samples} or \emph{height functions}.
For $\Lambda\subset\mathbb Z^d$
and $\phi\in\Omega$,
write $\phi_\Lambda\in E^\Lambda$
for the restriction $\phi|_\Lambda$.
If furthermore $\Delta\subset\mathbb Z^d$
and $\psi\in\Omega$ with $\Lambda$ and $\Delta$ disjoint,
then write $\phi_\Lambda\psi_\Delta\in E^{\Lambda\cup\Delta}$
for the unique function that restricts to $\phi$ on $\Lambda$
and to $\psi$ on $\Delta$.

\subsubsection{Subsets of $\mathbb Z^d$}

Write $\Lambda\subset\subset\mathbb Z^d$ if
$\Lambda$ is a finite subset of $\mathbb Z^d$.
Throughout this article, we shall reserve
the notation $(\Pi_n)_{n\in\mathbb N}$
for the sequence of subsets of $\mathbb Z^d$
defined by $\Pi_n:=[0,n)^d\subset\subset\mathbb Z^d$
for each $n\in\mathbb N$.
Remark that $|\Pi_n|=n^d$ for any $n\in\mathbb N$.

Next,  introduce two notions of boundary for
subsets $\Lambda$ of $\mathbb Z^d$.
Write $\partial\Lambda$ for set of the vertices
 which are
adjacent to $\Lambda$ in the square lattice.
Write $\partial^n\Lambda$
for the set of vertices in $\Lambda$
which are at $d_1$-distance at most $n$
from $\mathbb Z^d\smallsetminus\Lambda$,
for any $n\in\mathbb Z_{\geq 0}$; here $d_1$ is the graph metric corresponding to the square lattice.
Write also $\Lambda^{-n}$ for $ \Lambda\smallsetminus\partial^n\Lambda$.
If $D\subset\mathbb R^d$, then write
$\Lambda(D):=D\cap\mathbb Z^d$ and $\Lambda^{-n}(D):=(\Lambda(D))^{-n}$.

Now let $(\Lambda_n)_{n\in\mathbb N}$
denote a sequence of subsets of $\mathbb Z^d$.
If all sets $\Lambda_n$ are finite with
$|\Lambda_n|\to\infty$
and $|\partial\Lambda_n|/|\Lambda_n|\to 0$
as $n\to\infty$,
then $(\Lambda_n)_{n\in\mathbb N}$ is called a \emph{Van Hove sequence}.
We write $(\Lambda_n)_{n\in\mathbb N}\uparrow\mathbb Z^d$
to mean that $(\Lambda_n)_{n\in\mathbb N}$ is a Van Hove sequence.
The sequence $(\Pi_n)_{n\in\mathbb N}$
is an example of a Van Hove sequence.

\subsubsection{$\sigma$-Algebras and random fields}

If $(X,\mathcal X)$ is any measurable space,
then write $\mathcal P(X,\mathcal X)$ for the set of probability measures on it,
and $\mathcal M(X,\mathcal X)$ for the set of $\sigma$-finite measures.
Define the following $\sigma$-algebras on $\Omega$ for any $\Lambda\subset\mathbb Z^d$:
\begin{alignat*}{4}
  & \mathcal F       && :=\sigma(\phi(x):x\in \mathbb Z^d),
 \qquad
  && \mathcal F_\Lambda        && :=\sigma(\phi(x):x\in \Lambda),
\\
  & \mathcal F^\nabla    && :=\sigma(\phi(y)-\phi(x):x,y\in \mathbb Z^d),\qquad
  && \mathcal F^\nabla_\Lambda && :=\sigma(\phi(y)-\phi(x):x,y\in \Lambda).
\end{alignat*}
A \emph{random field} is a probability measure in $\mathcal P(\Omega,\mathcal A)$
for some $\sigma$-algebra $\mathcal A\subset\mathcal F$.
We introduce the gradient $\sigma$-algebra $\mathcal F^\nabla$ because
it is often not possible to measure the height $\phi(x)$
directly; only the height differences $\phi(y)-\phi(x)$ are measurable.
Note that, with the above definitions, $\mathcal F^\nabla_\Lambda=\mathcal F^\nabla\cap\mathcal F_\Lambda$.
For $\Lambda\subset\mathbb Z^d$,
write $\pi_\Lambda$ for the natural probability kernel
from $(\Omega,\mathcal F)$ to $(E^\Lambda,\mathcal E^\Lambda)$
which restricts random fields to $\Lambda$.

A \emph{cylinder set} is a measurable subset of $\Omega$
which is contained in $\mathcal F_\Lambda$
for some $\Lambda\subset\subset\mathbb Z^d$;
a \emph{cylinder function} is a function $\Omega\to\mathbb R$
which is $\mathcal F_\Lambda$-measurable for some $\Lambda\subset\subset\mathbb Z^d$.
A cylinder function is called \emph{continuous} if it is continuous with
respect to the topology of uniform convergence on $\Omega$.
Note that all cylinder functions are continuous whenever $E=\mathbb Z$.

Define the further $\sigma$-algebras on $\Omega$ for any $\Lambda\subset\mathbb Z^d$:
\[
 \mathcal T_\Lambda:=\mathcal F_{\mathbb Z^d\smallsetminus\Lambda},
 \quad
 \mathcal T:=\cap_{\Delta\subset\subset\mathbb Z^d}\mathcal T_\Delta,\quad
 \mathcal T^\nabla_\Lambda:=\mathcal T_\Lambda\cap\mathcal F^\nabla,\quad
 \mathcal T^\nabla:=\mathcal T\cap\mathcal F^\nabla.
\]
Sets in $\mathcal T$ are called \emph{tail-measurable}.

\subsubsection{The topology of (weak) local convergence}
\label{subsubsec:topo_local_conv}

The \emph{topology of local convergence} is the coarsest
topology on $\mathcal P(\Omega,\mathcal F^\nabla)$ that makes the map $\mu\mapsto\mu(f)$
continuous for any bounded cylinder function $f$.
The \emph{topology of weak local convergence} is the coarsest
topology on $\mathcal P(\Omega,\mathcal F^\nabla)$ that makes the map $\mu\mapsto\mu(f)$
continuous for any bounded continuous cylinder function $f$.
Note that the two topologies coincide whenever $E=\mathbb Z$.
Section~\ref{section_limit_equalities} uses a particular basis $\mathcal B$
for the topology of weak local convergence on
$\mathcal P(\Omega,\mathcal F^\nabla)$.
This basis $\mathcal B$ is defined such that it contains exactly all sets
$B\subset\mathcal P(\Omega,\mathcal F^\nabla)$
which can be written as finite intersections
of open sets of the form
$\{\mu:a<\mu(f)<b\}$, where $a,b\in\mathbb R$
and where $f$ is a continuous bounded cylinder function.

\subsubsection{Shift-invariance and ergodicity}

To see convergence of the model at a macroscopic scale it is important that the model exhibits shift-invariance.
For $x\in\mathbb Z^d$,
write $\theta_x:\mathbb Z^d\to\mathbb Z^d,\,y\mapsto y+x$.
Throughout this paper, the letter $\mathcal L$ denotes a fixed full-rank sublattice of $\mathbb Z^d$,
and $\Theta=\Theta(\mathcal L)=\{\theta_x:x\in\mathcal L\}$ is the corresponding group of translations of $\mathbb Z^d$.
If $\phi\in\Omega$ and $\theta\in\Theta$, then $\theta\phi$ denotes the
unique height function satisfying $(\theta\phi)(x)=\phi(\theta x)$ for all $x$.
Similarly, define
\[
  \theta A:=\{\theta\phi:\phi\in A\},\quad
  \theta\mathcal A:=\{\theta A:A\in\mathcal A\},\quad
  \theta\mu:\theta\mathcal A\to [0,\infty],\,
  \theta\mu(\theta A)\mapsto \mu(A)
\]
for $A\subset\Omega$,
for $\mathcal A$ a sub-$\sigma$-algebra of $\mathcal F$,
and
for $\mu$ a measure on $\mathcal A$.
Any of these three objects is called \emph{$\mathcal L$-invariant}
if they are invariant under $\theta$
for any $\theta\in\Theta$.
If $\mathcal A$ is an $\mathcal L$-invariant $\sigma$-algebra on $\Omega$,
then write
$\mathcal P_\mathcal L(\Omega,\mathcal A)$ for the collection of $\mathcal L$-invariant probability measures on $(\Omega,\mathcal A)$.
Note that $\mathcal P_\mathcal L(\Omega,\mathcal A)$ is
the set of probability measures on $(\Omega,\mathcal A)$
such that $\phi$ and $\theta\phi$ have the same distribution for any $\theta\in\Theta$.

Define finally
\[
\mathcal I_\mathcal L:=\{A\in\mathcal F:\text{$A=\theta A$ for all $\theta\in\Theta$}\},
\qquad
\mathcal I_\mathcal L^\nabla:=\mathcal I_\mathcal L\cap\mathcal F^\nabla.
\]
A gradient measure $\mu\in\mathcal P(\Omega,\mathcal F^\nabla)$
is called \emph{ergodic} if $\mu$ is $\mathcal L$-invariant and trivial on $\mathcal I_\mathcal L^\nabla$.
Write $\operatorname{ex \mathcal P_\mathcal L}(\Omega,\mathcal F^\nabla)$
for the set of all such ergodic gradient measures.
Write $e(\operatorname{ex \mathcal P_\mathcal L}(\Omega,\mathcal F^\nabla))$
for the smallest $\sigma$-algebra
that makes the map $A\mapsto \mu(A)$ measurable for all $A\in\mathcal F^\nabla$.

\subsubsection{Specifications}

A \emph{specification} is a family $\gamma=(\gamma_\Lambda)_{\Lambda\subset\subset\mathbb Z^d}$
of probability kernels, such that
\begin{enumerate}
 \item $\gamma_\Lambda$ is a probability kernel from
       $(\Omega,\mathcal T_\Lambda)$
       to $(\Omega,\mathcal F)$ for each $\Lambda\subset\subset\mathbb Z^d$,
 \item $\mu\gamma_\Lambda(A)=\mu(A)$
       for any $\Lambda\subset\subset\mathbb Z^d$, $A\in\mathcal T_\Lambda$, and $\mu\in\mathcal P(\Omega,\mathcal F)$,
 \item $\gamma_\Lambda\gamma_\Delta=\gamma_\Lambda$
       for any $\Delta\subset\Lambda\subset\subset\mathbb Z^d$.
\end{enumerate}
The specification defines the local behavior of the model,
and we think of $\gamma_\Lambda(\cdot,\phi)$
as the \emph{local Gibbs measure} in $\Lambda\subset\subset\mathbb Z^d$
with boundary conditions $\phi\in\Omega$.
A specification $\gamma$ is called \emph{$\mathcal L$-invariant}
if $\gamma_\Lambda(\cdot,\theta\phi)=\theta\gamma_{\theta\Lambda}(\cdot,\phi)$ for any $\Lambda\subset\subset\mathbb Z^d$, $\phi\in\Omega$,
and $\theta\in\Theta$.
Call $\gamma$ a \emph{gradient specification}
if the distribution of $\psi+a$
in $\gamma_\Lambda(\cdot,\phi)$
equals that of $\psi$ in  $\gamma_\Lambda(\cdot,\phi+a)$
for any $\Lambda\subset\subset\mathbb Z^d$, $\phi\in\Omega$,
and $a\in E$,
where $\psi$ denotes the random height function in
each local Gibbs measure.
Note that each kernel $\gamma_\Lambda$ restricts
to a kernel from $(\Omega,\mathcal T_\Lambda^\nabla)$
to $(\Omega,\mathcal F^\nabla)$ whenever $\gamma$ is a gradient specification.

\subsubsection{Monotonicity}

An event $A\in\mathcal F$ is called \emph{increasing}
if $\phi\in A$ and $\psi\geq\phi$
implies $\psi\in A$.
Consider two measures $\mu_1,\mu_2\in\mathcal P(\Omega,\mathcal F)$.
Say that $\mu_2$ \emph{stochastically dominates} $\mu_1$,
and write $\mu_1\preceq\mu_2$,
if $\mu_1(A)\leq \mu_2(A)$
for any increasing event $A$.
This is equivalent to asking
that there exists a coupling between the two measures such that
$\phi_1\leq \phi_2$ almost surely, where the distributions of $\phi_1$ and $\phi_2$
are prescribed by the measures $\mu_1$ and $\mu_2$ respectively.
A specification $\gamma$ is called \emph{monotone}
if for each $\Lambda\subset\subset\mathbb Z^d$,
the kernel $\gamma_\Lambda$
preserves the partial order $\preceq$ on $\mathcal P(\Omega,\mathcal F)$.
Now consider a fixed measurable set $A\in\mathcal F$,
and use---in this definition---the shorthand $\mathcal P_A$ for the set $\{\mu\in\mathcal P(\Omega,\mathcal F):\mu(A)=1\}$.
The specification $\gamma$
is called \emph{monotone over $A$}
if
 $\mu\gamma_\Lambda\in\mathcal P_A$ for any
 $\Lambda\subset\subset\mathbb Z^d$ and
 $\mu\in\mathcal P_A$,
and
if $\gamma_\Lambda$ preserves the partial order $\preceq$ on $\mathcal P_A$.
The assumption that $\gamma$ is monotone
over a suitable set of Lipschitz functions is
crucial to the proof of strict convexity
of the surface tension.

\subsubsection{The Lipschitz property}

Consider some fixed constant $K\in [0,\infty)$.
A height function is called
\emph{$K$-Lipschitz} if that height function is $K$-Lipschitz with
respect to the graph metric $d_1$ on the square lattice $\mathbb Z^d$.
A measure is called \emph{$K$-Lipschitz} if it is supported on $K$-Lipschitz functions.
The Lipschitz property is further refined in Subsection~\ref{subsection:formal_setting:lipschitz}.


\subsubsection{The slope}

Consider $\mu\in\mathcal P_\mathcal L(\Omega,\mathcal F^\nabla)$.
If $\phi(y)-\phi(x)$ is $\mu$-integrable for any
$x,y\in\mathbb Z^d$,
then $\mu$ is said to have \emph{finite slope}.
If $\mu$ has finite slope,
then shift-invariance of $\mu$ implies that the function
\[\mathcal L\to\mathbb R,\,x\mapsto \mu(\phi(x)-\phi(0))\]
is additive.
In particular, this means that there is a unique linear functional
$u\in(\mathbb R^d)^*$ such
that
\[
 u(x)=\mu(\phi(x)-\phi(0))
\]
for any $x\in\mathcal L\subset\mathbb R^d$.
This linear functional $u$ is called the \emph{slope}
of $\mu$,
and we write $S(\mu)$ for it.
The map $S$ is affine: it is clear that
$S((1-t)\mu+t\nu)=(1-t)S(\mu)+tS(\nu)$
for any $t\in[0,1]$
and for any $\mu,\nu\in\mathcal P_\mathcal L(\Omega,\mathcal F^\nabla)$
with finite slope.

If we restrict to $K$-Lipschitz measures in $\mathcal P_\mathcal L(\Omega,\mathcal F^\nabla)$ for fixed $K\in [0,\infty)$,
then all measures have finite slope,
and the map $\mu\mapsto S(\mu)$ is then
continuous with respect to the topology of (weak) local convergence.

\subsection{Interaction potentials, reference measures, and specifications}
\label{subsection:formal_setting:potentials_specifications}

\subsubsection{Interaction potentials}

The model of interest is formalized in terms of an \emph{interaction potential} $\Phi=(\Phi_\Lambda)_{\Lambda\subset\subset\mathbb Z^d}$, which is
a family
of \emph{potential functions} $\Phi_\Lambda:\Omega\to\mathbb R\cup\{\infty\}$
where each function $\Phi_\Lambda$ is required to be measurable with respect to $\mathcal F_\Lambda$.
The potential $\Phi$ is called a \emph{gradient potential} if each function $\Phi_\Lambda$ is in addition $\mathcal F_\Lambda^\nabla$-measurable.
The potential $\Phi$ is furthermore called \emph{$\mathcal L$-invariant}
or \emph{periodic} if
$\Phi_{\theta\Lambda}(\phi)=\Phi_{\Lambda}(\theta\phi)$ for all $\theta\in\Theta$
and for any $\phi\in\Omega$.
In the sequel, $\Phi$ shall always denote a fixed
periodic gradient potential.
It is always conventionally assumed that $\Phi_\Lambda\equiv 0$ whenever $\Lambda$ is a singleton or empty
because the $\sigma$-algebra $\mathcal F_\Lambda^\nabla$ is then trivial.

Next, introduce the Hamiltonian. For $\Lambda\subset\subset\mathbb Z^d$ and $\Delta\subset\mathbb Z^d$ containing
$\Lambda$, let $H_{\Lambda,\Delta}$ denote the $\mathcal F^\nabla_\Delta$-measurable function from $\Omega$ to $\mathbb R\cup\{\infty\}$ defined by
\[H_{\Lambda,\Delta}:=\sum\nolimits_{\text{$\Gamma\subset\subset\mathbb Z^d$ with $\Gamma\subset\Delta$ and with $\Gamma$ intersecting $\Lambda$}}\Phi_\Gamma.\]
In particular, we write $H_\Lambda:=H_{\Lambda,\mathbb Z^d}$ and $H_\Lambda^0:=H_{\Lambda,\Lambda}$.
We shall soon introduce further conditions on $\Phi$
which ensure that the sum in the display is always well-defined and bounded below.
The function $H_\Lambda$ is called the \emph{Hamiltonian} of $\Lambda$
and $H_\Lambda^0$ is called the \emph{interior Hamiltonian} of $\Lambda$.
We
add a superscript $\Phi$ to this notation whenever multiple interaction potentials
are considered and confusion might possibly arise.

\subsubsection{Reference measures}

For any fixed nonempty $\Lambda\subset\subset\mathbb Z^d$,
there exist natural reference measures on
the measurable spaces
$(\Omega,\mathcal F_\Lambda)$ and $(\Omega,\mathcal F^\nabla_\Lambda)$,
in terms of the previously introduced reference measure $\lambda$
on $(E,\mathcal E)$.
In the non-gradient setting this is straightforward:
the map $\phi\mapsto \phi_\Lambda$
extends to a bijection from $\mathcal F_\Lambda$
to $\mathcal E^\Lambda$,
and $\lambda^\Lambda$ is a measure on $(E^\Lambda,\mathcal E^\Lambda)$.
With only slight abuse of notation,
we write also $\lambda^\Lambda$
for the unique measure
on $(\Omega,\mathcal F_\Lambda)$
that makes the map $\phi\mapsto\phi_\Lambda$ into
a measure-preserving projection from
$(\Omega,\mathcal F_\Lambda,\lambda^\Lambda)$
to $(E^\Lambda,\mathcal E^\Lambda,\lambda^\Lambda)$.
We must be more subtle in the gradient setting:
we cannot measure the height of $\phi$ directly,
and so we cannot pullback the measure $\lambda^\Lambda$.
Fix therefore some reference point $x\in\Lambda$
and set $\Lambda':=\Lambda\smallsetminus\{x\}$,
and consider instead the map
$\phi\mapsto \phi_{\Lambda'}-\phi(x)$.
This map extends to a bijection from $\mathcal F^\nabla_\Lambda$
to $\mathcal E^{\Lambda'}$.
Abuse notation again by writing $\lambda^{\Lambda-1}$
for the unique measure on $(\Omega,\mathcal F^\nabla_\Lambda)$
that turns the map $\phi\mapsto \phi_{\Lambda'}-\phi(x)$ into a measure-preserving projection from
$(\Omega,\mathcal F^\nabla_\Lambda,\lambda^{\Lambda-1})$
to $(E^{\Lambda'},\mathcal E^{\Lambda'},\lambda^{\Lambda'})$.
The notation $\lambda^{\Lambda-1}$ bears no reference
to the choice of $x\in\Lambda$,
as the resulting measure $\lambda^{\Lambda-1}$ is indeed independent of this arbitrary choice.
The gradient reference measures
$\lambda^{\Lambda-1}$ are not used in the definition of the specification that $\Phi$ generates; they will first appear in the definition of the specific free energy.

\subsubsection{The specification generated by a potential}

The potential $\Phi$ generates a specification
$\gamma^\Phi=(\gamma_\Lambda^\Phi)_{\Lambda\subset\subset\mathbb Z^d}$
defined by
\[
 \gamma_\Lambda^\Phi(A,\phi):=\frac{1}{Z_\Lambda^\Phi(\phi)}\int_{E^\Lambda}
 1_A(\psi\phi_{\mathbb Z^d\smallsetminus\Lambda})
 e^{-H_\Lambda^\Phi(\psi\phi_{\mathbb Z^d\smallsetminus\Lambda})}
 d\lambda^\Lambda(\psi),
\]
for any $\Lambda\subset\subset\mathbb Z^d$, $\phi\in\Omega$,
and $A\in\mathcal F$,
where $Z_\Lambda^\Phi(\phi)$ is the normalizing constant
\[
 Z_\Lambda^\Phi(\phi):=\int_{E^\Lambda}
 e^{-H_\Lambda^\Phi(\psi\phi_{\mathbb Z^d\smallsetminus\Lambda})}
 d\lambda^\Lambda(\psi).
\]
We drop the superscript $\Phi$ in this notation unless the choice of
potential is ambiguous.
Of course, $\gamma_\Lambda(\cdot,\phi)$ is a well-defined probability measure
on $(\Omega,\mathcal F)$
only if
$Z_\Lambda(\phi)\in(0,\infty)$.
Say that $\phi$ has \emph{finite energy} if $\Phi_\Lambda(\phi)<\infty$
for any $\Lambda\subset\subset\mathbb Z^d$,
and say that $\phi$ is \emph{admissible} if
it has finite energy and $Z_\Lambda(\phi)\in(0,\infty)$ for any $\Lambda\subset\subset\mathbb Z^d$.
To draw a sample $\psi$ from $\gamma_\Lambda(\cdot,\phi)$,
set first $\psi$ equal to $\phi$ on the complement of $\Lambda$,
then sample $\psi_\Lambda$ proportional to $e^{-H_\Lambda}\lambda^\Lambda$.
Similarly, if $\mu$ is a probability measure on $(\Omega,\mathcal T_\Lambda)$
supported on admissible height functions,
then $\mu\gamma_\Lambda$ is a probability measure on $(\Omega,\mathcal F)$;
to sample from $\mu\gamma_\Lambda$ one first obtains an auxiliary sample $\phi$ from $\mu$;
then one draws the final sample $\psi$ from $\gamma_\Lambda(\cdot,\phi)$.

It is important to observe that $\gamma$ is a gradient specification.
This is due to the fact that $\Phi$ is a gradient potential which makes
$H_\Lambda$ measurable with respect to $\mathcal F^\nabla$,
and because the reference measures $\lambda$ and $\lambda^\Lambda$ are invariant under translations.

\subsection{The surface tension}
\label{subsection:formal_setting:surface_tension}

\subsubsection{Relative entropy}

Recall first the relative entropy.
If $(X,\mathcal X,\nu)$ is an arbitrary $\sigma$-finite
measure space and $\mu$ another probability measure on $(X,\mathcal X)$,
then the \emph{relative entropy} of $\mu$ with respect to $\nu$
is defined by
\[
 \mathcal H(\mu|\nu):=
 \begin{cases}
  \mu(\log f)=\nu(f\log f) & \text{if $\mu\ll\nu$ where $f=d\mu/d\nu$,} \\
  \infty                   & \text{otherwise}.
 \end{cases}
\]
Remark that $\mathcal H(\mu|\nu)\in\mathbb R\cup\{-\infty,\infty\}$
in general,
and that
$\mathcal H(\mu|\nu)\geq -\log\nu(X)$.
If $\nu$ is a finite measure, then we have equality
if and only if $\mu$ is a scalar multiple of $\nu$.
If $\mathcal A$ is a sub-$\sigma$-algebra of $\mathcal X$,
then use the shorthand $\mathcal H_\mathcal A(\mu|\nu)$
for $\mathcal H(\mu|_\mathcal A|\nu|_\mathcal A)$.

\subsubsection{The free energy}

We are now ready to introduce the free energy.
This already requires the presence of some gradient potential
$\Phi$, although we do not yet impose any condition on it.
Consider also some gradient random field $\mu\in\mathcal P(\Omega,\mathcal F^\nabla)$,
and some finite set $\Lambda\subset\subset\mathbb Z^d$.
Then the \emph{free energy} of $\mu$ in $\Lambda$ with respect to $\Phi$
is defined by
\[
 \mathcal H_\Lambda(\mu|\Phi):=
 \mathcal H_{\mathcal F_\Lambda^\nabla}(\mu|e^{-H_{\Lambda}^{0,\Phi}}\lambda^{\Lambda-1})
 =
 \mathcal H_{\mathcal F_\Lambda^\nabla}(\mu|\lambda^{\Lambda-1})
 +
 \mu(H_{\Lambda}^{0,\Phi}).
\]
The free energy is sometimes decomposed into the \emph{entropy}
and the \emph{energy} of $\mu$ in $\Lambda$---the two terms in the rightmost expression in the display respectively. (For the final equality,
we adopt the convention that $\infty-\infty=\infty$.)

\subsubsection{The specific free energy}

The \emph{specific free energy} of a shift-invariant random field $\mu\in\mathcal P_\mathcal L(\Omega,\mathcal F^\nabla)$ with respect to $\Phi$ is defined by the limit
\begin{equation*}
 \mathcal H(\mu|\Phi):=\lim_{n\to\infty}n^{-d}\mathcal H_{\Pi_n}(\mu|\Phi).
\end{equation*}
The specific free energy thus describes the asymptotic of the normalized
free energy of $\mu$ with respect to $\Phi$ over a large box.
In Section~\ref{section:specific_free_energy} we prove that the limit converges
for all $\Phi$ in the class $\mathcal S_\mathcal L+\mathcal W_\mathcal L$
which is described in Section~\ref{section:class_of_models}.
It is also shown in Section~\ref{section:specific_free_energy} that $\mathcal H(\cdot|\Phi)$
is affine and bounded below.

\subsubsection{The surface tension}

Consider a potential $\Phi$ in our class $\mathcal S_\mathcal L+\mathcal W_\mathcal L$,
which implies that the specific free energy is well-defined, affine and bounded below.
The \emph{surface tension} is the function $\sigma:(\mathbb R^d)^*\to\mathbb R\cup\{\infty\}$
defined by
\[
 \sigma(u):=\inf_{\text{$\mu\in\mathcal P_\mathcal L(\Omega,\mathcal F^\nabla)$ with $S(\mu)=u$}}\mathcal H(\mu|\Phi).
\]
The function $\sigma$ must be convex because both $S(\cdot)$ and $\mathcal H(\cdot|\Phi)$
are affine.
We shall write $U_\Phi$ for the interior of the convex set
$\{\sigma<\infty\}\subset (\mathbb R^d)^*$.
Slopes in $U_\Phi$ are called \emph{allowable}.
The major contribution of this article is that
we show that $\sigma$ is strictly convex on $U_\Phi$
whenever
$\gamma^\Phi$ is monotone over
the set of admissible height functions
and if
$\Phi$ is in our class $\mathcal S_\mathcal L+\mathcal W_\mathcal L$
(and under an additional condition whenever $E=\mathbb Z$).


\section{The class of models under consideration}
\label{section:class_of_models}
In the following four subsections, we describe the conditions
which are imposed on the model of interest: these are specific
to this article, and this is where we broaden the class of models
for which strict convexity of the surface tension can be derived.
Subsection~\ref{subsection:formal_setting:lipschitz} describes the Lipschitz setting in more detail.
We take great care in formulating the Lipschitz condition:
this is not necessary for the arguments to work,
but it rather minimizes the restrictions imposed on the class of models.
Let us now consider the potential which generates the model.
The potential of the model of interest must decompose as the sum of
two potentials,
where the first component is a strong, local potential which---at the very least---enforces
the Lipschitz condition (Subsection~\ref{subsection:formal_setting:strong}),
and where the second component is a weak interaction of infinite range
(Subsection~\ref{subsection:formal_setting:weak}).
The word \emph{weak} here is only relative to the word \emph{strong}
that was used to describe the first potential:
in particular, we do not mean to imply that
the second component demonstrates any sort of decay over long distances.
It is the second potential that allows us to assign energy to
large geometric objects, such as level sets.
Subsection~\ref{subsection:formal_setting:overview} finally gives
an overview of the objects describing the model of interest,
and which are considered fixed throughout most of the analysis.

\subsection{Local Lipschitz constraints}
\label{subsection:formal_setting:lipschitz}

We require that a height function has finite energy
if and only if it is Lipschitz with respect to the correct quasimetric.
We shall allow quasimetrics (subject to certain necessary constraints) in order to be as
general as possible.
The Lipschitz constraint must be enforced locally by the potential,
due to the nature of the arguments that we use to derive the main result.
This means that for each vertex $x\in\mathbb Z^d$ we are allowed
to enforce a Lipschitz constraint between $x$ and only finitely many
other vertices $y\in\mathbb Z^d$.
In other words,
what we have in mind is a set $A\subset\mathbb Z^d\times\mathbb Z^d\times\mathbb R$,
such that a height function $\phi$ is Lipschitz
if and only if $\phi(y)-\phi(x)\leq a$ for any $(x,y,a)\in A$,
and such that $A$ becomes a finite set
once we identify
each triple of the form $(x,y,a)$
with all triples of the form $(\theta x,\theta y,a)$
as $\theta$ ranges over $\Theta$.
The local Lipschitz constraint also enforces that the functions are globally Lipschitz
with respect to the correct quasimetric.
This is formalized as follows.

\begin{definition}[local Lipschitz constraint]
  \label{def_llc}
 Call an edge set $\mathbb A$ on $\mathbb Z^d$ an \emph{admissible graph} if
 $\mathbb A$ is $\mathcal L$-invariant and makes $(\mathbb Z^d,\mathbb A)$
 a connected graph of bounded degree.
 Call a function $q:\mathbb Z^d\times\mathbb Z^d\to \mathbb R$
 an \emph{admissible quasimetric} if
 \begin{enumerate}
  \item $q(x,x)=0$ for any $x\in\mathbb Z^d$,
  \item $q(x,y)+q(y,x)>0$ for any $x,y\in\mathbb Z^d$ distinct,
  \item $q(x,z)\leq q(x,y)+q(y,z)$ for any $x,y,z\in\mathbb Z^d$,
  \item $q(\theta x,\theta y)=q(x,y)$ for any $x,y\in\mathbb Z^d$ and $\theta\in\Theta$.

 \end{enumerate}
 Such a function is called
 \emph{integral} if it takes integral values.
 A \emph{local Lipschitz constraint} is a pair $(\mathbb A,q)$
 where
 \begin{enumerate}
  \item $\mathbb A$ is an admissible graph,
  \item $q$ is an admissible quasimetric,
  \item $q$ is maximal among all admissible quasimetrics that equal $q$ on $\mathbb A$,
  in the sense that $p\leq q$
  for any admissible quasimetric $p$   with $p(x,y)\leq q(x,y)$
  for all $\{x,y\}\in\mathbb A$.
 \end{enumerate}
 If $(\mathbb A,q)$ is a local Lipschitz constraint
 and $\varepsilon\geq 0$ a sufficiently small constant,
then write $q_\varepsilon$ for the largest admissible quasimetric
subject to $q_\varepsilon(x,y)\leq q(x,y)-\varepsilon$ for all $\{x,y\}\in\mathbb A$.
(It is demonstrated in Proposition~\ref{propo_qqq} that this is indeed well-defined for $\varepsilon>0$ sufficiently small.)
 Note that the resulting pair $(\mathbb A,q_\varepsilon)$ is also a local Lipschitz constraint.
\end{definition}

\begin{remarks*}
  \begin{enumerate}
    \item   The last condition in the definition of a local Lipschitz constraint
      guarantees that $q$ is fully determined by its
      values on the edges in $\mathbb A$.
\item   We shall sometimes omit the reference to $\mathbb A$
  and simply call $q$ the \emph{local Lipschitz constraint}.
  If $(\mathbb A,q)$ is a local Lipschitz constraint
  and $\mathbb B$ another admissible graph on $\mathbb Z^d$,
  then the pair $(\mathbb A\cup\mathbb B,q)$ is also a local Lipschitz constraint
  producing the same quasimetric $q$.
  We shall always assume, without loss of generality, that
  $\mathbb A$ contains the edges of the square lattice.
\item   If $q$ is a local Lipschitz constraint,
  then there is a constant $K<\infty$
  such that $Kd_1\geq q$.
\item   We do not impose that $q$ takes values in $[0,\infty)$. This restriction is not necessary to make the
  arguments work.
  \end{enumerate}
\end{remarks*}

From now on,
we shall always have in mind a
fixed local Lipschitz constraint $(\mathbb A,q)$.

\begin{definition}[$q$-Lipschitz]
 A function $\phi:\mathbb Z^d\to\mathbb R$ is called \emph{$q$-Lipschitz} if, for every $x,y\in\mathbb Z^d$,
 \[
  \phi(y)-\phi(x)\leq q(x,y).
 \]
 The function $\phi$ is called \emph{$q$-Lipschitz at $z\in\mathbb Z^d$} if
 this inequality is satisfied for any edge $\{x,y\}\in\mathbb A$
 containing $z$.
 Naturally extend these definitions to cover the cases that $\phi:\Lambda\to\mathbb R$
 for some $\Lambda\subset\mathbb Z^d$.
 Write $\Omega_q$ for the collection of $q$-Lipschitz height functions.
 A measure
 is called \emph{$q$-Lipschitz} if it is supported on $\Omega_q$.
 A specification is called \emph{$q$-Lipschitz} if it maps $q$-Lipschitz measures
 to $q$-Lipschitz measures.
 Finally, a function is called \emph{strictly $q$-Lipschitz}
 if it is $q_\varepsilon$-Lipschitz for $\varepsilon>0$ sufficiently small.
\end{definition}

We now construct a number of objects which derive from $q$.
These are necessary to state the main results,
which address the macroscopic behavior of Lipschitz surfaces.

\begin{definition}[$U_q$, $\|\cdot\|_q$]
 By a \emph{slope} we simply mean an element $u$
 in the dual space $(\mathbb R^d)^*$ of $\mathbb R^d$.
 Write $U_q$ for the interior of the set of slopes $u$
 such that $u|_\mathcal L$ is $q$-Lipschitz.
 The set $U_q$ is nonempty and convex---this
 follows from the definition of a local Lipschitz constraint;
 see Lemma~\ref{lemma_lipschitz_q_and_norm}.
 Introduce furthermore the function $\|\cdot\|_q:\mathbb R^d\to \mathbb R$ defined by
 \[
  \|x\|_q:=\sup\{u(x):u\in U_q\}.
 \]
\end{definition}

The function $\|\cdot\|_q$ is positive homogeneous: we have $\|ax\|_q=a\|x\|_q$ for $a\in[0,\infty)$
and $x\in\mathbb R^d$.
It also satisfies the triangle inequality,
in the sense that
$\|x+y\|_q\leq \|x\|_q+\|y\|_q$
for any $x,y\in\mathbb R^d$.

\begin{definition}[$\|\cdot\|$-Lipschitz]
 If $\|\cdot\|:\mathbb R^d\to\mathbb R$ is any positive homogeneous function satisfying the triangle inequality,
 then any other function
 $f:D\to\mathbb R$
 defined on a subset $D$ of $\mathbb R^d$
 is called \emph{$\|\cdot\|$-Lipschitz}
 if $f(y)-f(x)\leq \|y-x\|$
 for any $x,y\in D$.
 The function $f$
 is called \emph{strictly
 $\|\cdot\|_q$-Lipschitz} if it is $\|\cdot\|_{q_\varepsilon}$-Lipschitz
 for some $\varepsilon>0$.
 If $D$ is open, then
  $f$ is called \emph{locally strictly $\|\cdot\|_q$-Lipschitz}
 if $f|_K$ is strictly $\|\cdot\|_q$-Lipschitz
 for all compact sets $K\subset D$.
\end{definition}

For example,
$U_q$ is the interior of the set of slopes $u\in(\mathbb R^d)^*$
which are $\|\cdot\|_q$-Lipschitz.

\subsection{Strong interactions}
\label{subsection:formal_setting:strong}

Let $\Psi$ denote an arbitrary periodic gradient potential.
The potential $\Psi$ is called \emph{positive} if $\Psi_\Lambda\geq 0$
for any $\Lambda\subset\subset\mathbb Z^d$.
The potential $\Psi$ is said to have \emph{finite range} if
$\Psi_\Lambda\equiv 0$ whenever the diameter of $\Lambda$---in the graph
metric $d_1$ on the square lattice---exceeds some fixed constant $R\in\mathbb N$;
in that case the smallest such $R$ is called the \emph{range}
of $\Psi$.
The potential $\Psi$ is called
\emph{Lipschitz}
if there exists a local Lipschitz constraint $(\mathbb A,q)$
such that
$\Psi_\Lambda(\phi)=\infty$
if and only if $\Lambda=\{x,y\}\in\mathbb A$
and $\phi(y)-\phi(x)>q(x,y)$
for some $x,y\in\mathbb Z^d$.
If $E=\mathbb R$
and $\Psi$ Lipschitz with constraint $(\mathbb A,q)$,
then $\Psi$ is called
\emph{locally bounded}
if for any $\varepsilon>0$ sufficiently small,
there exists a fixed constant $C_\varepsilon<\infty$,
such that \[H_{\{x\}}^\Psi(\phi)\leq C_\varepsilon\]
for any $x\in\mathbb Z^d$ and for any $\phi\in\Omega$
which is $q_\varepsilon$-Lipschitz at $x$.

\begin{definition}[strong interaction, $\mathcal S_\mathcal L$]
 A potential $\Psi$ is called a \emph{strong interaction}
 if $\Psi$ has all of the above properties,
 that is,
 if $\Psi$ is a positive Lipschitz periodic gradient potential of finite range,
 and if it is locally bounded in the case that $E=\mathbb R$.
 We shall write $\mathcal S_\mathcal L$ for the collection of strong interactions.
\end{definition}

The class $\mathcal S_\mathcal L$ includes all so-called \emph{Lipschitz simply attractive
potentials}. These are convex Lipschitz nearest-neighbor interactions,
see~\cite{S05}.

\subsection{Weak interactions}
\label{subsection:formal_setting:weak}

Let $\Xi$ denote an arbitrary periodic gradient potential.

\begin{definition}[summability]
The potential $\Xi$ is called \emph{summable}
if it has finite norm
\[
 \|\Xi\|:=\sup_{(x,\phi)\in\mathbb Z^d\times\Omega}
 \sum_{\text{$\Lambda\subset\subset\mathbb Z^d$ with $x\in\Lambda$}}
 |\Xi_\Lambda(\phi)|.
\]
\end{definition}

This requirement is significantly weaker than the
\emph{absolutely summable} setting of Georgii~\cite{G11}.

\begin{definition}[amenability]
By an \emph{amenable function} we mean a function $f$
which assigns a number in $[0,\infty)$ to each finite subset
of $\mathbb Z^d$, such that:
\begin{enumerate}
  \item $f(\Lambda)=f(\theta\Lambda)$ for all $\Lambda\subset\subset\mathbb Z^d$ and for any $\theta\in\Theta$,
  \item $f(\Lambda\cup\Delta)\leq f(\Lambda)+f(\Delta)$
  for all $\Lambda,\Delta\subset\subset\mathbb Z^d$ disjoint,
  \item $f(\Lambda_n)=o(|\Lambda_n|)$ as $n\to\infty$ for any $( \Lambda_n)_{n\in\mathbb N}\uparrow\mathbb Z^d$.
\end{enumerate}
\end{definition}

\begin{definition}[lower exterior bound]
Let us now turn back to
the potential $\Xi$ and define, for any  $\Lambda\subset\subset\mathbb Z^d$,
\[
 e^-(\Lambda)
 :=
 \sup_{\phi\in\Omega}\sum_{\text{$\Delta\subset\subset\mathbb Z^d$
 with $\Delta$ intersecting both $\Lambda$ and $\mathbb Z^d\smallsetminus\Lambda$}}
 |\Xi_\Delta(\phi)|.
\]
The function $e^-(\cdot)$ is called the \emph{lower exterior bound} of $\Xi$.
\end{definition}

The key property of the function $e^-(\cdot)$
is that $|H_\Lambda^\Xi-H_\Lambda^{0,\Xi}|\leq e^-(\Lambda)$.
The lower exterior bound
satisfies
Properties~1 and~2 from the definition of an amenable function;
this is immediate from the definition.

\begin{definition}[weak interaction, $\mathcal W_\mathcal L$]
A \emph{weak interaction} is a summable periodic gradient potential
for which the lower exterior bound is amenable.
Write $\mathcal W_\mathcal L$ for the collection of weak interactions.
\end{definition}

It is straightforwardly verified that amenability of $e^-(\cdot)$ is
equivalent to asking that
$
 e^-(\Pi_n)=o(n^d)
$ as $n\to\infty$.
Remark that $(\mathcal W_\mathcal L, \|\cdot\|)$ is a Banach space.

\subsection{Overview}
\label{subsection:formal_setting:overview}

Let us fix a number of notations, in order to avoid an excessive number of
declarations.
We notify the reader of any deviation from this notation.
We had already agreed that
the choices for $d\geq 2$ and $E\in\{\mathbb Z,\mathbb R\}$ are fixed,
and that
$\mathcal L$ denotes a fixed
full-rank sublattice of $\mathbb Z^d$ with
corresponding translation group $\Theta$.
The letter $\Phi$ denotes a fixed potential
in $\mathcal S_\mathcal L+\mathcal W_\mathcal L$,
and
we fix some pair $(\Psi,\Xi)\in \mathcal S_\mathcal L\times\mathcal W_\mathcal L$
such that $\Phi=\Psi+\Xi$.
This decomposition is not unique,
but this is never a problem.
The specification generated by $\Phi$ is denoted
$\gamma=\gamma^\Phi$.
The pair $(\mathbb A,q)$ always denotes the local Lipschitz constraint
corresponding to $\Psi$,
and the range of $\Psi$ is denoted by $R$.
If
$E=\mathbb Z$,
then $q$ is always assumed to be integral. The function $e^-(\cdot)$ denotes the lower exterior bound of $\Xi$.
Finally, let $K\in(0,\infty)$
denote the smallest constant
such that $Kd_1\geq q$,
and let $N\in\mathbb N$
denote the smallest positive integer
such that $N\cdot\mathbb Z^d\subset\mathcal L$.

\begin{definition}
  The potential $\Phi\in\mathcal S_\mathcal L+\mathcal W_\mathcal L$
  is called \emph{monotone}
  if the induced specification $\gamma=\gamma^\Phi$
  is monotone over $\Omega_q$.
\end{definition}

\section{Main results}
\label{sec:main}

The motivation for writing this article
was to demonstrate that the surface tension is strictly convex on $U_\Phi$ if the potential
of interest is in the class $\mathcal S_\mathcal L+\mathcal W_\mathcal L$
and monotone.
If $E=\mathbb Z$, then we require an extra condition to be met,
but we also demonstrate that this condition is satisfied
for many natural models.
This section contains an overview of the main results,
including several results and applications which are of independent interest.
The results are presented roughly in the order in which they appear
in the article.


\subsection{The specific free energy and its minimizers}
\label{subsec_main_sfe_and_minimizers}

The specific free energy functional plays a fundamental role in the analysis.
The following result is therefore of independent interest in the study of Lipschitz random surfaces;
it is a direct extension of a result of Sheffield~\cite{S05} to the setting of this article.

\begin{theorem}[specific free energy]
  \label{thm_main_sfe}
  If $\Phi\in\mathcal S_\mathcal L+\mathcal W_\mathcal L$,
  then the specific free energy
   functional
   \[
    \mathcal H(\cdot|\Phi):\mathcal P_\mathcal L(\Omega,\mathcal F^\nabla)\to\mathbb R\cup\{\infty\},\,
    \mu\mapsto \lim_{n\to\infty}n^{-d}\mathcal H_{\Pi_n}(\mu|\Phi)
    \]
  is well-defined, affine, bounded below, lower-semicontinuous,
  and for each $C\in\mathbb R$ its lower level set
  \[
   M_C:=\{\mu\in\mathcal P_\mathcal L(\Omega,\mathcal F^\nabla):\mathcal H(\mu|\Phi)\leq C\}
  \]
  is a compact Polish space, with respect to the topology of (weak) local convergence. In fact, the two topologies coincide on each set $M_C$.
\end{theorem}

A measure $\mu\in\mathcal P_\mathcal L(\Omega,\mathcal F^\nabla)$
is called a \emph{minimizer of the specific free energy},
or simply a \emph{minimizer},
if it satisfies the equation
\[
  \mathcal H(\mu|\Phi)=\sigma(S(\mu))<\infty.
\]
For the purpose of deriving the main result,
all that we require is that such minimizers have finite energy,
in a sense which is similar to the notion of finite energy
in the original paper of Burton and Keane~\cite{BURTON}.
There is a canonical way to translate
the concept of finite energy to
the gradient Lipschitz setting:
 we shall see that
the following result fits our arguments.
Recall that $\Omega_q$ denotes the set of $q$-Lipschitz height functions,
and that $\pi_\Lambda$ is the kernel which restrict measures to $\Lambda$,
for any $\Lambda\subset\mathbb Z^d$.

\begin{theorem}[finite energy]
  \label{thm_main_minimizers_finite_energy}
  Consider $\Phi\in\mathcal S_\mathcal L+\mathcal W_\mathcal L$,
  and suppose that $\mu\in\mathcal P_\mathcal L(\Omega,\mathcal F^\nabla)$
  is a minimizer.
  Then for any $\Lambda\subset\subset\mathbb Z^d$,
  we have
  \[
    1_{\Omega_q}(\mu\pi_{\mathbb Z^d\smallsetminus\Lambda}\times\lambda^\Lambda)\ll\mu.
  \]
\end{theorem}

In~\cite{S05}, finite energy follows from the \emph{variational
principle}, which asserts that shift-invariant measures $\mu$
which satisfy $\mathcal H(\mu|\Phi)=\sigma(S(\mu))$
must also be Gibbs measures with respect
to the specification $\gamma=\gamma^\Phi$ induced by the potential $\Phi$---which has finite range.
In the infinite-range setting one cannot
hope for such a statement, because the specification $\gamma$
is not necessarily quasilocal.
This pathology, and its relation
to the variational principle, is discussed
extensively in~\cite{lammers2019variational}.
One of the key observations in that article
is that minimizers of the specific free energy
must have finite energy, even if the concept of a Gibbs measure is not
well-defined because the specification fails to be quasilocal.
There,
finite energy is an immediate corollary of a result (Lemma~5.4) which is not quite
equivalent to the variational principle,
but it is ``as close as one expects to get'' to it in the non-quasilocal setting.
We shall follow the same strategy here:
the following theorem states the strongest result
on minimizers of the specific free energy,
implies directly that such minimizers have finite energy,
and is a direct translate of Lemma~5.4 from~\cite{lammers2019variational}
to the Lipschitz gradient setting.
Let us first introduce the necessary definitions
for the analysis of quasilocality.

\begin{definition}[quasilocality, almost Gibbs measure]
  \label{definition_quasilocality_aGm}
  Consider two finite sets $\Lambda\subset\Delta\subset\subset\mathbb Z^d$.
Denote by $\mathcal A_{\Lambda,\Delta,\phi}$
the set of probability measures on $(E^\Lambda,\mathcal E^\Lambda)$
of the form
$\mu\gamma_\Lambda\pi_\Lambda$,
where $\mu$ is any measure in $\mathcal P(\Omega,\mathcal F)$
subject only to $\mu\pi_\Delta=\delta_{\phi_\Delta}$.
In other words, $\mathcal A_{\Lambda,\Delta,\phi}$
is the set of local Gibbs measures in $\Lambda$
(and restricted to $\Lambda$) given (mixed) boundary conditions which
match $\phi$ on $\Delta$.
Write $\mathcal C(\mathcal A)$
for the closure of any $\mathcal A\subset\mathcal P(E^\Lambda,\mathcal E^\Lambda)$
in the strong topology,
and define
\[
  \mathcal A_{\Lambda,\phi}:=\cap_{\Delta\subset\subset\mathbb Z^d}
  \mathcal C(\mathcal A_{\Lambda,\Delta,\phi}).
\]
A height function $\phi\in\Omega$ is called
a \emph{point of quasilocality}
if $\mathcal A_{\Lambda,\phi}=\{\delta_\phi\gamma_\Lambda\pi_\Lambda\}=\{\gamma_\Lambda(\cdot,\phi)\pi_\Lambda\}$
for any $\Lambda\subset\subset\mathbb Z^d$.
Write $\Omega_\gamma$ for the set of points of quasilocality.
A measure $\mu\in\mathcal P(\Omega,\mathcal F)$
is called an \emph{almost Gibbs measure}
whenever $\mu(\Omega_\gamma)=1$
and
 $\mu=\mu\gamma_\Lambda$ for any $\Lambda\subset\subset\mathbb Z^d$.
 The definition of an almost Gibbs measure is the same for
 gradient measures
 $\mu\in\mathcal P(\Omega,\mathcal F^\nabla)$---noting that
 $\Omega_\gamma\in\mathcal F^\nabla$ as $\gamma$ is a
 gradient specification.
 Almost Gibbs measures are also called \emph{Gibbs measures}
 whenever $\Omega_\gamma=\Omega$.
\end{definition}

Let us now state the strongest result on minimizers,
which is of independent interest.

\begin{theorem}[minimizers of the specific free energy]
  \label{thm_mr_minimizers}
  Consider $\Phi\in\mathcal S_\mathcal L+\mathcal W_\mathcal L$,
  and suppose that $\mu\in\mathcal P_\mathcal L(\Omega,\mathcal F^\nabla)$
  is a minimizer.
  Fix $\Lambda\subset\subset\mathbb Z^d$,
  and write
  $\mu^{\phi}$
  for the regular conditional
  probability distribution of
  $\mu$ on $(\Omega,\mathcal F)$
  corresponding to the projection map $\Omega\to E^{\mathbb Z^d\smallsetminus\Lambda}$.
  Then for $\mu$-almost every $\phi\in\Omega$,
  we have
  $\mu^\phi\pi_\Lambda\in\mathcal A_{\Lambda,\phi}$.
  In particular,
  if $\mu(\Omega_\gamma)=1$,
  then $\mu$ is an almost Gibbs measure,
  and if $\Omega_\gamma=\Omega$, then $\mu$ is a Gibbs measure.
\end{theorem}

We shall furthermore demonstrate that
in each of our applications,
all minimizers are indeed
(almost) Gibbs measures.
We finally derive the following result.

\begin{theorem}[existence of ergodic minimizers]
  \label{thm_existence_ergodic_minimizers}
  Suppose that $\Phi\in\mathcal S_\mathcal L+\mathcal W_\mathcal L$.
  Then for any exposed point $u\in\bar U_\Phi$
  of $\sigma$,
  there exists an ergodic gradient measure $\mu$ of slope $u$
  which is also a minimizer.
  In particular, if $\sigma$ is strictly convex on $U_\Phi$,
  then for each $u\in U_\Phi$, there is an ergodic minimizer of that slope.
\end{theorem}

Theorem~\ref{thm_main_sfe}
is proven in Section~\ref{section:specific_free_energy}.
Theorems~\ref{thm_main_minimizers_finite_energy} and~\ref{thm_mr_minimizers}
are proven in Section~\ref{section_minimizers}.
Theorem~\ref{thm_existence_ergodic_minimizers} is proven in Section~\ref{section:ergo_decomp}.


\subsection{Large deviations principle and variational principle}
\label{subsec:main_results:ldp}
In Section~\ref{sec:LDP} we prove a large deviations principle (LDP) of similar strength to the one stated in Chapter~7 of~\cite{S05}, with the noteworthy difference that we express it directly in terms of the Gibbs specification.
This LDP captures both the macroscopic profile of each sample, as well as its local statistics.
In this subsection however, we shall state a simpler LDP:
one that captures only the macroscopic profile.
By doing so we deliver on the premise that limit shapes
are characterized by a variational principle,
without spending many pages discussing the exact topology
for the LDP with local statistics.
However, the full LDP is also of independent interest,
and we refer the interested reader to Subsection~\ref{subsec:LDP_FULL}.
Before stating the LDP,
we must first describe how a sequence of discrete boundary conditions can approximate
a continuous boundary profile,
and we must also introduce a topology which captures the macroscopic
profile of each sample.
Let $\Phi$ denote a fixed potential throughout this subsection,
and adopt the standard notation from Subsection~\ref{subsection:formal_setting:overview}.

\begin{definition}[asymptotic boundary profile]
	A \emph{domain} is a nonempty
	bounded open subset of $\mathbb R^d$
  such that its boundary has zero Lebesgue measure.
	An \emph{asymptotic boundary profile}
	is a pair $(D,b)$
	where $D$ is a domain
	and $b$ a $\|\cdot\|_q$-Lipschitz function
	on $\partial D$.
	If $E=\mathbb R$, then call an asymptotic boundary profile
	$(D,b)$ \emph{good}
	if $b$ is strictly $\|\cdot\|_q$-Lipschitz.
	If $E=\mathbb Z$, then call an asymptotic boundary
	profile \emph{good}
	if it is \emph{non-taut}.
	An asymptotic boundary profile
	$(D,b)$ is called \emph{non-taut}
	if $b$ has an extension $\bar b$
	to $\bar D$ such that $\bar b|_D$
	is locally strictly $\|\cdot\|_q$-Lipschitz.
	This is equivalent to asking that
	the largest and smallest
	$\|\cdot\|_q$-Lipschitz
	extensions $b^\pm$ of $b$
	to $\bar D$
	satisfy $b^-<b^+$ on $D$.
\end{definition}

\begin{definition}[discrete approximations]
	Let $(D,b)$ denote an asymptotic boundary profile.
	Call a sequence of pairs $(D_n,b_n)_{n\in \mathbb N}$
	of finite subsets of $\mathbb Z^d$ and height functions
 an \emph{approximation} of $(D,b)$
if
	\begin{enumerate}
		\item
		For all $n\in\mathbb N$,
			the function
			$b_n$ is $q$-Lipschitz
			if $E=\mathbb Z$ or strictly $q$-Lipschitz if $E=\mathbb R$,
		\item We have $\frac1n D_n\to D$ in the Hausdorff metric on $\mathbb R^d$,
		\item We have  $\frac{1}{n}\operatorname{Graph}(b_n|_{\partial D_n})\to\operatorname{Graph}(b)$ in the Hausdorff metric on $\R^d\times\mathbb R$.
	\end{enumerate}
	Moreover,
	if $E=\mathbb R$,
	then an approximation
	$(D_n,b_n)_{n\in \mathbb N}$ is called \emph{good}
	if the constant $\varepsilon>0$
	which makes each function $b_n$
	a
	$q_\varepsilon$-Lipschitz
	function,
	is independent of $n$.
	If $E=\mathbb Z$,
	then any approximation is called
	\emph{good}.
\end{definition}

We have in mind a good approximation $(D_n,b_n)_{n\in\mathbb N}$
of some fixed good asymptotic boundary profile $(D,b)$.
The sequence of local Gibbs measures
which are of interest in the LDP is the sequence
$(\gamma_n)_{n\in\mathbb N}$ defined
by $\gamma_n:=\gamma_{D_n}(\cdot,b_n)$.
All samples from the sequence of measures
$(\gamma_n)_{n\in\mathbb N}$
must be brought to the same topological space,
in order for us to formulate the LDP.
	 We will now describe this topology,
	 as well as the map
	 from $\Omega$ to
	 this topological space.

\begin{definition}[topology for macroscopic profiles]
	\label{definition_main_ldp_topology_macroscopic_profiles}
	For any $U\subset\mathbb R^d$,
 write $\operatorname{Lip}(U)$
 for the set of real-valued $K\|\cdot\|_1$-Lipschitz functions on $U$,
 where we recall that $K$ is minimal subject to $Kd_1\geq q$.
 Suppose given
 a sample
 $\phi$ from $\gamma_n$.
 Define the
 scaled interpolation $\mathfrak G_n(\phi)\in\operatorname{Lip}(\bar D)$
 of $\phi$, which captures the global shape of $\phi$, as follows.
 The sample $\phi$ is almost surely
 $q$-Lipschitz, and therefore
 also $Kd_1$-Lipschitz.
 First, write
$\bar\phi:\mathbb R^d\to\mathbb R$
for the smallest
$K\|\cdot\|_1$-Lipschitz extension
of $\phi$ to $\mathbb R^d$.
Next, we simply scale back each sample by $n$ and
restrict it to the set $\bar D$.
Formally, this means that we define
\[
	 \mathfrak G_n(\phi):\bar D\to\mathbb R,\,x\mapsto \frac{1}{n}\bar\phi(nx).
\]
This function is $K\|\cdot\|_1$-Lipschitz,
that is,
$\mathfrak G_n(\phi)\in\operatorname{Lip}(\bar D)$.
Endow the space $\operatorname{Lip}(\bar D)$ with the topology of uniform
convergence, denoted by $\mathcal X^\infty$.
The map $\mathfrak G_n:\Omega\to\operatorname{Lip}(\bar D)$
captures the global profile of the height functions in the large
deviations principle.
\end{definition}

\begin{definition}[rate function, pressure]
	The \emph{rate function} associated to the profile $(D,b)$
	is the function $I:\operatorname{Lip}(\bar D)\to[0,\infty]$
	defined by
	\[
		I(f):=-P_{\Phi}(D,b)+\int_D\sigma(\nabla f(x))dx
	\]
	if $f|_{\partial D}=b$ and $I(f):=\infty$
	otherwise.
	Here $P_{\Phi}(D,b)$ is
	the \emph{pressure} associated to this profile,
	which is defined precisely such that the minimum of $I$
	is zero.
\end{definition}

\begin{theorem}[large deviations principle]
	\label{thm:sec_mr_ldp}
	Let $\Phi\in\mathcal S_\mathcal L+\mathcal W_\mathcal L$,
	 and let $(D_n, b_n)_{n\in \mathbb N}$ denote a good approximation of some good
	 asymptotic profile $(D,b)$.
	Let $\gamma_n^*$ denote the pushforward
 of $\gamma_n:=\gamma_{D_n}(\cdot,b_n)$ along the map $\mathfrak G_n$,
 for any $n\in\mathbb N$.
	Then the sequence of probability measures
	$(\gamma_n^*)_{n\in\mathbb N}$
	satisfies a large deviations principle with speed $n^d$ and rate function
	$I$
	 on the topological space $(\operatorname{Lip}(\bar D),\mathcal X^\infty)$.
Moreover, the sequence of normalizing constants $(Z_n)_{n\in\mathbb N}:=(Z_{D_n}(b_n))_{n \in\N}$  satisfies
$
-
n^{-d} \log Z_n \to P_{\Phi}(D,g)
$
as $n\to\infty$.
\end{theorem}

\begin{corollary}[variational principle]
	\label{cor_vp}
	Let $\Phi\in\mathcal S_\mathcal L+\mathcal W_\mathcal L$,
	 and let $(D_n, b_n)_{n\in \mathbb N}$ denote a good approximation of some good
	 asymptotic profile $(D,b)$.
	Let $\gamma_n^*$ denote the pushforward
 of $\gamma_n:=\gamma_{D_n}(\cdot,b_n)$ along the map $\mathfrak G_n$,
 for any $n\in\mathbb N$.
 Write $f_n$ for the random function in $\gamma_n^*$,
 which---as a random object---takes values in $\operatorname{Lip}(\bar D)$.
 If $\sigma$ is strictly convex on $U_\Phi$,
 then the random function $f_n$
 converges to the unique minimizer $f^*$ of the rate function $I$,
 in probability in the topology
 of uniform convergence
 as $n\to\infty$.
 In other words, $f^*$ is the unique minimizer of the integral
 \[
		\int_D\sigma(\nabla f(x))dx
 \]
 over all Lipschitz functions $f:\bar D\to\mathbb R$ which equal $b$ on the boundary of
 $D$.
 If however $\sigma$ fails to be strictly convex
 on $U_\Phi$,
 then for any neighborhood $A$ of the set of
 minimizers of the integral
 in the topology of uniform convergence,
 we have $f_n\in A$
 with high probability as $n\to\infty$.
\end{corollary}


\subsection{The surface tension}

Let us now state the motivating result on the surface tension.

\begin{theorem}[strict convexity of the surface tension]
  \label{thm_main_main}
Let $\Phi$ denote a potential in $\mathcal S_\mathcal L+\mathcal W_\mathcal L$
which is monotone.
\begin{enumerate}
  \item If $E=\mathbb R$,
  then $\sigma$ is strictly convex on $U_\Phi$,
  \item If $E=\mathbb Z$,
  then $\sigma$ is strictly convex on $U_\Phi$
  if for any affine map $h:(\mathbb R^d)^*\to\mathbb R$
  with $h\leq \sigma$,
  the set $\{h=\sigma\}\cap\partial U_\Phi$
  is convex.
  In particular, $\sigma$ is strictly convex on $U_\Phi$
  if at least one of the following conditions is satisfied:
  \begin{enumerate}
    \item $\sigma$ is affine on $\partial U_\Phi$,
    but not on $\bar U_\Phi$,
    \item $\sigma$ is not affine on $[u_1,u_2 ]$ for any distinct $u_1,u_2 \in \partial U_\Phi$ such that  $[u_1,u_2 ] \not\subset \partial U_\Phi$.
  \end{enumerate}
\end{enumerate}
\end{theorem}

Strict convexity of the surface tension is important because
of Theorem~\ref{thm_existence_ergodic_minimizers},
Theorem~\ref{thm:sec_mr_ldp},
and Corollary~\ref{cor_vp}.
Let us also mention some other properties of the surface tension
which are useful to keep in mind.

\begin{theorem}[general properties of the surface tension]
  \label{thm_main_st_general}
  If $\Phi\in\mathcal S_\mathcal L+\mathcal W_\mathcal L$,
  then
  \begin{enumerate}
    \item We have $U_\Phi=U_q$,
    \item If $E=\mathbb R$,
    then $\sigma(u)$ tends to $\infty$ as $u$
    approaches the boundary of $U_\Phi$,
    \item If $E=\mathbb Z$,
    then $\sigma$ is bounded and continuous on the closure of $U_\Phi$.
  \end{enumerate}
\end{theorem}

Theorem~\ref{thm_main_main}
is proven in Section~\ref{sec_strict_convex},
and Theorem~\ref{thm_main_st_general} is proven in Section~\ref{section:specific_free_energy}.


\subsection{Note on the Lipschitz setting}
\label{subsec_main_Lip_note}
Local Lipschitz constraints
are designed to be as flexible as possible.
Essential in the argument is that a height function $\phi:\mathbb Z^d\to E$
has finite energy \emph{if and only if} it is Lipschitz with respect to the local Lipschitz
constraint.
This means that we can rely on the Kirszbraun theorem (Theorem~\ref{propo_Kirszbraun})
to join together Lipschitz functions defined
on disjoint parts of the space.
However, this formulation is sometimes inconvenient.
There are, as we shall see, several natural models in which
the admissible height functions
are exactly the \emph{graph homomorphisms}
from $\mathbb Z^d$ to $\mathbb Z$:
these are functions $\phi:\mathbb Z^d\to\mathbb Z$
which satisfy $\phi(0)\in2\mathbb Z$
and $|\phi(y)-\phi(x)|=1$
for each edge $\{x,y\}$
of the square lattice.
For example, the canonical height functions
corresponding to the six-vertex model
are precisely the graph homomorphisms from $\mathbb Z^2$
to $\mathbb Z$.
Since the zero transition is not allowed,
it might appear that this model does not fit the Lipschitz framework:
it is the first \emph{if} in the \emph{if and only if}
that is violated.
However, this problem is only cosmetic in nature:
by a simple transformation one can move from graph homomorphisms to
the Lipschitz framework.
Write $h:\mathbb Z^d\to\mathbb Z$ for the function
$h(x):=\sum_i x_i$,
and consider the map
\[
  \phi\mapsto(\phi+h)/2.
\]
This map is a bijection from the set
of graph homomorphisms
to the set of functions which are
$q$-Lipschitz
for $q$ defined by
\[
    q(x,y):=\sum\nolimits_i 0\vee (y-x)_i.
\]
By applying this transformation, it is thus clear
that models of graph homomorphisms \emph{do}
fit into the local Lipschitz setting of this article.
In fact, the exact same trick applies to dimer models,
and perhaps other models of discrete height functions.


\subsection{Application to submodular potentials}

A potential $\Phi$ is said to be \emph{submodular} if for every $\Lambda\subset\subset\Z^d$,
$\Phi_{\Lambda}$ has the property that
\[
\Phi_{\Lambda} (\phi\wedge\psi) + \Phi_{\Lambda}(\phi\vee \psi)
\leq
\Phi_{\Lambda}(\phi) + \Phi_{\Lambda}(\psi).
\]
Sheffield proposes this family of potentials as a natural generalization of simply attractive potentials, and asks if similar results as the ones proved for simply attractive potentials in~\cite{S05} could be proved for finite-range submodular potentials.
We provide an answer to this question for the case that the model is also Lipschitz.
(In fact, we do not even require the potential to be finite-range.)
It is easy to see that submodular potentials
generate monotone specifications.
If $E=\mathbb R$ and $\Phi$ a submodular Lipschitz
potential fitting the framework of this article (which is a very mild requirement),
then we derive immediately from Theorem~\ref{thm_main_main} that the surface tension is strictly convex.
If $E=\mathbb Z$, then we must also fulfill the extra condition in Theorem~\ref{thm_main_main}.
We show that we can fulfill the extra condition if all shift-invariant measures
$\mu$ which are supported on $q$-Lipschitz functions
and which have $S(\mu)\in\partial U_\Phi$,
are \emph{frozen}, in the sense that for any $\Lambda\subset\subset\mathbb Z^d$,
the values of $\phi_\Lambda$ depend deterministically on $\phi_{\partial^R\Lambda}$ in $\mu$.
This is a property of the local Lipschitz constraint $q$,
and such local Lipschitz constraints are called \emph{freezing}.

\begin{theorem}[strict convexity for submodular potentials]
  \label{thm:main_results_submodular}
  Suppose that the potential $\Phi\in\mathcal S_\mathcal L+\mathcal W_\mathcal L$
  is submodular.
  Then it is monotone. Moreover,
  \begin{enumerate}
    \item If $E=\mathbb R$,
    then $\sigma$ is strictly convex on $U_\Phi$,
    \item If $E=\mathbb Z$,
    then $\sigma$ is strictly convex on $U_\Phi$
    if the local Lipschitz constraint $q$ is freezing.
  \end{enumerate}
  Note that $q$ is automatically freezing if it is $\mathbb Z^d$-invariant.
\end{theorem}

Of course, Theorem~\ref{thm_existence_ergodic_minimizers}
applies,
and if the potential is finite-range, then the specification is
quasilocal ($\Omega=\Omega_\gamma$) so that all minimizers are Gibbs measures
(Theorem~\ref{thm_mr_minimizers}).

\subsection{Application to tree-valued graph homomorphisms}

The flexibility of the main theorem in this article can also be used to prove statements about the behavior of random functions taking values in target spaces other than $\Z$ and $\R$.
A noteworthy example is the model of tree-valued graph homomorphisms described in~\cite{MT16}.
In this context, tree-valued graph homomorphisms are functions from $\Z^d$ to a
$k$-regular tree $\mathcal T_k$ which also map the edges of the square lattice to
the edges of the tree.
Regular trees are natural objects in
several fields of mathematics:
in group theory, for example, they arise as Cayley graphs
of free groups on finitely many generators.
As a significant result in~\cite{MT16}, the authors characterize
the surface tension for the model (there named \emph{entropy}) and show that
it is equivalent to the number of graph homomorphisms with nearly-linear boundary
conditions.
This entropy function describes the macroscopic behavior of the model,
as is extensively discussed in~\cite{MT16}.
We confirm the conjecture in~\cite{MT16}, which asserts that
this entropy function is strictly convex.
We can do so because the model of uniformly random $\mathcal T_k$-valued
graph homomorphisms can be translated into a model of $\mathbb Z$-valued
graph homomorphisms after introducing an infinite-range interaction.

Let us now rigorously describe the conjecture which we prove is correct.
Write $U$ for the set of slopes $u\in(\mathbb R^d)^*$
such that $|u(e_i)|<1$ for each element $e_i$
in the natural basis of $\mathbb R^d$.
For fixed $u\in\bar U$,
write $\phi^u:\mathbb Z^d\to\mathbb Z$ for the graph homomorphism
defined by
\[
	\phi^u(x):=
	\lfloor u(x)\rfloor+
	\begin{cases}
		0&\text{if $d_1(0,x)\equiv \lfloor u(x)\rfloor \mod 2$,}\\
		1&\text{if $d_1(0,x)\equiv \lfloor u(x)\rfloor + 1 \mod 2$.}
	\end{cases}
\]
Then $\phi^u$ approximates $u$ and it thus nearly linear,
in the sense that $\|\phi^u-u|_{\mathbb Z^d}\|_\infty\leq 1$.
Let $g$ denote a \emph{bi-infinite geodesic} through $\mathcal T_k$,
that is, a $\mathbb Z$-indexed sequence of vertices $g=(g_n)_{n\in\mathbb Z}\subset\mathcal T_k$
such that $d_{\mathcal T_k}(g_n,g_m)=|m-n|$ for any $n,m\in\mathbb Z$.
The geodesic $g$ is thought of as a copy of $\mathbb Z$
in $\mathcal T_k$, and is used as reference frame.
Write $\tilde\phi^u:\mathbb Z^d\to\mathcal T_k$
for the graph homomorphism defined by
$\tilde\phi^u(x):=g_{\phi^u(x)}$ for every $x\in\mathbb Z^d$.
It is shown in~\cite{MT16} that the macroscopic behavior of uniformly random
$\mathcal T_k$-valued graph homomorphisms is characterized by the function
\[
  \operatorname{Ent}:\bar U\to [-\log k,0],\,
  u\mapsto\lim_{n \to \infty}-n^{-d}\log |\{\tilde\phi\in\tilde\Omega:\tilde\phi_{\mathbb Z^d\smallsetminus\Pi_n}=\tilde\phi_{\mathbb Z^d\smallsetminus\Pi_n}^u\}|,
\]
where $\tilde\Omega$ denotes the set of all graph homomorphisms from $\mathbb Z^d$
to $\mathcal T_k$.
It is conjectured in~\cite{MT16} that $\operatorname{Ent}$ is strictly convex on $U$,
which we prove is correct.
Figure~\ref{fig:tree_valued_large} displays
a sample from the model; the limit shape is clearly visible.

\begin{theorem}[strict convexity of the entropy for tree-valued graph homomorphisms]
  \label{thm:main_tree_valued_statement}
  For any $d,k\geq 2$, the entropy function $\operatorname{Ent}:\bar U\to[-\log k,0]$
  associated to uniformly random
  graph homomorphisms from $\mathbb Z^d$
  to a $k$-regular tree,
  is strictly convex on $U$.
\end{theorem}


\begin{figure}
	\centering
  \includegraphics[width=\textwidth]{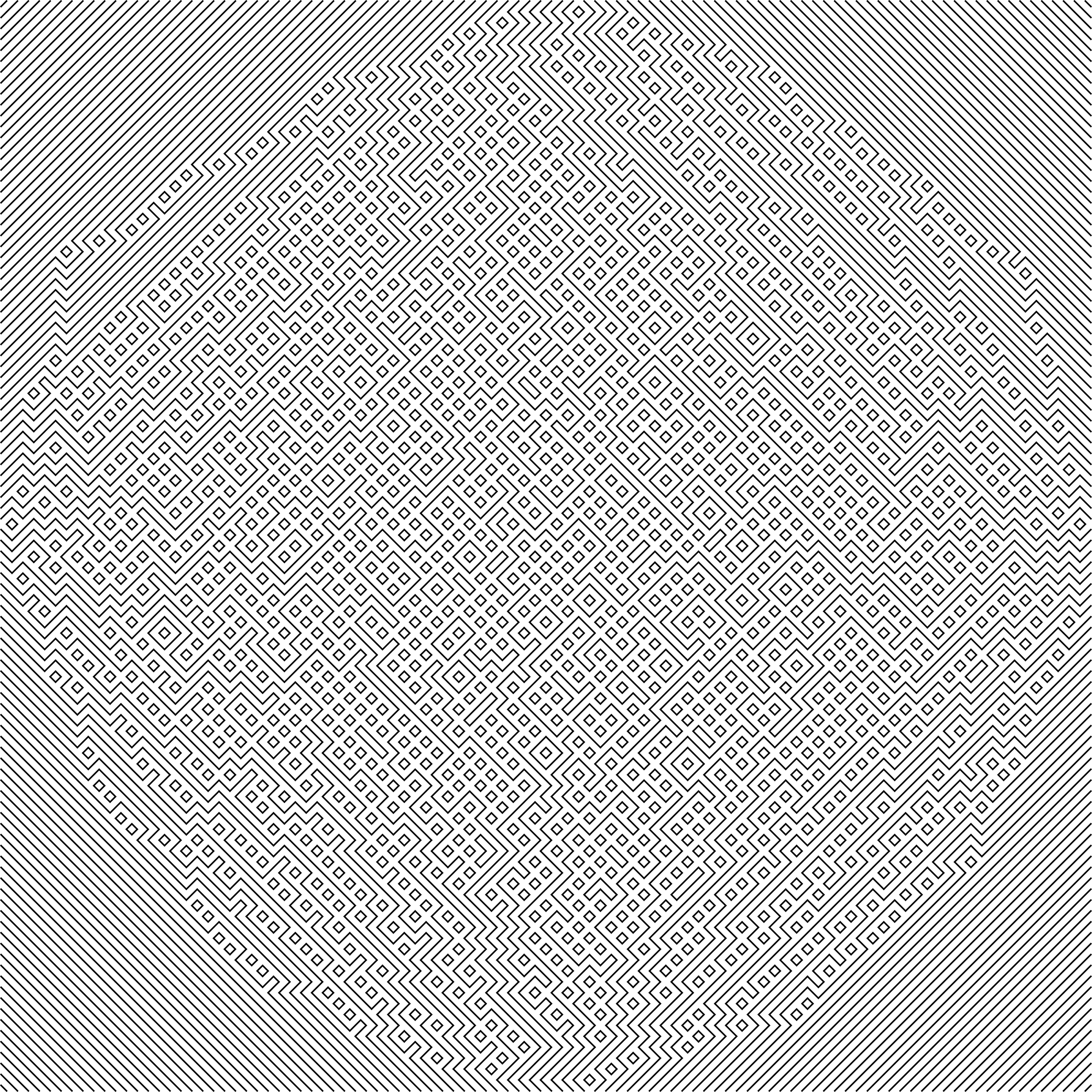}
	\caption[The upper level sets of a random $\mathcal T_3$-valued graph homomorphism]{This figure
	shows the boundaries
	of the upper level sets of the horocylic height function
	(presented in Subsection~\ref{subsection:applications_tree})
	of a random $\mathcal T_3$-valued graph homomorphism.
	The boundary conditions resemble
	the Aztec diamond for domino tilings.
	The simulation hints at the presence of an arctic circle,
	alongside the limit shape which we prove appears inside.
	}
	\label{fig:tree_valued_large}
\end{figure}


\section{Moats}
\label{section:moats_heart}
The following section is at the heart of this work. Its purpose is to show that for a specification which is stochastically monotone, two configurations sampled independently with the same boundary conditions are, on the scale of the specific free energy, at least as likely to oscillate a large number of times than to deviate from each other macroscopically.
\emph{Moats} are introduced in Definition~\ref{def:moats}
to formalize this statement.
Informally, moats are clusters surrounding a given connected set, and on which the height difference between two configurations is prescribed between two fixed bounds.
The proof relies crucially on the reflection principle which is
stated in Lemma~\ref{lemma:reflection}.

In this section, the implicit graph structure on $\mathbb Z^d$ is always the square lattice.
As per usual,
$(\mathbb A,q)$
denotes the local Lipschitz constraint,
and $K\in(0,\infty)$
is chosen minimal subject to $Kd_1\geq q$.
We have in mind a gradient specification $\gamma$ which is $q$-Lipschitz and monotone over $\Omega_q$.
From this specification we draw two height
functions $\phi_1,\phi_2\in\Omega$,
and $f$ shall generally denote the difference function $\phi_1-\phi_2$,
which is thus $2K$-Lipschitz.

\subsection{Reflection principle}

We first state and prove the reflection principle, which does not rely on
the Lipschitz property.
Throughout this section only, we shall adopt the following notation.
Suppose that $f_1$ and $f_2$ are random functions in $\Omega$,
in some probability measures $\mu_1$ and $\mu_2$
respectively.
Then write $f_1\preceq f_2$
if $f_1$ is \emph{stochastically dominated}
by $f_2$,
that is, $\mu_1(f_1\in A)\leq \mu_2(f_2\in A)$
for any increasing set $A\in\mathcal F$.
Note that this notation still makes sense if $\mu_1=\mu_2$,
even if $\mu_1$ and $\mu_2$ are finite measures rather than probability measures.

\begin{lemma}[Reflection principle]\label{lemma:reflection}
	Let $\gamma=(\gamma_{\Lambda})_{\Lambda\subset\subset\mathbb Z^d}$ denote a monotone gradient specification.
	Fix $\Lambda\subset\subset\mathbb Z^d$,
	and consider a probability measure $\mu$
	on the product space $(\Omega^2,\mathcal F^2)$,
	 writing $(\phi_1,\phi_2)$
	for the random pair of height functions,
		and with $f:=\phi_1-\phi_2$.
		Suppose that
		\[
			\mu=\mu(\gamma_\Lambda\times\gamma_\Lambda).
		\]
		If $\mu$-almost surely
		$f_{\mathbb Z^d\smallsetminus \Lambda}\geq a$
		for some $a\in\mathbb R$, then
		\[
			-f\preceq f-2a.
		\]
		Similarly, if
		$\mu$-almost surely
		$f_{\mathbb Z^d\smallsetminus \Lambda}\leq b$ for some $b\in\mathbb R$, then
		\[
			f\preceq -f+2b.
		\]
		The same holds true if $\mu$ is a finite measure
		rather than a probability measure.
\end{lemma}

\begin{proof}
	We focus on the first statement; the second statement then
	follows by symmetry.
	Fix $a\in\mathbb R$.
	Suppose first that $\mu$ restricted to $\mathbb Z^d\smallsetminus\Lambda$
	is a Dirac measure,
	that is,
	\[
		\mu=\gamma_\Lambda(\cdot,\psi_1)\times\gamma_\Lambda(\cdot,\psi_2)
	\]
	for some $\psi_1,\psi_2\in\Omega$
	with $\psi_1-\psi_2\geq a$.
	As $\gamma$ is a monotone gradient specification,
	we have
	\[
		\phi_1\succeq \phi_2+a.
	\]
	But $\phi_1$ and $\phi_2$ are independent, and therefore
	\[
		-f=\phi_2-\phi_1\preceq (\phi_1-a)-(\phi_2+a)=f-2a.
	\]
	This inequality is generalized
	to the case that $\mu$ restricted to $\mathbb Z^d\smallsetminus\Lambda$
	is not a Dirac measure, simply by averaging the inequality over all possible
	 values
	of $\phi_1$ and $\phi_2$ on $\mathbb Z^d\smallsetminus\Lambda$
	with respect to $\mu$.
\end{proof}

\subsection{Definition of moats}

\begin{definition}[Moats]\label{def:moats}
	Let $f:\mathbb Z^d\to \mathbb R$ be a $2K$-Lipschitz function and $\Lambda\subset\subset\mathbb Z^d$ connected.
	Consider two real numbers $a$ and $b$ with $b-a\geq 4K$.
 	\begin{enumerate}
		\item A set $M\subset \mathbb Z^d$ is a called an \emph{$a,b$-moat of $(f,\Lambda)$}
		or simply a \emph{moat}
		if
		$M$ is a finite connected component of the set
		$\{a\leq f< b\}=\{x\in\mathbb Z^d:a\leq f(x)< b\}\subset \mathbb Z^d$
		such that $\Lambda$ is contained in a bounded connected component
		of $\mathbb Z^d\smallsetminus M$.
		\item The boundary of $M$, that is, the set of vertices $x\in\mathbb Z^d\smallsetminus M$
		adjacent to $M$, is denoted by $\partial M$.
		Write $\bar M$ for the closure of $M$,
		that is, $M\cup \partial M$.
		\item The connected component of $\mathbb Z^d\smallsetminus M$ containing $\Lambda$ is called
		the \emph{inside} of $M$,
		and the \emph{inside boundary} is the intersection of the inside with $\partial M$.
		Write $M^\Lambda$ and $\partial_\Lambda M$ for the inside and the inside boundary respectively.
		\item The unbounded connected component of $\mathbb Z^d\smallsetminus M$ is called
		the \emph{outside} of $M$,
		and the \emph{outside boundary} is the intersection of the outside with $\partial M$.
		Write $M^\infty$ and $\partial_\infty M$ for the outside and the outside boundary respectively.
		\item A moat $M$ is said to \emph{surround}
		another moat $N$, if $N\subset M^\Lambda$.
		\item A moat $M$ is called a \emph{climbing moat} if $f_{\partial_\infty M}<a$
		and $f_{\partial_\Lambda M}\geq b$
		and it is called a \emph{descending moat} if $f_{\partial_\infty M}\geq b$
		and $f_{\partial_\Lambda M}<a$.
		From now on, we shall only consider moats
		which are either climbing or descending;
		when speaking of a moat, it is implicit that it
		belongs to one of these categories.
		\item A finite sequence of moats $(M_k)_{1\leq k\leq n}$
		is called \emph{nested} if
		$M_k$ surrounds $ M_{k+1}$ for all $1\leq k<n$,
		and if the moats are alternatingly climbing and descending, with $M_1$ climbing.
	\end{enumerate}
\end{definition}

We immediately collect a number
of important properties.

\begin{proposition}
	\label{propo:collection}
	Work in the context of the previous definition.
	\begin{enumerate}
		\item There exists at most one moat $M$
		with $x\in M$,
		for any fixed $x\in\mathbb Z^d$.
		\item If $M$ is a moat, then
		$a-2K\leq f<  b+2K$ on $\bar M$.
		\item
		\label{propo:collection:in_out_dist}
	 Suppose that $M$ is a moat,
	 and that $p=(p_k)_{0\leq k\leq n}\subset \mathbb Z^d$
	 is a path through the square lattice
	 from $M^\Lambda$ to $M^\infty$.
	 Then $p_k\in M$
	 for at least $\lfloor(b-a)/2K\rfloor\geq 2$
	 consecutive integers $k$.
	 \item If $\Delta\subset\subset\mathbb Z^d$ contains $\Lambda$,
	 then the number of moats $M$ of $(f,\Lambda)$
	 for which $M\cup M^\Lambda\subset\Delta$, is bounded by
	 the $d_1$-distance from $\Lambda$ to $\mathbb Z^d\smallsetminus\Delta$.

		\item
		Suppose that $M$ is a moat of $(f,\Lambda)$,
		and that $g$ is another $2K$-Lipschitz function with
		$g=f$ on $\bar M$.
		Then $M$ is also a moat of $(g,\Lambda)$.
		If $M$ was climbing (resp.\ descending) w.r.t.\ $(f,\Lambda)$
		then it is climbing (resp.\ descending) w.r.t.\ $(g,\Lambda)$.
		In other words, for $M\subset\mathbb Z^d$, the event
		\[
		\{\text{$M$ is a (climbing or descending) moat of $(f,\Lambda)$}\}
		\] is $\mathcal F_{\bar M}^2$-measurable.
		\item
		Suppose that $A\subset \mathbb Z^{d}$
		such that $\Lambda$ is contained in a finite connected component of $\mathbb Z^d\smallsetminus A$,
		and write $A^\Lambda$ for this connected component.
		If $f<a$ on $A$ and $f\geq b$ on $\Lambda$, then $A^\Lambda$ contains a climbing moat.
		If $f\geq b$ on $A$ and $f<a$ on $\Lambda$, then $A^\Lambda$ contains a descending moat.

		\item
		\label{moatcontainsmoat}
		If $a'$ and $b'$ are real numbers
		with $b'-a'\geq 4K$ and $[a',b']\subset [a,b]$,
		then any $a,b$-moat contains an $a',b'$-moat.
	\end{enumerate}
\end{proposition}

\begin{proof}[Proof]
	The first three statements
	follow from the definitions,
	where it is important that $f$ is $2K$-Lipschitz
	and that any moat is either climbing or descending.
	For the fourth statement, observe that a path of minimal length
	from $\Lambda$ to $\mathbb Z^d\smallsetminus\Delta$ through
	the square lattice must intersect any moat $M$ for which $M\cup M^\Lambda\subset\Delta$.
	The fifth statement is immediate from the definition.
	The sixth statement follows from the
	connectivity properties of the square lattice,
	as well as the fact that $f$ is $2K$-Lipschitz.
	The final statement is a corollary of the sixth.
\end{proof}

\subsection{Moats and macroscopic deviations}

\begin{theorem}\label{thm:moats}
	Let $\gamma=(\gamma_{\Lambda})_{\Lambda\subset\subset\mathbb Z^d}$ denote a $q$-Lipschitz gradient specification
	which is monotone over $\Omega_q$.
	Fix $\Delta\subset\subset\mathbb Z^d$,
	and consider a $q$-Lipschitz probability measure $\mu$
	on the product space $(\Omega^2,\mathcal F^2)$,
	 writing $(\phi_1,\phi_2)$
	for the random pair of height functions,
		and with $f:=\phi_1-\phi_2$.
		Suppose that
		\[
		\text{$
			\mu=\mu(\gamma_\Delta\times\gamma_\Delta)$
			\qquad and\qquad
			$\mu$-almost surely $|f_{\mathbb Z^d\smallsetminus\Delta}|\leq 2K$.}
		\]
		Fix a connected set $\Lambda\subset\Delta$,
		and write $E(n)$ for the event that there exists a sequence of $n$ nested $a,b$-moats
 	of $(f,\Lambda)$,
	where $a=4K$ and $b\geq 8K$.
	Then
	\begin{equation}
	\label{eq:thmgoal}
	m^{2n}\mu(E(2n))\geq \mu(f_\Lambda\geq 3bn)
	\end{equation}
	for all $n\in\mathbb N $,
	where $m=d_1(\Lambda,\mathbb Z^d\smallsetminus\Delta)$.
\end{theorem}

The idea of the proof is as follows.
If $f\geq 3bn$ on $\Lambda$ and $f\leq 2K$ on $\mathbb Z^d\smallsetminus\Delta$, then
$\Delta$ must contain a climbing $a,b$-moat.
Suppose now that we fix a subset $M$ of $\Delta$,
and condition on the event
\[
	A:=\{\text{$M$ is a climbing $a,b$-moat of $(f,\Lambda)$}\}\in\mathcal F_{\bar M}^2.
\]
If we write $\Gamma$ for the set $\Delta\smallsetminus\bar M$,
then the conditioned measure $\mu(\cdot|A)$ satisfies
\[
\text{$
	\mu(\cdot|A)=\mu(\cdot|A)(\gamma_\Gamma\times\gamma_\Gamma)$
	\qquad and\qquad
	$\mu(\cdot|A)$-almost surely $-2K\leq f_{\mathbb Z^d\smallsetminus\Gamma}\leq b+2K$.}
\]
By the reflection principle,
we thus have
\[
	\mu(f_\Lambda\leq -3bn+2b+4K|A)\geq \mu(f_\Lambda\geq 3bn|A).
\]
In other words, this means that it is as least as likely to observe
the set $M$ as a climbing moat and a large negative deviation
on $\Lambda$, than to see the set $M$ as a climbing moat
and a slightly larger positive deviation on $\Lambda$.
But if $f$ is negative on $\Lambda$ then we can find a descending moat
in the inside $M^\Lambda$ of $M$.
One repeats this reflection procedure
to generate a full nested sequence of moats, while retaining a sufficiently large probability.
The formalism is slightly more convoluted because one needs
to choose the set $M$ appropriately.
This produces the extra factor $m^{2n}$ in~\eqref{eq:thmgoal}.

\begin{proof}[Proof of Theorem~\ref{thm:moats}]
	We proceed along the same spirit.
	Write
	$\Delta_k:=\{\Lambda\subset\Delta\}^k$,
	and define
	\[
	A(M)
	:=
	\{\text{$M$ is a nested sequence of $a,b$-moats of $(f,\Lambda)$}\}\in\mathcal F_{\cup_i\bar M_i}^2
	\]
	for $M\in\Delta_k$.
	We also write $\Gamma(M):=\Delta\smallsetminus\cup_i\bar M_i$.
	Define
	$\mu_B:=\mu(\cdot\cap B)$
	for any $B\in\mathcal F^2$.

	For any $k\in\mathbb N$ and $M\in\Delta_k$,
	we have
	\[
		\text{$
			\mu_{A(M)}=\mu_{A(M)}(\gamma_{\Gamma(M)}\times\gamma_{\Gamma(M)})$
			\qquad and\qquad
			$\mu_{A(M)}$-a.e.\ $-2K\leq f_{\mathbb Z^d\smallsetminus\Gamma(M)}\leq b+2K$,}
	\]
	which means that the reflection principle applies
	to this measure.
	Claim that
	\begin{align}
		\label{al:moats:1}\mu(f_\Lambda\geq 3bn) &\leq \sum\nolimits_{M\in\Delta_1}\mu_{A(M)}(f_\Lambda\geq 3bn)
		\\
		&\label{al:moats:2}\leq
		\sum\nolimits_{M\in\Delta_1}\mu_{A(M)}(f_\Lambda\leq -3bn+2b+4K)
		\\
		&\label{al:moats:3}\leq
		\sum\nolimits_{M\in\Delta_2}\mu_{A(M)}(f_\Lambda\leq -3bn+2b+4K)
		\\
		&\label{al:moats:4}\leq
		\sum\nolimits_{M\in\Delta_2}\mu_{A(M)}(f_\Lambda\geq 3bn-2b-8K)
		\\
		&\label{al:moats:5}\leq
		\sum\nolimits_{M\in\Delta_2}\mu_{A(M)}(f_\Lambda\geq 3b(n-1)).
	\end{align}
	Here~\eqref{al:moats:1} follows from the fact that $\Delta$
	contains a moat whenever $f\leq 2K$ on the complement of $\Delta$
	and $f\geq 3bn$ on $\Lambda$, and
	\eqref{al:moats:2} follows from the reflection principle applied to each measure in the finite sum.
	Now isolate one set $M\in\Delta_1$ and consider the measure
	$\mu_{A(M)}$.
	If $f_\Lambda\leq -3bn+2b+4K$,
	then there must be a descending moat in the inside of $M_1$---recall
	that $M_1$ is a climbing moat, by definition of a nested sequence of moats.
	In particular, this proves~\eqref{al:moats:3}.
	Inequality~\eqref{al:moats:4} follows again from the reflection principle
	applied to each separate measure, and~\eqref{al:moats:5}
	follows from the fact that $3bn-2b-8K\geq 3b(n-1)$.
	A continuation of this series of inequalities leads to the equation
	\[
	\mu(f_\Lambda\geq 3bn)\leq \sum_{M\in\Delta_{2n}}\mu_{A(M)}(f_\Lambda\geq 0).
	\]

	The proof is nearly done.
	Note that
	$\mu_{A(M)}(\Omega^2)=\mu_{A(M)}(E(2n))$ for $M\in \Delta_{2n}$
	by definition of $A(M)$ and $E(2n)$,
	and therefore
	\[
	\mu(f_\Lambda\geq 3bn)\leq
	\sum_{M\in\Delta_{2n}}\mu_{A(M)}(f_\Lambda\geq 0)
	\leq \sum_{M\in\Delta_{2n}}\mu_{A(M)}(E(2n)).
	\]
	To deduce~\eqref{eq:thmgoal},
	it suffices to demonstrate that,
	as measures,
	\[
		\sum_{M\in\Delta_{2n}}\mu_{A(M)}\leq m^{2n}\mu.
	\]
	The measure on the left equals $X\mu$,
	where $X$ is the number of ways to choose
	a nested sequence of $2n$ moats
	contained in $\Delta$.
		Since $\Delta$ contains at most $m$ moats, we have $X\leq {m\choose 2n}\leq m^{2n}$.
\end{proof}

We state an immediate corollary, which is an adaptation of the previous
result to the case that $\Delta$ and $\Lambda$ are not connected.

\begin{proposition}
	\label{proposition_moats_effective_application}
	Assume the setting of the previous theorem,
	only suppose now that $\Delta$ and $\Lambda$ each decompose into
	$k$ connected components denoted by $(\Delta_i)_i$ and $(\Lambda_i)_i$
	respectively with $\Lambda_i\subset \Delta_i$,
	and write $E(n)$ for the event that each $\Delta_i$ contains
	a sequence of $n$ nested $a,b$-moats of $(f,\Lambda_i)$.
	Then~\eqref{eq:thmgoal} holds true once we replace $m$ by
	\[
		m=\prod_{i=1}^k d_1(\Lambda_i,\mathbb Z^d\smallsetminus\Delta_i).
	\]
%
\end{proposition}


\section{Analysis of local Lipschitz constraints}

This section
contains several results
on local Lipschitz constraints---most are deduced directly from
Definition~\ref{def_llc}.
Fix, throughout this section,
a local Lipschitz constraint $(\mathbb A,q)$,
and let $R\in\mathbb N$
denote a fixed constant
such that $d_1(x,y)\leq R$
for all $\{x,y\}\in\mathbb A$.
For example, one can take $(\mathbb A,q)$ to be the local Lipschitz constraint of $\Psi$, and $R$ its range.
These results are near-trivial for most commonly studied models;
they require some work in the generality of Definition~\ref{def_llc}.

Throughout this section,
we adopt the following notation.
If $p=(p_k)_{0\leq k\leq n}$
is a path through $(\mathbb Z^d,\mathbb A)$,
then we write $q(p)$ for $\sum_{k=1}^nq(p_{k-1},p_k)$.
If $q(p)=q(p_0,p_n)$,
then $p$ is called an \emph{optimal
path}.

\subsection{Homogenization of local Lipschitz constraints}

The following lemma characterizes $U_q$
in terms of $q$.
It also provides a relation between the local Lipschitz constraint
$q$ and the map $\|\cdot\|_q$ that it generates.
The proof is similar to the proof in~\cite{S05},
although the formulation of the lemma is different.

\begin{lemma}
 \label{lemma_lipschitz_q_and_norm}
 The set $U_q$ is nonempty.
 Its closure $\bar U_q$ can be written as the intersection of finitely many half-spaces.
 For each contributing half-space $H$, there exists a path
 $p=(p_k)_{0\leq k\leq n}$ through $(\mathbb Z^d,\mathbb A)$
 with $p_n-p_0\in\mathcal L$
 such that
 \[
  \textstyle
  H=H(p):=\left\{u\in(\mathbb R^d)^*:u(p_n-p_0)\leq q(p)\right\}.
 \]
 Moreover, there exists a constant $C<\infty$
 such that
 \[
  \|y-x\|_q-C
  \leq
  q(x,y)
  \leq
  \|y-x\|_q+C
 \]
 for any $x,y\in\mathbb Z^d$.
\end{lemma}

\begin{proof}
 Call some path $p=(p_k)_{0\leq k\leq n}$
 through $(\mathbb Z^d,\mathbb A)$
 a \emph{cycle lift} if
 the projection of $p$ onto $\mathbb Z^d/\mathcal L$ is
 a cycle.
 Since $\mathbb Z^d/\mathcal L$ is finite
 and $(\mathbb Z^d,\mathbb A)$ of bounded degree,
 there exist only finitely many cycle lifts
 once we identify paths which differ by a shift
 by a vector in $\mathcal L$.

 Claim that
 \[
 \numberthis
 \label{eq:U_q_bar}
  \{u\in(\mathbb R^d)^*:\text{$u|_\mathcal L$ is $q$-Lipschitz}\}
  =
  \cap_{p:\:\text{$p$ is a cycle lift}}
  H(p).
 \]
 It is clear that the left set is contained in the right set.
 Focus now on the other containment.
 Fix a slope $u$ in the set on the right.
 Suppose, for the sake of contradiction, that $u$ is not in the set on the left,
 that is, that $u|_\mathcal L$ is not $q$-Lipschitz.
 Then there is some vertex $x\in\mathcal L$ and a path $p$
 from $0$ to $x$ through $(\mathbb Z^d,\mathbb A)$ such that
 $u(x)>q(0,x)=q(p)$.
 But $p$ decomposes into a finite collection of cycle lifts $(p^i)_i$.
 By choice of $u$, we have $u(x)\leq\sum_iq(p^i)=q(p)$,
 a contradiction.
 This proves the claim.

 The set $U_q$ equals the interior of the left
 and right in~\eqref{eq:U_q_bar}.
 Suppose that $U_q$ is empty.
 Select a minimal family of cycle lifts $(p^k)_{1\leq k\leq m}$
 such that the corresponding
 intersection of interiors of half-spaces $\cap_k \mathring H(p^k)$
 is empty---by minimal we simply mean that $m$ is as small as possible.
 For each $1\leq k\leq m$,
 write $x^k\in\mathcal L$ for the endpoint of $p^k$
 minus $p^k_0$.
 Then each vector $x^k$ is orthogonal
 to the affine hyperplane $\partial H(p^k)$.
 By Helly's theorem, we observe that $m\leq d+1$.
 In fact, it is easy to see that, regardless of the value of $m$, the set
 $\{x^k:1\leq k\leq m\}$
 is linearly dependent,
 with any strict subset linearly independent.
 It is a simple exercise in linear algebra to derive from the fact that the intersection of half-spaces
 $\cap_k \mathring H(p^k)$
 is empty,
 that there is some slope
 $u$ which is contained in the complement of $\mathring H(p^k)$
 for any $k$,
 and that
 there exists a family of positive integers $(a^k)_{1\leq k\leq m}\subset \mathbb N$ such that
 $\sum_k a^kx^k=0$.
 Since $u(x^k)\geq q(0,x^k)$
 for each $k$ by choice of $u$, we have
 \[
  \sum\nolimits_k q(0,a^kx^k)\leq \sum\nolimits_k a^kq(0,x^k)\leq \sum\nolimits_k a^ku(x^k)=\textstyle u(\sum_ka^kx^k)=0.
  \]
 However,
 the triangle inequality and the inequality
 $q(x,y)+q(y,x)>0$ for $x\neq y$
 from the definition of an admissible quasimetric
 imply that
 \begin{align*}
   \sum\nolimits_{k=1}^m q(0,a^kx^k)&=  q(0,a^1x^1)+\sum\nolimits_{k=2}^m q(0,a^kx^k)
   \\&\geq q(0,a^1x^1)+q(0,-a^1x^1)=q(0,a^1x^1)+q(a^1x^1,0)>0,
 \end{align*}
 a contradiction.
 This proves that $U_q$ is nonempty.

 Now let $x,y\in\mathbb Z^d$ arbitrary,
 and let $p$ denote an optimal path from $x$
 to $y$. Then $p$ decomposes
 into cycle lifts and at most $|\mathbb Z^d/\mathcal L|-1$
 remaining edges.
 It is straightforward to derive from this decomposition
 that the difference between $q(x,y)$ and $\|y-x\|_q$ is bounded uniformly over
 the choice of $x$ and $y$.
\end{proof}

Let us also state the following result, which follows
immediately from the definition of $\|\cdot\|_q$
in terms of $q$.

\begin{proposition}
  If $f:D\to \mathbb R$
  is $\|\cdot\|_q$-Lipschitz for $D\subset\mathbb R^d$,
  then $f|_{D\cap\mathcal L}$ is $q$-Lipschitz.
  If furthermore $q$ is integral,
  then $\lfloor f\rfloor|_{D\cap\mathcal L}$ is also $q$-Lipschitz.
\end{proposition}

\subsection{General observations}

First state the Kirszbraun theorem: this is an elementary result
in the theory of Lipschitz functions.
It asserts that a Lipschitz function
defined on part of the space
can be extended to a Lipschitz function on the entire space, with the same Lipschitz constant.

\begin{proposition}[Kirszbraun theorem]
  \label{propo_Kirszbraun}
    If $\Lambda\subset\mathbb Z^d$ is nonempty and if $\phi:\Lambda\to\mathbb R$
    is $q$-Lipschitz,
    then the function
    \[
      \phi^*:\mathbb Z^d\to\mathbb R,
      x\mapsto \sup_{y\in\Lambda}\phi(y)-q(x,y)
    \]
    is the unique smallest $q$-Lipschitz
    extension of $\phi$ to $\mathbb Z^d$.
    If $\phi$ and $q$ are integral, then so is $\phi^*$.
    Suppose that $\|\cdot\|:\mathbb R^d\to\mathbb R$ is any positive homogeneous function satisfying the triangle inequality.
    If $D\subset\mathbb R^d$ is nonempty and if $f:D\to\mathbb R$
    is $\|\cdot\|$-Lipschitz,
    then the function
    \[
      f^*:\mathbb R^d\to\mathbb R,
      x\mapsto \sup_{y\in D}f(y)-\|y-x\|
    \]
    is the unique smallest $\|\cdot\|$-Lipschitz
    extension of $f$ to $\mathbb R^d$.
\end{proposition}

Next, we discuss the derived local Lipschitz constraint $q_\varepsilon$
for $\varepsilon$ sufficiently small.
For example, if $\mathbb A$ is the edge set of the square lattice
and $q=Kd_1$ for $K\in[0,\infty)$,
then $q_\varepsilon$ is well-defined for $\varepsilon\in[0,K)$,
and $q_\varepsilon=(K-\varepsilon)d_1$ for such $\varepsilon$.
For the more general case, we use a technical construction to understand
the derived local Lipschitz constraint $q_\varepsilon$.

\begin{proposition}
  \label{propo_qqq}
  There exist constants $\eta>0$ and $C<\infty$
  such that for any $0\leq\varepsilon\leq\eta$,
  \begin{enumerate}
    \item We have
    $\varepsilon d_1/R\leq  q-q_\varepsilon\leq C\varepsilon d_1$,
    \item We have
    $q_{\varepsilon+\varepsilon'}=(q_\varepsilon)_{\varepsilon'}=(q_{\varepsilon'})_\varepsilon$
    for any $\varepsilon'\geq 0$ with $\varepsilon+\varepsilon'\leq\eta$,
    \item
    \label{propo_qqq:flex}
    For any $\varepsilon'\geq 0$ with $\varepsilon+2\varepsilon'\leq\eta$,
    if $\phi,\psi:\Lambda\to\mathbb R$ are functions for some $\Lambda\subset\mathbb Z^d$
    where $\phi$ is $q_{\varepsilon+2\varepsilon'}$-Lipschitz
    and $\|\phi-\psi\|_\infty\leq \varepsilon'$, then $\psi$ is $q_{\varepsilon}$-Lipschitz.
  \end{enumerate}
\end{proposition}

\begin{proof}[Proof outline]
  Claim that there exists a uniform constant $C<\infty$
  such that $n(p)\leq Cd_1(x,y)$ for any optimal path $p$ from $x$
  to $y$,
  where $n(p)$ denotes the length of that path.
  To see that the claim is true,
  observe that
   $U_q$ is nonempty and open,
  and therefore there exists a constant $\alpha>0$
  such that $\|x\|_q+\|-x\|_q\geq \alpha\|x\|_1$
  for any $x\in\mathbb R^d$.
  Moreover, the difference between $q(x,y)$ and $\|y-x\|_q$
  is bounded uniformly over $x,y\in\mathbb Z^d$
  (Lemma~\ref{lemma_lipschitz_q_and_norm}).
  It is straightforward to deduce the claim from these two facts.

  One now defines the map $X_{q}:\mathbb Z^d\times\mathbb Z^d
  \to\mathbb Z_{\geq 0}$
  by
  \[
    X_{q}(x,y):=\max\{n(p):\text{$p$ is an optimal path from $x$ to $y$}\}.
  \]
  Then $d_1/R\leq X_q\leq Cd_1$ by the previous discussion.
  It is straightforward, but slightly technical,
  to see that $q_\varepsilon=q-\varepsilon X_q$
  for $\varepsilon$ sufficiently small.
  This implies the three statements of the proposition.
\end{proof}

\begin{proposition}
  We have $U_q=\cup_{\varepsilon>0}\bar U_{q_\varepsilon}$.
\end{proposition}

\subsection{Approximation of continuous profiles}

Recall that $\Lambda^{-m}(D):=(\mathbb Z^d\cap D)\smallsetminus \partial^m(\mathbb Z^d\cap D)$
for any $m\in\mathbb Z_{\geq 0}$ and $D\subset\mathbb R^d$.

\begin{theorem}
  Consider $\varepsilon>0$ sufficiently small so
  that $q_\varepsilon$ is well-defined,
  and fix $C<\infty$.
  Then
  there is a constant $m\in\mathbb Z_{\geq 0}$
  such that the following statement holds true.
  Suppose given a collection $(D_i)_i$ of disjoint subsets of $\mathbb R^d$,
  and write $D:=\cup_iD_i$.
  Let $f:D\to \mathbb R$ denote a $\|\cdot\|_q$-Lipschitz
  function such that
  $f(y)-f(x)\leq \|y-x\|_{q_\varepsilon}$ for $x\in D_i$
  and $y\in D_j$ with $i\neq j$.
  Define $\Lambda_i:=\Lambda^{-m}(D_i)$
  and $\Lambda:=\cup_i\Lambda_i$.
  Let $\phi:\Lambda\to E$ denote a function such that $\phi_{\Lambda_i}$ is $q$-Lipschitz for all $i$
  and with
   $|\phi-f|_{\Lambda}|\leq C$.
   Then $\phi$ is $q$-Lipschitz,
   and has a $q$-Lipschitz extension to $\mathbb Z^d$.
\end{theorem}

\begin{proof}
  This follows from Lemma~\ref{lemma_lipschitz_q_and_norm} and Proposition~\ref{propo_qqq}.
\end{proof}

In the remainder of this section,
we specialize to the case that
 $(\mathbb A,q)$ is the local Lipschitz
constraint associated
to the strong interaction $\Psi$ as described
in Subsection~\ref{subsection:formal_setting:overview}.
The previous theorem is particularly useful
in the case that the function $f$
is affine on each set $D_i$,
say with slope $u_i\in U_q$.
In that case, we want the height function $\phi$ to approximate
the slope
$u_i$ on each set $\Lambda_i$.
To this end we will choose
for each $u\in U_q$
a canonical Lipschitz height function $\phi^u$
to represent that slope $u$.
This is the purpose of the following definition.

\begin{definition}
  Consider some fixed slope $u\in U_q$.
  If $E=\mathbb Z$, then
  write $\phi^u\in\Omega$
  for the unique smallest $q$-Lipschitz
  extension of the function $\lfloor u\rfloor|_{\mathcal L}$
  to $\mathbb Z^d$.
  If $E=\mathbb R$,
  then write $\phi^u\in\Omega$ for the unique smallest
  $q_\varepsilon$-Lipschitz extension of $u|_{\mathcal L}$
  to $\mathbb Z^d$,
  where $\varepsilon$ is the largest positive real number
  such that $u|_{\mathcal L}$ is $q_\varepsilon$-Lipschitz
  (subject to $\varepsilon\leq\eta$, where $\eta$ is as in Proposition~\ref{propo_qqq}).
\end{definition}

If $E=\mathbb Z$, then $q$ is integral,
and therefore the smallest $q$-Lipschitz extension
of $\lfloor u\rfloor|_\mathcal L$
to $\mathbb Z^d$ is also integer-valued.
The rounding procedure in the discrete setting
makes that the gradient of $\phi^u$ is not
$\mathcal L$-invariant.
In the continuous setting $E=\mathbb R$
there is no rounding, and therefore
the gradient of $\phi^u $ is $\mathcal L$-invariant.
Finally, we want to remark that, in both the
discrete and the continuous setting, there exists
a constant $C<\infty$ such that
$|\phi^u-u|_{\mathbb Z^d}|\leq C$
for any $u\in U_q$.
This is due to Lemma~\ref{lemma_lipschitz_q_and_norm}.
This observation, combined with the previous theorem, implies the following result.

\begin{theorem}
  \label{general_lips_ext_less_general}
  Let $C<\infty$ denote the smallest
  constant such that
  $|\phi^u-u|_{\mathbb Z^d}|+1\leq C$
  for all $u\in U_q$.
  Consider $\varepsilon>0$ so small
  that $q_{\varepsilon}$ is well-defined.
  Then there exists a constant $m\in\mathbb Z_{\geq 0}$ such that the following holds true.
  Suppose given a collection $(D_i)_i$ of disjoint subsets of $\mathbb R^d$,
  and write $D:=\cup_iD_i$,
  $\Lambda_i:=\Lambda^{-m}(D_i)$,
  and $\Lambda:=\cup_i\Lambda_i$.
  Let $f:D\to \mathbb R$ denote a $\|\cdot\|_{q_{\varepsilon}}$-Lipschitz
  function which is affine
  with slope $u_i\in\bar U_{q_{\varepsilon}}$ whenever restricted to $D_i$.
  Then there exists a $q$-Lipschitz function $\phi:\Lambda\to E$
  which satisfies $|\phi-f|_\Lambda|\leq C$
  and $\nabla\phi|_{\Lambda_i}=\nabla \phi^{u_i}|_{\Lambda_i}$
  for all $i$.
  If $E=\mathbb R$
  then we may furthermore impose that $\phi$
  is $q_{\varepsilon'}$-Lipschitz for fixed $0<\varepsilon'<\varepsilon$ (that $m$ is allowed to depend upon).
\end{theorem}

For this result, the notation $\nabla\phi=\nabla\psi$
means that the difference $\phi-\psi$
is constant.


\section{The specific free energy}
\label{section:specific_free_energy}


\subsection{The attachment lemmas}

The letter $\Phi$
denotes a potential in $\mathcal S_\mathcal L+\mathcal W_\mathcal L$
throughout this section.
For the thermodynamical formalism, it is crucial that we are able to attach
height functions defined on disjoint subsets of $\mathbb Z^d$ without
losing or
gaining too much energy.
More precisely,
if $\Lambda_1,\Lambda_2\subset\subset\mathbb Z^d$
are disjoint with $\Lambda:=\Lambda_1\cup\Lambda_2$,
then
we want to find bounds on the difference
between $H_\Lambda^0(\phi)$
and $H_{\Lambda_1}^0(\phi)+H_{\Lambda_2}^0(\phi)$.
Similarly, we will require bounds
on the difference between $H_\Lambda^0(\phi)$
and $H_\Lambda(\phi)$.
In this section, we present simple tools for doing this: the attachment lemmas.
We first state and prove the lower attachment lemma,
which is easier.

\begin{lemma}[Lower attachment lemma]
  \label{lemma_lower_attachments}
  Let $\Lambda_1,\Lambda_2\subset\subset\mathbb Z^d$ disjoint, and write $\Lambda:=\Lambda_1\cup\Lambda_2$.
  Then
  \[
  H_{\Lambda}^0
  \geq
  H_{\Lambda_1}^0+
  H_{\Lambda_2}^0-
  \min_{i\in\{1,2\}}e^-(\Lambda_i),
  \]
  where $e^-$ is the lower exterior bound of $\Xi$.
  We also have $H_\Lambda\geq H_\Lambda^0-e^-(\Lambda)$
  for any $\Lambda\subset\subset\mathbb Z^d$.
\end{lemma}

\begin{proof}
  The inequality $H_{\Lambda}^{0,\Psi}
  \geq
  H_{\Lambda_1}^{0,\Psi}+
  H_{\Lambda_2}^{0,\Psi}$ is obvious because $\Psi$
  is positive.
  The inequality
  $H_{\Lambda}^{0,\Xi}
  \geq
  H_{\Lambda_1}^{0,\Xi}+
  H_{\Lambda_2}^{0,\Xi}-\min_{i\in\{1,2\}}e^-(\Lambda_i)$
  is immediate from the definition of $e^-$ in terms of $\Xi$.
  This proves the inequality in the display.
  The other inequality follows from a
  similar decomposition.
\end{proof}

More care is required for the upper bound.
There is a difference between the discrete case $E=\mathbb Z$
and the continuous case $E=\mathbb R$.
If $E=\mathbb Z$ then the strong interaction $\Psi$
can be described by finite information. The effect of this
is that there exists a uniform bound $C<\infty$
such that
\[
\numberthis \label{eq_discrete_singleton_hamil_bound}
  H_{\{x\}}^\Psi(\phi)\leq C
\]
for any $x\in\mathbb Z^d$ and any $q$-Lipschitz function $\phi\in\Omega$.
If $E=\mathbb R$ then there exists no such \emph{a priori} bound,
and it is this specific reason reason that we introduce the \emph{locally bounded}
property in Subsection~\ref{subsection:formal_setting:strong},
so that at least
\[
\numberthis \label{eq_cts_singleton_hamil_bound}
  H_{\{x\}}^\Psi(\phi)\leq C_\varepsilon
\]
whenever $\phi$ is $q_\varepsilon$-Lipschitz at $x$.

For the upper bound, one requires control especially over
the potential $\Psi$ which enforces the Lipschitz
constraint.
The height function $\phi$ must therefore be sufficiently
well-behaved for the lemma to work,
at least on the boundary where $\Lambda_1$
meets $\Lambda_2$.
%
%

\begin{lemma}[Upper attachment lemma]
  \label{lemma_attachment}
  Let $\phi\in\Omega$
  and $\Lambda_1,\Lambda_2\subset\subset\mathbb Z^d$ disjoint, and write $\Lambda:=\Lambda_1\cup\Lambda_2$.
  If $E=\mathbb Z$,
  then there exists an amenable function
  $e^+$, dependent only on $\Phi$,
  such that
  \begin{equation}
    \label{eq_gluing_basic}
    H_{\Lambda}^0(\phi)
    \leq
    H_{\Lambda_1}^0(\phi)+
    H_{\Lambda_2}^0(\phi)+
    \min_{i\in\{1,2\}}e^+(\Lambda_i)
  \end{equation}
  whenever $\phi_{\partial^R\Lambda_1\cup\partial^R\Lambda_2}$
  is $q$-Lipschitz,
  and such that
  \begin{equation}
    \label{eq_gluing_ext_0}
    H_{\Lambda}(\phi)
    \leq
    H_{\Lambda}^0(\phi)+e^+(\Lambda)
  \end{equation}
  whenever $\phi_{\partial^R\Lambda\cup\partial^R(\mathbb Z^d\smallsetminus\Lambda)}$
  is $q$-Lipschitz.
  If $E=\mathbb R$ and $\varepsilon>0$,
  then there exists an amenable function $e^+_\varepsilon$, dependent only on $\Phi$ and $\varepsilon$,
  such that~\eqref{eq_gluing_basic} and~\eqref{eq_gluing_ext_0} hold true  whenever the restrictions of $\phi$
    are $q_\varepsilon$-Lipschitz,
    and with $e^+$ replaced by $e^+_\varepsilon$.
\end{lemma}

\begin{definition}
  The functions
  $e^+$ and $e_\varepsilon^+$ are called \emph{upper exterior bounds}.
\end{definition}

\begin{proof}[Proof of Lemma~\ref{lemma_attachment}]
  It suffices to consider
  the contributions of the potentials
  $\Psi$ and $\Xi$ to each Hamiltonian separately;
  one can simply
  sum the two upper exterior bounds $e^{+,\Psi}$
  and $e^{+,\Xi}$ so obtained.
  In fact, the upper exterior bound $e^{+,\Xi}:=e^-$ suffices
  for the long-range interaction $\Xi$.
  Let us therefore focus on the contribution from the potential $\Psi$.

  We shall simultaneously
  consider the discrete case and the continuous case.
  In this proof we shall reserve the name \emph{Lipschitz} for \emph{$q$-Lipschitz} whenever $E=\mathbb Z$
  and for \emph{$q_\varepsilon$-Lipschitz} whenever $E=\mathbb R$.
  Write $C$ for a fixed constant
  such that
  $H_{
  \{x\}
  }^\Psi(\psi)\leq C$
  for any $x\in\mathbb Z^d$
  and for any Lipschitz height function $\psi$.
  Because $\Psi$ is positive and of range $R$
  and because the restriction of $\phi$
  to $\partial^R\Lambda_1\cup\partial^R\Lambda_2$
  is Lipschitz,
  we have
  \begin{align*}
    &H_{\Lambda}^{0,\Psi}(\phi)-
    H_{\Lambda_1}^{0,\Psi}(\phi)-
    H_{\Lambda_2}^{0,\Psi}(\phi)
    =
    \sum_{
      \Delta\subset\Lambda,\,
      \Delta\not\subset\Lambda_1,\,
      \Delta\not\subset\Lambda_2
    }
    \Psi_\Delta(\phi)
    \\&\qquad=
    \sum_{
      \Delta\subset\partial^R\Lambda_1\cup\partial^R\Lambda_2,\,
      \Delta\not\subset\Lambda_1,\,
      \Delta\not\subset\Lambda_2
    }
    \Psi_\Delta(\phi)
    \leq\min_{i\in\{1,2\}} H^\Psi_{\partial^R\Lambda_i,\partial^R\Lambda_1\cup\partial^R\Lambda_2}(\phi)
    \\&\qquad
    \leq
    \min_{i\in\{1,2\}}
    C |\partial^R \Lambda_i|
    \leq
      \min_{i\in\{1,2\}}e^{+,\Psi}(\Lambda_i)
  \end{align*}
  if we define $e^{+,\Psi}(\Lambda):=C (2R+1)^d|\partial \Lambda|$;
  this function satisfies the desired constraints.
  It is clear that this choice for $e^{+,\Psi}$
  also implies that
  \[
    H_\Lambda^\Psi(\phi)\leq
    H_\Lambda^{0,\Psi}(\phi)
    +e^{+,\Psi}(\Lambda)
  \]
  whenever the restriction of $\phi$
  to $\partial^R\Lambda\cup\partial^R(\mathbb Z^d\smallsetminus\Lambda)$ is Lipschitz.
\end{proof}

\subsection{Density limits of functions on finite subsets of $\mathbb Z^d$}

\begin{proposition}
  \label{propo_large_set_limit}
  Consider two $\mathcal L$-invariant real-valued
  functions $f$ and $b$ on the finite subsets of $\mathbb Z^d$,
  with $b$ amenable and
\[
\numberthis
\label{equation_for_f_a_superadditivty}
  f(\Lambda_1\cup\Lambda_2)\leq f(\Lambda_1)+f(\Lambda_2)+\min_{i\in\{1,2\}}b(\Lambda_i)
\]
for disjoint $\Lambda_1,\Lambda_2\subset\subset\mathbb Z^d$.
Then
$(n^{-d}f(\Pi_n))_{n\in\mathbb N}$ tends
to a limit in $[-\infty,\infty)$ as $n\to\infty$,
and
\[
\lim_{n\to\infty}n^{-d}f(\Pi_n)=\inf_{n\in N\cdot\mathbb N}
n^{-d}(f(\Pi_{n})+b(\Pi_{n}))
\]
where $N\in\mathbb N$ is minimal subject to $N\cdot\mathbb Z^d\subset\mathcal L$.
Finally, if $(\Lambda_n)_{n\in\mathbb N}\uparrow\mathbb Z^d$,
then
\[
  \limsup_{n\to\infty}|\Lambda_n|^{-1}f(\Lambda_n)\leq \lim_{n\to\infty}n^{-d}f(\Pi_n).
\]

If we weaken the assumptions, and
suppose only that~\eqref{equation_for_f_a_superadditivty}
holds true whenever $\Lambda_1$ contains some vertex $x$
adjacent to some vertex $y$ in $\Lambda_2$,
then each statement in this proposition remains valid,
except that, for the final assertion, we also require
that each set $\Lambda_n$ is connected.
\end{proposition}

\begin{definition}
  Write $\langle\cdot|\Phi\rangle:\mathcal P_\mathcal L(\Omega,\mathcal F^\nabla)\to [-\|\Xi\|,\infty]$
  for the unique functional which satisfies
  \[
    \langle\mu|\Phi\rangle:=\mu(\Phi):=\lim_{n\to\infty}n^{-d}\mu(H_{\Pi_n}^0).
  \]
  The limit on the right converges due to the lower attachment lemma and the previous proposition.
  This quantity
  is called the \emph{specific energy} of $\mu$ with respect to $\Phi$.
\end{definition}

\subsection{Free energy attachment lemma}

\begin{definition}
  Define $e^*:=e^-+\log(2K+1)$,
  where $K$ is minimal subject to $Kd_1\geq q$.
  Call the amenable function $e^*$ the \emph{free energy exterior bound}.
\end{definition}

\begin{lemma}[Free energy attachment lemma]
  \label{lemma_fe_superadditivity}
  Fix $\mu\in\mathcal P(\Omega,\mathcal F^\nabla)$,
  and consider some disjoint sets $\Lambda_1,\Lambda_2\subset\subset\mathbb Z^d$
  with some vertex $x$ of $\Lambda_1$
  adjacent to some vertex $y$ of $\Lambda_2$ in the square lattice.
  Write $\Lambda:=\Lambda_1\cup\Lambda_2$.
  Then
  \[
    \mathcal H_\Lambda(\mu|\Phi)\geq \mathcal H_{\Lambda_1}(\mu|\Phi)
    +\mathcal H_{\Lambda_2}(\mu|\Phi)-\min_{i\in\{1,2\}}e^*(\Lambda_i).
  \]
  Moreover, for $\Lambda\subset\subset\mathbb Z^d$ connected and nonempty, we have
  \[\mathcal H_\Lambda(\mu|\Phi)\geq -(|\Lambda|-1)\max_{x\in\mathbb Z^d/\mathcal L}
  e^*(\{x\}).\]
\end{lemma}

\begin{proof}
  Fix $K$ minimal subject to $Kd_1\geq q$.
  Recall that $\mu\pi_\Lambda$ is the restriction of $\mu$ to $\Lambda$.
  We assume that $\mu\pi_\Lambda$ is supported on $Kd_1$-Lipschitz functions;
  if this is not the case, then $\mathcal H_\Lambda(\mu|\Phi)$
  is infinite, and we are done.
  For any $\Delta\subset\Lambda$,
   we have
   \[
    \mathcal H_\Delta(\mu|\Phi)=\mathcal H_{\mathcal F_\Delta^\nabla}(\mu|\lambda^{\Delta-1})+\mu(H_{\Delta}^0).
   \]
   By the lower attachment lemma,
   we have
   $
    \mu(H_{\Lambda}^0)\geq \mu(H_{\Lambda_1}^0)+\mu(H_{\Lambda_2}^0)-\min_{i\in\{1,2\}}e^-(\Lambda_i).
   $
   Therefore it suffices to show that
   \begin{equation}
     \label{equation_remains_to_prove_for_quasi_superadditivity_fe}
     \mathcal H_{\mathcal F_\Lambda^\nabla}(\mu|\lambda^{\Lambda-1})
     \geq
     \mathcal H_{\mathcal F_{\Lambda_1}^\nabla}(\mu|\lambda^{\Lambda_1-1})
     +
     \mathcal H_{\mathcal F_{\Lambda_2}^\nabla}(\mu|\lambda^{\Lambda_2-1})
     -\log(2K+1)
   \end{equation}
   whenever $\mu\pi_\Lambda$ is supported on $Kd_1$-Lipschitz functions.
      This follows from the following two facts:
   \begin{enumerate}
     \item We have
     $\mathcal H_{\mathcal F_{\{x,y\}}^\nabla}(\mu|\lambda^{\{x,y\}-1})\geq -\log(2K+1)$,
     \item If $\Delta_1,\Delta_2\subset \Lambda$ share a single vertex $z$
     and $\Delta:=\Delta_1\cup\Delta_2$,
     then
     \[
     \mathcal H_{\mathcal F_\Delta^\nabla}(\mu|\lambda^{\Delta-1})
     \geq
     \mathcal H_{\mathcal F_{\Delta_1}^\nabla}(\mu|\lambda^{\Delta_1-1})
     +
     \mathcal H_{\mathcal F_{\Delta_2}^\nabla}(\mu|\lambda^{\Delta_2-1}).
     \]
   \end{enumerate}
   Note that~\eqref{equation_remains_to_prove_for_quasi_superadditivity_fe} then follows by applying the
   second fact twice, first to the sets $\Lambda_1$ and $\{x,y\}$,
   then to the sets $\Lambda_1\cup\{y\}$ and $\Lambda_2$.
   Let us first prove the first fact.
   Since $\mu$ is supported on $Kd_1$-Lipschitz functions,
   we have
   \[
   \mathcal H_{\mathcal F_{\{x,y\}}^\nabla}(\mu|\lambda^{\{x,y\}-1})
   \geq
   -\log\lambda^{\{x,y\}-1}(\{|\phi(y)-\phi(x)|\leq K\})
   \geq -\log (2K+1).
   \]
   For the second fact, we can simply choose the point $z$ as a reference
   point for all gradient measures,
   such that the measurable space $(\Omega,\mathcal F^\nabla_\Delta)$
   becomes effectively a product space;
   the measure $\lambda^{\Delta-1}$ is then the product measure
   of $\lambda^{\Delta_1-1}$ and $\lambda^{\Delta_2-1}$.
   The second fact now follows; the inequality in the display
  is
   well-known for product spaces.

   The final assertion of the lemma is a direct consequence of the first
   assertion and the fact that $\mathcal H_\Lambda(\mu|\Phi)=0$
   whenever $\Lambda$ is a singleton.
\end{proof}

\subsection{Convergence and properties of the
specific free energy}

The two results in this subsection jointly imply Theorem~\ref{thm_main_sfe}.

\begin{theorem}
  If $\Phi\in\mathcal S_\mathcal L+\mathcal W_\mathcal L$,
  then the
   functional $\mathcal H(\cdot|\Phi):\mathcal P_\mathcal L(\Omega,\mathcal F^\nabla)\to\mathbb R\cup\{\infty\}$
  is well-defined
  and satisfies
  \[
    \mathcal H(\mu|\Phi):=
    \lim_{n\to\infty}n^{-d}\mathcal H_{\Pi_n}(\mu|\Phi)
    =
    \sup_{n\in N\cdot\mathbb N}
    n^{-d}\left(\mathcal H_{\Pi_n}(\mu|\Phi)-e^*(\Pi_n)\right)
    \geq -\max_{x\in\mathbb Z^d/\mathcal L}
    e^*(\{x\}),
  \]
  where $N$ is minimal subject to $N\cdot\mathbb Z^d\subset\mathcal L$.
  Moreover, $\mathcal H(\cdot|\Phi)$ is lower-semicontinuous,
  and for each $C\in\mathbb R$ the lower level set
  \[
   M_C:=\{\mu\in\mathcal P_\mathcal L(\Omega,\mathcal F^\nabla):\mathcal H(\mu|\Phi)\leq C\}
  \]
  is a compact Polish space, with respect to the topology of (weak) local convergence. In fact, the two topologies coincide on each set $M_C$.
\end{theorem}

\begin{proof}
  The statements in the first display follow from Lemma~\ref{lemma_fe_superadditivity}
  and Proposition~\ref{propo_large_set_limit}.
  For the remainder of the theorem,
  observe that
  \[
    M_C=\mathcal P_\mathcal L(\Omega,\mathcal F^\nabla)\cap\bigcap_{n\in N\cdot\mathbb N}\{\mu\in\mathcal P(\Omega,\mathcal F^\nabla):\mathcal H_{\Pi_n}(\mu|\Phi)\leq n^d C+e^*(\Pi_n)\}.
  \]
  Each of these sets is closed (in the topology of weak local convergence),
  and therefore $M_C$ is closed;
  the functional
  $\mathcal H(\cdot|\Phi)$ must be lower-semicontinuous (in either topology).
  Moreover, for each $n\in N\cdot\mathbb N$,
  the set
  \[
  \{\mu\in\mathcal P(\Omega,\mathcal F^\nabla_{\Pi_n}):\mathcal H_{\Pi_n}(\mu|\Phi)\leq n^d C+e^*(\Pi_n)\}
  \]
  is a compact Polish space with respect to both the weak and strong topologies,
  which coincide on this set.
  Write $\delta_n$ for the corresponding metric.
  Then $M_C$ is a compact Polish space with metric
  $\delta(\mu,\nu):=\sum_{n\in N\cdot\mathbb N}e^{-n}(\delta_n(\mu,\nu)\wedge 1)$.
\end{proof}

\begin{theorem}
If $\Phi\in\mathcal S_\mathcal L+\mathcal W_\mathcal L$,
then the functional $\mathcal H(\cdot|\Phi)$
is affine, in the sense that
\[\mathcal H((1-t)\mu+t\nu|\Phi)=(1-t)\mathcal H(\mu|\Phi)+t\mathcal H(\nu|\Phi)\]
for $\mu,\nu\in\mathcal P_\mathcal L(\Omega,\mathcal F^\nabla)$ and $0\leq t\leq1$.
\end{theorem}

\begin{proof}
  It follows from a direct entropy calculation that
  for fixed $\Lambda\subset\subset\mathbb Z^d$,
  \[
    0\leq
    (1-t)\mathcal H_\Lambda(\mu|\Phi)+t\mathcal H_\Lambda(\nu|\Phi)
    -\mathcal H_\Lambda((1-t)\mu+t\nu|\Phi)
\leq 2\log 2.
  \]
  This error term vanishes in the normalization of the specific free energy.
\end{proof}

\subsection{The surface tension}

Recall that the surface tension $\sigma:(\mathbb R^d)^*\to\mathbb R\cup\{\infty\}$ is
defined by
\[
\sigma(u):=\inf_{\text{$\mu\in\mathcal P_\mathcal L(\Omega,\mathcal F^\nabla)$ with $S(\mu)=u$}}\mathcal H(\mu|\Phi).
\]
The function $\sigma$ must be convex because both $S(\cdot)$ and $\mathcal H(\cdot|\Phi)$
are affine.
It is also bounded from below because $\mathcal H(\cdot|\Phi)$
is bounded from below by $-\max_{x\in\mathbb Z^d/\mathcal L}e^*(\{x\})$.
Recall that $U_\Phi$
is  defined to be the interior of the set $\{\sigma<\infty\}\subset (\mathbb R^d)^*$.
The set $U_\Phi$ is convex, and  $\sigma$ is continuous on $U_\Phi$.
Moreover, $\sigma$ must equal $\infty$ on the complement
of the closure of $U_\Phi$.
Recall the statement of Theorem~\ref{thm_main_st_general},
for which we now provide a proof.


\begin{proof}[Proof of Theorem~\ref{thm_main_st_general}]
  \label{thm_main_st_general_proof}
  Observe that $\sigma$ is lower-semicontinuous,
  because $S(\cdot)$ is continuous
  and because
  $\mathcal H(\cdot|\Phi)$ is lower-semicontinuous
  with compact lower level sets.

  Let us first prove that $U_\Phi\subset U_q$.
  Suppose that the slope $u:=S(\mu)$ of
  $\mu\in\mathcal P_\mathcal L(\Omega,\mathcal F^\nabla)$
  is not in $\bar U_q$.
  It suffices to demonstrate that $\mathcal H(\mu|\Phi)=\infty$.
  Since $u\not\in \bar U_q$,
  we know that $u|_\mathcal L$ is not $q$-Lipschitz,
  and therefore with positive $\mu$-probability, $\phi|_\mathcal L$
  is not $q$-Lipschitz. In particular, this means that
  $\mu(H_{\Pi_n}^0)=\infty$ for $n$ sufficiently large.
  This proves that $\mathcal H(\mu|\Phi)=\infty$.

  For the remainder of the proof, we distinguish between the
  discrete and the continuous setting.
  Consider first the case that $E=\mathbb Z$.
  For the lemma, it suffices to demonstrate that
  $\sigma$ is bounded on $U_q$.
  If $\mu\in\mathcal P_\mathcal L(\Omega,\mathcal F^\nabla)$
  is supported on $q$-Lipschitz functions,
  then
  \[
    \mathcal H_{\Pi_n}(\mu|\Phi)
    =\mu(H_{\Pi_n}^0)+\mathcal H_{\mathcal F_{\Pi_n}^\nabla}(\mu|\lambda^{\Pi_n-1})
    \leq
    Cn^d
    \qquad \text{where}
    \qquad
    C:=\max_{x\in\mathbb Z^d/\mathcal L}e^+(\{x\});
  \]
  the energy term is bounded by $Cn^d$ because $\phi$
  is $q$-Lipschitz $\mu$-almost surely,
  and the entropy term is nonpositive because $\lambda^{\Pi_n-1}$
  is a counting measure.
  In particular, $\mathcal H(\mu|\Phi)\leq C$.
  Fix $u\in U_q$,
  and consider a subsequential limit $\mu$ of the sequence
  \[
  \mu_n:= \frac{1}{|\Pi_n\cap\mathcal L|}\sum_{x\in\Pi_n\cap\mathcal L}
  \delta_{\theta_x\phi^u}.
  \]
  This limit $\mu$ is clearly supported on $q$-Lipschitz functions
  and is automatically shift-invariant
  and satisfies $S(\mu)=u$;
  in particular, $\sigma(u)\leq C<\infty$. This proves that $\sigma$
  is bounded by $C$ on $U_q$.

  Consider now the continuous case $E=\mathbb R$.
  For the lemma, we must show that $\sigma$
  is finite on $U_q$,
  and infinite on $\partial U_q$.
  Fix $u\in U_q$.
  Then $\phi^u$ is $q_{3\varepsilon}$-Lipschitz
  for $\varepsilon>0$ sufficiently small.
  Let $X=(X_{x})_{x\in\mathbb Z^d}$ denote an i.i.d.\ family
  of random variables which are uniformly random in the interval $[0,\varepsilon]$.
  Write $\mu\in\mathcal P_\mathcal L(\Omega,\mathcal F^\nabla)$
  for the measure in which $\phi$ has the distribution of $\phi^u+X$.
  Then $\phi$ is $q_\varepsilon$-Lipschitz almost surely.
  It is straightforward to see that
  \[
    \mathcal H_{\Pi_n}(\mu|\Phi)=\mu(H_{\Pi_n}^0)+\mathcal H_{\mathcal F_{\Pi_n}^\nabla}(\mu|\lambda^{\Pi_n-1})
    \leq (C -\log \varepsilon)n^d
    \quad \text{where}
    \quad
    C:=\max_{x\in\mathbb Z^d/\mathcal L}e^+_\varepsilon(\{x\});
  \]
  in particular, $\mathcal H(\mu|\Phi)\leq C-\log\varepsilon<\infty$.
  Clearly $S(\mu)=u$, and so $\sigma(u)<\infty$.
  Finally, consider $u\in\partial U_q$.
  Suppose that $\mu\in\mathcal P_\mathcal L(\Omega,\mathcal F^\nabla)$
  has slope $u$.
  Then at least one of the following two must hold true:
  \begin{enumerate}
    \item $\phi$ is not $q$-Lipschitz, with positive $\mu$-probability,
    \item $\phi(y)-\phi(x)$ is deterministic in $\mu$ for some distinct vertices
    $x$ and $y$.
  \end{enumerate}
  This follows from Lemma~\ref{lemma_lipschitz_q_and_norm} which
  gives a characterization of $U_q$.
  In the former case we have $\mathcal H(\mu|\Phi)=\infty$
  as was shown at the beginning of this proof.
  In the latter case,
  we observe that
  \[
    \mathcal H_{\mathcal F_{\Pi_n}^\nabla}(\mu|\lambda^{\Pi_n-1})=\infty
    \]
    for $n$ sufficiently large,
    because $\mu\pi_{\Pi_n}$ is not absolutely continuous with respect
    to $\lambda^{\Pi_n-1}$.
    This also implies that $\mathcal H(\mu|\Phi)=\infty$.
    We have now shown that $\sigma=\infty$ on $\partial U_q$.
\end{proof}

\section{Minimizers of the specific free energy}
\label{section_minimizers}

Recall that a minimizer is a shift-invariant measure
$\mu\in\mathcal P_\mathcal L(\Omega,\mathcal F^\nabla)$
which satisfies
\[
  \mathcal H(\mu|\Phi)=\sigma(S(\mu))<\infty,
\]
and recall the discussion of minimizers
in Subsection~\ref{subsec_main_sfe_and_minimizers},
in particular Definition~\ref{definition_quasilocality_aGm}.
The purpose of this section
is to prove the following theorem,
which provides us with several properties of minimizers,
and is equivalent to the conjunction of Theorem~\ref{thm_main_minimizers_finite_energy} and Theorem~\ref{thm_mr_minimizers}.

\begin{theorem}
  \label{theorem_stronger_than_finite_energy}
  Let $\Phi\in\mathcal S_\mathcal L+\mathcal W_\mathcal L$,
  and consider a minimizer
  $\mu\in\mathcal P_\mathcal L(\Omega,\mathcal F^\nabla)$.
  Fix $\Lambda\subset\subset\mathbb Z^d$,
  and write
  $\mu^{\phi}$
  for the regular conditional
  probability distribution of
  $\mu$ on $(\Omega,\mathcal F)$
  corresponding to the projection map $\Omega\to E^{\mathbb Z^d\smallsetminus\Lambda}$.
  Then for $\mu$-almost every $\phi\in\Omega$,
  we have
  $\mu^\phi\pi_\Lambda\in\mathcal A_{\Lambda,\phi}$.
  In particular,
  if $\mu(\Omega_\gamma)=1$,
  then $\mu$ is an almost Gibbs measure.
  In general,
  the former implies that
  $\mu$ has finite energy,
  in the sense that
  \[
1_{\Omega_q}(\mu\pi_{\mathbb Z^d\smallsetminus\Lambda}\times \lambda^\Lambda)\ll \mu,
  \]
  where $\Omega_q$ is the set of $q$-Lipschitz height functions.
\end{theorem}

We first introduce the definition of the max-entropy,
which is due to Datta~\cite{DATTA}.

\begin{definition}
  Let $(X,\mathcal X)$ denote a measurable space,
  endowed with some finite measures $\mu$ and $\nu$.
  Then the \emph{max-entropy} of $\mu$ with respect to $\nu$
  is defined by
  \[
    \mathcal H^\infty(\mu|\nu):=
    \log\inf\{\lambda\geq 0:\mu\leq\lambda\nu\}
    =
    \begin{cases}
      \operatorname{ess\,sup}\log f&\text{if $\mu\ll\nu$ where $f=d\mu/d\nu$,}\\
      \infty&\text{otherwise}.
    \end{cases}
  \]
  The \emph{max-diameter} of a non-empty set $\mathcal A$ of  finite measures on $(X,\mathcal X)$
  is defined by
  \[
    \operatorname{Diam}^\infty\mathcal A:=\sup_{\mu,\nu\in\mathcal A}
    \mathcal H^\infty(\mu|\nu).
  \]
\end{definition}

If $\operatorname{Diam}^\infty\mathcal A<\infty$,
then all measures in $\mathcal A$ are absolutely continuous with respect to one another,
with uniform lower and upper bounds on the Radon-Nikodym derivatives.

\begin{proposition}
  \label{propo_all_abs_cts_wrt_one_another}
  Suppose that $\Lambda\subset\Delta\subset\subset\mathbb Z^d$ with $\Lambda\subset\Delta^{-R}$.
  Then $\operatorname{Diam}^{\infty}\mathcal C(\mathcal A_{\Lambda,\Delta,\phi})\leq 4e^-(\Lambda)<\infty$.
  In particular, $\operatorname{Diam}^{\infty}\mathcal A_{\Lambda,\phi}\leq 4e^-(\Lambda)<\infty$.
\end{proposition}

\begin{proof}
Claim first that $\operatorname{Diam}^{\infty}\mathcal A_{\Lambda,\Delta,\phi}\leq 4e^-(\Lambda)$.
Consider two random fields $\nu_1,\nu_2\in\mathcal P(\Omega,\mathcal F)$ with $\nu_1\pi_\Delta=\nu_2\pi_\Delta=\delta_{\phi_\Delta}$.
Then
\[
  \nu_i\gamma_\Lambda\pi_\Lambda=\int \frac1{Z_\Lambda(\psi)}e^{-H(\cdot,\psi_{\mathbb Z^d\smallsetminus\Lambda})}\lambda^\Lambda d\nu_i(\psi).
\]
But since $\psi_\Delta=\phi_\Delta$ almost surely in both $\nu_1$ and $\nu_2$,
the dependence of $H(\cdot,\psi_{\mathbb Z^d\smallsetminus\Delta})$
on $\psi$ is bounded by $e^{-}(\Lambda)$.
This error term appears twice in each measure $\nu_i\gamma_\Lambda\pi_\Lambda$;
directly in the Hamiltonian, and indirectly in the normalization constant.
Thus, in calculating the Radon-Nikodym derivative between the two measures,
the term appears four times. This proves the claim.
By Lemma~5.1 in~\cite{lammers2019variational},
this also implies that $\operatorname{Diam}^{\infty}\mathcal C(\mathcal A_{\Lambda,\Delta,\phi})\leq 4e^-(\Lambda)$.
\end{proof}

We also need the following lemma,
which is an adaptation of an intermediate
result in~\cite{lammers2019variational} to the gradient setting.

\begin{lemma}
  Fix $\mu\in\mathcal P(\Omega,\mathcal F^\nabla)$,
  and define, for $\Lambda\subsetneq \Delta\subset\subset\mathbb Z^d$,
  \[
    K_\mu(\Lambda,\Delta):=\inf_{\text{$\nu\in\mathcal P(\Omega,\mathcal F^\nabla)$ with $
    \nu\pi_\Delta=\mu\pi_\Delta$}}\mathcal H_{\mathcal F^\nabla_\Delta}(\mu|\nu\gamma_\Lambda)\geq 0.
  \]
  Then $K_\mu(\cdot,\cdot)$ is superadditive in the first argument,
  and increasing in the second argument.
\end{lemma}

\begin{proof}
  It is straightforward to see that $K_\mu(\Lambda,\Delta)$ is increasing
  in $\Delta$: increasing $\Delta$ restricts the set of measures
  $\nu$ for the infimum, while increasing the
  $\sigma$-algebra $\mathcal F_{\Delta}^\nabla$ for the entropy.
  Both operations increase the value of $K_\mu(\Lambda,\Delta)$.
  For superadditivity in $\Lambda$,
  it suffices to prove that
  \[
  \inf_{\text{$\nu\in\mathcal P(\Omega,\mathcal F^\nabla)$ with $
  \nu\pi_\Delta=\mu\pi_\Delta$}}\mathcal H_{\mathcal F^\nabla_\Delta}(\mu|\nu\gamma_\Lambda)
  \geq
  \sum_{i\in\{1,2\}}
  \inf_{\text{$\nu\in\mathcal P(\Omega,\mathcal F^\nabla)$ with $
  \nu\pi_\Delta=\mu\pi_\Delta$}}\mathcal H_{\mathcal F^\nabla_\Delta}(\mu|\nu\gamma_{\Lambda_i})
  \]
  for $\Lambda_1$ and $\Lambda_2$ disjoint with
  $\Lambda:=\Lambda_1\cup\Lambda_2\subsetneq\Delta$.
  This follows from Lemma~4.1 in~\cite{lammers2019variational}.
  Observe that that lemma does not concern the gradient setting,
  which provides us with a slight complication.
  However, since we choose $\Lambda$ to be a strict subset of $\Delta$,
  we can fix a vertex $x\in\Delta\smallsetminus\Lambda$
  to serve as a reference vertex for the gradient setting
  for all three entropy calculations in the display,
  thus translating the inequality to the non-gradient setting.
\end{proof}

\begin{lemma}
  If
  $\mu$
  is a minimizer,
  then $K_\mu\equiv 0$.
\end{lemma}

\begin{proof}
  Fix $\mu\in\mathcal P_\mathcal L(\Omega,\mathcal F^\nabla)$.
  Then $K_\mu$ is shift-invariant, in the sense that
  $K_\mu(\Lambda,\Delta)=K_\mu(\theta\Lambda,\theta\Delta)$
  for $\Lambda\subsetneq\Delta\subset\subset\mathbb Z^d$
  and $\theta\in\Theta$.
  Using also the properties of $K_\mu$
  in the previous defining lemma, it is immediate
  that
  $K_\mu\equiv 0$ if and only if
  $
    K_\mu(\Pi_n^{-R},\Pi_n)=o(n^d)
  $
  as $n\to\infty$.
  Moreover, by definition of $K_\mu$,
  it is immediate that
  \[
    K_\mu(\Pi_n^{-R},\Pi_n)\leq \mathcal H_{\mathcal F_{\Pi_n}^\nabla}(\mu|\mu\gamma_{\Pi_n^{-R}}).
  \]
  We must therefore prove that $\mathcal H(\mu|\Phi)=\sigma(S(\mu))<\infty$
  implies
  that the expression on the right in this display
  is of order
  $o(n^d)$
  as $n\to\infty$.
  If this expression is not of order $o(n^d)$,
  then there is an $n\in\mathbb N$ and an $\varepsilon>0$
  such that
  \[
  \numberthis
  \label{qwineioqwnqiweqewqw}
    \mathcal H_{\mathcal F_{\Pi_n}^\nabla}(\mu|\mu\gamma_{\Pi_n^{-R}})\geq 2e^-(\Pi_n)+\varepsilon.
  \]
  We will use this inequality to construct another $\mathcal L$-invariant
  measure $\mu''$ of the same slope as $\mu$ and with a strictly smaller
  specific free energy.
  This proves that $\mathcal H(\mu|\Phi)\neq\sigma(S(\mu))$.

       For $\Lambda\subset\subset\mathbb Z^d$,
       we denote by $\gamma_\Lambda^*$
       the kernel $\gamma_{\Lambda^{-R}}$,
       only now with respect to the partial Hamiltonian
       $H_{\Lambda^{-R},\Lambda}$ rather than the full Hamiltonian
       $H_{\Lambda^{-R}}$.
       With a straightforward entropy calculation one
       can demonstrate that
       \[\mathcal H_{\Lambda}(\nu\gamma_{\Pi_n}^*|\Phi)\leq
       \mathcal H_{\Lambda}(\nu|\Phi)-\varepsilon
       \]
       for any $\Lambda\subset\subset\mathbb Z^d$ containing $\Pi_n$,
       and for any $\nu\in\mathcal P(\Omega,\mathcal F^\nabla)$
       with $\nu\pi_{\Pi_n}=\mu\pi_{\Pi_n}$.
       This can be done by
       calculating each free energy term first over the $\sigma$-algebra
       generated by the vertices in $\Lambda\smallsetminus\Pi_n^{-R}$,
       then over the remaining vertices. The first term is the same for
       $\nu\gamma_{\Pi_n}^*$ and $\nu$ since the kernel modifies the values of $\phi$
       in $\Pi_n^{-R}$ only;
       the difference between the two measures
       for the second term is at least $\varepsilon$ due to~\eqref{qwineioqwnqiweqewqw}
       and because \[|H_{\Pi_n^{-R},\Lambda}-H_{\Pi_n^{-R},\Pi_n}|\leq e^-(\Pi_n).\]

       If $\Lambda$ and $\Delta$ are disjoint,
       then clearly $\gamma_\Lambda^*$ and $\gamma_\Delta^*$ commute.
       Let $M$ denote the smallest multiple of $N$ which exceeds
       $n$,
       and write
       \[\mu':=\mu\prod_{x\in M\cdot\mathbb Z^d}\gamma_{\Pi_n+x}^*;\]
       this measure is $M\cdot\mathbb Z^d$-invariant,
       but not necessarily $\mathcal L$-invariant.
       By the inequality in the previous paragraph, we have $\mathcal H_{\Pi_{kM}}(\mu'|\Phi)\leq \mathcal H_{\Pi_{kM}}(\mu|\Phi)-k^d\varepsilon$
       for any $k\in\mathbb N$.
       As $M\cdot\mathbb Z^d$-invariant measures,
       we have $S(\mu')=S(\mu)$ and $\mathcal H(\mu'|\Phi)\leq \mathcal H(\mu|\Phi)-\varepsilon/M^d<\mathcal H(\mu|\Phi)$.
       To make $\mu'$ also $\mathcal L$-invariant,
       simply define
       \[
        \mu'':=\frac{1}{|\mathcal L/(M\cdot\mathbb Z^d)|}
        \sum_{x\in \mathcal L/(M\cdot\mathbb Z^d)}\theta_x\mu'.
       \]
       The averaging procedure does not change the slope
       or the specific free energy.
       This is the desired measure.
\end{proof}

\begin{proof}[Proof of Theorem~\ref{theorem_stronger_than_finite_energy}]
  The theorem contains three claims.
  The second claim follows directly from the first claim
  and the definition of an almost Gibbs measure.
  We shall quickly demonstrate that the third claim
  also follows from the first claim,
   before focusing on that first claim.
   Assume that the first claim holds true.
   Observe first that, by assumption, for $\mu$-almost every $\phi$,
   \[
   1_{\Omega_q}(\delta_{\phi_{\mathbb Z^d\smallsetminus\Lambda}}\times \lambda^\Lambda)\pi_\Lambda\ll
   \gamma_\Lambda(\cdot,\phi)\pi_\Lambda\in\mathcal A_{\Lambda,\phi}\ni \mu^{\phi}\pi_\Lambda.
   \]
   But all measures in $\mathcal A_{\Lambda,\phi}$ are absolutely
   continuous with respect to one another, by Proposition~\ref{propo_all_abs_cts_wrt_one_another}
   and the comment preceding it.
   Therefore
   \[
    1_{\Omega_q}(\delta_{\phi_{\mathbb Z^d\smallsetminus\Lambda}}\times \lambda^\Lambda)
    \ll
    \mu^{\phi}
   \]
   for $\mu$-almost every $\phi$,
   which implies that
   $1_{\Omega_q}(\mu\pi_{\mathbb Z^d\smallsetminus\Lambda}\times \lambda^\Lambda)\ll \mu$.

   Focus finally on the first claim.
   By the previous lemma,
   it suffices to prove that $K_\mu\equiv 0$ implies that
   $\mu^\phi\pi_\Lambda\in\mathcal A_{\Lambda,\phi}$
   for $\mu$-almost every $\phi$.
   The proof is nearly identical to the proof of Lemma~5.4 in~\cite{lammers2019variational}.
   Fix $\Delta\subset\subset\mathbb Z^d$
   with $\Lambda\subset\Delta^{-R}$;
   it suffices to demonstrate that $K_\mu\equiv 0$ implies that
   $\mu^\phi\pi_\Lambda\in\mathcal C(\mathcal A_{\Lambda,\Delta,\phi})$
   for $\mu$-almost every $\phi$.
   By choice of $\Delta$, we have $\operatorname{Diam}^\infty\mathcal C(\mathcal A_{\Lambda,\Delta,\phi})<\infty$.
   Write $\Delta_n$ for $\{-n,\dots,n\}^d\subset\subset\mathbb Z^d$,
   and write $\mu^\phi_n$ for the
   regular conditional probability distribution of
   $\mu$ on $(\Omega,\mathcal F)$
   corresponding to the projection map $\Omega\to E^{\Delta_n\smallsetminus\Lambda}$.
   We only consider $n\in\mathbb N$ so large that $\Lambda\subset\Delta\subset\Delta_n$.
   As in the proof of Lemma~5.4 in~\cite{lammers2019variational},
   we observe that $K_\mu(\Lambda,\Delta_n)=0$ implies that
   for $\mu$-almost every $\phi$,
   \begin{enumerate}
     \item $\mu_n^\phi\pi_\Lambda\in \mathcal C( \mathcal A_{\Lambda,\Delta_n,\phi}) \subset\mathcal C( \mathcal A_{\Lambda,\Delta,\phi})$
     for fixed $n$---this follows from Lemma 5.1 in~\cite{lammers2019variational},
     \item $\mu_n^\phi(A)\to \mu^\phi(A)$ for fixed $A\in\mathcal F_\Lambda$, by the bounded martingale convergence theorem,
     \item $\mu_n^\phi\pi_\Lambda\to\mu^\phi\pi_\Lambda\in \mathcal C( \mathcal A_{\Lambda,\Delta,\phi})$
     by compactness of $\mathcal C(\mathcal A_{\Lambda,\Delta,\phi})$
     in the strong topology.
   \end{enumerate}
   Compactness of $\mathcal C(\mathcal A_{\Lambda,\Delta,\phi})$ follows from Lemma~5.1 in~\cite{lammers2019variational}
   and the fact that $\mathcal A_{\Lambda,\Delta,\phi}$
   has finite max-diameter.
\end{proof}


\section{Ergodic decomposition of shift-invariant measures}
\label{section:ergo_decomp}

In this section we cite some standard
results on ergodic decompositions of
shift-invariant random fields from the work of Georgii~\cite{G11}.
Recall that $\mathcal I_\mathcal L^\nabla$
is the $\sigma$-algebra of shift-invariant gradient events,
and that $\operatorname{ex \mathcal P_\mathcal L}(\Omega,\mathcal F^\nabla)$
is the set of ergodic gradient measures,
endowed with the $\sigma$-algebra
 $e(\operatorname{ex \mathcal P_\mathcal L}(\Omega,\mathcal F^\nabla))$.


The following result is a direct adaptation of Theorem~14.10 in~\cite{G11}
to the gradient setting of this article.
Informally, the theorem asserts that if $\mu$ is a shift-invariant gradient random field, then the regular conditional probability distribution
of $\mu$ given the information in $\mathcal I_\mathcal L^\nabla$
is well-defined.

\begin{theorem}
  \label{theorem_ergodic_decomposition}
  There is a unique affine bijection
  \[
    w:\mathcal P_\mathcal L(\Omega,\mathcal F^\nabla)\to
    \mathcal P(\operatorname{ex \mathcal P_\mathcal L}(\Omega,\mathcal F^\nabla),e(\operatorname{ex \mathcal P_\mathcal L}(\Omega,\mathcal F^\nabla)))
    ,\,
    \mu\mapsto w_\mu
  \]
   such that
  \[
    \mu=\int \nu dw_\mu(\nu)
  \]
  for all $\mu\in\mathcal P_\mathcal L(\Omega,\mathcal F^\nabla)$.
  For any $A\in\mathcal F^\nabla$ and $c\in\mathbb R$, this bijection satisfies
  \[
    w_\mu(\nu(A)\leq c)=\mu(\mu(A|\mathcal I_\mathcal L^\nabla)\leq c).
  \]
\end{theorem}

\begin{definition}
  The measure $w_\mu$
  is called the \emph{ergodic decomposition} of $\mu$.
\end{definition}

\begin{proof}[Proof of Theorem~\ref{theorem_ergodic_decomposition}]
  Let $(e_1,\dots,e_d)$ denote the standard basis of $\mathbb R^d$.
  The measure $\mu$ can be considered a non-gradient measure,
  by associating to each vertex $x\in\mathbb Z^d$
  the tuple $(\phi(x+e_1)-\phi(x),\dots,\phi(x+e_d)-\phi(x))\in E^d$.
  Theorem~14.10 in~\cite{G11} applies to this non-gradient measure, which immediately implies
  the current theorem.
\end{proof}

It was shown in previous sections that
the slope and specific free energy are
affine.
In fact, these functionals are also strongly affine.
This is the subject of the following two results.

\begin{proposition}
  \label{propo_S_strongly_affine}
  The functional $S$
  is strongly affine, that is,
  \[
  S(\mu)=\int S(\nu) dw_\mu(\nu)
  \]
  for any $\mu\in\mathcal P_\mathcal L(\Omega,\mathcal F^\nabla)$ with finite slope.
\end{proposition}

This proposition is immediate from the definition of
$S$.

\begin{theorem}
  \label{thm_SFE_strongly_affine}
  If $\Phi\in\mathcal S_\mathcal L+\mathcal W_\mathcal L$,
  then the functional $\mathcal H(\cdot|\Phi)$
  is strongly affine, that is,
  \[
  \numberthis
  \label{eq_thm_SFE_strongly_affine}
  \mathcal H(\mu|\Phi)=\int \mathcal H(\nu|\Phi) dw_\mu(\nu)
  \]
  for any $\mu\in\mathcal P_\mathcal L(\Omega,\mathcal F^\nabla)$.
\end{theorem}

\begin{proof}
  If $\mu$ is not supported on $q$-Lipschitz functions,
  then the left and right of~\eqref{eq_thm_SFE_strongly_affine}
  equal $\infty$;
  recall that $\mathcal H(\cdot|\Phi)$
  is bounded below by Theorem~\ref{thm_main_sfe}
  so that the integral on the right in~\eqref{eq_thm_SFE_strongly_affine}
  is always well-defined.

  Consider now the case that $\mu$ is supported on $q$-Lipschitz functions,
  which means in particular that $\mu$ is $K$-Lipschitz
  for $K$ minimal subject to $Kd_1\geq q$.
  In that case we have
  \[
  \numberthis
  \label{eq_decomp_of_SFE_useful_occas}
    \mathcal H(\mu|\Phi)=\langle\mu|\Phi\rangle+\lim_{n\to\infty}n^{-d}\mathcal H_{\mathcal F_{\Pi_n}^\nabla}(\mu|\lambda^{\Pi_n-1}),
  \]
  once it is established that the sequence on the right
  tends to some limit in $(-\infty,\infty]$.
  The functional $\langle\cdot|\Phi\rangle$
  is clearly strongly affine on $\mathcal P_\mathcal L(\Omega,\mathcal F^\nabla)$.
  Let us therefore focus on the limit on the right in the display.
  It suffices to demonstrate that the second limit in the display
  is well-defined, bounded below,
  and strongly affine in its dependence on $\mu$,
  once restricted to $K$-Lipschitz measures.
  The idea is to use Theorem~15.20 in~\cite{G11},
  which concerns the non-gradient setting.
  The measure $\mu$ can be made into a shift-invariant,
  non-gradient measure by considering the values of $\phi$
  modulo $4K$.
  It is clear that the gradient of $\phi$
  can be reconstructed from this reduced height function,
  if we use the extra information that $\phi$ is $K$-Lipschitz.
  This is formalized as follows.
  Write $\hat E$ for the set $E/4K\mathbb Z$,
  and endow it with the Borel $\sigma$-algebra $\hat{\mathcal E}$
  and the Lebesgue measure $\hat\lambda$ which satisfies $\hat\lambda(\hat E)=4K$.
  Write $\hat\Omega$ for the set of functions from $\mathbb Z^d$
  to $\hat E$,
  and $\hat{\mathcal F}$ for the product $\sigma$-algebra on $\Omega$.
  Define the measure $\hat\mu$ on $(\hat\Omega,\hat{\mathcal F})$
  as follows:
  first sample a pair $(\phi,a)$ from $\mu\times (\hat\lambda/4K)$,
  the final sample $\hat\phi$ is then obtained by
  setting $\hat\phi(x)=\phi(x)-\phi(0)+a\in \hat E$.
  The measure $\hat\mu$ is clearly $\mathcal L$-invariant.
  Note that, for $\Lambda\subset\subset\mathbb Z^d$ nonempty,
  \[
    \mathcal H_{\mathcal F_\Lambda^\nabla}(\mu|\lambda^{\Lambda-1})
    =
    \mathcal H_{\hat{\mathcal F}_\Lambda}(\hat \mu|\hat \lambda^\Lambda)+\log 4K,
  \]
  where $\hat{\mathcal F_\Lambda}:=\sigma(\hat\phi(x):x\in\Lambda)$.
  By Theorem~15.20 in~\cite{G11},
  the limit
  \[
  \lim_{n\to\infty}n^{-d}\mathcal H_{\mathcal F_{\Pi_n}^\nabla}(\mu|\lambda^{\Pi_n-1})
  =
    \lim_{n\to\infty}n^{-d}\mathcal H_{\hat{\mathcal F}_{\Pi_n}}(\hat \mu|\hat \lambda^{\Pi_n})
  \]
  is well-defined, bounded below by $-\log 4K$,
  and strongly affine over $\mu$.
\end{proof}

\begin{definition}
  For $\mu\in\mathcal P_\mathcal L(\Omega,\mathcal F^\nabla)$ a $K$-Lipschitz measure,
  define
  \[
    \mathcal H(\mu|\lambda):=\lim_{n\to\infty}n^{-d}\mathcal H_{\mathcal F_{\Pi_n}^\nabla}(\mu|\lambda^{\Pi_n-1})
    \in[-\log 4K,\infty],
  \]
  the \emph{specific entropy} of $\mu$.
  This quantity is well-defined and
  strongly affine over $\mu$ due to the proof of the previous theorem.
  Remark that $\mathcal H(\mu|\lambda)\leq 0$
  whenever $E=\mathbb Z$.
\end{definition}

\begin{lemma}
  Consider a potential $\Phi\in\mathcal S_\mathcal L+\mathcal W_\mathcal L$
  and a measure $\mu\in\mathcal P_\mathcal L(\Omega,\mathcal F^\nabla)$.
  Fix $K$ minimal subject to $Kd_1\geq q$.
  If $\mu$ is not $K$-Lipschitz,
  then $\mathcal H(\mu|\Phi)=\langle\mu|\Phi\rangle=\infty$,
  and if $\mu$ is $K$-Lipschitz,
  then $\mathcal H(\mu|\Phi)=\langle\mu|\Phi\rangle+\mathcal H(\mu|\lambda)$.
\end{lemma}

\begin{proof}
  This also follows from the proof of the previous theorem.
\end{proof}

We are now able to prove Theorem~\ref{thm_existence_ergodic_minimizers}.

\begin{proof}[Proof of Theorem~\ref{thm_existence_ergodic_minimizers}]
  Suppose that $u\in\bar U_\Phi$ is an exposed
  point of $\sigma$.
  By compactness of the lower level sets of $\mathcal H(\cdot|\Phi)$
  (Theorem~\ref{thm_main_sfe})
  and by continuity of $S(\cdot)$,
  there exists a minimizer $\mu\in\mathcal P_\mathcal L(\Omega,\mathcal F^\nabla)$
  of slope $u$.
  Write $w_\mu$ for the ergodic decomposition of $\mu$.
  Since both $S(\cdot)$ and $\mathcal H(\cdot|\Phi)$
  are strongly affine (due to Proposition~\ref{propo_S_strongly_affine} and
  Theorem~\ref{thm_SFE_strongly_affine})
  and because $u$ is an exposed point,
  we observe that $w_\mu$-almost every component $\nu$ is an ergodic minimizer of slope $u$.
\end{proof}

\section{Limit equalities}
\label{section_limit_equalities}

This section provides the fundamental building
blocks for the large deviations principle
in the next section.
The motivating thesis for this section is that
$\sigma(u)$
can be approximated by integrals of $\exp-H_{\Pi_n}^0$
after restricting to height functions
which are close to the slope $u$ on $\partial^R\Pi_n$.
It is possible to be more subtle:
if one considers a measure $\mu\in\mathcal P_\mathcal L(\Omega,\mathcal F)$
with $\mathcal H(\mu|\Phi)<\infty$
and $S(\mu)\in U_\Phi$,
then one can approximate $\mathcal H(\mu|\Phi)$
by integrals of $\exp-H_{\Pi_n}^0$
after restricting to height functions
which are close to the slope $S(\mu)$ on $\partial^R\Pi_n$,
and after restricting further to height functions $\phi$
whose \emph{empirical measure} in $\Pi_n$
approximates $\mu$.
The empirical measure of $\phi$ in $\Pi_n$ is obtained
by randomly shifting $\phi$ by a vertex in $\mathcal L\cap\Pi_n$.
Analogous results for finite-range non-Lipschitz potentials
can be found in Chapter~6 in~\cite{S05}.
However, the proof presented here differs from the proof in~\cite{S05} to account for the generality of our
setting, and the specificity of the discrete Lipschitz case.


\subsection{Formal statement}

Let us first introduce some simple notation for fixing
boundary conditions.

\begin{definition}
  Write $0_\Lambda$
  for the smallest
  element in $\Lambda$
  in the dictionary
  order on $\mathbb Z^d$
  whenever $\Lambda\subset\subset\mathbb Z^d$.
  Let $u\in U_\Phi$.
If $E=\mathbb Z$,
then write
$C_\Lambda^u$ for the set of height functions
\[
  \{
    \phi\in\Omega:
    \phi_{\partial^R\Lambda}-\phi(0_\Lambda)
    =
    \phi_{\partial^R\Lambda}^u-\phi^u(0_\Lambda)
  \}\in \mathcal F^\nabla_{\partial^R\Lambda}.
\]
Now consider $E=\mathbb R$,
and fix $\varepsilon>0$.
Write $C_{\Lambda,\varepsilon}^u$
for the set
\[
\{
  \phi\in\Omega:
  |(\phi_{\partial^R\Lambda}-\phi(0_\Lambda))
  -(
  \phi_{\partial^R\Lambda}^u-\phi^u(0_\Lambda)
  )|\leq \varepsilon
\}\in \mathcal F^\nabla_{\partial^R\Lambda}.
\]
Abbreviate
$C_{\Pi_n}^u$
and
$C_{\Pi_n,\varepsilon}^u$
to $C_n^u$
and $C_{n,\varepsilon}^u$ respectively.
\end{definition}

Next, we formally define the empirical measure of a
height function $\phi$ in $\Lambda$.
Recall the definition of the basis $\mathcal B$
of the topology of weak local convergence
on $\mathcal P(\Omega,\mathcal F^\nabla)$
from Subsection~\ref{subsubsec:topo_local_conv}.

\begin{definition}
    In this definition, we adopt the following notation:
    if $\phi$ is a height function and
    $\Lambda\subset\subset\mathbb Z^d$, then write $\bar\phi_\Lambda$
    for unique extension of $\phi_\Lambda$ to $\mathbb Z^d$
    which equals $\phi(0_\Lambda)$
    on the complement of $\Lambda$.
    For $\Lambda\subset\subset\mathbb Z^d$ and $\phi\in\Omega$,
    we define the measure
    $L_\Lambda(\phi)$ by
    \[
    L_\Lambda(\phi):=\frac{1}{|\mathcal L\cap\Lambda|}\sum_{x\in\mathcal L\cap\Lambda}\delta_{\theta_x\bar \phi_{\Lambda}}.
    \]
    This is called the \emph{empirical measure}
    of $\phi$ in $\Lambda$.
    The kernel $L_\Lambda$
    is thus a probability
    kernel from $(\Omega,\mathcal F_\Lambda)$
    to $(\Omega,\mathcal F)$
    which restricts to a kernel from
    $(\Omega,\mathcal F_\Lambda^\nabla)$
    to $(\Omega,\mathcal F^\nabla)$.
    Now consider $B\in\mathcal B$.
    Write $B_\Lambda$ for the event
    $B_\Lambda:=\{\phi\in\Omega:L_\Lambda(\phi)\in B\}$; this event
    is $\mathcal F_{\Lambda}^\nabla$-measurable.
    We shall also write
    $L_n$
    and $B_n$
    for $L_{\Pi_n}$
    and $B_{\Pi_n}$
    respectively.
\end{definition}

We start with the introduction of free boundary limits,
which is slightly easier than the definition of pinned boundary limits.
For free boundary limits, we integrate
over all height functions
having the appropriate
empirical measure,
irrespective of
boundary conditions.
It will be useful to define free boundary limits also
for measures $\mu\in\mathcal P(\Omega,\mathcal F^\nabla)$
which are not shift-invariant.

\begin{definition}
  Let
  $\Lambda\subset\subset\mathbb Z^d$
  and $B\in\mathcal B$.
  The \emph{free boundary}
  estimate
  of $B$ over $\Lambda$
  is given by
  \[
    \FB_\Lambda(B):=
    -\log\int_{B_\Lambda}
    e^{-H^0_\Lambda}d\lambda^{\Lambda-1}.
  \]
  Let $\mu\in\mathcal P(\Omega,\mathcal F^\nabla)$.
  The \emph{free boundary}
  limits
  of $B$ and $\mu$
  respectively are given
  by
  \[
    \FB(B)
    :=
    \liminf_{n\to\infty}
    n^{-d}
    \FB_{\Pi_n}(B)
    \qquad\text{and}
    \qquad
    \FB(\mu):=\sup_{\text{$A\in\mathcal B$ with $\mu\in A$}}
    \FB(A).
  \]
\end{definition}

Free boundary limits should be thought of as an asymptotic upper bound
on the integral in the display,
and this is why we take the limit inferior in the definition
of $\FB(B)$---taking
into account the minus sign which appears in the definition of $\FB_\Lambda(B)$.
Indeed, the free boundary estimates
are useful in proving the upper bound on probabilities
in the large deviations principle in the next section.
Remark that it is immediate from the definition of $\FB(\mu)$
that $\FB(\cdot)$ is lower-semicontinuous on
the set of gradient measures
in the topology of weak local convergence for which $\mathcal B$
forms a basis.

Finally, we introduce pinned boundary limits,
which take into consideration also the value of $\phi$
on the boundary of $\Pi_n$.
In this case, it is the lower bound on the integral of interest
that matters to us; pinned boundary limits play a crucial
role in the proof of the lower bound on probabilities in the large
deviations principle.

\begin{definition}
  Fix
  $u\in U_\Phi$
  and $\varepsilon>0$,
  and let $\Lambda\subset\subset\mathbb Z^d$
  and $B\in\mathcal B$.
  If $E=\mathbb R$,
  then define
  \[
  \PB_{\Lambda,u,\varepsilon}(B):=
  -\log\int_{C_{\Lambda,\varepsilon}^u\cap B_\Lambda}e^{-H_\Lambda^0}d\lambda^{\Lambda-1}.
  \]
  If $E=\mathbb Z$,
  then define
  \[
  \PB_{\Lambda,u}(B):=
  -\log\int_{C_{\Lambda}^u\cap B_\Lambda}e^{-H_\Lambda^0}d\lambda^{\Lambda-1}.
  \]
  These are called the \emph{pinned boundary}
  estimates of $B$ over $\Lambda$.
  In either case,
  we set $\PB_{\Lambda,u,\varepsilon}(B):=\infty$ and $\PB_{\Lambda,u}(B):=\infty$
  whenever $u\not\in U_\Phi$.
  Consider now also some random field $\mu\in\mathcal P_\mathcal L(\Omega,\mathcal F^\nabla)$.
  The \emph{pinned boundary} limits of $B$ and $\mu$
  are defined as follows:
  \[
    \PB_{u,\varepsilon}(B):=\limsup_{n\to\infty}n^{-d}\PB_{\Pi_n,u,\varepsilon}(B)
    ,\qquad
    \PB(\mu):=
    \sup_{\text{$\varepsilon>0$ and $A\in\mathcal B$ with $\mu\in A$}}
    \PB_{S(\mu),\varepsilon}(A)
  \]
  whenever $E=\mathbb R$,
  and if $E=\mathbb Z$, then
  \[
  \PB_{u}(B):=\limsup_{n\to\infty}n^{-d}\PB_{\Pi_n,u}(B)
  ,\qquad
  \PB(\mu):=
  \sup_{\text{$A\in\mathcal B$ with $\mu\in A$}}
  \PB_{S(\mu)}(A).
  \]
\end{definition}

It is again immediate from these definitions
that for fixed $u\in U_\Phi$,
the functional $\PB(\cdot)$ is lower-semicontinuous
on the set $\{S(\cdot)=u\}\subset\mathcal P_\mathcal L(\Omega,\mathcal F^\nabla)$.

For the proof of the large deviations principle
in the next section, we require the following equalities
and inequalities.

\begin{theorem}
  \label{thm_limit_equalities_overview}
  If $\Phi\in\mathcal S_\mathcal L+\mathcal W_\mathcal L$
  and $\mu\in\mathcal P_\mathcal L(\Omega,\mathcal F^\nabla)$,
  then
  \[
    \mathcal H(\mu|\Phi)=\FB(\mu)=\PB(\mu),
  \]
  unless $E=\mathbb Z$ and $S(\mu)\in\partial U_\Phi$.
  If however $E=\mathbb Z$ and $S(\mu)\in\partial U_\Phi$,
  then
  \[
    \FB(\mu)\geq\mathcal H(\mu|\Phi).
  \]
  Finally, if $\mu\in\mathcal P(\Omega,\mathcal F^\nabla)\smallsetminus\mathcal P_\mathcal L(\Omega,\mathcal F^\nabla)$, then $\FB(\mu)=\infty$.
\end{theorem}

Free and pinned boundary limits are calculated along the
sequence $(\Pi_n)_{n\in\mathbb N}$.
This choice is convenient, but by no means necessary.
In the following sections, we do not only prove the inequalities
 presented in the theorem: we also prove some generalizations thereof
 where these quantities
 are calculated over sequences of the form
$(\Lambda_n)_{n\in\mathbb N}$
with
 $\Lambda_n:=\Lambda^{-m}(nD)$,
 where $D$ is a bounded convex subset of $\mathbb R^d$
 of positive Lebesgue measure,
 and where $m\in\mathbb Z_{\geq 0}$.
 Observe that in this notation, $\Pi_n=\Lambda_n$ for $m=0$ and $D=[0,1)^d\subset\mathbb R^d$.

 \begin{definition}
   Write $\mathcal C$ for the set of bounded convex subsets of $\mathbb R^d$
   of positive Lebesgue measure.
 \end{definition}

The definitions imply that
$\PB(\mu)\geq \FB(\mu)$ for $\mu$ shift-invariant.
In Subsection~\ref{subsec:emp_measure} we discuss free boundary limits.
In particular, we show that
$\FB(\mu)\geq\mathcal H(\mu|\Phi)$ whenever $\mu$ is shift-invariant,
and that $\FB(\mu)=\infty$ whenever $\mu$ is not shift-invariant.
In Subsection~\ref{subsec:rta} we prove that
$\PB(\mu)\leq \mathcal H(\mu|\Phi)$
whenever $\mu$ is ergodic with $S(\mu)\in U_\Phi$.
In Subsection~\ref{subsec:ea} we extend this inequality
to shift-invariant measures $\mu$ which are not ergodic.


\subsection{Free boundary limits: empirical measure argument}
\label{subsec:emp_measure}

The idea in this subsection is always
to use the set $B$, the empirical measures
$L_n(\phi)$ for $\phi\in B_n$,
as well as the subsequential limits thereof as $n\to\infty$,
to derive the desired inequalities which were mentioned in the previous subsection.
Let us first cover the case that $\mu$ is not shift-invariant.

\begin{lemma}
  \label{lemma_FB_non_shift-inv}
  If $\mu\in\mathcal P(\Omega,\mathcal F^\nabla)\smallsetminus\mathcal P_\mathcal L(\Omega,\mathcal F^\nabla)$,
  then $\FB(\mu)=\infty$.
\end{lemma}

\begin{proof}
  If $\mu$ is not shift-invariant,
  then there is a shift $\theta\in\Theta$
  and a continuous cylinder function $g:\Omega\to[0,1]$
  such that $\mu(g-\theta g)\neq 0$.
  Define $f:=g-\theta g$;
  this is a bounded continuous cylinder function
  such that $\mu(f)\neq 0$.
  Define $\varepsilon:=|\mu(f)|/2$
  and $B:=\{\nu:|\nu(f)-\mu(f)|<\varepsilon\}\in\mathcal B$.
  For $\Lambda\subset\subset\mathbb Z^d$
  fixed and for $n$ large,
  the measure $L_n(\phi)=L_{\Pi_n}(\phi)$
  restricted to $\mathcal F_\Lambda^\nabla$
  looks almost shift-invariant.
  More precisely, the sequence of functions
  \[
  \Omega\to[-1,1],\,
    \phi\mapsto L_{n}(\phi)(f)
  \]
  converges to $0$ uniformly over $\phi\in\Omega$
  as $n\to\infty$.
  This proves that $B_n=B_{\Pi_n}$ is empty for $n$ sufficiently large,
  that is, $\FB(\mu)\geq\FB(B)=\infty$.
\end{proof}

Next, we consider shift-invariant gradient random fields.

\begin{lemma}
  \label{lemma_limit_equalities_fbl}
For any $\mu\in\mathcal P_\mathcal L(\Omega,\mathcal F^\nabla)$,
we have $\FB(\mu)\geq\mathcal H(\mu|\Phi)$.
\end{lemma}

We start with the following auxiliary lemma.

\begin{lemma}
  Suppose that $B\in\mathcal B$ satisfies
  $\FB(B)<\infty$.
  Then $\bar B$ contains a shift-invariant measure
  $\mu$ with $\mathcal H(\mu|\Phi)\leq\FB(B)$.
\end{lemma}

\begin{proof}
  Write
  $\nu_n^B$ for the normalized version
  of the measure
     $1_{B_n}e^{-H_{\Pi_n}^0}\lambda^{\Pi_n-1}$
    for each $n\in\mathbb N$,
    and observe that
    \[
      \FB(B)=\liminf_{n\to\infty}n^{-d}\mathcal H_{\Pi_n}(\nu_n^B|\Phi).
    \]
    We focus
    on \emph{good} subsequences of $n$,
    that is, subsequences
    along which the limit inferior is reached.

      Write $m:\mathbb N\to\mathbb N$ for a
      sequence of integers with $m(n)\to\infty$ and $m(n)/n\to 0$
      as $n\to\infty$,
      and set
       $\Pi_n':=\Pi_n^{-m(n)}=\{m(n),\dots,n-m(n)-1\}^d
       \subset\Pi_n$.
       Fix $N\in\mathbb N$
       minimal subject to
       $N\cdot\mathbb Z^d\subset\mathcal L$,
       and let $k$ denote an integer
       multiple of $N$.
       Let $n$ denote another integer,
       which is so large that  $m(n)>k$.
       The idea is now to apply Lemma~\ref{lemma_fe_superadditivity}
       to translates of $\Pi_k$.
       In particular, if we write $\Pi_{n,k}''$
       for the set
       $\Pi_n$
       with the sets
       $\Pi_k+x$ removed
       for all $x$ in $\Pi_n'\cap (k\cdot\mathbb Z^d)$,
       then
       that lemma asserts
       that
       \[
       \numberthis
       \label{eq:lim_eq_pbl_first_inequ}
        \mathcal H_{\Pi_n}(\nu_n^B|\Phi)\geq
        \mathcal H_{\Pi_{n,k}''}(\nu_n^B|\Phi)+
        \sum_{x\in \Pi_n'\cap (k\cdot\mathbb Z^d)}
        \mathcal H_{\Pi_k+x}(\nu_n^B|\Phi)
        -e^*(\Pi_k+x).
       \]
       The set $\Pi_{n,k}''$ is
       always connected and,
       as $n\to\infty$, we have $|\Pi_{n,k}''|=o(n^d)$.
       Therefore the first term
       on the right in~\eqref{eq:lim_eq_pbl_first_inequ}
       has a lower bound of order $o(n^d)$.
       Moreover, the value of $e^*(\Pi_k+x)$
       is independent of $x$ as long as $x$ lies
       in $\mathcal L$,
       and therefore we obtain the asymptotic
       bound
       \[
       \numberthis\label{eq_Leb_here}
          \frac{1}{|\Pi_n'\cap (k\cdot\mathbb Z^d)|}
          \sum_{x\in \Pi_n'\cap (k\cdot\mathbb Z^d)}
          \mathcal H_{\Pi_k+x}(\nu_n^B|\Phi)
          \leq k^d\FB(B)+e^*(\Pi_k)+o(1)
       \]
       as $n\to\infty$
       along a good subsequence.
       Moreover, if we write
       $\mu^{n,k}$ for the measure
       \[
        \frac{1}{|\Pi_n'\cap (k\cdot\mathbb Z^d)|}
        \sum_{x\in \Pi_n'\cap (k\cdot\mathbb Z^d)}
        \theta_x\nu_n^B,
       \]
       then the previous inequality
       and
       convexity of relative entropy
         imply that
       \[
        \mathcal H_{\Pi_k}(\mu^{n,k}|\Phi)\leq k^d\FB(B)+e^*(\Pi_k)+o(1)
       \]
       as $n\to\infty$ along a good subsequence.
       We may replace the sublattice $k\cdot\mathbb Z^d$
       by another set $k\cdot\mathbb Z^d+y$
       for $y\in \mathcal L/(k\cdot\mathbb Z^d)$
       in the previous discussion,
       and by doing so and averaging further,
       it is immediate that
       the sequence of measures
       $\mu^n$ defined by
       \[
        \frac{1}{|\Pi_n'\cap\mathcal L|}
        \sum_{x\in \Pi_n'\cap\mathcal L}
        \theta_x\nu_n^B,
       \]
       also satisfies
       \[
        \mathcal H_{\Pi_k}(\mu^{n}|\Phi)\leq k^d\FB(B)+e^*(\Pi_k)+o(1)
       \]
       as $n\to\infty$
       along a good subsequence.
       Compactness of the lower level sets of relative entropy
       implies that the sequence $\mu^n$
       has a subsequential limit---at least
       when restricted to $\mathcal F^\nabla_{\Pi_k}$.
       Using a standard diagonal argument
       for convergence for all integers
       $k\in N\cdot\mathbb N$,
       one obtains a subsequential limit $\mu$
       which is shift-invariant and satisfies
       \[
        \mathcal H_{\Pi_k}(\mu|\Phi)\leq k^d\FB(B)+e^*(\Pi_k)
       \]
       for all $k$,
       that is,
       $\mathcal H(\mu|\Phi)\leq\FB(B)$.
       This measure must clearly
       lie in $\bar B$ by construction.
\end{proof}

\begin{proof}[Proof of Lemma~\ref{lemma_limit_equalities_fbl}]
Fix $\mu\in\mathcal P_\mathcal L(\Omega,\mathcal F^\nabla)$,
and suppose that $\FB(\mu)<\infty$.
Then the lower level set of
the specific free energy $M_{\FB(\mu)}$
endowed with the topology of weak local convergence
is metrizable,
and therefore we may choose for
each $n$
an open set $B^n\in\mathcal B$,
containing $\mu$
and
with $B^n\cap M_{\FB(\mu)}$ of diameter at most $1/n$ in this metric.
Then each set $\bar B^n$ contains a measure $\mu^n$
with $\mathcal H(\mu^n|\Phi)\leq \FB(B^n)\leq \FB(\mu)$.
By choice of $B^n$ we must have $\mu^n\to\mu$,
and lower-semicontinuity implies that
\[
\mathcal H(\mu|\Phi)\leq \liminf_{n\to\infty}\mathcal H(\mu^n|\Phi)
\leq
\liminf_{n\to\infty}\FB(B^n)\leq \FB(\mu).\qedhere
\]
\end{proof}

Finally, we discuss how to extend this result
to other shapes.

\begin{definition}
  For fixed $D\in\mathcal C$ and $\mu\in\mathcal P(\Omega,\mathcal F^\nabla)$,
  we write
  \[
    \FB(\mu:D):=
    \sup_{\text{$B\in\mathcal B$ with $\mu\in B$}}
      \liminf_{n\to\infty}n^{-d}\FB_{\Lambda_n}(B),
  \]
  where we write $\Lambda_n$ for $\Lambda(nD)=nD\cap\mathbb Z^d$.
\end{definition}

The previous results extend immediately as follows---the Lebesgue measure
$\operatorname{Leb}(D)$
would first appear as a factor on the left in~\eqref{eq_Leb_here} in the generalized argument.

\begin{lemma}
  Consider $D\in\mathcal C$ and $\mu\in\mathcal P(\Omega,\mathcal F^\nabla)$.
  If $\mu$ is not shift-invariant, then $\FB(\mu:D)=\infty$,
  and if $\mu$ is shift-invariant, then
  $
  \FB(\mu:D)\geq\operatorname{Leb}(D)\cdot\mathcal H(\mu|\Phi)$.
\end{lemma}


\subsection{Pinned boundary limits for $\mu$ ergodic: truncation argument}
\label{subsec:rta}

The goal of this section is to derive the following lemma.
The proof
starts with a simple reduction,
and is then intermitted to state an auxiliary result
and to give an overview of the remainder
of the proof.
The proof extends the random truncation argument in~\cite{S05} to the infinite-range Lipschitz setting.

\begin{lemma}
  \label{lemma:PBL_upper_bound}
  If $\mu\in\mathcal P_\mathcal L(\Omega,\mathcal F^\nabla)$
  is ergodic and $u:=S(\mu)\in U_\Phi$,
  then $\PB(\mu)\leq \mathcal H(\mu|\Phi)$.
\end{lemma}

\begin{proof}
  It suffices to consider the case that $\mathcal H(\mu|\Phi)<\infty$,
  which implies in particular that $\mu$ is $K$-Lipschitz.
  We first focus on the discrete case $E=\mathbb Z$,
  then generalize to the continuous case $E=\mathbb R$;
  the latter comes with some additional technical complications.

  \emph{The discrete case.}
  Pick $B\in\mathcal B$ with $\mu\in B$.
  It suffices to show that
  \[
    \mathcal H(\mu|\Phi)=\lim_{n\to\infty}n^{-d}
    \mathcal H_{\Pi_n}(\mu|\Phi)
    \geq\limsup_{n\to\infty}n^{-d}
    \PB_{\Pi_n,u}(B)
    =\limsup_{n\to\infty}n^{-d}\mathcal H_{\Pi_n}(\nu_n^B|\Phi),
  \]
  where
   $\nu_n^B$ is the normalized measure
  \[
    \nu_n^B:=\frac1Z1_{C_n^u\cap B_n}e^{-H_{\Pi_n}^0}\lambda^{\Pi_n-1}.
  \]
  Observe that $\nu_n^B$ minimizes $\mathcal H_{\Pi_n}(\cdot|\Phi)$
  over all measures which are supported on $C_n^u\cap B_n$.
  Therefore it suffices to construct a sequence of measures
  $(\mu_n)_{n\in\mathbb N}$,
  with each $\mu_n$ supported on $C_n^u\cap B_n$,
  and such that
  $\mathcal H_{\Pi_n}(\mu_n|\Phi)\leq\mathcal H_{\Pi_n}(\mu|\Phi)+o(n^d)$
  as $n\to\infty$.
  Let us now intermit the proof to give an overview
  of the remainder of the proof, before continuing.
  \let\qed\relax
\end{proof}

One continues roughly as follows.
Always take $0$ as a reference point for all gradient
measures. This means that $\mu$-almost surely $\phi(0)=0$.
Write $\phi^\pm_n$ for the largest and smallest $q$-Lipschitz
extensions of $\phi^u_{\partial^R\Pi_n}$ to $\Pi_n$ respectively,
for each $n\in\mathbb N$.
Define the random sets
\[A_n^-:=\{x\in\Pi_n:\phi(x)<\phi^-_n(x)\}\quad\text{and}\quad A_n^+:=\{x\in\Pi_n:\phi(x)>\phi^+_n(x)\}.\]
Note that $\phi_n^\pm(0)=0$ by definition; $0$
is $\mu$-almost surely not contained in $A_n^\pm$.

Since $\mu$ is ergodic and $K$-Lipschitz,
almost every sample $\phi$ from $\mu$ is asymptotically
close to $u$, in the sense of Theorem~\ref{thm:superergodresult}.
As $u$ belongs to $U_q$, the interior of the set of Lipschitz slopes,
the function $\phi_n^+$ is substantially
larger than $u$
on most vertices in
$\Pi_n$.
This means that
$\mu(|A^+_n|)=o(n^d)$
as $n\to\infty$,
and similarly $\mu(|A^-_n|)=o(n^d)$.
For each $n\in\mathbb N$,
define the measure $\mu_n^\pm$
as follows:
to draw a sample from $\mu_n^\pm$,
sample first a height function $\phi$
from $\mu$,
then replace this sample by
$\psi:=\phi_n^-\vee \phi_{\Pi_n}\wedge \phi_n^+$.
Note that $\phi$ and $\psi$ differ at at most $o(n^d)$
vertices in $\Pi_n$ on average as $n\to\infty$.
Moreover, the modified height function $\psi$ is $q$-Lipschitz
if the original height function $\phi$ was $q$-Lipschitz.
In particular, we deduce
that
\[
\mathcal H_{\Pi_n}(\mu_n^\pm|\Phi)
=\mathcal H_{\Pi_n}(\mu|\Phi)+o(n^d).
\]
The measure $\mu_n^{\pm}$
is clearly supported on $C_n^u$,
because $\phi_n^-$, $\phi^+_n$, and $\phi^u$
are equal
on $\partial^R\Pi_n$.
Using again the ergodicity of $\mu$
through Theorem~\ref{thm:superergodresult},
one can show that $\mu(B_n)\to 1$ as $n\to\infty$,
and consequently $\mu_n^\pm(B_n)\to 1$
because $\phi$ and $\psi$ agree on most vertices of $\Pi_n$.
This proves that the sequence
$(\mu_n)_{n\in\mathbb N}$ defined by
$\mu_n:=\mu_n^\pm(\cdot|B_n)$
is the desired sequence of measures.
This concludes the proof overview for $E=\mathbb Z$.
In the real case $E=\mathbb R$,
the details are more involved,
owing to the following two difficulties:
\begin{enumerate}
  \item We cannot simply replace $\phi$
  by
  $\phi_n^-\vee \phi_{\Pi_n}\wedge \phi_n^+$,
  because the measure so produced would
  not be absolutely continuous with respect to Lebesgue measure,
  \item We only have a bound on $H_{\{x\}}(\phi)$
  if $\phi$ is $q_\varepsilon$-Lipschitz
  at $x$; it is not sufficient to make
  modifications which are $q$-Lipschitz.
\end{enumerate}
Let us now state Theorem~\ref{thm:superergodresult}
before continuing the proof of   Lemma~\ref{lemma:PBL_upper_bound}.

\begin{theorem}
  \label{thm:superergodresult}
  Consider $\mu\in\mathcal P_\mathcal L(\Omega,\mathcal F^\nabla)$
  ergodic.
Then $L_n(\phi)(f)\to\mu(f)$ as $n\to\infty$ for $\mu$-almost every $\phi$, for any bounded cylinder function $f$.
Suppose now that $\mu$
is also $K$-Lipschitz with slope $u:=S(\mu)$.
Then $\mu$-almost surely $\|\phi_{\Pi_n}-\phi(0)-u|_{\Pi_n}\|_\infty\leq \varepsilon n$ for $n$ sufficiently large,
    for any fixed constant $\varepsilon>0$.
\end{theorem}

The first assertion is the ergodic theorem.
The second assertion is straightforward:
in the Lipschitz setting,
the height difference $(\phi(x)-\phi(0))/\|x\|_1$ is approximately equal
to the average of the gradient---which is bounded in magnitude---of
$\phi$ over a large
set $\Lambda\subset\subset\mathbb Z^d$.

\begin{proof}[Continuation of the proof of Lemma~\ref{lemma:PBL_upper_bound}]
  Recall that $E=\mathbb Z$.
  By taking a smaller set $B\in\mathcal B$ if necessary,
  we suppose that $B$ is of the form
  \[
    B=\{\nu\in\mathcal P(\Omega,\mathcal F^\nabla):\text{$|\nu(f_i)-\mu(f_i)|<2\eta$ for all $i$}\}
  \]
  for a finite collection $(f_i)_i$ of continuous cylinder functions
  $f_i:\Omega\to[0,1]$ and for some $\eta>0$,
  and we write $B^*$ for the same set with $2\eta$
  replaced by $\eta$.
  The ergodic theorem asserts that $\mu(B_n^*)\to 1$
  as $n\to\infty$.
  Consider $\mu$ a non-gradient measure
  on $(\Omega,\mathcal F)$ by taking $0\in\Pi_n$
  as a reference point:
  this means that $\phi(0)=0$ almost surely in $\mu$.

  Recall the definitions of $\phi^\pm_n$
  and $A_n^\pm$ from the proof overview,
  and
  claim that
  $\mu(|A_n^\pm|)=o(n^d)$ as $n\to\infty$.
  The function $\phi^+_n$ is pyramid-shaped,
  as in Figure~\ref{fig:random_truncation}---that
  figure concerns the more complicated continuous setting
  $E=\mathbb R$, but the shape of $\phi_n^+$ is the same.
  Formally, this means that there exist constants $C'>0$ and $\varepsilon'>0$
  such that for any $n\in\mathbb N$ and for any $x\in\Pi_n$,
  \begin{equation}
    \label{eq:phipyr}
    \phi^+_n(x)\geq u(x)+\varepsilon' d_1(x,\partial^R\Pi_n)-C'.
  \end{equation}
  This is a consequence of Lemma~\ref{lemma_lipschitz_q_and_norm} and the fact that $u$ is in $U_\Phi$, the \emph{interior} of the set of slopes
  $u'$ for which $u'|_\mathcal L$ is $q$-Lipschitz.
  Now fix $\varepsilon''>0$.
  By~\eqref{eq:phipyr}, the number of points $x\in\Pi_n$
  at which $\phi_n^+(x)\leq u(x)+\varepsilon''n$
  is bounded from above by $(2d\varepsilon''/\varepsilon')n^d+o(n^d)$
  as $n\to\infty$.
  Theorem~\ref{thm:superergodresult} tells us that
  \[\mu(|\{x\in\Pi_n:\phi(x)> u(x)+\varepsilon'' n\}|)=o(n^d).\]
  Combining the two bounds gives
  $\mu(|A_n^+|)\leq (2d\varepsilon''/\varepsilon')n^d+o(n^d)$.
  The constant $\varepsilon''$ may be chosen arbitrarily small,
  and therefore we obtain $\mu(|A_n^+|)=o(n^d)$.
  In the same spirit, one obtains $\mu(|A_n^-|)=o(n^d)$.
  This proves the claim.

  Next, we construct for each $n\in\mathbb N$
  a new measure $\mu^+_n$, the \emph{upper truncation} of $\mu$.
  To sample from $\mu^+_n$,
  first sample $\phi$ from $\mu$,
  then replace
  $\phi(x)$
  by $\phi^+_n(x)$
  for any $x\in A_n^+$.
  This means that the distribution of
  $\phi_{\Pi_n}$ in $\mu_n^+$ is the same as the distribution of
  $\phi_{\Pi_n}\wedge \phi^+_n$ in $\mu$.
  Assert that
  \begin{equation*}
    \mathcal H_{\Pi_n}(\mu_n^+|\Phi)=\mathcal H_{\Pi_n}(\mu|\Phi)+o(n^d).
  \end{equation*}

  We present an alternative three-stage construction of $\mu_n^+$,
  and demonstrate that the free energy changes by no more than $o(n^d)$ at every stage.
  Write $\mathcal S$ for the set of finite subsets of $\mathbb Z^d$,
  which is countable, and write $\alpha$ for the counting measure on $\mathcal S$.
  Write $\mathcal G_n$ for the smallest $\sigma$-algebra
  on $\Omega\times\mathcal S$
  containing $A\times\{\Lambda\}$ for any $A\in\mathcal F_{\Pi_n}^\nabla$
  and any $\Lambda\subset\Pi_n$.

   For the first stage, write $\tilde \mu_n$ for the measure $\mu$ with the
       set $A_n^+$ attached to every sample $\phi\in\Omega$.
       The measure $\tilde \mu_n$ is thus a probability measure on the measurable space $(\Omega\times \mathcal S,\mathcal G_n)$.
       Moreover, the distribution of $\phi_{\Pi_n}$ is the same in $\mu$ as it is in $\tilde \mu_n$,
       and the set $A_n^+$ depends deterministically on $\phi_{\Pi_n}$.
       Therefore
       \begin{equation}
         \label{eq:reftoentrterm}
         \mathcal H_{\Pi_n}(\mu|\Phi)
         =\mathcal H_{\mathcal F_{\Pi_n}^\nabla}\left(\mu\middle|e^{-H_{\Pi_n}^0}\lambda^{\Pi_n-1}\right)
         =\mathcal H_{\mathcal G_n}\left(\tilde\mu_n\middle|\left(e^{-H_{\Pi_n}^0}\lambda^{\Pi_n-1}\right)\times\alpha\right).
       \end{equation}

     For the second stage, introduce a new measure $\tilde\mu_n^+$ on $(\Omega\times\mathcal S,\mathcal G_n)$.
     To sample from $\tilde\mu^+_n$, sample first a pair $(\phi,A)$ from $\tilde\mu_n$,
     then replace $\phi(x)$ by $\phi_n^+(x)$ for every $x\in A$.
     Remark that the distribution of $\phi_{\Pi_n}$ is the
     same in $\tilde \mu_n^+$ as it is in $\mu_n^+$.
     Write $A'$ for the set $\Pi_n\smallsetminus A$.
     The entropies of $\tilde\mu_n$ and $\tilde \mu_n^+$
     (relative to the reference measure in the final term of~\eqref{eq:reftoentrterm})
     can be calculated in three steps.
     First, calculate the entropy of the choice of the set $A$.
     Second, calculate the entropy of the choice of the values of $\phi$
     on $A'$.
     Third, calculate the entropy of the choice of the values of $\phi$
     on $A$.
     In the construction of $\tilde\mu_n^+$ we only change the values
     of $\phi$ on $A$, and therefore the third step is the only step that produces a different entropy term.
     We have
     \begin{align*}
       &
\mathcal H_{\mathcal G_n}\left(\tilde\mu_n^+\middle|\left(e^{-H_{\Pi_n}^0}\lambda^{\Pi_n-1}\right)\times\alpha\right)
-
\mathcal H_{\mathcal G_n}\left(\tilde\mu_n\middle|\left(e^{-H_{\Pi_n}^0}\lambda^{\Pi_n-1}\right)\times\alpha\right)
         \\&\qquad=\int \left(
         \mathcal H\left(\delta_{\phi_n^+|_A}\middle|e^{-H_{A,\Pi_n}\left(\cdot,\phi_{A'}\right)}\lambda^{A}\right)
         -
         \mathcal H\left(\mu^{(A,\phi_{A'})}\pi_A\middle|e^{-H_{A,\Pi_n}\left(\cdot,\phi_{A'}\right)}\lambda^{A}\right)
       \right)d\tilde\mu_n(\phi,A)
       \\&\qquad=
       \tilde\mu_n\left(H_{A,\Pi_n}(\phi\wedge \phi_n^+)-H_{A,\Pi_n}(\phi)\right)-
       \int
       \mathcal H\left(\mu^{(A,\phi_{A'})}\pi_A\middle|\lambda^{A}\right)
       d\tilde\mu_n(\phi,A).\numberthis\label{eq:finalentrdiff}
     \end{align*}
     In these equations, $\delta$ denotes the Dirac measure,
     $\pi_A$ is the projection kernel
     onto $A$, and
     $\mu^{(A,\phi_{A'})}$
     denotes the original measure $\mu$ conditioned on seeing $A_n^+=A$
     and on the values of $\phi$ on the set $A'$.
     For the first term in~\eqref{eq:finalentrdiff} we observe
     that
     \[
      \left|
      \tilde\mu_n\left(H_{A,\Pi_n}(\phi\wedge \phi_n^+)-H_{A,\Pi_n}(\phi)\right)
      \right|
      =O(\mu(|A_n^+|))=o(n^d);
     \]
     this follows from the claim and~\eqref{eq_discrete_singleton_hamil_bound}---noting
     that $\phi$ and $\phi\wedge \phi_n^+$ are $q$-Lipschitz.
     For $\tilde\mu_n$-a.e.\ $(\phi,A)$,
     we observe that
     the measure $\mu^{(A,\phi_{A'})}$
     produces $K$-Lipschitz height functions
     almost surely,
     and consequently the same measure---restricted to $A$---is
     supported on a set of cardinality
     at most $(2K+1)^{|A|}$.
     Conclude that the second term in~\eqref{eq:finalentrdiff}
     is bounded absolutely by
     \[
     \mu(|A_n^+|)\log(2K+1)=o(n^d).
     \]

  To sample from $\mu^+_n$, sample a pair $(\phi,A)$ from $\tilde\mu^+_n$,
  then simply forget about the set $A$.
  This is the third stage.
  Write $\nu_n$ for marginal of $\tilde\mu^+_n$ on $\mathcal S$.
   Then
  \[
  \mathcal H_{\Pi_n}(\mu_n^+|\Phi)+\mathcal H(\nu_n|\alpha) \leq \mathcal H_{\mathcal G_n}\left(\tilde\mu^+_n\middle|\left(e^{-H_{\Pi_n}^0}\lambda^{\Pi_n-1}\right)\times\alpha\right)\leq\mathcal H_{\Pi_n}(\mu_n^+|\Phi).
  \]
  Evidently $\mathcal H(\nu_n|\alpha)\leq 0$; the goal is to
   find a lower bound on $\mathcal H(\nu_n|\alpha)$.
   The measure $\nu_n$ is a probability measure on the set of subsets of $\Pi_n$
   and we also know that
   $\nu_n(|A|)=\mu(|A_n^+|)$.
   The entropy of $\nu_n$ is minimized (among all probability measures with these two properties)
   if one samples from $\nu_n$
   by flipping a coin independently for every vertex $x\in\Pi_n$
   to determine if $x\in A$.
   The Bernoulli parameter of the coin is $\mu(|A_n^+|)/n^d$
   so that $\nu_n(|A|)=\mu(|A_n^+|)$.
   Write $f(p)=p\log p + (1-p)\log(1-p)$,
   the entropy of a Bernoulli trial with parameter $p$.
   Then the entropy of the entropy-minimizing measure
   is
   $n^df(\mu(|A_n^+|)/n^d)$.
   Now $\lim_{p\to 0}f(p)=0$
   and therefore $\mathcal H(\nu_n|\alpha)=o(n^d)$.
   Conclude that the assertion holds true,
   that is,
   \[
   \mathcal H_{\Pi_n}(\mu_n^+|\Phi)
   =\mathcal H_{\Pi_n}(\mu|\Phi)+o(n^d).\]

   The measure $\mu_n^\pm$ is now obtained from $\mu_n^+$
   by applying a \emph{lower truncation}.
   To sample from $\mu_n^\pm$, sample first a height function
   $\phi$ from $\mu_n^+$,
   then replace $\phi(x)$ by $\phi_n^-(x)$
   for any $x\in A_n^-$.
   Alternatively, sample $\phi$ from $\mu$,
   then replace $\phi_{\Pi_n}$ by
   $\phi_n^-\vee \phi_{\Pi_n}\wedge \phi_n^+$.
   By similar arguments as before we have
   \[
   \mathcal H_{\Pi_n}(\mu_n^\pm|\Phi)=
   \mathcal H_{\Pi_n}(\mu_n^+|\Phi)+o(n^d)
   =\mathcal H_{\Pi_n}(\mu|\Phi)+o(n^d)
    \]
    as $n\to\infty$.
   The measure $\mu_n^\pm$ is supported on $C_n^u$.
   Moreover, because $\mu(|A_n^\pm|)=o(n^d)$
   and because $\mu(B_n^*)\to 1$ as $n\to\infty$,
   we have $\mu^\pm_n(B_n)\to 1$ as $n\to\infty$.
   In particular, this means that the measures $\mu_n:=\mu_n^\pm(\cdot|B_n)$
   are supported on $C_n^u\cap B_n$
   and satisfy
   $\mathcal H_{\Pi_n}(\mu_n|\Phi)\leq \mathcal H_{\Pi_n}(\mu|\Phi)+o(n^d)$
   as $n\to\infty$. This concludes the proof for $E=\mathbb Z$.

   \emph{The continuous case.}
     Fix $\varepsilon>0$ so small that $\phi^u$ is $q_{4\varepsilon}$-Lipschitz,
      and pick $B\in\mathcal B$ with $\mu\in B$.
     Assume a choice of $B$ and $B^*$ as for the discrete case.
      It suffices to find a sequence of measures $(\mu_n)_{n\in\mathbb N}$
     with $\mu_n$ supported on $C_{n,2\varepsilon}^u\cap B_n $
     and with $\mathcal H_{\Pi_n}(\mu_n|\Phi)\leq \mathcal H_{\Pi_n}(\mu|\Phi)+o(n^d)$.
   Write $\phi_n^\pm$
   for the largest and smallest $q_{3\varepsilon}$-Lipschitz
   extensions of $\phi^u_{\partial^R\Pi_n}$
   to $\Pi_n$ respectively, for each $n\in\mathbb N$.
   Take again $0$ as reference vertex for the gradient setting
   (as for the discrete case),
   and define the random sets
   \[
     A_n^-:=\{x\in\Pi_n:\phi(x)<\phi^-_n(x)-2\varepsilon\}\quad\text{and}\quad A_n^+:=\{x\in\Pi_n:\phi(x)>\phi^+_n(x)+2\varepsilon\}.
     \]
     Note that $\phi_n^\pm(0)=0$ by definition,
     and therefore $\mu$-almost surely $0\not\in A_n^\pm$.
     Observe that
     $\mu(|A_n^\pm|)=o(n^d)$ by arguments identical to the case $E=\mathbb Z$;
     one can show that $\phi^+_n$ is pyramid-shaped in the sense of~\eqref{eq:phipyr}
     because $\phi^u$ is $q_{4\varepsilon}$-Lipschitz and because we chose the extension
     $\phi^+_n$ to be the largest $q_{3\varepsilon}$-Lipschitz extension.

\begin{figure}
	\centering
  \begin{tikzpicture}[x=0.07cm,y=0.03cm]
	  \input{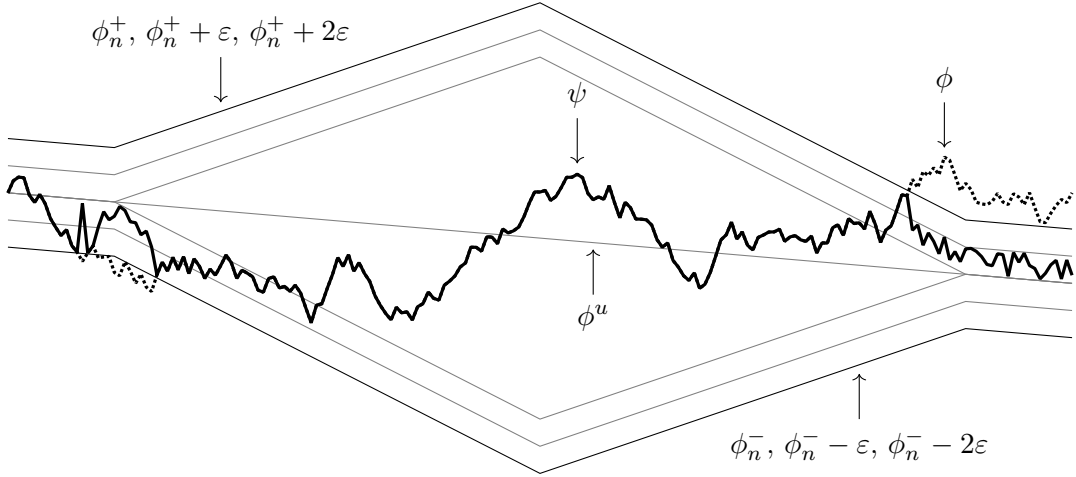}
		\draw[<-] (40,40) -- (40,60) node[above] {$\phi^+_n$, $\phi^+_n+\varepsilon$, $\phi^+_n+2\varepsilon$};
		\draw[<-] (160,-80) -- (160,-100) node[below] {$\phi^-_n$, $\phi^-_n-\varepsilon$, $\phi^-_n-2\varepsilon$};
		\draw[<-] (176,20) -- (176,40) node[above] {$\phi$};
		\draw[<-] (107,13) -- (107,33) node[above] {$\psi$};
		\draw[<-] (110,-25) -- (110,-45) node[below] {$\phi^u$};

  \end{tikzpicture}
  \caption[The random truncation for $E=\mathbb R$]{The random truncation for $E=\mathbb R$.
	The randomly truncated sample $\psi$
  remains between $\phi^-_n-2\varepsilon$ and $\phi^+_n+2\varepsilon$.}
	\label{fig:random_truncation}
\end{figure}

     For each $n\in\mathbb N$ we construct a new measure $\mu_n^+$ on $(\Omega,\mathcal F^\nabla_{\Pi_n})$, the \emph{upper truncation}
     of $\mu$.
     Let $(X(x))_{x\in\mathbb Z^d}$ be a process of i.i.d.\
     random variables, uniformly random
     in the interval $[0,\varepsilon]$,
     in some new measure $\nu$.
     To sample from $\mu_n^+$, first sample $(\phi,X)$ from $\mu\times\nu$.
     Then, for each $x\in A_n^+$, replace
     $\phi(x)$ by $\phi_n^+(x)+X(x)$.
     Figure~\ref{fig:random_truncation}
     displays the original function $\phi$
     and the randomly truncated function $\psi$;
     the upper truncation is located on the right hand side,
     a lower truncation (which is defined at a later stage) occurs
     on the left.
     The new measure $\mu_n^+$ is absolutely continuous with respect to $\lambda^{\Pi_n-1}$
     because we replaced each value $\phi(x)$ by a continuously distributed random variable.
     Assert
     that
     \begin{equation}
       \label{eq:bigclaimREAL}
       \mathcal H_{\Pi_n}
       (\mu_n^+|\Phi)
       \leq
       \mathcal H_{\Pi_n}(\mu|\Phi)+o(n^d).
     \end{equation}
     Again, we present an alternative three-stage construction of $\mu_n^+$,
     and we demonstrate that the entropy does not increase by more than $o(n^d)$ at every stage.

     For the first stage, write $\tilde \mu_n$ for the measure $\mu$ with the
         set $A_n^+$ attached to every sample $\phi\in\Omega$.
         Define $\alpha$ and $\mathcal G_n$ as before.
         The measure $\tilde\mu_n$ is a probability measure on $(\Omega\times\mathcal S,\mathcal G_n)$.
         Note that~\eqref{eq:reftoentrterm} holds for this measure
         as
          $A_n^+$ depends deterministically on $\phi_{\Pi_n}$.

            For the second stage, introduce a new measure $\tilde\mu_n^+$ on $(\Omega\times\mathcal S,\mathcal G_n)$.
            To sample from $\tilde\mu_n^+$, sample first a triple $(\phi,A,X)$ from $\tilde\mu_n\times\nu$,
            then replace $\phi(x)$ by $\phi_n^+(x)+X(x)$ for every $x\in A$.
            Write $A'$ for $\Pi_n\smallsetminus A$.
            Write $\psi$ for the function on
            $\Pi_n$ defined by
            $\psi_A=\phi^+_n|_A+X_A$ and $\psi_{A'}=\phi_{A'}$.
            One calculates the entropies of $\tilde\mu_n$ and $\tilde \mu_n^+$
            as in the discrete case to deduce that
            \begin{align*}
              &
              \mathcal H_{\mathcal G_n}\left(\tilde\mu_n^+\middle|\left(e^{-H_{\Pi_n}^0}\lambda^{\Pi_n-1}\right)\times\alpha\right)
              -
   \mathcal H_{\mathcal G_n}\left(\tilde\mu_n\middle|\left(e^{-H_{\Pi_n}^0}\lambda^{\Pi_n-1}\right)\times\alpha\right)
              \\&=\int \left(
              \mathcal H\left((\nu+\phi_n^+)\pi_A\middle|e^{-H_{A,\Pi_n}\left(\cdot,\phi_{A'}\right)}\lambda^{A}\right)
              -
              \mathcal H\left(\mu^{(A,\phi_{A'})}\pi_A\middle|e^{-H_{A,\Pi_n}\left(\cdot,\phi_{A'}\right)}\lambda^{A}\right)
            \right)d\tilde\mu_n(\phi,A)
            \\&=
            \tilde\mu_n\times\nu\left(H_{A,\Pi_n}(\psi)-H_{A,\Pi_n}(\phi)\right)-
            \tilde\mu_n(|A|)\log\varepsilon
            -
            \int
            \mathcal H\left(\mu^{(A,\phi_{A'})}\pi_A\middle|\lambda^{A}\right)
            d\tilde\mu_n(\phi,A).
            \end{align*}
            In these equations, $\mu^{(A,\phi_{A'})}$
            denotes the original measure $\mu$ conditioned on seeing $A_n^+=A$
            and on the values of $\phi$ on the set $A'$.
            By $\nu+\phi_n^+$ we simply mean the measure obtained
            by shifting each sample $X$ from $\nu$ by $\phi_n^+$.
            As in the discrete setting,
            the last two terms have an upper bound of
            order $o(n^d)$
            as $n\to\infty$.
            It suffices to find an appropriate upper bound
            for the first term in the final expression.

            Let $(\mathbb A,q)$
            denote the local Lipschitz constraint.
            By Proposition~\ref{propo_qqq},
            it is possible
            to find a constant $0<\varepsilon'\leq \varepsilon$,
            such that
            for any $\{x,y\}\in\mathbb A$, we have
            \[
              q_{\varepsilon'}(x,y)\geq q(x,y)-\varepsilon
              .
            \]
            Claim that $\tilde\mu_n\times\nu$-almost surely,
            $\psi$ is $q_{\varepsilon'}$-Lipschitz at every $x\in A$.
            In other words, we claim that
            \begin{equation}\label{eq:cutclaim}
              -q_{\varepsilon'}(y,x)\leq \psi(y)-\psi(x)
              \leq
              q_{\varepsilon'}(x,y)
            \end{equation}
            whenever $x\in A$, $y\in\Pi_n$, and $\{x,y\}\in\mathbb A$.
            Suppose first that $y\in A$.
            The function $\phi_n^+$ is $q_{3\varepsilon}$-Lipschitz
            and $0\leq (\psi-\phi_n^+)_{\{x,y\}}=X_{\{x,y\}}\leq \varepsilon$ for $x,y\in A$,
            and therefore~\eqref{eq:cutclaim}
            holds true with $\varepsilon'$
            replaced by $\varepsilon$.
            But $q_\varepsilon\leq q_{\varepsilon'}$,
            which implies~\eqref{eq:cutclaim} without said replacement.
            Now suppose that $y\not\in A$, so that $\psi(y)-\psi(x)=\phi(y)-\phi_n^+(x)-X(x)$.
            For the righthand inequality of~\eqref{eq:cutclaim}
            we have (almost surely)
            \[
            \phi(y)-\phi_n^+(x)-X(x)\leq (\phi_n^+(y)+2\varepsilon)-\phi_n^+(x)\leq q_{3\varepsilon}(x,y)+2\varepsilon\leq q_{\varepsilon}(x,y)
            \leq q_{\varepsilon'}(x,y).
            \]
            For the inequality on the left we see that (using $\phi_n^+(x)+X(x)\leq\phi(x)-\varepsilon$ for the first inequality)
            \[\phi(y)-\phi_n^+(x)-X(x)\geq  \phi(y)-\phi(x)+\varepsilon\geq -q(y,x)+\varepsilon\geq-q_{\varepsilon'}(y,x).\]
            The middle inequality in this equation is due to the fact that $\phi$ is $\mu$-almost surely $q$-Lipschitz.
            This proves the claim.

         By the claim and~\eqref{eq_cts_singleton_hamil_bound},
         we have
         \[
           \tilde\mu_n\times\nu\left(H_{A,\Pi_n}(\psi)\right)\leq O(\mu(|A_n^+|))=o(n^d).
           \]
           For the other Hamiltonian we simply observe that
           \[
             \tilde\mu_n\times\nu\left(H_{A,\Pi_n}(\phi)\right)
             =
             \tilde\mu_n\left(H_{A,\Pi_n}(\phi)\right)
             \geq -\|\Xi\|\mu(|A_n^+|)=o(n^d).
           \]
           Putting all estimates together, we see that
           \begin{equation*}
             \mathcal H_{\mathcal G_n}\left(\tilde\mu_n^+\middle|\left(e^{-H_{\Pi_n}^0}\lambda^{\Pi_n-1}\right)\times\alpha\right)
             \leq
             \mathcal H_{\Pi_n}(\mu|\Phi)+o(n^d).
           \end{equation*}
           To prove the original assertion,
           simply observe that, as in the discrete
           case,
           forgetting about the information
           encoded in the set $A$ changes the entropy
           of $\tilde\mu_n^+$ by no more than $o(n^d)$:
           \[
           \mathcal H_{\Pi_n}(\mu_n^+|\Phi)=
           \mathcal H_{\mathcal G_n}\left(\tilde\mu_n^+\middle|\left(e^{-H_{\Pi_n}^0}\lambda^{\Pi_n-1}\right)\times\alpha\right)
           +o(n^d).
           \]
           This proves the assertion~\eqref{eq:bigclaimREAL}.

         Finally one constructs a \emph{lower truncation} $\mu_n^\pm$ from $\mu_n^+$.
         To sample from $\mu_n$, one first samples $\phi$ from $\mu_n^+$.
         Then, for every $x\in A_n^-$,
         one resamples $\phi(x)$ independently and uniformly at random from the interval $[\phi_n^-(x)-\varepsilon,\phi_n^-(x)]$.
         As before, we have
         \[
         \mathcal H_{\Pi_n}(\mu_n^\pm|\Phi)\leq \mathcal H_{\Pi_n}(\mu_n^+|\Phi)+o(n^d)\leq \mathcal H_{\Pi_n}(\mu|\Phi)+o(n^d).
         \]
         Now $\phi_n^--2\varepsilon\leq \phi_{\Pi_n}\leq \phi_n^++2\varepsilon$ almost surely in the measure $\mu_n^\pm$
         and this implies in particular that
         \[\phi^u_{\partial^R\Pi_n}-2\varepsilon\leq\phi_{\partial^R\Pi_n}\leq \phi^u_{\partial^R\Pi_n}+2\varepsilon,\]
         that is, $\mu_n^\pm$ is supported on $C_{n,2\varepsilon}^u$.
         Since $\mu(B_n^*)\to 1$
         and $\mu(|A_n^\pm|)=o(n^d)$ as $n\to\infty$,
         we have $\mu_n^\pm(B_n)\to 1$ as $n\to\infty$.
         This proves that the sequence $\mu_n:=\mu_n^\pm(\cdot|B_n)$
         has the desired properties.
\end{proof}


%
%
%
%

We now proceed as for free boundary limits,
and define pinned boundary limits over other
Van Hove sequences.

\begin{definition}
  Fix $D\in\mathcal C$
  and
  $m\in\mathbb Z_{\geq 0}$,
  and write $\Lambda_n$ for $\Lambda^{-m}(nD)=(nD\cap\mathbb Z^d)^{-m}$.
  Consider $\mu\in\mathcal P_\mathcal L(\Omega,\mathcal F^\nabla)$.
  If $E=\mathbb Z$, then define
  \[
    \PB(\mu:D,m):=
    \sup_{\text{$B\in\mathcal B$ with $\mu\in B$}}
      \limsup_{n\to\infty}n^{-d}\PB_{\Lambda_n,S(\mu)}(B),
  \]
  and if $E=\mathbb R$, then define
  \[
    \PB(\mu:D,m):=
    \sup_{\text{$\varepsilon>0$ and $B\in\mathcal B$ with $\mu\in B$}}
      \limsup_{n\to\infty}n^{-d}\PB_{\Lambda_n,S(\mu),\varepsilon}(B).
  \]
  Finally, write
  $\PB^*(\mu):=\sup_{(D,m)\in \mathcal C\times\mathbb Z_{\geq 0}}
  \PB(\mu:D,m)/\operatorname{Leb}(D)$.
\end{definition}

It is immediate that $\PB^*(\mu)\geq \PB(\mu)$
because one can take $D=[0,1)^d$ and $m=0$
in the supremum in this new definition.
By reordering the suprema in the definitions, it is also clear that
$\PB^*$ is lower-semicontinuous on the set
$\{S(\cdot)=u\}\subset\mathcal P_\mathcal L(\Omega,\mathcal F^\nabla)$
for any $u$.


\begin{lemma}
  \label{lemma_pb_convex_shape}
  Consider $D\in\mathcal C$,
  $m\in\mathbb Z_{\geq 0}$,
  and $\mu\in\mathcal P_\mathcal L(\Omega,\mathcal F^\nabla)$
 ergodic with $S(\mu)\in U_\Phi$.
Then   \[
  \PB(\mu:D,m)\leq\operatorname{Leb}(D)\cdot\mathcal H(\mu|\Phi).\]
  In other words, $\PB^*(\mu)\leq \mathcal H(\mu|\Phi)$.
\end{lemma}

\begin{proof}
   Write $u:=S(\mu)$
   and $\Lambda_n:=\Lambda^{-m}(nD)$,
   and fix $B\in\mathcal B$ with $\mu\in B$.
  The truncation argument in the proof of Lemma~\ref{lemma:PBL_upper_bound}
  implies that
  $\PB_{\Lambda_n,u}(B)\leq \mathcal H_{\Lambda_n}(\mu|\Phi)+o(n^d)$
  as $n\to\infty$
  if $E=\mathbb Z$,
  and   $\PB_{\Lambda_n,u,\varepsilon}(B)\leq \mathcal H_{\Lambda_n}(\mu|\Phi)+o(n^d)$
  as $n\to\infty$
  for any $\varepsilon>0$
  if $E=\mathbb R$.
  Therefore it suffices to demonstrate
  that
  \[
  \numberthis\label{eq:wrongsideineq}
    \limsup_{n\to\infty}n^{-d}\mathcal H_{\Lambda_n}(\mu|\Phi)\leq \operatorname{Leb}(D)\cdot
    \mathcal H(\mu|\Phi).
  \]
  Without loss of generality,
  we suppose that $D\subset [\varepsilon,1-\varepsilon]^d\subset [0,1)^d\subset\mathbb R^d$
  for some $\varepsilon>0$.
  Define $\Delta_n:=\Pi_n\smallsetminus\Lambda_n$.
  Then
   $n^{-d}|\Delta_n|\to 1-\operatorname{Leb}(D)$
  as $n\to\infty$,
  and therefore Proposition~\ref{propo_large_set_limit}
  implies that
  \[
  \liminf_{n\to\infty}n^{-d}\mathcal H_{\Delta_n}(\mu|\Phi)\geq (1-\operatorname{Leb}(D))\cdot
  \mathcal H(\mu|\Phi).
  \]
  Note that~\eqref{eq:wrongsideineq}
  now follows from the fact that
  \[
    n^{-d}(\mathcal H_{\Lambda_n}(\mu|\Phi)+
    \mathcal H_{\Delta_n}(\mu|\Phi))
    \leq
    \mathcal H(\mu|\Phi)+o(1)
  \]
  as $n\to\infty$.
\end{proof}


\subsection{Pinned boundary limits for $\mu$ not ergodic: washboard argument}

\label{subsec:ea}

The purpose of this subsection is to demonstrate that $\PB^*(\mu)\leq\mathcal H(\mu|\Phi)$
for any shift-invariant random field $\mu$ with $S(\mu)\in U_\Phi$.
The previous subsection proved this for $\mu$
ergodic.
First, we demonstrate that $\PB^*$ is convex (Lemma~\ref{lemma_PB_star_convex})---recall that $\mathcal H(\cdot|\Phi)$
is affine.
The idea is then to use lower-semicontinuity of $\PB^*$
in the topology of weak local convergence to derive the inequality
for all non-ergodic measures (Lemma~\ref{lemma_PB_star_uncontrained_ineq}).
Extra care must be taken whenever $E=\mathbb Z$,
because in that case there exist
ergodic measures with finite specific free energy
which have their slope in $\partial U_\Phi$
rather than $U_\Phi$.
This pathology is dealt with in Lemma~\ref{lemma_ergodic_approx_aux_1}.

\begin{lemma}
  \label{lemma_PB_star_convex}
  The functional $\PB^*$ is convex.
\end{lemma}

Consider $\nu_1,\nu_2\in\mathcal P_\mathcal L(\Omega,\mathcal F^\nabla)$
with $S(\nu_1),S(\nu_2)\in U_\Phi$,
and define $\mu:=(1-t)\nu_1+t\nu_2$
for some $t\in(0,1)$.
If we take the value of
$\PB^*(\nu_1)$
and $\PB^*(\nu_2)$ for granted,
then we look for an upper bound on
$\PB^*(\mu)$.
This means that we look for asymptotic
lower bounds on the integrals defining
the pinned boundary estimates of $\mu$.

The proof of the lemma uses a general strategy
which produces an asymptotic lower bound
on this particular integral,
and which is used again twice
in this article:
in the lower bound on probabilities in the proof of the large
deviations principle
in Subsection~\ref{subsection:lower_bound_on_probabilities_in_the_LDP},
and when constructing
the contradiction
which leads to a proof of strict convexity of the surface tension
in Subsection~\ref{subsection_attachment_applications}.
The general idea is as follows:
Lemma~\ref{lemma:PBL_upper_bound},
and later (once it is proven)
Lemma~\ref{lemma_PB_star_uncontrained_ineq},
 provide the fundamental building blocks for the lower
bounds.
One then shows that these building blocks
can be put together without
gaining too much energy,
that is, without decreasing the value of the integral
of interest by too much.
For this, one appeals to
Theorem~\ref{general_lips_ext_less_general},
which allows one to find suitable discrete approximations
of continuous Lipschitz profiles,
and
the upper attachment lemma (Lemma~\ref{lemma_attachment}),
which allows one to bound the energy increase
due to combining height functions defined on different parts
of $\mathbb Z^d$.
This is already sufficient to understand the
macroscopic shape of the height functions. In the
context of boundary limits,
this is expressed through the pinning
of the height functions
on the boundary $\partial^R\Lambda$
of the set $\Lambda$ of interest---essentially
by restricting to the set $C_{\Lambda}^u$
or $C_{\Lambda,\varepsilon}^u$.
It is, however, also necessary
to understand the behavior of the local statistics
of the height functions---expressed
in the boundary limits through the
sets $B_\Lambda$---under the operation of putting
together the fundamental building blocks.
For this, one appeals to the following result.

\begin{proposition}
  \label{propo_compare_empirical_measures}
  Consider some set $D\in\mathcal C$
  and a nonnegative integer $m\in\mathbb Z_{\geq 0}$.
  Consider also
  some finite family $(D_i,m_i)_i\subset
  \mathcal C\times\mathbb Z_{\geq m}$
  with the sets $D_i$
  disjoint and contained in $D$.
  Write $\Lambda_n:=\Lambda^{-m}(nD)$
  and $\Lambda_n^i:=\Lambda^{-m_i}(nD_i)$.
  Then for any cylinder function $f:\Omega\to[0,1]$,
  we have
  \[
    \lim_{n\to\infty}
    \sup_{\phi\in\Omega}\left|
    \textstyle
      L_{\Lambda_n}(\phi)(f)-\sum_i\frac{\operatorname{Leb}(D_i)}{\operatorname{Leb}(D)}
      L_{\Lambda_n^i}(\phi)(f)
    \right|
    \leq \frac{\operatorname{Leb}(D\smallsetminus\cup_iD_i)}
    {\operatorname{Leb}(D)}.
      \]
\end{proposition}

\begin{proof}
  Note that
  \[
    \frac{|\mathcal L\cap \Lambda_n^i|}{|\mathcal L\cap \Lambda_n|}\to
    \frac{\operatorname{Leb}(D_i)}{\operatorname{Leb}(D)}
  \]
  as $n\to\infty$,
  and therefore it suffices to prove the proposition
  for the latter fraction replaced by the former.
  Suppose that $f$ is $\mathcal F_\Delta$-measurable
  for some $\Delta\subset\subset\mathbb Z^d$
  which contains $0$.
  Write $\mathbb P_n$
  for the uniform probability measure on $\{\theta_x:x\in\mathcal L\cap \Lambda_n\}$.
  By coupling the measures
   in the obvious way,
  we observe that
  \[
  L_{\Lambda_n}(\phi)(f)-\sum_i\frac{|\mathcal L\cap \Lambda_n^i|}{|\mathcal L\cap \Lambda_n|}
  L_{\Lambda_n^i}(\phi)(f)
  =\mathbb E_n(g_n)
  \]
  where $g_n$ is defined by
  \[
    g_n(\theta)=\begin{cases}
      0 &\text{if $\theta\Delta\subset\Lambda_n^i$ for some $i$,}\\
      f(\theta \bar\phi_{\Lambda_n})-f(\theta \bar\phi_{\Lambda_n^i})
      &\text{if $\theta 0\in\Lambda_n^i$ for some $i$ but $\theta\Delta\not\subset\Lambda_n^i$,}\\
      f(\theta \bar\phi_{\Lambda_n})&\text{otherwise.}
  \end{cases}
  \]
  Now $|g_n|\leq 1$
  and $\mathbb P_n(\text{$\theta\Delta\not\subset \Lambda_n^i$
  for any $i$})= \operatorname{Leb}(D\smallsetminus\cup_iD_i)/
  \operatorname{Leb}(D)+o(1)$
  as $n\to\infty$, which implies the proposition.
\end{proof}

The particular proof of Lemma~\ref{lemma_PB_star_convex}
utilizes the so-called washboard construction
(see Figure~\ref{fig:washboard}),
which appears in the work of Sheffield~\cite{S05},
and is adapted here to the particular Lipschitz setting.

\begin{proof}[Proof of Lemma~\ref{lemma_PB_star_convex}]
  Consider $\mu:=s\nu_1+t\nu_2$
  for $s,t\in(0,1)$ with $s+t=1$ and
  for some measures $\nu_1,\nu_2\in\mathcal P_\mathcal L(\Omega,\mathcal F^\nabla)$
  which have their slope in $U_\Phi$.
  The goal is to prove that
  $\PB^*(\mu)\leq s\PB^*(\nu_1)+t\PB^*(\nu_2)$.

  Write $u:=S(\mu)$, $u_1:=S(\nu_1)$, and $u_2:=S(\nu_2)$.
  Consider $D\in\mathcal C$, $m\in\mathbb Z_{\geq 0}$,
  and $B\in\mathcal B$ with $\mu\in B$.
Write $\Lambda_n:=\Lambda^{-m}(nD)$.
Fix also some $\varepsilon>0$.
If $E=\mathbb Z$, then we must show that
\[
\limsup_{n\to\infty}n^{-d}\PB_{\Lambda_n,u}(B)
\leq \operatorname{Leb}(D)( s\PB^*(\nu_1)+t\PB^*(\nu_2)),
\]
and if $E=\mathbb R$, then we must show that
\[
\limsup_{n\to\infty}n^{-d}\PB_{\Lambda_n,u,\varepsilon}(B)
\leq \operatorname{Leb}(D)( s\PB^*(\nu_1)+t\PB^*(\nu_2)).
\]

By choosing $\varepsilon>0$ smaller if necessary,
we suppose that $u,u_1,u_2\in U_{q_{7\varepsilon}}$.
By choosing $B$ smaller if necessary,
we suppose that
$B$ is of the form
\[
  B=\{\pi:\text{$|\mu(f_i)-\pi(f_i)|<2\eta$ for all $i$}\}\in\mathcal B
\]
for some finite family $(f_i)_i$ of continuous cylinder functions $f_i:\Omega\to[0,1]$
and for some $\eta>0$,
and we write
\[
B^j:=\{\pi:\text{$|\nu_j(f_i)-\pi(f_i)|<\eta$ for all $i$}\}\in\mathcal B
\]
for $j\in\{1,2\}$.

The idea of the proof is roughly as follows.
First, we partition a large subset of $D$ into finitely many convex shapes.
Second, we find a continuous Lipschitz function $f$ which equals $u$
on $\partial D$,
and which is affine on each convex shape in this partition, with slope
either $u_1$ or $u_2$.
This function is chosen such that the Lebesgue measure
of the convex shapes with slope $u_j$ is roughly $s\operatorname{Leb}(D)$
for $j=1$ and roughly $t\operatorname{Leb}(D)$ for $j=2$.
Informally, the function $f$ looks like a ``washboard''.
Next, we define $f_n:=nf(\cdot/n)$,
and use the existence of the function $f_n$
and Theorem~\ref{general_lips_ext_less_general} to find for each $n\in\mathbb N$
a corresponding height function $\phi_n$.
The existence of the height function $\phi_n$
and the previously described general strategy allow us to build a direct comparison
between $\PB_{\Lambda_n,u}(B)$
or $\PB_{\Lambda_n,u,\varepsilon}(B)$ and the numbers $\PB^*(\nu_1)$ and $\PB^*(\nu_2)$.




\begin{figure}
	\centering
  \begin{tikzpicture}
    \clip
     (-2,-.2) rectangle (12,7.2);
		\fill[color=gray, fill opacity=0.5, rounded corners=.4cm] (-.2,-.2) rectangle (10.2,7.2);
		\draw[very thick] (.2,.2) rectangle (9.8,6.8);
		\fill[color=white] (.6,.6) rectangle (9.4,6.4);
		\draw[very thick,dashed] (1,1) rectangle (9,6);
		\fill[color=lightgray, fill opacity=0.5](1.0277777777777777,1.0833333333333335)--(1.0277777777777777,1.0833333333333335)--(1.0277777777777777,1.0833333333333335)--(1.0277777777777777,1.0833333333333335)--cycle;
\fill[color=darkgray, fill opacity=0.5](1.0277777777777777,1.0833333333333335)--(1.0277777777777777,1.0833333333333335)--(1.0277777777777777,1.0833333333333335)--(1.0277777777777777,1.0833333333333335)--cycle;
\fill[color=lightgray, fill opacity=0.5](1.0277777777777777,1.0833333333333335)--(1.0277777777777777,1.0833333333333335)--(1.0277777777777777,1.0833333333333335)--(1.0277777777777777,1.0833333333333335)--cycle;
\fill[color=darkgray, fill opacity=0.5](1.0277777777777777,1.0833333333333335)--(1.1666666666666667,1.0)--(1.0,1.5)--(1.0277777777777777,1.0833333333333335)--cycle;
\fill[color=lightgray, fill opacity=0.5](1.1666666666666667,1.0)--(1.4166666666666667,1.15)--(1.0499999999999998,2.2499999999999996)--(1.0,1.5)--cycle;
\fill[color=darkgray, fill opacity=0.5](1.4166666666666667,1.15)--(1.6666666666666667,1.0)--(1.0,3.0)--(1.0499999999999998,2.2499999999999996)--cycle;
\fill[color=lightgray, fill opacity=0.5](1.6666666666666667,1.0)--(1.916666666666667,1.15)--(1.05,3.7500000000000004)--(1.0,3.0)--cycle;
\fill[color=darkgray, fill opacity=0.5](1.916666666666667,1.15)--(2.1666666666666665,1.0)--(1.0,4.5)--(1.05,3.7500000000000004)--cycle;
\fill[color=lightgray, fill opacity=0.5](2.1666666666666665,1.0)--(2.4166666666666665,1.15)--(1.05,5.25)--(1.0,4.5)--cycle;
\fill[color=darkgray, fill opacity=0.5](2.4166666666666665,1.15)--(2.6666666666666665,1.0)--(1.0,6.0)--(1.0,6.0)--(1.05,5.25)--cycle;
\fill[color=lightgray, fill opacity=0.5](2.6666666666666665,1.0)--(2.9166666666666665,1.15)--(1.2500000000000002,6.15)--(1.0,6.0)--(1.0,6.0)--cycle;
\fill[color=darkgray, fill opacity=0.5](2.9166666666666665,1.15)--(3.1666666666666665,1.0)--(1.5,6.0)--(1.2500000000000002,6.15)--cycle;
\fill[color=lightgray, fill opacity=0.5](3.1666666666666665,1.0)--(3.4166666666666665,1.15)--(1.7500000000000002,6.15)--(1.5,6.0)--cycle;
\fill[color=darkgray, fill opacity=0.5](3.4166666666666665,1.15)--(3.6666666666666665,1.0)--(2.0,6.0)--(1.7500000000000002,6.15)--cycle;
\fill[color=lightgray, fill opacity=0.5](3.6666666666666665,1.0)--(3.9166666666666665,1.15)--(2.25,6.15)--(2.0,6.0)--cycle;
\fill[color=darkgray, fill opacity=0.5](3.9166666666666665,1.15)--(4.166666666666667,1.0)--(2.5,6.0)--(2.25,6.15)--cycle;
\fill[color=lightgray, fill opacity=0.5](4.166666666666667,1.0)--(4.416666666666667,1.15)--(2.75,6.15)--(2.5,6.0)--cycle;
\fill[color=darkgray, fill opacity=0.5](4.416666666666667,1.15)--(4.666666666666667,1.0)--(3.0,6.0)--(2.75,6.15)--cycle;
\fill[color=lightgray, fill opacity=0.5](4.666666666666667,1.0)--(4.916666666666667,1.15)--(3.25,6.15)--(3.0,6.0)--cycle;
\fill[color=darkgray, fill opacity=0.5](4.916666666666667,1.15)--(5.166666666666667,1.0)--(3.5,6.0)--(3.25,6.15)--cycle;
\fill[color=lightgray, fill opacity=0.5](5.166666666666667,1.0)--(5.416666666666667,1.15)--(3.75,6.15)--(3.5,6.0)--cycle;
\fill[color=darkgray, fill opacity=0.5](5.416666666666667,1.15)--(5.666666666666667,1.0)--(4.0,6.0)--(3.75,6.15)--cycle;
\fill[color=lightgray, fill opacity=0.5](5.666666666666667,1.0)--(5.916666666666667,1.15)--(4.25,6.15)--(4.0,6.0)--cycle;
\fill[color=darkgray, fill opacity=0.5](5.916666666666667,1.15)--(6.166666666666667,1.0)--(4.5,6.0)--(4.25,6.15)--cycle;
\fill[color=lightgray, fill opacity=0.5](6.166666666666667,1.0)--(6.416666666666667,1.15)--(4.75,6.15)--(4.5,6.0)--cycle;
\fill[color=darkgray, fill opacity=0.5](6.416666666666667,1.15)--(6.666666666666667,1.0)--(5.0,6.0)--(4.75,6.15)--cycle;
\fill[color=lightgray, fill opacity=0.5](6.666666666666667,1.0)--(6.916666666666667,1.15)--(5.25,6.15)--(5.0,6.0)--cycle;
\fill[color=darkgray, fill opacity=0.5](6.916666666666667,1.15)--(7.166666666666667,1.0)--(5.5,6.0)--(5.25,6.15)--cycle;
\fill[color=lightgray, fill opacity=0.5](7.166666666666667,1.0)--(7.416666666666667,1.15)--(5.75,6.15)--(5.5,6.0)--cycle;
\fill[color=darkgray, fill opacity=0.5](7.416666666666667,1.15)--(7.666666666666667,1.0)--(6.0,6.0)--(5.75,6.15)--cycle;
\fill[color=lightgray, fill opacity=0.5](7.666666666666667,1.0)--(7.916666666666666,1.1499999999999995)--(6.249999999999999,6.1499999999999995)--(6.0,6.0)--cycle;
\fill[color=darkgray, fill opacity=0.5](7.916666666666666,1.1499999999999995)--(8.166666666666666,1.0)--(6.5,6.0)--(6.249999999999999,6.1499999999999995)--cycle;
\fill[color=lightgray, fill opacity=0.5](8.166666666666666,1.0)--(8.416666666666664,1.1499999999999995)--(6.749999999999999,6.1499999999999995)--(6.5,6.0)--cycle;
\fill[color=darkgray, fill opacity=0.5](8.416666666666664,1.1499999999999995)--(8.666666666666666,1.0)--(7.0,6.0)--(6.749999999999999,6.1499999999999995)--cycle;
\fill[color=lightgray, fill opacity=0.5](8.666666666666666,1.0)--(8.916666666666664,1.1499999999999995)--(7.249999999999999,6.1499999999999995)--(7.0,6.0)--cycle;
\fill[color=darkgray, fill opacity=0.5](8.916666666666664,1.1499999999999995)--(9.027777777777779,1.083333333333333)--(9.0,1.5)--(7.5,6.0)--(7.249999999999999,6.1499999999999995)--cycle;
\fill[color=lightgray, fill opacity=0.5](9.0,1.5)--(9.05,2.2499999999999973)--(7.749999999999999,6.1499999999999995)--(7.5,6.0)--cycle;
\fill[color=darkgray, fill opacity=0.5](9.05,2.2499999999999973)--(9.0,3.0)--(8.0,6.0)--(7.749999999999999,6.1499999999999995)--cycle;
\fill[color=lightgray, fill opacity=0.5](9.0,3.0)--(9.05,3.7499999999999973)--(8.25,6.1499999999999995)--(8.0,6.0)--cycle;
\fill[color=darkgray, fill opacity=0.5](9.05,3.7499999999999973)--(9.0,4.5)--(8.5,6.0)--(8.25,6.1499999999999995)--cycle;
\fill[color=lightgray, fill opacity=0.5](9.0,4.5)--(9.05,5.249999999999997)--(8.75,6.1499999999999995)--(8.5,6.0)--cycle;
\fill[color=darkgray, fill opacity=0.5](9.05,5.249999999999997)--(9.0,6.0)--(9.0,6.0)--(8.75,6.1499999999999995)--cycle;
\fill[color=lightgray, fill opacity=0.5](9.0,6.0)--(9.0,6.0)--(9.0,6.0)--(9.0,6.0)--cycle;
\fill[color=darkgray, fill opacity=0.5](9.0,6.0)--(9.0,6.0)--(9.0,6.0)--(9.0,6.0)--cycle;
\draw[very thin] (2.25,6.15)--(2.25,6.45);
\draw[very thin] (2.75,6.15)--(2.75,6.45);
\draw[very thin] (2.25,6.300000000000001)--(2.75,6.300000000000001);
\draw[<-] (2.15,6.300000000000001)--(-0.5,6.300000000000001)node[left] {$\propto\varepsilon_2$};

    \draw[-Circle] (-0.5,3.5) node[left] {$\partial'D$} -- (0,3.5) ;
    \draw[->] (-0.5,4.5) node[left] {$\partial D$} -- (.1,4.5) ;
    \draw[-Circle] (-0.5,2.5) node[left] {$D'$} -- (1.5,2.5) ;

        \draw[very thin] (-.2,.2)--(-.5,.2);
        \node[anchor=east] at (-.5,.4) {$\varepsilon_1$};
        \draw[very thin] (-.2,.6)--(-.5,.6);
        \node[anchor=east] at (-.5,.8) {$\varepsilon_1$};
        \draw[very thin] (-.2,1)--(-.5,1);
        \draw[very thin] (-.35,.2)--(-.35,1);

  \end{tikzpicture}
  \caption{The washboard and the function $f:\partial'D\cup D'\to\mathbb R$}
	\label{fig:washboard}
\end{figure}
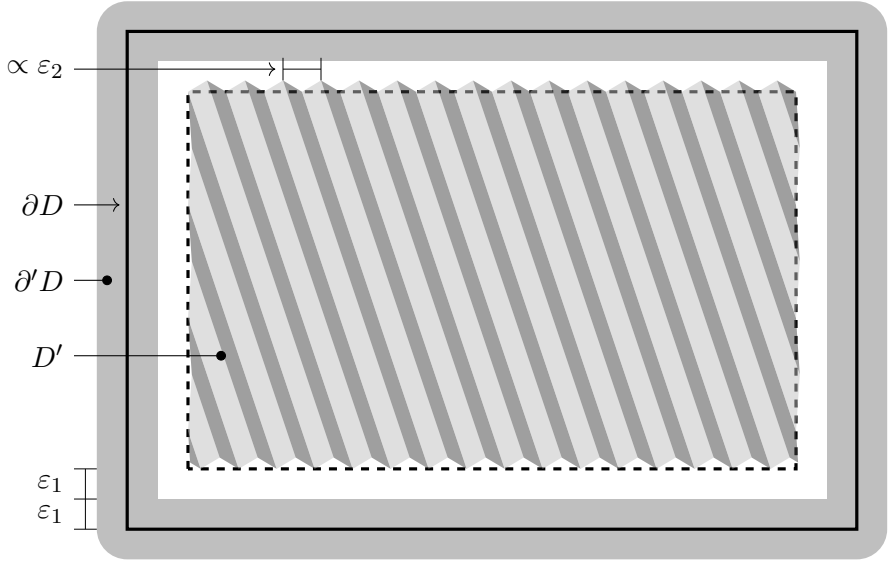

We start by constructing the continuous ``washboard''---see
Figure~\ref{fig:washboard}.
\label{washboarddefstart}
Set $v:=u_1-u_2$ if $u_1\neq u_2$,
and choose $v\in(\mathbb R^d)^*\smallsetminus \{0\}$
arbitrary otherwise.
Define
\[
  w:\mathbb R^d\to\mathbb Z,\,
  x\mapsto \begin{cases}
    2\lfloor v(x)\rfloor &\text{if $v(x)-\lfloor v(x)\rfloor\in [0,t)$},\\
    2\lfloor v(x)\rfloor+1&\text{if $v(x)-\lfloor v(x)\rfloor\in [t,1)$}.
\end{cases}
\]
Write $p:\mathbb R^d\to\mathbb R$ for the unique continuous function that maps
$0$ to $0$,
and
which
has gradient $u_j$ on
the interior of $\{w\in 2\mathbb Z+j\}\subset\mathbb R^d$
for $j\in\{1,2\}$.
For $\alpha>0$,
write $w_\alpha$ for the map $w_\alpha(\cdot):=w(\cdot/\alpha)$,
and
write $p_\alpha$ for the map $p_\alpha(\cdot):=\alpha p(\cdot/\alpha)$.
It is straightforward to see that $p_\alpha$ has gradient
$u_j$ on $\{w_\alpha\in 2\mathbb Z+j\}$ for $j\in\{1,2\}$,
and that $\|p_\alpha-u\|_\infty\propto \alpha$.
Observe also that $p$ and $p_\alpha$ are $\|\cdot\|_{q_{7\varepsilon}}$-Lipschitz.

In the remainder of the proof, we shall
work with three limits. First we take $n\to\infty$,
then $\varepsilon_2\to 0$, then $\varepsilon_1\to 0$.
Reference to these variables is sometimes omitted for brevity.
Define
\begin{alignat*}{2}
 &\partial'D&&:=\{x\in\mathbb R^d:d_2(x,\partial D)<\varepsilon_1\},
\\
 & D'&&:=\{x\in D:d_2(x,\partial D)>2\varepsilon_1\},
\\
&  D'_k&&:=D'\cap \{w_{\varepsilon_2}=k\},
\end{alignat*}
where $d_2$ denotes Euclidean distance.
Write
$f:\partial'D\cup D'\to\mathbb R$ for the function defined by
\[
f(x):=\begin{cases}
  u(x)&\text{if $x\in\partial'D$},\\
  p_{\varepsilon_2}(x)&\text{if $x\in D'$}.
\end{cases}
\]
This function is $\|\cdot\|_{q_{6\varepsilon}}$-Lipschitz
for $\varepsilon_2$ sufficiently small,
depending on $\varepsilon_1$.
Note that $f$ is affine with gradient $u_1$
on $D_k'$ for $k$ odd and with gradient $u_2$
on $D_k'$ for $k$ even.
Moreover, the family $(D'_k)_{k\in\mathbb Z}$ is a partition of $D'$.
Only finitely many members are nonempty, and
the nonempty members are convex, bounded, and have positive Lebesgue
measure.
The merit of this construction is that
\begin{alignat}{3}
  \label{lebesgue_1}
  &\operatorname{Leb}(\cup_{k\in2\mathbb Z+1} D'_k)&&\to_{\varepsilon_2\to 0} s\operatorname{Leb}(D')
  &&\to_{\varepsilon_1\to 0} s\operatorname{Leb}(D),\\
  \label{lebesgue_2}
  &\operatorname{Leb}(\cup_{k\in2\mathbb Z}D'_k)&&\to_{\varepsilon_2\to 0} t\operatorname{Leb}(D')
  &&\to_{\varepsilon_1\to 0} t\operatorname{Leb}(D).
\end{alignat}

For $n\in\mathbb N$, define
$f_n:n(\partial'D\cup D')\to\mathbb R$
by $f_n(\cdot):=nf(\cdot/n)$---this function is also $\|\cdot\|_{q_{6\varepsilon}}$-Lipschitz.
In particular,
Theorem~\ref{general_lips_ext_less_general}
implies that for some $M\in\mathbb Z_{\geq m}$ depending only on $\varepsilon$,
there exists a $q$-Lipschitz height function $\phi_n\in\Omega$
such that
\begin{enumerate}
  \item $\nabla\phi_n|_{\Lambda^{-M}(n\partial'D)}=\nabla\phi^u|_{\Lambda^{-M}(n\partial'D)}$,
  \item $\nabla\phi_n|_{\Lambda^{-M}(nD'_k)}=\nabla\phi^{u_1}|_{\Lambda^{-M}(nD'_k)}$
  for all $k$ odd,
  \item $\nabla\phi_n|_{\Lambda^{-M}(nD'_k)}=\nabla\phi^{u_2}|_{\Lambda^{-M}(nD'_k)}$
  for all $k$ even,
  \item $\phi_n$ is $q_{5\varepsilon}$-Lipschitz if $E=\mathbb R$.
\end{enumerate}
Recall the definition of $\Lambda_n$,
and define
\[
  \Lambda_{n,k}:=\Lambda^{-M}(nD'_k),\qquad
  \Lambda_n^0:=\Lambda_n\smallsetminus\cup_k(\Lambda_{n,k}\smallsetminus\{0_{\Lambda_{n,k}}\})
  ,\qquad
  \Lambda_n^*:=\Lambda_n\smallsetminus\cup_k\Lambda_{n,k}^{-R}.
\]
Note that $\partial^R\Lambda_n\subset\Lambda^{-M}(n\partial'D)$
for $n$ sufficiently large,
and consequently $\nabla\phi_n|_{\partial^R\Lambda_n}=\nabla\phi^u|_{\partial^R\Lambda_n}$.
This
 also implies that the sets $\partial^R\Lambda_n$ and $\Lambda_{n,k}$ are all disjoint
 for fixed $n$ as $k$ ranges over $\mathbb Z$.
 Finally, $\Lambda_{n,k}\subset \Lambda_n$ for all $k$.

The idea is now to use the existence of the function $\phi_n$
to derive the inequalities.
We distinguish two cases, depending on whether $E=\mathbb Z$
or $E=\mathbb R$.
Start with the former, which is easier.
Write $A_n$ for the set of height functions $\phi$
such that
\begin{enumerate}
  \item $\nabla\phi$ equals $\nabla\phi_n$ on $\Lambda_n^*$,
  \item $\phi\in B^1_{\Lambda_{n,k}}$ for all $k$ odd,
  \item $\phi\in B^2_{\Lambda_{n,k}}$ for all $k$ even.
\end{enumerate}
Note that $A_n\subset C^u_{\Lambda_n}$ because $\partial^R\Lambda_n\subset\Lambda_n^*$
and because $\nabla\phi_n=\nabla\phi^u$ on $\partial^R\Lambda_n$.
It is straightforward to
work out that $A_n\subset B_{\Lambda_n}$ for $n$ sufficiently large
and $\varepsilon_1,\varepsilon_2$ sufficiently small,
by application of Proposition~\ref{propo_compare_empirical_measures}
combined with~\eqref{lebesgue_1} and~\eqref{lebesgue_2}.

Therefore it suffices to demonstrate that
\[
\numberthis\label{eq_proof_again_one}
  \liminf n^{-d}\log\int_{A_n}
  e^{-H_{\Lambda_n}^0}d\lambda^{\Lambda_n-1}
  \geq -\operatorname{Leb}(D)( s\PB^*(\nu_1)+t\PB^*(\nu_2))
\]
where the limit is in the variables $n$, $\varepsilon_2$, and $\varepsilon_1$.
Moreover, since $\nabla\phi$ equals $\nabla\phi_n$ on $\Lambda_n^*$
for any $\phi\in A_n$,
this restriction to $\Lambda_n^* $ is $q$-Lipschitz, and the upper attachment lemma
(Lemma~\ref{lemma_attachment}) implies that
\[
\numberthis\label{eq_proof_again_two}
  H_{\Lambda_n}^0\leq
  \sum_kH_{\Lambda_{n,k}}^0+\sum_k e^+(\Lambda_{n,k})
  +
  |\Lambda_n\smallsetminus\cup_k\Lambda_{n,k}| \max_{x\in\mathbb Z^d/\mathcal L }e^+(\{x\})
\]
on $A_n$.
For the third term we have $n^{-d}|\Lambda_n\smallsetminus\cup_k\Lambda_{n,k}|\to_{n\to\infty} \operatorname{Leb}(D\smallsetminus D')\to_{\varepsilon_1\to 0} 0$,
and the second term is of order $o(n^d)$ as $n\to\infty$.
This implies that
\[
\numberthis\label{eq_proof_again_three}
\liminf n^{-d}\log\int_{A_n}
e^{-H_{\Lambda_n}^0}d\lambda^{\Lambda_n-1}
\geq
\liminf
n^{-d}\log
\int_{A_n}
e^{-\sum_kH_{\Lambda_{n,k}}^0}d\lambda^{\Lambda_n-1}.
\]
Recall the definition of $\Lambda_n^0$,
and consider $\lambda^{\Lambda_n-1}$ a product measure,
by writing
\[
  \lambda^{\Lambda_n-1}:=\lambda^{\Lambda_n^0-1}\times\prod_{k}\lambda^{\Lambda_{n,k}-1}.
\]
Note that $A_n$ contains exactly all height functions $\phi$ such that
\begin{enumerate}
  \item $\nabla\phi$ equals $\nabla\phi_n$ on $\Lambda_n^0$,
  \item $\phi\in C_{\Lambda_{n,k}}^{u_1}\cap B^1_{\Lambda_{n,k}}$ for all $k$ odd,
  \item $\phi\in C_{\Lambda_{n,k}}^{u_2}\cap B^2_{\Lambda_{n,k}}$ for all $k$ even,
\end{enumerate}
and therefore
\[\begin{split}
  &\int_{A_n}
  e^{-\sum_kH_{\Lambda_{n,k}}^0}d\lambda^{\Lambda_n-1}
=
\int_{\{\text{$\nabla\phi$ equals $\nabla\phi_n$ on $\Lambda_n^0$}\}}d\lambda^{\Lambda_n^0-1}(\phi)
\\&\qquad\qquad\cdot\prod_{k\in2\mathbb Z+1}\int_{C_{\Lambda_{n,k}}^{u_1}\cap B_{\Lambda_{n,k}}^1}
e^{-H_{\Lambda_{n,k}}^0}d\lambda^{\Lambda_{n,k}-1}
\cdot\prod_{k\in2\mathbb Z}\int_{C_{\Lambda_{n,k}}^{u_2}\cap B_{\Lambda_{n,k}}^2}
e^{-H_{\Lambda_{n,k}}^0}d\lambda^{\Lambda_{n,k}-1}.
\end{split}\]
The first factor equals one since we are dealing with the counting measure,
and therefore
\[
\log\int_{A_n}
e^{-\sum_kH_{\Lambda_{n,k}}^0}d\lambda^{\Lambda_n-1}
=-\sum_{k\in2\mathbb Z+1}\PB_{\Lambda_{n,k},u_1}(B^1)
-\sum_{k\in2\mathbb Z}\PB_{\Lambda_{n,k},u_2}(B^2).
\]
For fixed $\varepsilon_1,\varepsilon_2$ only finitely many
terms are possibly nonzero---those corresponding to nonempty sets $D_k'$---and
for each term we have (for $j\in\{1,2\}$)
\[
  \limsup_{n\to\infty} n^{-d}\PB_{\Lambda_{n,k},u_j}(B^j)\leq
  \PB(\nu_j:D_k',M)
  \leq
  \operatorname{Leb}(D_k')\PB^*(\nu_j).
\]
Therefore~\eqref{lebesgue_1} and~\eqref{lebesgue_2} imply
\[
\liminf n^{-d}\log\int_{A_n}
e^{-H_{\Lambda_n}^0}d\lambda^{\Lambda_n-1}
\geq -
\operatorname{Leb}(D)(s\PB^*(\nu_1)+t\PB^*(\nu_2)),
\]
 the desired inequality.

Let us now discuss what changes for $E=\mathbb R$.
Write $A_n$ for the set of samples $\phi$
such that
\begin{enumerate}
  \item $|(\phi_{\Lambda_n^0}-\phi(0_{\Lambda_n}))
  -(
  \phi_n|_{\Lambda_n^0}-\phi_n(0_{\Lambda_n})
  )|\leq \varepsilon$,
  \item $\phi\in C^{u_1}_{\Lambda_{n,k},\varepsilon}$ and $\phi\in B^1_{\Lambda_{n,k}}$ for all $k$ odd,
  \item $\phi\in C^{u_2}_{\Lambda_{n,k},\varepsilon}$ and $\phi\in B^2_{\Lambda_{n,k}}$ for all $k$ even.
\end{enumerate}
Note that $A_n\subset C^u_{\Lambda_n,\varepsilon}$.
The proof that $A_n\subset B_{\Lambda_n}$ is the same as before.
We must again prove~\eqref{eq_proof_again_one}.
The definition of $A_n$ implies that
$|(\phi_{\Lambda_n^*}-\phi(0_{\Lambda_n}))
-(
\phi_n|_{\Lambda_n^*}-\phi_n(0_{\Lambda_n})
)|\leq 2\varepsilon$
for any $\phi\in A_n$,
which in turn implies that $\phi_{\Lambda_n^*}$ is $q_\varepsilon$-Lipschitz
as $\phi_n$ was $q_{5\varepsilon}$-Lipschitz---see Proposition~\ref{propo_qqq}.
Therefore~\eqref{eq_proof_again_two} holds true with $e^+(\cdot)$
replaced by $e^+_{\varepsilon}(\cdot)$,
which implies~\eqref{eq_proof_again_three}.
We now have
\begin{multline*}
  \int_{A_n}
  e^{-\sum_kH_{\Lambda_{n,k}}^0}d\lambda^{\Lambda_n-1}
=
\int_{\{|(\phi_{\Lambda_n^0}-\phi(0_{\Lambda_n}))
-(
\phi_n|_{\Lambda_n^0}-\phi_n(0_{\Lambda_n})
)|\leq \varepsilon\}}
d\lambda^{\Lambda_n^0-1}(\phi)
\\\cdot\prod_{k\in2\mathbb Z+1}\int_{C_{\Lambda_{n,k},\varepsilon}^{u_1}\cap B_{\Lambda_{n,k}}^1}
e^{-H_{\Lambda_{n,k}}^0}d\lambda^{\Lambda_{n,k}-1}
\cdot\prod_{k\in2\mathbb Z}\int_{C_{\Lambda_{n,k},\varepsilon}^{u_2}\cap B_{\Lambda_{n,k}}^2}
e^{-H_{\Lambda_{n,k}}^0}d\lambda^{\Lambda_{n,k}-1}.
\end{multline*}
The first integral equals
$(2\varepsilon)^{|\Lambda_n^0|-1}$,
and therefore
\begin{multline*}
\log\int_{A_n}
e^{-\sum_kH_{\Lambda_{n,k}}^0}d\lambda^{\Lambda_n-1}
\\=(|\Lambda_n^0|-1)\log (2\varepsilon)-\sum_{k\in2\mathbb Z+1}\PB_{\Lambda_{n,k},u_1,\varepsilon}(B^1)
-\sum_{k\in2\mathbb Z}\PB_{\Lambda_{n,k},u_2,\varepsilon}(B^2).
\end{multline*}
The first term vanishes in the limit in the three variables after normalizing by $n^{-d}$.
The remainder of the proof is the same as before.
\end{proof}

Let us now discuss briefly how to deal with
ergodic measures with finite specific free energy
which have their slope in $\partial U_\Phi$,
before proving that $\PB^*(\mu)\leq\mathcal H(\mu|\Phi)$
for any shift-invariant random field $\mu$
with $S(\mu)\in U_\Phi$.

\begin{definition}
  Consider a measure $\mu\in\mathcal P_\mathcal L(\Omega,\mathcal F^\nabla)$ with finite specific
free energy.
Classify $\mu$ as \emph{taut} if $w_\mu$-almost surely $S(\nu)\in\partial U_\Phi$,
and as \emph{non-taut} if $w_\mu$-almost surely $S(\nu)\in U_\Phi$.
A \emph{non-taut approximation} of $\mu$ is a sequence
 $(\mu_n)_{n\in\mathbb N}\subset\mathcal P_\mathcal L(\Omega,\mathcal F^\nabla)$
 of non-taut measures
such that
$\mathcal H(\mu_n|\Phi)\to\mathcal H(\mu|\Phi)$
and
$\mu_n\to\mu$
in the topology of weak local convergence
as $n\to\infty$.
\end{definition}

If $E=\mathbb R$ and $\mu$ a shift-invariant random field
with finite specific free energy,
then $w_\mu$-almost surely $S(\nu)\in U_\Phi$,
due to Theorem~\ref{thm_main_st_general}
and because $\mathcal H(\cdot|\Phi)$
is strongly affine.
In other words, $\mu$ is automatically non-taut.
The following lemma is therefore meaningful for $E=\mathbb Z$ only.

\begin{lemma}
  \label{lemma_ergodic_approx_aux_1}
  Any ergodic gradient random field with finite specific free energy
  has a non-taut approximation.
\end{lemma}

\begin{proof}
  Let $E=\mathbb Z$,
  and let $\mu$ denote an ergodic
  random field with $\mathcal H(\mu|\Phi)<\infty$ and $S(\mu)\in\partial U_\Phi$.
  In this pathological case, we must modify $\mu$
  slightly, so that the modified measure
  is non-taut,
  and without changing the
  specific free energy
  too much.
  Let $\xi$ denote another ergodic measure
  in $\mathcal P_\mathcal L(\Omega,\mathcal F^\nabla)$
  with $\mathcal H(\xi|\Phi)<\infty$
  and with $S(\xi)\in U_\Phi$---such measures exist,
  due to the proof of Theorem~\ref{thm_main_st_general}
  on Page~\pageref{thm_main_st_general_proof}.
  Write $\rho_n$ for the uniform probability
  measure on the set $\{0,\dots,n-1\}$.
  Fix $n\in\mathbb N$;
  we are going to define a new measure $\mu_n$.
  To sample a height function $\phi$
  from $\mu_n$,
  sample first a triple $(\phi^\mu,\phi^\xi,a)$
  from the measure $\mu\times\xi\times\rho_n$.
  The final sample $\phi$
  is then given by the equation
  \[
    \phi:=
    \phi^\mu
    -
    \left\lfloor
      \frac{(\phi^\mu-\phi^\mu(0))-(\phi^\xi-\phi^\xi(0))+a}{n}
    \right\rfloor.
  \]
  The random choice of $a$ makes the rounding operation shift-invariant.
  Note that the numerator in this fraction
  is $2K$-Lipschitz almost surely for $K$ minimal subject to $Kd_1\geq q$,
  and therefore the rounded function is $1$-Lipschitz for $n$ sufficiently large.
  In fact, the density of edges on which
  the rounded function is not constant, has a bound of order $O(1/n)$ as $n\to\infty$.
  In particular, this implies that $\mu_n\to\mu$
  in the topology of (weak) local convergence.
  Recall~\eqref{eq_decomp_of_SFE_useful_occas}
  from the proof of Theorem~\ref{thm_SFE_strongly_affine},
  and observe that
  the specific free energy of $\mu$ and $\mu_n$ can be calculated
  as in this equation because either measure is $K$-Lipschitz.
  If $f(p)$ denotes the entropy function of a Bernoulli trial with
  parameter $p$ as in the proof of Lemma~\ref{lemma:PBL_upper_bound},
  then by arguments similar to those used in that proof,
  we can bound the difference in the specific entropy between
  $\mu$ and $\mu_n$:
  \[
    |\mathcal H(\mu|\lambda)
    -
    \mathcal H(\mu_n|\lambda)|=O(f(O(1/n)))=o(1)
  \]
  as $n\to\infty$.
  For $E=\mathbb Z$,
  we have a lower and upper bound on $H_{\{x\}}(\phi)$
  for $q$-Lipschitz $\phi$,
  and this and
  amenability of the weak interaction $\Xi$
  imply that the specific energy functional
  \[
    \mu\mapsto \mu(\Phi)
  \]
  is continuous with respect to the topology of local convergence
  whenever restricted to shift-invariant random fields
  which are supported on $q$-Lipschitz functions.
  Jointly these two observations imply that $\mathcal H(\mu_n|\Phi)\to\mathcal H(\mu|\Phi)$.
  It suffices to demonstrate that each
  measure $\mu_n$ is non-taut.
  Claim that $w_{\mu_n}$-almost every ergodic component $\nu$
  satisfies $S(\nu)=(1-\frac1n)S(\mu)+\frac1nS(\xi)\in U_\Phi$.
  Recall Theorem~\ref{thm:superergodresult}.
  The final assertion of that theorem tells us that
  the slope $S(\nu)$ of each ergodic
  component can be read off from almost every sample $\phi$
  from $\nu$,
  since
  the slope $u:=S(\nu)$ is almost surely the unique slope such that
  for any fixed $\varepsilon>0$,
  \[
    \|\phi_{\Pi_m}-\phi(0)-u|_{\Pi_m}\|_\infty\leq \varepsilon m
  \]
  for $m$ sufficiently large.
  The slope $(1-\frac1n)S(\mu)+\frac1nS(\xi)$  makes
  this inequality work
  for samples $\phi$ from the original measure $\mu_n$,
  because $\mu$ and $\xi$ are ergodic,
  and because
    $\phi$
  equals $(1-\frac1n)\phi^\mu+\frac1n\phi^\xi$
  up to bounded differences.
\end{proof}

\begin{lemma}
  \label{lemma_PB_star_uncontrained_ineq}
  For any $\mu\in\mathcal P_\mathcal L(\Omega,\mathcal F^\nabla)$
  with $S(\mu)\in U_\Phi$, we have $\PB^*(\mu)\leq \mathcal H(\mu|\Phi)$.
\end{lemma}

\begin{proof}
  Let $\mu$ denote an arbitrary
  shift-invariant random field with $H:=\mathcal H(\mu|\Phi)+1<\infty$
  and $u:=S(\mu)\in U_\Phi$.
  If $\mu$ is non-taut and a convex combination of finitely many ergodic random fields,
  then the lemma follows immediately from
  Theorem~\ref{thm_main_sfe} and Lemmas~\ref{lemma_pb_convex_shape} and~\ref{lemma_PB_star_convex}.
  Let us now consider the case that $\mu$ is non-taut,
  but not a convex combination of finitely many ergodic random fields.
  The lower level set of the specific free energy
  $M_{H}$ is a compact Polish space,
  and therefore
  there exists a sequence
  of continuous cylinder functions
  $(f_k)_{k\in\mathbb N}$
  with $f_k:\Omega\to[0,1]$
  such that some sequence
  $(\mu_n)_{n\in\mathbb N}\subset
  M_{H}$
  satisfies
  $\mu_n\to\mu$ in the
  topology of weak local convergence if
  and only if $\mu_n(f_k)\to\mu(f_k)$
  as  $n\to\infty$ for every $k\in\mathbb N$.
  Write $w_\mu$ for the ergodic
  decomposition of $\mu$.
  Let $(\nu_i)_{i\in\mathbb N}$
  denote an i.i.d.\ sequence
  of samples from $w_\mu$.
  Define
  \[
    \mu_n:=\sum\nolimits_{i=1}^{n}{\textstyle\frac 1 n}\nu_i.
  \]
  Then $w_\mu$-almost surely,
  $\mathcal H(\mu_n|\Phi)\to \mathcal H(\mu|\Phi)$
  and
  $\mu_n(f_k)\to\mu(f_k)$
  as $n\to\infty$
  for all $k\in\mathbb N$.
  This implies that $\mu_n\to\mu$
  in the topology of weak local convergence.
  Finally, we have $S(\mu_n)\to u$.
  By altering the coefficients in the definition
  of each measure $\mu_n$
  slightly, we can make sure that $S(\mu_n)=u$
  for $n$ sufficiently large,
  while retaining the other properties of this sequence.
  For each measure $\mu_n$
  we have $\PB^*(\mu_n)\leq\mathcal H(\mu_n|\Phi)$
  by the first part of this proof,
  and $\PB^*(\mu)\leq\mathcal H(\mu|\Phi)$
  because $\PB^*(\cdot)$ is lower-semicontinuous when restricted
  to $\{S(\cdot)=u\}$,
  while $\mathcal H(\mu_n|\Phi)\to\mathcal H(\mu|\Phi)$ as $n\to\infty$.

  We now prove the lemma for the case that $\mu$
  is a convex combination of finitely
  many ergodic measures, but without imposing that $\mu$ is non-taut.
  Write
  \[
    \mu=\sum\nolimits_{i=1}^n a_i\nu^i
  \]
  for the decomposition of $\mu$ into ergodic components.
  Since each $\nu^i$ is ergodic, it has a non-taut approximation
  $(\nu^i_k)_{k\in\mathbb N}$.
  Define $\mu_k:=\sum_{i=1}^na_i\nu_k^i$,
  so that $\mu_k\to\mu$ in the topology of weak local convergence
  with $\mathcal H(\mu_k|\Phi)\to \mathcal H(\mu|\Phi)$
  as $k\to\infty$.
  This implies also that $S(\mu_k)\to S(\mu)$,
  and by altering the coefficients in the definition
  of each measure $\mu_k$ slightly,
  we may ensure that $S(\mu_k)=u$ for $k$
  sufficiently large,
  while retaining the previously mentioned properties.
  By arguing as before,
  we have $\PB^*(\mu_k)\leq\mathcal H(\mu_k|\Phi)$
  and therefore $\PB^*(\mu)\leq\mathcal H(\mu|\Phi)$.
  The generalization to those measures $\mu$ which
  are not a convex combination of finitely many
  ergodic measures
  and not non-taut is the same as before.
\end{proof}


\section{Large deviations principle}
\label{sec:LDP}

Large deviations are the subject of a vast literature within statistical physics~\cite{MR997938,MR2571413,MR3309619}.
In the context of gradient models, the pioneering result was derived by Sheffield in~\cite{S05}.
In this section we prove a large deviations principle (LDP) of similar strength to the one contained in Chapter 7 of~\cite{S05}, with the noteworthy difference that we express it directly in terms of the Gibbs specification.
The large deviations principle applies to all models described in the introduction,
including for example perturbed dimer models~\cite{MR3606736,
  GMT2019} which are not monotone, even if the perturbation has infinite range.
This LDP captures both the macroscopic profile of each sample,
as well as its local statistics.
We will be using some notations and ideas from~\cite{S05} and~\cite{krieger2019deducing}.
Recall Subsection~\ref{subsec:main_results:ldp}
for a description of good asymptotic boundary profiles
and good approximations.
That subsection also contains a description of the topology for the macroscopic profile of each function.
The letter $\Phi$ denotes a fixed potential belonging to the class $\mathcal S_\mathcal L+\mathcal W_\mathcal L$ throughout this section.


\subsection{Formal description of the LDP}

\label{subsec:LDP_FULL}

Recall Subsection~\ref{subsec:main_results:ldp},
which gave an overview of the large deviations principle
without local statistics.
Throughout this section,
the sequence $(D_n,b_n)_{n\in\mathbb N}$
denotes a good approximation of
some fixed good asymptotic boundary profile $(D,b)$.
The sequence of local Gibbs measures
which are of interest in the LDP is the sequence
$(\gamma_n)_{n\in\mathbb N}$ defined
by $\gamma_n:=\gamma_{D_n}(\cdot,b_n)$.
We shall also write $Z_n$ for $Z_{D_n}(b_n)$; the
normalizing constant in the definition of the measure
$\gamma_{D_n}(\cdot,b_n)$.
Finally, $\tilde\gamma_n$ shall denote the non-normalized
version of $\gamma_n$, that is, $\tilde\gamma_n:=Z_n\gamma_n$.

	\subsubsection{The topological space}
	\label{formal_LDP_topo}

All samples from the sequence of measures
$(\gamma_n)_{n\in\mathbb N}$
must be brought to the same topological space,
in order to formulate the large deviations principle.
	We want our large
	 deviations principle to describe both
	 the global profile
	 of each sample as well
	 as its local statistics,
	 and this is reflected
	 in the choice of topological
	 space.
	 More concretely,
	 the topological space that we have
	 in mind decomposes as the product of two topological spaces, each describing one
	 of the two aspects of each sample.
   Recall from Definition~\ref{definition_main_ldp_topology_macroscopic_profiles}
   that $(\operatorname{Lip}(\bar D),\mathcal X^\infty)$
   is space of $K\|\cdot\|_1$-Lipschitz functions functions
   on $\bar D$ endowed with the topology of uniform convergence.
   Recall also the definition of $\mathfrak G_n$;
	 each map $\mathfrak G_n$ is used to map samples
	 from $\gamma_n$ to $\operatorname{Lip}(\bar D)$.
	 This map characterizes the macroscopic profile of each sample.
%


Next, we define the empirical measure profile
$\mathfrak L_n(\phi)$
of the sample $\phi$
from $\gamma_n$.
The empirical measure profile captures the local statistics
of the height function $\phi$ in the large deviations principle.

\begin{definition}[topology for local statistics]
Write $\mathcal D$ for the Borel $\sigma$-algebra
on $D$,
and
recall that
$\mathcal M(X,\mathcal X)$
denotes the set of $\sigma$-finite measures on the measurable space $(X,\mathcal X)$.
Throughout this article,
we shall write
$\mathcal M^D$
for the set of measures
$\mu\in\mathcal M(D\times\Omega,\mathcal D\times\mathcal F^\nabla)$
which have the property that
the first marginal
$\mu_D=\mu(\cdot,\Omega)$
equals the Lebesgue measure
on $D$.
The empirical profile $\mathfrak L_n(\phi)\in \mathcal M^D$ of $\phi$ is now defined
by the equation
\[
    \mathfrak L_n(\phi):=\int_D \delta_{(x,\theta_{[nx]_\mathcal L}\phi)}dx,
\]
where $\delta$ denotes the Dirac measure
and $[nx]_\mathcal L$ is the vertex in $\mathcal L$
closest to $nx$ in the Euclidean metric---this is well-defined
for almost every $x$ with respect to the Lebesgue measure.
Thus, to ``sample'' from $\mathfrak L_n(\phi)$---this
language is abusive because the size of the measure
$\mathfrak L_n(\phi)$ is $\operatorname{Leb}(D)$
and therefore not generally a probability measure---one first
samples $x$ from $D$ uniformly at random;
then one shifts the sample $\phi$
by $[nx]_\mathcal L$.
The map $\mathfrak L_n:\Omega\to\mathcal M^D$
 thus captures the local statistics of the height functions in the large
 deviations principle.
 For the statement of the large deviations principle,
 we endow the space $\mathcal M^D$
 with the topology $\mathcal X^\mathfrak L$.
 This is defined to be the weakest topology
 which makes the map $\mu\mapsto \mu(R,f)$ continuous
 for any rectangular subset $R$ of $D$,
 and for any continuous cylinder function $f:\Omega\mapsto[0,1]$.
\end{definition}

\begin{remark*}
If $\phi$ is a height function, $R\subset D$
a bounded convex set of positive Lebesgue measure,
 and $n$ large,
then
\[
    \operatorname{Leb}(R)^{-1}\mathfrak L_n(\phi)(R,\cdot)\approx
    L_{\Lambda(nR)}(\phi).
\]
More precisely, the total variation distance between
the two measures goes to zero as $n\to\infty$,
uniformly over the choice of $\phi$.
\end{remark*}

\begin{definition}[Product topology for the large deviation principle]
  The large deviations principle is
  formulated on the space $X^\mathfrak P:=\operatorname{Lip}(\bar D)\times \mathcal M^D$
  endowed with the topology $\mathcal X^\mathfrak P:=\mathcal X^\infty\times\mathcal X^\mathfrak L$,
  and we map each sample $\phi$ from $\gamma_n$
  to this space by applying the map $\mathfrak P_n:=\mathfrak G_n\times\mathfrak L_n$.
\end{definition}

\subsubsection{The rate function}

Before proceeding, a few definitions for measures  $\mu\in\mathcal M^D$ are introduced.
The measure $\mu$ is called \emph{$\mathcal L$-invariant} if $\operatorname{Leb}(U)^{-1}\mu(U,\cdot)\in\mathcal P_\mathcal L(\Omega,\mathcal F^\nabla)$ for any $U\in\mathcal D$ of positive Lebesgue measure.
Write $\mathcal M^D_\mathcal L$
for the set of all such shift-invariant measures.
If $\mu$ is $\mathcal L$-invariant
and
$U\in\mathcal D$ has positive Lebesgue measure,
then we write
$S(\mu(U,\cdot))$ for the slope of  $\operatorname{Leb}(U)^{-1}\mu(U,\cdot)$.
Call a pair $(g,\mu)\in X^\mathfrak P$ \emph{compatible},
and write $g\sim\mu$, if $\mu$ is $\mathcal L$-invariant
with $\nabla g(x)=S(\mu(x,\cdot))$ as a distribution on $D$.
Finally, write $w_\mu$ for the ergodic decomposition
of the shift-invariant non-normalized measure $\mu(D,\cdot)$,
and define
\[
	\mathcal H(\mu|\Phi):=\mathcal H(\mu(D,\cdot)|\Phi):=\int \mathcal H(\nu|\Phi)dw_\mu(\nu)
	=\operatorname{Leb}(D)\mathcal H(\operatorname{Leb}(D)^{-1} \mu(D,\cdot)|\Phi).
\]

\begin{definition}
	Consider a good asymptotic boundary profile $(D,b)$.
The \emph{rate function} associated to this profile
is the function $I:X^\mathfrak P\to\mathbb R\cup\{\infty\}$ defined by
\[
I(g,\mu) :=\tilde I(g,\mu)-P_\Phi(D,b)
\quad\text{where}\quad
\tilde I(g,\mu) := \begin{cases}
\mathcal H(\mu|\Phi) &
\text{if $g|_{\partial D} = b$ and $g\sim\mu$,}\\
\infty & \text{otherwise.}
\end{cases}
\]
Here $P_\Phi(D,b)$ denotes the
 \emph{pressure}
	of $(D,b)$,  which is given by
	\[
		P_\Phi(D,b)
		:=
		\min_{\text{$g \in \operatorname{Lip}(\bar D)$ with $g|_{\partial D} = b$}} \int_D \sigma(\nabla g(x)) dx.
	\]
\end{definition}

The function $\tilde I$ is useful because its definition
does not appeal to the pressure.
It will later appear as the rate function of the LDP
corresponding to the sequence of measures $(\tilde\gamma_n)_{n\in\mathbb N}$
defined by $\tilde\gamma_n:=Z_n\gamma_n$,
the non-normalized versions of the local Gibbs measures
$\gamma_{D_n}(\cdot,b_n)$.

\begin{lemma}
	\label{lemma_rate_fcts}
	The following hold true:
	\begin{enumerate}
		\item The rate functions $I$ and $\tilde I$ are convex,
		\item The rate functions $I$ and $\tilde I$ are lower-semicontinuous,
		\item The lower level sets $\{I\leq C\}$ and $\{\tilde I\leq C\}$ are compact
		Polish spaces for $C<\infty$,
    \item There is a probability kernel $u\mapsto \mu_u$
		such that for any $u\in\{\sigma<\infty\}$,
		we have $\mu_u\in\mathcal P_\mathcal L(\Omega,\mathcal F^\nabla)$
		with $S(\mu_u)=u$
		and $\mathcal H(\mu_u|\Phi)=\sigma(u)$,
		%
		\item
    \label{minimizerindeedasexpected}
     For fixed $g\in\operatorname{Lip}(\bar D)$
		with $g|_{\partial D}=b$, we have
		\[
			\min_{\mu\in\mathcal M^D}\tilde I(g,\mu)=\int_D\sigma(\nabla g(x))dx,
		\]
		\item The minimum of $I$ is $0$, and the minimum of $\tilde I$ is $P_\Phi(D,b)$.
	\end{enumerate}
\end{lemma}

We provide a proof in the next subsection.

\subsubsection{Statement of the LDP}

\begin{theorem}[Large deviations principle]\label{p_LDP}
	Let $\Phi\in\mathcal S_\mathcal L+\mathcal W_\mathcal L$,
	 and let  $(D_n, b_n)_{n\in \mathbb N}$ denote a good approximation of some good
	 asymptotic profile $(D,b)$.
	Let $\gamma_n^*$ denote the pushforward
 of $\gamma_n:=\gamma_{D_n}(\cdot,b_n)$ along the map $\mathfrak P_n$,
 for any $n\in\mathbb N$.
	Then the sequence of probability measures
	$(\gamma_n^*)_{n\in\mathbb N}$
	satisfies a large deviations principle with speed $n^d$ and rate function
	$I$
	 on the topological space $(X^\mathfrak P,\mathcal X^\mathfrak P)$.
Moreover, the sequence of normalizing constants $(Z_n)_{n\in\mathbb N}:=(Z_{D_n}(b_n))_{n \in\N}$  satisfies
$
-
n^{-d} \log Z_n \to P_{\Phi}(D,g)
$
as $n\to\infty$.
\end{theorem}

Remark that Theorem~\ref{thm:sec_mr_ldp} follows immediately from this theorem
in combination with Lemma~\ref{lemma_rate_fcts}, Statement~\ref{minimizerindeedasexpected}.


\subsection{Proof overview}

Let us start with a proof
of some key properties of the rate functions $I$ and $\tilde I$.

\begin{proof}[Proof of Lemma~\ref{lemma_rate_fcts}]
	Note that $\mathcal M_\mathcal L^D$
	is closed in $(\mathcal M^D,\mathcal X^\mathfrak L)$.
	It follows immediately from the properties
	of the original specific free energy functional
	(Theorem~\ref{thm_main_sfe}) that the map
	\[
		\mathcal H(\cdot|\Phi):\mathcal M_\mathcal L^D
		\to \mathbb R\cup\{\infty\}
	\]
	is affine and lower-semicontinuous,
	and that its lower level sets
	are compact Polish spaces with respect to the $\mathcal X^\mathfrak L$-topology.

	Observe that the set $\{g\sim\mu\}\subset X^\mathfrak P$
	is convex.
	This implies that $I$ and $\tilde I$
	are convex,
	since the map $(g,\mu)\mapsto\mathcal H(\mu|\Phi)$ is affine
	on $\{g\sim\mu\}$.
	Observe that the set $\{g\sim\mu\}$
	is also closed in $\mathcal X^\mathfrak P$.
	The lower level sets of $I$ and $\tilde I$
	are compact Polish spaces because
	\[
		\{\tilde I\leq C\}=(\{g\in\operatorname{Lip}(\bar D):g|_{\partial D}=b\}\times\{\mu\in\mathcal M^D_\mathcal L:\mathcal H(\mu|\Phi)\leq C\})\cap\{g\sim\mu\},
	\]
	that is, $\{\tilde I\leq C\}$ is as a closed
	subset of a product of two
	compact Polish spaces.
	This also implies that $I$ and $\tilde I$ are lower-semicontinuous.

	The fourth statement is a simple exercise
	in measure theory; it follows from the topological
	properties of the specific free energy
	stated in Theorem~\ref{thm_main_sfe}.
	If $g\sim\mu$,
	then it is clear that
	\[
		\tilde I(g,\mu)= \int_D \mathcal H(\mu(x,\cdot)|\Phi) dx
		\geq \int_D\sigma(S(\mu(x,\cdot))) dx
		=\int_D\sigma(\nabla g(x)) dx.
	\]
	For fixed $g$, this inequality can be turned
	into an equality, by constructing
	$\mu$ in terms of $\nabla g$ and the kernel from the fourth
	statement. This proves the fifth statement.
	The final statement is now obvious.
\end{proof}

Theorem~\ref{p_LDP} states the LDP for the sequence of normalized
measures $(\gamma_n)_{n\in\mathbb N}$.
For the proof, however, it will be beneficial to consider also the
sequence of non-normalized measures $(\tilde\gamma_n)_{n\in\mathbb N}$.
Write $\tilde\gamma_n^*$ for the pushforward of $\tilde\gamma_n$
along $\mathfrak P_n$.
Theorem~\ref{p_LDP} is equivalent to the conjunction of the following two
statements:
\begin{enumerate}
	\item The minimum of $\tilde I$ is $P_\Phi(D,b)$,
	\item The sequence $(\tilde\gamma_n^*)_{n\in\mathbb N}$
	satisfies an LDP with speed $n^d$
	and rate function $\tilde I$ in $(X^\mathfrak P,\mathcal X^\mathfrak P)$.
\end{enumerate}
The first statement was proven in Lemma~\ref{lemma_rate_fcts}.
The second statement is somewhat easier to prove than the
original LDP, because it appeals to non-normalized measures only.

Let us first describe a particular basis for the topological space
\[
(X^\mathfrak P,\mathcal X^\mathfrak P)=(\operatorname{Lip}(\bar D),\mathcal X^\infty)
\times (\mathcal M^D,\mathcal X^\mathfrak L).
\]
As a basis $\mathcal B^\infty$ for $\mathcal X^\infty$,
we take the sets of the form
\[
	B_\varepsilon^\infty(g):=\{h\in\operatorname{Lip}(\bar D):\|h-g\|_\infty<\varepsilon\}
\]
where $g\in \operatorname{Lip}(\bar D)$ and $\varepsilon>0$.
 Write
 \[
 	B^\mathfrak L_\varepsilon(\mu,(R_i)_{ i},(f_j)_{j}):=\{\nu:\text{$|\mu(R_i,f_j)-\nu(R_i,f_j)|<\operatorname{Leb}(R_i)\varepsilon$ for all $i,j$}\}
	\subset \mathcal M^D
	,
 \]
 where $\varepsilon>0$,
 $\mu$ is a measure in $\mathcal M^D$,
 $(R_i)_{ i}$ is a finite collection
 of closed rectangular subsets of $D$,
 and
 $(f_j)_{j}$ is a finite collection
 of continuous cylinder functions $f_j:\Omega\to[0,1]$.
 The collection $\mathcal B^\mathfrak L$
 of such sets
 forms a basis of $\mathcal X^\mathfrak L$.
As a basis $\mathcal B^\mathfrak P$  for $\mathcal X^\mathfrak P$,
we choose the collection of
open sets of the form $B_\varepsilon^\mathfrak P(\cdot,\cdot,\cdot,\cdot):=
B^\infty_\varepsilon(\cdot)\times B^\mathfrak L_\varepsilon(\cdot,\cdot,\cdot)$.

To prove a large deviations principle,
it must first be checked that the rate function is lower-semicontinuous.
For this refer again to Lemma~\ref{lemma_rate_fcts}.
The large deviations principle
(with non-normalized measures)
is now a corollary of the following three claims:
	\begin{enumerate}
		\item \emph{Lower bound on probabilities.}
		For any $(g,\mu)\in A\in\mathcal B^\mathfrak P$,
		we have
		\[
			\liminf_{n\to\infty}n^{-d}\log\tilde\gamma_n^*(A)\geq -\tilde I(g,\mu).
		\]
		\item \emph{Upper bound on probabilities.}
		For any $(g,\mu)\in X^\mathfrak P$,
		we have
		\[
		\inf_{\text{$A\in\mathcal B^\mathfrak P$ with $(g,\mu)\in A$}}	\limsup_{n\to\infty}n^{-d}\log\tilde\gamma_n^*(A)\leq -\tilde I(g,\mu).
		\]
		\item \emph{Exponential tightness.} For all $\alpha >- \infty$, there is a compact set $K_\alpha \subset X^\mathfrak P$ such that
		\[
		\limsup_{n \to \infty}  n^{-d}\log\tilde\gamma_n^*(X^\mathfrak P \smallsetminus K_{\alpha} ) \leq \alpha
		\]
	\end{enumerate}
The next subsection contains an auxiliary result on approximations of Lipschitz functions
which is useful for proving the lower bound.
Each of the three subsequent sections addresses one of the three claims formulated above.


\subsection{Simplicial approximations of Lipschitz function}

This subsection is dedicated to providing some results on affine approximations of Lipschitz functions necessary to prove the lower bound on probabilities.
For  $x \in \mathbb R^d$,
the point $\lfloor x\rfloor\in\mathbb Z^d$
is obtained by rounding down each coordinate.

\begin{definition}
	Let $S_d$ denote the group of permutations on $\{1, \dots, d\}$.
	For $x\in\mathbb R$,
	we write $s(x)\in S_d$ for the permutation which
	rank-orders the coordinate indices of $x-\lfloor x\rfloor$.
	For $x \in \Z^d$ and $s \in S_d$, we define
	the \emph{simplex} $C(x,s)$
	to be the closure of the set
	\begin{equation*}
	\{ y \in \R^d : \lfloor y\rfloor  = x, \, s(y) = s \} .
\end{equation*}
	By a \emph{simplex of scale $\varepsilon$},
	we simply mean a scaled simplex of the form
	$\varepsilon C(x,s)$.
	A \emph{simplex domain of scale $\varepsilon$} is a
	 union of finitely many simplices of
	scale $\varepsilon$.
	If $D$ is a domain,
	then write $D_\varepsilon$ for the largest
	simplex domain of scale $\varepsilon$
	contained in $D$.
\end{definition}

\begin{definition}
	Let $D$ denote a domain,
	and $g$ a real-valued function on $D$.
	Consider $\varepsilon>0$.
	Write $F_\varepsilon=F_\varepsilon(g)$ for the unique real-valued function
	on $D_\varepsilon$ which equals $g$ on $D_\varepsilon\cap \varepsilon\mathbb Z^d$,
	interpolated linearly on each simplex.
\end{definition}

We will make use of the simplicial Rademacher theorem proven in~\cite{krieger2019deducing} for which we recall a statement here.

\begin{lemma}[Lemma 6.1 from~\cite{krieger2019deducing}] \label{lem_approx_tri}
	Consider a positive homogeneous function
	$\|\cdot\|:\mathbb R^d\to\mathbb R$ satisfying the triangle inequality.
	Let $D\subset\mathbb R^d$ be a domain and $g:D \to \R$  a $\|\cdot\|$-Lipschitz function.
	For any $\delta > 0$
	and any $\varepsilon > 0$ sufficiently small (depending on $\delta$),
	we have
	\begin{enumerate}
		\item
		$\operatorname{Leb}(D \smallsetminus D_\varepsilon) \leq \delta$,
		\item
		$\|F_{\varepsilon} -g|_{{D_\varepsilon}}\|_\infty \leq \delta \varepsilon$,
		\item
		$
		 \operatorname{Leb} (\{x\in D_\varepsilon: \|\nabla F_{\varepsilon}(x) - \nabla g(x)\|_2 \geq \delta\}
		)
		\leq \delta
		$.
	\end{enumerate}
	Moreover, $F_\varepsilon$ is $\|\cdot\|$-Lipschitz for any $\varepsilon>0$.
\end{lemma}

The first property is obvious,
and the proof of the second and third property is identical
to the proof in~\cite{krieger2019deducing}.


\subsection{The lower bound on probabilities}
\label{subsection:lower_bound_on_probabilities_in_the_LDP}
Fix
$(g,\mu)\in A\in\mathcal B^\mathfrak P$ and $\beta>0$;
the goal of this subsection is to prove that
\[
	\liminf_{n\to\infty}n^{-d}\log\tilde\gamma_n^*(A)\geq -\tilde I(g,\mu)-\beta.
\]
We suppose of course that $\tilde I(g,\mu)$ is finite.
For the proof, we require the following result.

%

\begin{lemma}
	\label{lemma_boring_slope_change}
	Consider some fixed $\varepsilon>0$.
	Then there exists a sufficiently
	small constant
	$\alpha>0$
	 such that the following statement holds true.
	 Suppose that $\mu\in\mathcal P_\mathcal L(\Omega,\mathcal F)$
	 satisfies $u:=S(\mu)\in \bar U_{q_\varepsilon}\subset U_q$,
	 and that $v\in U_q$
	 is another slope with $\|u-v\|_2\leq \alpha$.
	 Then there is another measure
	 $\nu\in\mathcal P_\mathcal L(\Omega,\mathcal F)$
	 such that $S(\nu)=v$
	 and $\mathcal H(\nu|\Phi)\leq \mathcal H(\mu|\Phi)+\varepsilon$
	 and
	 $\|\mu-\nu\|_{\operatorname{TV}}:=\|\mu-\nu\|_{\infty}< \varepsilon$.
	 In particular,
	 if $f:\Omega\to[0,1]$ is measurable, then
	 $|\mu(f)-\nu(f)|<\varepsilon$.
\end{lemma}

\begin{proof}
	Note that $\sigma$ is bounded uniformly
	on a neighborhood $A$ of $\bar U_{q_\varepsilon}$.
	The proof of the lemma is straightforward:
	one simply defines $\nu:=(1-t)\mu+t\mu'$
	for $t$ small and $\mu'$
	some minimizer with
	$S(\mu')\in A$ in order to adjust the slope of the measure of interest.
\end{proof}

\begin{proof}[Proof of the lower bound on probabilities]
	It suffices to consider the case that $\tilde I(g,\mu)$
	is finite.
	We claim that it is sufficient
	to consider the case
	that
	$g$ is strictly $\|\cdot\|_q$-Lipschitz
	(if $E=\mathbb R$)
	or that $g|_D$ is locally strictly
	$\|\cdot\|_q$-Lipschitz (if $E=\mathbb Z$).
	If $g$ were not $\|\cdot\|_q$-Lipschitz and
	$g\sim\mu$, then $\mu$ cannot be supported
	on $q$-Lipschitz functions, and consequently $\tilde I(g,\mu)=\infty$.
	Therefore $g$ must be $\|\cdot\|_q$-Lipschitz.
	There is some pair $(h,\nu)\in X^\mathfrak P$
	such that $h$ is strictly $\|\cdot\|_q$-Lipschitz
	(if $E=\mathbb R$)
	or such that $h|_D$ is locally strictly
	$\|\cdot\|_q$-Lipschitz (if $E=\mathbb Z$),
	and such that $\tilde I(h,\nu)<\infty$---this follows from
	the definition of a good
	asymptotic boundary profile
	and
	from Lemma~\ref{lemma_rate_fcts}.
Define $g_t:=(1-t)g+th$
and $\mu_t:=(1-t)\mu+t\nu$.
Then
$(g_t,\mu_t)\in A$ for $t$ sufficiently small
and $\limsup_{t\to 0}\tilde I(g_t,\mu_t)\leq \tilde I(g,\mu)$
as $\tilde I$ is convex.
Moreover, for any $t>0$,
the function $g_t$ has the desired properties.
Thus, we may replace $(g,\mu)$ by $(g_t,\mu_t)$
for small $t$,
by choosing $\beta$ smaller if necessary.
This proves the claim.

The proof follows the general strategy that was outlined after the statement
of Lemma~\ref{lemma_PB_star_convex}.
Let us first consider the case that $E=\mathbb R$.
We find an appropriate approximation of $g$
using the simplicial Rademacher theorem,
and then apply Lemma~\ref{lemma_boring_slope_change}
and the limit equalities to obtain the desired
lower bound on probabilities.
For the approximations,
it is necessary to take limits in three variables:
first we take $n\to\infty$,
then $\varepsilon_2\to 0$, and finally
$\varepsilon_1\to 0$.
There is also another variable $\varepsilon$;
it is not necessary to take a limit in this variable,
but it must be small for the arguments to work.

First fix $\varepsilon>0$
so small that $b$ and $g$ are $\|\cdot\|_{q_{8\varepsilon}}$-Lipschitz, and such that
all functions $b_n$ are $q_{8\varepsilon}$-Lipschitz.
We also suppose that $A=B^\mathfrak P_{8\varepsilon}(g,\mu,(R_i)_i,(f_j)_j)$,
by choosing $\varepsilon$
and $A$ smaller if necessary,
where $(R_i)_i$ is a finite family of rectangular subsets
of $D$, and $(f_j)_j$ a finite family of continuous cylinder functions $f_j:\Omega\to[0,1]$.

Consider some
 $\varepsilon_1>0$, and write
$D'$
for the points in $D$ at distance more than $\varepsilon_1$
from the complement of $D$.
Consider additionally some $\varepsilon_2>0$,
and write $D''$ for $D'_{\varepsilon_2}$:
the largest simplex domain of scale $\varepsilon_2$ contained in $D'$.
See Figure~\ref{fig:lower_bound}
for a drawing of this construction.
Write
$F=F(g)$
for the unique $\|\cdot\|_{q_{8\varepsilon}}$-Lipschitz function on $D''$
which
equals $g$ on $\varepsilon_2\mathbb Z^d\cap D''$,
and which is affine on each simplex of $D''$.
For $\varepsilon_2$ sufficiently small,
this function has a $\|\cdot\|_{q_{7\varepsilon}}$-Lipschitz
extension $\bar F$ to $\bar D$ which equals $b$ on $\partial D$.
It is clear that any such extension $\bar F$ is
contained in $ B_\varepsilon^\infty(g)$,
that is, $\|\bar F-g\|_\infty<\varepsilon$,
for $\varepsilon_1$ and $\varepsilon_2$
sufficiently small.

\begin{figure}
	\centering
  \begin{tikzpicture}
		\clip (-5, -3) rectangle (7, 5.2);
		\input{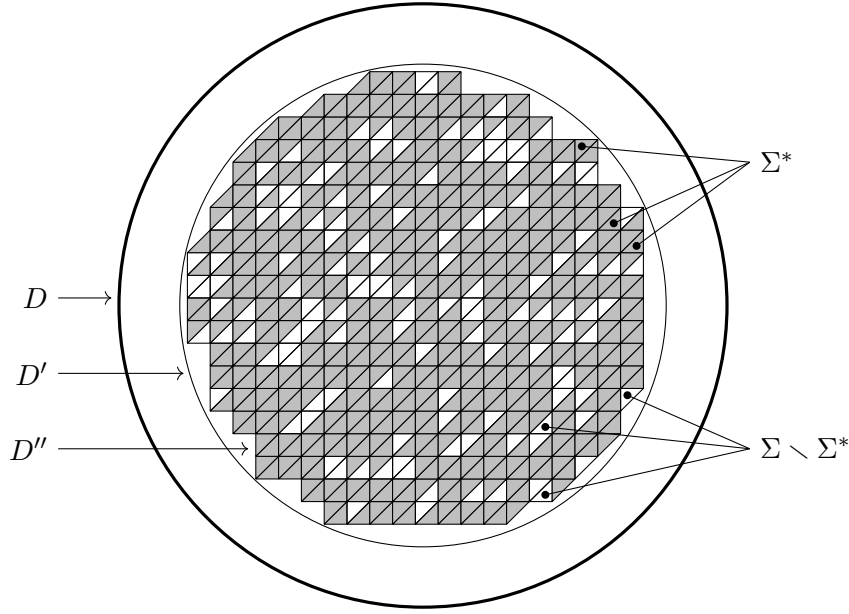}
		\draw[->] (-3.8,-0.8) node[left] {$D''$}-- (-1.3,-0.8);
		\draw[->] (-3.8,.2)node[left] {$D'$} -- (-2.15,.2);
		\draw[->] (-3.8,1.2)node[left] {$D$} -- (-3.1,1.2);

		\draw (5.3,-0.8)node[right] {$\Sigma\smallsetminus\Sigma^*$} -- (2.61,-.51);
		\fill (2.61,-.51) circle(1.5pt);
		\draw (5.3,-0.8) -- (2.61,-1.41);
		\fill (2.61,-1.41) circle(1.5pt);
		\draw (5.3,-0.8) -- (3.69,-.09);
		\fill (3.69,-.09) circle(1.5pt);

		\draw (5.3,3)node[right] {$\Sigma^*$} -- (3.81,1.89);
		\fill (3.81,1.89) circle(1.5pt);
		\draw (5.3,3) -- (3.51,2.19);
		\fill (3.51,2.19) circle(1.5pt);
		\draw (5.3,3) -- (3.09,3.21);
		\fill (3.09,3.21) circle(1.5pt);

  \end{tikzpicture}
  \caption{The sets $D''\subset D'\subset D\subset\mathbb R^d$,
	and the sets $\Sigma^*\subset\Sigma$ of simplices
	of scale $\varepsilon_2$}
	\label{fig:lower_bound}
\end{figure}

Write $\Sigma$ for the set of simplices
of scale $\varepsilon_2$ in $D''$---this is
a finite set.
The slope $\nabla F$ of $F$ is constant
on any $\Delta\in\Sigma$;
write $S(\Delta)\in \bar U_{q_{8\varepsilon}}$ for this slope.
Write $\Sigma^*$ for the set of
simplices $\Delta\in\Sigma$
for which $\|S(\Delta)-S(\mu(\Delta,\cdot))\|_2\leq \varepsilon_1$;
Lemma~\ref{lem_approx_tri} asserts
that $|\Sigma^*|/|\Sigma|\geq 1-\varepsilon_1$
for $\varepsilon_2$
sufficiently small.
See again Figure~\ref{fig:lower_bound}
for an example of the sets $\Sigma$
and $\Sigma^*$.

Choose $C$ minimal subject to
 $\|\phi^u-u|_{\mathbb Z^d}\|_\infty+1\leq C$ for all $u\in U_\Phi$.
Let $M$ denote a constant which
makes Theorem~\ref{general_lips_ext_less_general}
work for the local Lipschitz constraint
$q_{6\varepsilon}$,
and for the constants $\varepsilon$
and $C$---this constant $M$ depends on $\varepsilon$
 only. We shall also suppose that $M\geq R$, by choosing $M$ larger if necessary.
 For $\Delta\in \Sigma$
 and $n\in\mathbb N$,
 define
 $
		\Delta_n:=\Lambda^{-M}(n\Delta)$.
 Write also $D_n'':=\cup_{\Delta\in\Sigma}\Delta_n$
 and $D_n^*:=\cup_{\Delta\in\Sigma^*}\Delta_n$.
 It follows from the definition of an
  approximation that $D_n''\subset D_n$
 for $n$ sufficiently large.
 By Theorem~\ref{general_lips_ext_less_general}
 there exists, for any $n\in\mathbb N$, a $q_{6\varepsilon}$-Lipschitz
 function $F_n:D_n''\to
 E$
  	such that:
		\begin{enumerate}
			\item $|F_n(x)-nF(x/n)| \leq C$
			for all $x\in D_n''$,
			\item $\nabla F_n|_{\Delta_n} =\nabla \phi^{S(\Delta)}|_{\Delta_n}$ for all $\Delta\in\Sigma$.
		\end{enumerate}
		It is straightforward to see
		that for $n$ sufficiently large,
		the function $F_n$ extends to
		a $q_{5\varepsilon}$-Lipschitz height function $\bar F_n$
		which equals $b_n$ on the complement of $D_n$.

		We now use the existence of the function
		$\bar F_n$
		to demonstrate that
		there exists a set
		$A_n\in\mathcal F$
		such that $\mathfrak P_n(A_n)\subset A$,
		and for which we show that $\tilde\gamma_n(A_n)$
		is sufficiently large
		as $n\to\infty$.
		Define  $A_n$ to be the set of height functions $\phi$
		which are $q$-Lipschitz, and
		which satisfy the following criteria:
		\begin{enumerate}
			\item If $x\in\mathbb Z^d\smallsetminus D_n$,
			then $\phi(x)=\bar F_n(x)=b_n(x)$,
			\item If $x\in D_n\smallsetminus D_n^*$,
			then $|\phi(x)-\bar F_n(x)|\leq \varepsilon$,
			\item If $x=0_{\Delta_n}$
			for some $\Delta\in\Sigma^*$,
			then $|\phi(x)-\bar F_n(x)|\leq \varepsilon$,
			\item For each $\Delta\in\Sigma^*$,
			we have $\phi\in C_{\Delta_n,\varepsilon}^{S(\Delta)}$,
			\item
			For each $\Delta\in\Sigma^*$,
			we have $\phi\in B^\Delta_{\Delta_n}$, that is,
			 $L_{\Delta_n}(\phi)\in B^\Delta$,
			where
			\[B^\Delta:=\{\nu\in\mathcal P(\Omega,\mathcal F^\nabla):\text{$|\operatorname {Leb}(\Delta)\nu(f_j)
			-\mu(\Delta,f_j)|<
			\operatorname{Leb}(\Delta)\varepsilon$ for all $j$}\}\in\mathcal B.
			 \]
		\end{enumerate}
		It suffices to demonstrate that
		for $\varepsilon$, $\varepsilon_1$,
		and $\varepsilon_2$ sufficiently small,
		and for $n$ sufficiently large,
		we have $\mathfrak P_n(A_n)\subset A$
		and
		\[
			\liminf_{n\to\infty}n^{-d}\log\tilde\gamma_n(A_n)\geq -I(g,\mu)-\beta.
		\]
		Claim first that, in the limit, $\mathfrak P_n(A_n)\subset A$.
		This is equivalent to asking that  $\mathfrak G_n(A_n)\subset B_{8\varepsilon}^\infty(g)$
		and $\mathfrak L_n(A_n)\subset B_{8\varepsilon}^\mathfrak L(\mu,(R_i)_i,(f_j)_j)$.
		The former of the two holds true because
		$\|\bar F-g\|_\infty<\varepsilon$ and because
		$\|\bar F-\mathfrak G_n(\phi)\|_\infty$ is small in the described limit,
		uniformly over the choice of $\phi\in A_n$.
		The proof that $\mathfrak L_n(A_n)\subset B_{8\varepsilon}^\mathfrak L(\mu,(R_i)_i,(f_j)_j)$ in the
		limit relies again on Proposition~\ref{propo_compare_empirical_measures};
		observe in particular that in the limit
		most of the volume of each fixed rectangle $R_i$
		is covered by simplices in $\Sigma^*$
		which are entirely contained in $R_i$.

		In the sequel, we shall pretend that $A_n\in \mathcal E^{D_n}$ by
		restricting each height function in $A_n$ to $D_n$.
		If $\phi\in E^{D_n}$,
		then we write $\psi$ for the height function which
		restricts to $\phi$ on $D_n$
		and to $b_n$ on the complement of $D_n$.
		We aim to find an asymptotic lower bound on
		\[
		n^{-d}	\log\tilde\gamma_n(A_n)=n^{-d}\log\int_{A_n}e^{-H_{D_n}(\psi)}d\lambda^{D_n}(\phi).
		\]
		If $\phi\in A$, then $\psi$ is $q_{\varepsilon}$-Lipschitz
		whenever restricted to $\mathbb Z^d\smallsetminus\cup_{\Delta\in\Sigma^*}\Delta_n^{-R}$,
		because $\bar F_n$
		is $q_{5\varepsilon}$-Lipschitz and
		because
		$\psi$ and $\bar F_n$ differ by at most $2\varepsilon$
		at each vertex in this set.
		Therefore the upper attachment lemma
		(Lemma~\ref{lemma_attachment}) implies
		\[
			H_{D_n}(\psi)\leq H_{D_n\smallsetminus D_n^*}^0(\psi)+e_\varepsilon^+(D_n)
			+\sum_{\Delta\in\Sigma^*}
			H_{\Delta_n}^0(\psi)+e_\varepsilon^+(\Delta_n)
		\]
		for any $\phi\in A$.
		For fixed $\varepsilon$, $\varepsilon_1$, and $\varepsilon_2$,
		the terms of the form $e_\varepsilon^+(\cdot)$
		in this expression are of order $o(n^d)$ as $n\to\infty$,
		and therefore we may omit them in calculating the limit inferior.
		Moreover, since $\psi$ is $q_{\varepsilon}$-Lipschitz on $D_n\smallsetminus D_n^*$,
		the term $H_{D_n\smallsetminus D_n^*}^0(\psi)$ has an upper bound
		$C'|D_n\smallsetminus D_n^*|$, where $C'$ depends on $\varepsilon$ only.
		In particular,
		\[
			\liminf_{n\to\infty}n^{-d}\log\tilde\gamma_n(A_n)\geq \liminf_{n\to\infty}n^{-d}\left[-C'|D_n\smallsetminus D_n^*|+\log\int_{A_n}e^{-\sum_{\Delta\in\Sigma^*} H_{\Delta_n}^0(\psi)}d\lambda^{D_n}(\phi)\right].
		\]
		It follows from the definition of $A_n$,
		that the integral decomposes as follows:
		\begin{multline*}
			\int_{A_n}e^{-\sum_{\Delta\in\Sigma^*}H_{\Delta_n}^0(\psi)}d\lambda^{D_n}(\phi)
			\\=\left[\prod_{x\in D_n\smallsetminus D_n^*}\int_{\bar F_{n}(x)-\varepsilon}^{\bar F_{n}(x)+\varepsilon} d\lambda
			\right]
			\left[
			\prod_{\Delta\in\Sigma^*}\int_{\bar F_{n}(0_{\Delta_n})-\varepsilon}^{\bar F_{n}(0_{\Delta_n})+\varepsilon} d\lambda
			\right]
			\left[
			\prod_{\Delta\in\Sigma^*}
			\int_{C_{\Delta_n,\varepsilon}^{S(\Delta)}\cap B^\Delta_{\Delta_n}}
			e^{-H^0_{\Delta_n}}d\lambda^{\Delta_n-1}
			\right],
		\end{multline*}
		and therefore the logarithm of this integral equals
		\[
		(|D_n\smallsetminus D_n^*|+|\Sigma^*|)\log 2\varepsilon
		-\sum_{\Delta\in\Sigma^*}\PB_{\Delta_n,S(\Delta),\varepsilon}(B^\Delta).
		\]
		But $|\Sigma^*|$ does not depend on $n$,
		and by choosing $C'$ larger, we obtain
		\[
			\liminf_{n\to\infty}n^{-d}\log\tilde\gamma_n(A_n)\geq  \liminf_{n\to\infty}n^{-d}\left[-C'|D_n\smallsetminus D_n^*|-\sum_{\Delta\in\Sigma^*}\PB_{\Delta_n,S(\Delta),\varepsilon}(B^\Delta)\right]
			.
		\]
		It is easy to see that $n^{-d}|D_n\smallsetminus D_n^*|\to \operatorname{Leb}(D\smallsetminus \cup \Sigma^*)$
		as $n\to\infty$.
		Fix $\Delta\in\Sigma^*$.
		By definition of $\Sigma^*$,
 		we have $\|S(\Delta)-S(\mu(\Delta,\cdot))\|_2\leq \varepsilon_1$.
		Then Lemma~\ref{lemma_boring_slope_change} tells us
		that for $\varepsilon_1$ sufficiently small,
		the set $B^\Delta$ contains another
		shift-invariant measure $\nu$ of slope $S(\Delta)$ such
		that $\mathcal H(\nu|\Phi)\leq \mathcal H(\operatorname{Leb}(\Delta)^{-1}\mu(\Delta,\cdot)|\Phi)+\varepsilon$.
		In particular, this means that
		\begin{multline*}
			\limsup_{n\to\infty}n^{-d}\PB_{\Delta_n,S(\Delta),\varepsilon}(B^\Delta)
			\leq
			\operatorname{Leb}(\Delta)\left(\mathcal H(\operatorname{Leb}(\Delta)^{-1}\mu(\Delta,\cdot)|\Phi)+\varepsilon\right)
			\\=
			\mathcal H(\mu(\Delta,\cdot)|\Phi)+\operatorname{Leb}(\Delta)\varepsilon.
		\end{multline*}
		Conclude that
		\[
			\liminf_{n\to\infty}n^{-d}\log\tilde\gamma_n(A_n)\geq
			-C'\operatorname{Leb}(D\smallsetminus\cup\Sigma^*)
			-\mathcal H(\mu(\cup\Sigma^*,\cdot)|\Phi)
			-\varepsilon\operatorname{Leb}(\cup\Sigma^*)
			.
		\]
		As $\varepsilon_2\to 0$ and then $\varepsilon_1\to 0$,
		we have
		\[
			\operatorname{Leb}(\cup\Sigma^*)\to \operatorname{Leb}(D),
			\qquad
			\operatorname{Leb}(D\smallsetminus\cup\Sigma^*)
			\to
			0,
			\qquad
			\mathcal H(\mu(\cup\Sigma^*,\cdot)|\Phi)\to
			\mathcal H(\mu|\Phi).
		\]
		The desired lower bound is thus obtained by setting
		$\varepsilon$ so small that $\varepsilon\operatorname{Leb}(D)<\beta$.

		Let us finally describe what changes for $E=\mathbb Z$.
		The first part of the proof is the same,
		except that the functions $F_n$ and $\bar F_n$ are $q$-Lipschitz,
		and not $q_{6\varepsilon}$-Lipschitz or $q_{5\varepsilon}$-Lipschitz.
		The $q$-Lipschitz extension $\bar F_n$ exists for $n$ sufficiently
		large, because $g|_D$ is locally strictly $\|\cdot\|_q$-Lipschitz.
		The only thing that changes in the remainder of the proof
		is that $\lambda$ is now the counting measure
		rather than the Lebesgue measure.
		This makes the remainder of the proof easier, exactly as in the proof of Lemma~\ref{lemma_PB_star_convex}.
	\end{proof}


\subsection{The upper bound on probabilities}

\begin{proof}[Proof of the upper bound on probabilities]
Let us first consider the case $\tilde I(g,\mu)<\infty$,
in which case $\tilde I(g,\mu)=\mathcal H(\mu|\Phi)$.
Let $\Sigma$ denote a finite set of closed disjoint rectangles,
contained in $D$.
Define $R_n:=\Lambda(nR)$ for $R\in\Sigma$
and $\Sigma_n:=\cup_{R\in\Sigma}R_n$,
and note that $\Sigma_n\subset D_n$ for $n$ sufficiently large.
Now choose for each $R\in\Sigma$ an open set $B^R\in\mathcal B$
with $\mu(R,\cdot)/\operatorname{Leb}(R)\in B^R$,
and define \[A_n:=\cap_{R\in\Sigma}B^R_{R_n}.\]
It is straightforward to show that
$(g,\mu)$ has a fixed neighborhood which is contained in all sets $\mathfrak P_n(A_n)$
for $n$ sufficiently large.
Fix $\beta>0$.
It suffices to find an appropriate choice for the set of rectangles $\Sigma$
and the collection of balls $(B^R)_{R\in\Sigma}$, such that
 \[
 \limsup_{n\to\infty}n^{-d}\log
 \tilde\gamma_n(A_n)
 \leq -\mathcal H(\mu|\Phi)+\beta.
 \]

Remove all height functions from $A_n$ which
do not equal $b_n$ on $\mathbb Z^d\smallsetminus D_n$
or which are not $Kd_1$-Lipschitz;
this obviously does not change the value of $\tilde\gamma_n(A_n)$.
As in the proof of the lower bound,
we shall sometimes pretend that $A_n\in \mathcal E^{D_n}$ by
restricting each height function in $A_n$ to $D_n$.
If $\phi\in E^{D_n}$,
then we write $\psi$ for the height function which
restricts to $\phi$ on $D_n$
and to $b_n$ on the complement of $D_n$.
We are thus interested in the asymptotic behavior of
\[
  \tilde\gamma_n(A_n)=\int_{A_n}e^{-H_{D_n}(\psi)}d\lambda^{D_n}(\phi).
\]

The lower attachment lemma (Lemma~\ref{lemma_lower_attachments}) asserts that
\begin{align*}
  H_{D_n}&\geq
  H_{D_n\smallsetminus\Sigma_n}^0
  -
  e^-(D_n)
  +
  \sum_{R\in\Sigma}
  H_{R_n}^0
  -e^-(R_n)
  \\&\geq
  -\|\Xi\|\cdot|D_n\smallsetminus\Sigma_n|
  -
  e^-(D_n)
  +
  \sum_{R\in\Sigma}
  H_{R_n}^0
  -e^-(R_n).
\end{align*}
The terms of the form $e^-(\cdot)$ are of order $o(n^d)$
as $n\to\infty$. Moreover, $n^{-d}|D_n\smallsetminus\Sigma_n|\to\operatorname{Leb}(D\smallsetminus\cup\Sigma)$
as $n\to\infty$,
and
therefore
\[
\limsup_{n\to\infty}n^{-d}\log \tilde\gamma_n(A_n)
\leq \|\Xi\|\operatorname{Leb}(D\smallsetminus\cup\Sigma)
+\limsup_{n\to\infty}n^{-d}\log
\int_{A_n}e^{-\sum_{R\in\Sigma}
H_{R_n}^0(\psi)}d\lambda^{D_n}(\phi).
\]
Write $D_n^0:=D_n\smallsetminus\cup_{R\in\Sigma}(R_n\smallsetminus\{0_{R_n}\})$,
so that
$
\lambda^{D_n}=\lambda^{D_n^0}\times\prod_{R\in\Sigma}\lambda^{R_n-1}
$.
Then
\[
\int_{A_n}e^{-\sum_{R\in\Sigma}
H_{R_n}^0(\psi)}d\lambda^{D_n}(\phi)
\leq
\left[
\int_{W_n}
d
\lambda^{D_n^0}
\right]
\left[\prod_{R\in\Sigma}
\int_{B^R_{R_n}} e^{-H_{R_n}^0}d\lambda^{R_n-1}
\right],
\]
where $W_n$ is the set of $Kd_1$-Lipschitz functions
$\phi:D_n^0\to E$
such that $\phi b_n|_{\mathbb Z^d\smallsetminus D_n}$
is also $Kd_1$-Lipschitz.
Remark that
\[\log \int_{W_n}
d
\lambda^{D_n^0}
\leq |D_n^0|\log (2K+1)\]
and that
\[
\log\int_{B^R_{R_n}} e^{-H_{R_n}^0}d\lambda^{R_n-1}=-\FB_{R_n}(B^R).
\]
If we write $m:=\|\Xi\|+\log(2K+1)$,
then we have now shown that
\[
\limsup_{n\to\infty}n^{-d}\log \tilde\gamma_n(A_n)
\leq
m\operatorname{Leb}(D\smallsetminus\cup\Sigma)
-
\sum_{R\in\Sigma}\operatorname{Leb}(R)\FB(B^R).
\]
It now suffices to show that the expression
on the right is at most $-\mathcal H(\mu|\Phi)+\beta$
for an appropriate choice of the set of rectangles in $\Sigma$
and for the collection $(B^R)_{R\in\Sigma}\subset\mathcal B$.
By choosing the rectangles in $\Sigma$ such that they
exhaust most of the space, we can ensure that $\operatorname{Leb}(D\smallsetminus\cup\Sigma)\leq \beta/2m$,
and by also taking each ball $B^R$ sufficiently small,
we can ensure that $\sum_{R\in\Sigma}\operatorname{Leb}(R)\FB(B^R)$
is at least $\mathcal H(\mu|\Phi)-\beta/2$.
This proves the upper bound on probabilities.

Consider now the case that $\tilde I(g,\mu)=\infty$.
We distinguish several reasons which may cause $\tilde I(g,\mu)$ to be infinite.
If $\mu$ is shift-invariant but $\mathcal H(\mu|\Phi)=\infty$,
then the proof is the same as before.
If $\mu$ is not shift-invariant, then there is a closed rectangle $R\subset D$
such that $\mu(R,\cdot)$ is not shift-invariant,
and by including $R$ in $\Sigma$ and using the free boundary limits
for non-shift-invariant measures,
we obtain the same result. In fact, in that case it is readily seen
that $A_n$ is empty for $n$ sufficiently large.
(See also the proof of Lemma~\ref{lemma_FB_non_shift-inv}).

The remaining cases are:
either $g|_{\partial D}$ does not equal $b$,
or it is not true that $\nabla g(x)=S(\mu(x,\cdot))$
as a distribution on $D$.
Consider first the case that $g|_{\partial D}$ does not equal $b$.
Choose $\varepsilon:=\|g|_{\partial D}-b\|_\infty/2$.
In that case, it is readily seen that $\tilde\gamma_n(\mathfrak G_n^{-1}(B^\infty_\varepsilon(g)))=0$
for $n$ sufficiently large.
Finally consider the case that it is not true that $\nabla g(x)=S(\mu(x,\cdot))$
as a distribution on $D$.
In that case, there is a closed rectangle $R\subset D$
such that
the average of $\nabla g$ over $R$ does not equal
$S(\mu(R,\cdot))$.
But if $\mathfrak G_n(\phi)$
is close to $g$,
then $\mathfrak L_n(\phi)(R,\cdot)$
must have its approximate slope close to the average of $\nabla g$ over $R$.
Note that we use the words \emph{approximate slope} here rather than the word \emph{slope},
because $\mathfrak L_n(\phi)(R,\cdot)$ is not shift-invariant,
but it is almost shift-invariant in the sense that
$\mathfrak L_n(\phi)(R,f-\theta f)$ goes to zero uniformly over $\phi$ as $n\to\infty$ for
$f$ a bounded continuous cylinder function and $\theta\in\Theta(\mathcal L)$;
see the proof
of Lemma~\ref{lemma_FB_non_shift-inv}.
In particular,
by including $R$ in $\Sigma$ in the previous discussion
and choosing $B^R$ sufficiently small, it can again be seen
that $A_n$ is empty for $n$ sufficiently large, which leads to the desired bound.
\end{proof}


\subsection{Exponential tightness}
\begin{proof}[Proof of exponential tightness]
  The proof is easy.
  Fix a positive constant $\varepsilon>0$,
  and let $K$ denote the smallest constant such that $Kd_1\geq q$.
  Define
  \[
    K^\infty_\varepsilon:=\{
      g\in \operatorname{Lip}(\bar D):\|g|_{\partial D}-b\|_\infty\leq \varepsilon
    \}
    ,
    \qquad
    K^\mathfrak L:=\{
      \mu\in\mathcal M^D:\text{$\mu(D,\cdot)$ is $K$-Lipschitz}
    \}.
  \]
  It is clear that $K^\infty_\varepsilon$ is compact in
  $(\operatorname{Lip}(\bar D),\mathcal X^\infty)$,
  and that $K^\mathfrak L$ is compact in $(\mathcal M^D,\mathcal X^\mathfrak L)$.
  This means that $K^\infty_\varepsilon\times K^\mathfrak L$
  is compact in $(X^\mathfrak P,\mathcal X^\mathfrak P)$.
  As in the proof of the upper bound of probabilities,
  we observe that $\tilde\gamma_n^*$ is supported on
  $K^\infty_\varepsilon\times K^\mathfrak L$
  for $n$ sufficiently large.
  This completes the proof;
  the compact set that we have found is independent of the choice of $\alpha$
  that appeared in the original formulation of exponential tightness.
\end{proof}

\section{Proof of strict convexity}
\label{sec_strict_convex}

\subsection{The product setting}

For the proof of strict convexity of $\sigma$,
it is useful to work in the product setting $\Omega\times\Omega$,
because one is then able to study the difference $\phi_1-\phi_2$
of a pair of height functions $(\phi_1,\phi_2)$
and apply the theory of moats from Section~\ref{section:moats_heart}.
Almost all constructions and results
in the previous sections generalize to
the product setting.
An alternative way of viewing
the product setting is by considering
a height function to take values in the two-dimensional
space $E^2$ rather than $E$.
This section gives an overview of the definitions
and results for the product setting as required for the proof of strict convexity
of $\sigma$.

Write $\mathcal P^2(X, \mathcal X)$
for the set of probability measures on
$(X,\mathcal X)^2$
whenever
 $(X,\mathcal X)$ is a measurable space.
 If $\mu\in\mathcal P^2(X, \mathcal X)$,
 then write $\mu_1$ and $\mu_2$
 for the marginals of $\mu$ on the first
 and second space respectively.

\begin{definition}
  The \emph{topology of weak local convergence}
  is the coarsest topology on $\mathcal P^2(\Omega,\mathcal F^\nabla)$ that makes
  the evaluation map $\mu\mapsto\mu(f)$
  continuous for any bounded continuous cylinder function $f$ on $\Omega^2$,
  that is, a bounded function $f:\Omega^2\to\mathbb R$
  which is $\mathcal F^\nabla_\Lambda\times\mathcal F^\nabla_\Lambda$-measurable
  for some $\Lambda\subset\subset\mathbb Z^d$,
  and continuous with respect to the topology of uniform convergence
  on $\Omega^2$---the set of functions from $\mathbb Z^d$
  to $E^2$.
\end{definition}


\begin{definition}
Write $\mathcal P_\mathcal L^2(\Omega, \mathcal F^\nabla)$
for the set of $\mathcal L$-invariant probability
measures in $\mathcal P^2(\Omega, \mathcal F^\nabla)$;
a measure $\mu\in\mathcal P^2(\Omega, \mathcal F^\nabla)$
is called \emph{$\mathcal L$-invariant}
if
$\mu(A\times B)=\mu(\theta A\times\theta B)$
for any $A,B\in\mathcal F^\nabla$ and $\theta\in\Theta$.
This is equivalent to asking
that $(\phi_1,\phi_2)$ and $(\theta\phi_1,\theta\phi_2)$
have the same distribution under $\mu$.
\end{definition}

\begin{definition}
By the \emph{slope} of $\mu\in\mathcal P^2_\mathcal L(\Omega, \mathcal F^\nabla)$
we simply mean the pair
of slopes of the two marginals of $\mu$;
$S^2(\mu):=(S(\mu_1),S(\mu_2))$.
The slope functional $S^2$ is clearly strongly
affine, as in the non-product setting.
\end{definition}

\begin{definition}
  For $\mu\in\mathcal P^2(\Omega, \mathcal F^\nabla)$ and $\Lambda\subset\subset\mathbb Z^d$,
  define the \emph{free energy} of $\mu$ in $\Lambda$
  by
  \[
    \mathcal H^2_\Lambda(\mu|\Phi):=
    \mathcal H_{\mathcal F_\Lambda^\nabla\times\mathcal F_\Lambda^\nabla}
    (\mu|\lambda^{\Lambda-1}\times\lambda^{\Lambda-1})
    +
    \mu(H_\Lambda^{0,\Phi}(\phi_1)+H_\Lambda^{0,\Phi}(\phi_2)).
  \]
  Note that we immediately have
  \[
  \numberthis
  \label{eq_ineq_marginals}
      \mathcal H^2_\Lambda(\mu|\Phi)
      \geq
        \mathcal H_\Lambda(\mu_1|\Phi)
        +
          \mathcal H_\Lambda(\mu_2|\Phi),
  \]
  with equality if and only
  if the restriction of $\mu$ to $\mathcal F_\Lambda^\nabla\times\mathcal F_\Lambda^\nabla$
  decomposes as the product of $\mu_1$ and $\mu_2$,
  or if either side equals $\infty$.
  If $\mu$ is $\mathcal L$-invariant, then
    define the \emph{specific free energy} of $\mu$ by
    \[
      \mathcal H^2(\mu|\Phi)=\lim_{n\to\infty}n^{-d}\mathcal H_{\Pi_n}^2(\mu|\Phi).
      \]
      It follows immediately from~\eqref{eq_ineq_marginals}
      that $\mathcal H^2(\mu|\Phi)\geq \mathcal H(\mu_1|\Phi)+
      \mathcal H(\mu_2|\Phi)$.
      In particular,
      this implies that $\mathcal H^2(\mu|\Phi)\geq
      \sigma(S(\mu_1))+\sigma(S(\mu_2))$.
      For convenience, we shall write
      $\sigma^2(u,v):=\sigma(u)+\sigma(v)$.
      Note that
      \[
      \sigma^2(u,v):=\inf_{\text{$\mu\in\mathcal P_\mathcal L^2(\Omega,\mathcal F^\nabla)$ with $S^2(\mu)=(u,v)$}}\mathcal H^2(\mu|\Phi).
      \]
\end{definition}

With these definitions, the following results
generalize naturally to the product setting:
\begin{enumerate}
  \item
  Theorem~\ref{thm_main_sfe}
  for existence of the specific free energy,
\item Theorem~\ref{thm_main_minimizers_finite_energy} for finite energy,
where the result applies if
\[
\mathcal H^2(\mu|\Phi)=
\sigma^2(S^2(\mu))<\infty,
\]
\item
  Theorem~\ref{theorem_ergodic_decomposition},
  Proposition~\ref{propo_S_strongly_affine}, and Theorem~\ref{thm_SFE_strongly_affine}
  for ergodic decompositions,
  \item Theorem~\ref{thm_limit_equalities_overview}
  for limit equalities and
  Theorem~\ref{p_LDP}
  for the large deviations principle.
\end{enumerate}
Rather than repeating each result here, we
state clearly the generalized result that is used whenever
referring to it.

\subsection{Moats in the empirical limit}
\label{subsection_attachment_applications}

In this section, we suppose that $\sigma$ is not strictly
convex, and construct the pathological measure which derives from this assumption.
Let $K$ denote the smallest real number such that $Kd_1\geq q$,
and write $\rho$ for the uniform probability measure
on the set
$E\cap [0,4K)$, with random variable $U$.
For fixed $(\phi_1,\phi_2,U)\in \Omega\times\Omega\times E$,
we shall write $\xi=\xi(\phi_1,\phi_2,U)$
for the function
\[
  \xi:=\left\lfloor{\frac{1}{4K}}(\phi_1-\phi_1(0)-\phi_2+\phi_2(0)-U)\right\rfloor:\mathbb Z^d\to\mathbb Z.
\]
We will refer to $\xi$ as the \textit{difference function} associated to the triplet  $(\phi_1,\phi_2,U)$.
Remark that the law of $\nabla\xi$ is $\mathcal L$-invariant in $\mu\times\rho$
for any
 $\mu\in\mathcal P_\mathcal L^2(\Omega,\mathcal F)$;
 the random variable $U$
 makes the rounding operation shift-invariant
 as in the proof of Lemma~\ref{lemma_ergodic_approx_aux_1}.


\begin{theorem}
  \label{thm_beyond_moats}
  Let $\Phi$ denote a potential which is monotone and
  in $\mathcal S_\mathcal L+\mathcal W_\mathcal L$.
  Assume that $\sigma$ is affine on the line segment $[u_1,u_2]$ connecting two distinct slopes $u_1,u_2\in U_\Phi$,
  and set $u=(u_1+u_2)/2$.
Select two vertices $x\in\mathcal L$ and $y\in\mathbb Z^d$ subject only
  to $(u_1-u_2)(x)\neq 0$.
  Then
  there exists a product measure $\mu\in\mathcal P_\mathcal L^2(\Omega,\mathcal F^\nabla)$
  such that $S^2(\mu)=(u,u)$
  and $\mathcal H^2(\mu|\Phi)=\sigma^2(u,u)= 2\sigma(u)$,
  and such that with positive $\mu\times\rho$-probability,
  the following two events occur simultaneously:
  \begin{enumerate}
    \item The function $\xi$ is not constant on the set $y+\mathbb Zx$,
    \item The set $\{\xi=0\}\subset\mathbb Z^d$ has at least
     three distinct infinite
    connected components.
  \end{enumerate}
\end{theorem}

In the next section, we discuss rigorously how to derive a
contradiction from this theorem (under the additional condition whenever $E=\mathbb Z$),
using Theorem~\ref{thm_main_minimizers_finite_energy}
and the argument for uniqueness of the infinite cluster  of Burton and Keane~\cite{BURTON}.
The purpose of the remainder of this section is to prove Theorem~\ref{thm_beyond_moats}.

Let us assume the setting of Theorem~\ref{thm_beyond_moats}:
$\Phi$ is a monotone potential in $\mathcal S_\mathcal L+\mathcal W_\mathcal L$,
$u_1$ and $u_2$ are distinct slopes in $U_\Phi$
such that $\sigma$ is affine on $[u_1,u_2]$,
and $u:=(u_1+u_2)/2$.
In the proof of the theorem, we shall suppose
that $y=0$, without loss of generality.
Fix $0<\varepsilon<K$ so small that $u_1,u_2\in U_{q_\varepsilon}$.
We shall use the large deviations principle
with the
good asymptotic profile $(D,b)$
where $D:=(0,1)^d\subset\mathbb R^d$ and
$b:=u|_{\partial D}$,
and with the good approximation
$(D_n,b_n)_{n\in\mathbb N}$
of $(D,b)$ defined by $D_n:=\Pi_n$
and $b_n:=\phi^u$
 for all $n\in\mathbb N$.
 As per usual, we write $\gamma_n:=\gamma_{D_n}(\cdot,b_n)$,
 and we shall also write $\gamma_n^2:=\gamma_n\times\gamma_n$.


\begin{figure}
	\centering
	\begin{subfigure}{0.5\textwidth}
		\centering
		\begin{tikzpicture}[x=5cm,y=5cm]
			\clip (-.1,-.3) rectangle (1.1,1.3);
			\fill[color=lightgray!50] (-0.1,-0.1)--(-0.1,-0.1)--(-0.1,-0.1)--(-0.1,-0.1)--cycle;
\fill[color=darkgray] (0.1,0.1)--(0.1,0.1)--(0.1,0.1)--(0.1,0.1)--cycle;
\fill[color=lightgray!50] (-0.07266666666666668,-0.1)--(-0.06066666666666667,-0.1)--(-0.1,0.018000000000000016)--(-0.1,-0.018000000000000016)--cycle;
\fill[color=gray] (0,0)--(0,0)--(0,0)--(0,0)--cycle;
\fill[color=lightgray!50] (-0.02266666666666667,-0.1)--(-0.010666666666666672,-0.1)--(-0.1,0.168)--(-0.1,0.132)--cycle;
\fill[color=darkgray] (0.1,0.1)--(0.1,0.1)--(0.1,0.1)--(0.1,0.1)--cycle;
\fill[color=lightgray!50] (0.027333333333333334,-0.1)--(0.03933333333333333,-0.1)--(-0.1,0.31800000000000006)--(-0.1,0.28200000000000003)--cycle;
\fill[color=gray] (0,0)--(0.006,0)--(0,0.018000000000000002)--(0,0)--cycle;
\fill[color=lightgray!50] (0.07733333333333334,-0.1)--(0.08933333333333333,-0.1)--(-0.1,0.46799999999999997)--(-0.1,0.43200000000000005)--cycle;
\fill[color=darkgray] (0.1,0.1)--(0.1,0.1)--(0.1,0.1)--(0.1,0.1)--cycle;
\fill[color=lightgray!50] (0.12733333333333333,-0.1)--(0.13933333333333334,-0.1)--(-0.1,0.6180000000000001)--(-0.1,0.5820000000000001)--cycle;
\fill[color=gray] (0.094,0)--(0.10600000000000001,0)--(0,0.31800000000000006)--(0,0.28200000000000003)--cycle;
\fill[color=lightgray!50] (0.17733333333333334,-0.1)--(0.18933333333333335,-0.1)--(-0.1,0.768)--(-0.1,0.7320000000000001)--cycle;
\fill[color=darkgray] (0.11066666666666669,0.1)--(0.1226666666666667,0.1)--(0.1,0.16800000000000007)--(0.1,0.13200000000000003)--cycle;
\fill[color=lightgray!50] (0.22733333333333333,-0.1)--(0.23933333333333334,-0.1)--(-0.1,0.9180000000000001)--(-0.1,0.8820000000000001)--cycle;
\fill[color=gray] (0.194,0)--(0.20600000000000002,0)--(0,0.6180000000000001)--(0,0.5820000000000001)--cycle;
\fill[color=lightgray!50] (0.2773333333333333,-0.1)--(0.28933333333333333,-0.1)--(-0.1,1.068)--(-0.1,1.032)--cycle;
\fill[color=darkgray] (0.21066666666666667,0.1)--(0.22266666666666668,0.1)--(0.1,0.46799999999999997)--(0.1,0.43199999999999994)--cycle;
\fill[color=lightgray!50] (0.32733333333333337,-0.1)--(0.3393333333333334,-0.1)--(-0.06066666666666665,1.1)--(-0.07266666666666666,1.1)--cycle;
\fill[color=gray] (0.29400000000000004,0)--(0.30600000000000005,0)--(0,0.9180000000000001)--(0,0.8820000000000001)--cycle;
\fill[color=lightgray!50] (0.37733333333333335,-0.1)--(0.38933333333333336,-0.1)--(-0.010666666666666658,1.1)--(-0.02266666666666667,1.1)--cycle;
\fill[color=darkgray] (0.3106666666666667,0.1)--(0.3226666666666667,0.1)--(0.1,0.768)--(0.1,0.7320000000000001)--cycle;
\fill[color=lightgray!50] (0.42733333333333334,-0.1)--(0.43933333333333335,-0.1)--(0.03933333333333333,1.1)--(0.02733333333333332,1.1)--cycle;
\fill[color=gray] (0.394,0)--(0.406,0)--(0.07266666666666671,1)--(0.0606666666666667,1)--cycle;
\fill[color=lightgray!50] (0.47733333333333333,-0.1)--(0.48933333333333334,-0.1)--(0.08933333333333332,1.1)--(0.07733333333333331,1.1)--cycle;
\fill[color=darkgray] (0.4106666666666667,0.1)--(0.4226666666666667,0.1)--(0.15600000000000003,0.9)--(0.14400000000000002,0.9)--cycle;
\fill[color=lightgray!50] (0.5273333333333333,-0.1)--(0.5393333333333333,-0.1)--(0.1393333333333333,1.1)--(0.1273333333333333,1.1)--cycle;
\fill[color=gray] (0.494,0)--(0.506,0)--(0.1726666666666667,1)--(0.16066666666666668,1)--cycle;
\fill[color=lightgray!50] (0.5773333333333334,-0.1)--(0.5893333333333334,-0.1)--(0.18933333333333335,1.1)--(0.17733333333333334,1.1)--cycle;
\fill[color=darkgray] (0.5106666666666667,0.1)--(0.5226666666666667,0.1)--(0.25600000000000006,0.9)--(0.24400000000000005,0.9)--cycle;
\fill[color=lightgray!50] (0.6273333333333334,-0.1)--(0.6393333333333334,-0.1)--(0.2393333333333334,1.1)--(0.2273333333333334,1.1)--cycle;
\fill[color=gray] (0.5940000000000001,0)--(0.6060000000000001,0)--(0.2726666666666668,1)--(0.26066666666666677,1)--cycle;
\fill[color=lightgray!50] (0.6773333333333333,-0.1)--(0.6893333333333334,-0.1)--(0.28933333333333333,1.1)--(0.2773333333333333,1.1)--cycle;
\fill[color=darkgray] (0.6106666666666667,0.1)--(0.6226666666666667,0.1)--(0.35600000000000004,0.9)--(0.34400000000000003,0.9)--cycle;
\fill[color=lightgray!50] (0.7273333333333334,-0.1)--(0.7393333333333334,-0.1)--(0.3393333333333334,1.1)--(0.32733333333333337,1.1)--cycle;
\fill[color=gray] (0.6940000000000001,0)--(0.7060000000000001,0)--(0.37266666666666676,1)--(0.36066666666666675,1)--cycle;
\fill[color=lightgray!50] (0.7773333333333333,-0.1)--(0.7893333333333333,-0.1)--(0.3893333333333333,1.1)--(0.3773333333333333,1.1)--cycle;
\fill[color=darkgray] (0.7106666666666667,0.1)--(0.7226666666666667,0.1)--(0.456,0.9)--(0.444,0.9)--cycle;
\fill[color=lightgray!50] (0.8273333333333334,-0.1)--(0.8393333333333334,-0.1)--(0.43933333333333335,1.1)--(0.42733333333333334,1.1)--cycle;
\fill[color=gray] (0.794,0)--(0.806,0)--(0.47266666666666673,1)--(0.4606666666666667,1)--cycle;
\fill[color=lightgray!50] (0.8773333333333334,-0.1)--(0.8893333333333334,-0.1)--(0.4893333333333334,1.1)--(0.4773333333333334,1.1)--cycle;
\fill[color=darkgray] (0.8106666666666668,0.1)--(0.8226666666666668,0.1)--(0.556,0.9)--(0.544,0.9)--cycle;
\fill[color=lightgray!50] (0.9273333333333333,-0.1)--(0.9393333333333334,-0.1)--(0.5393333333333333,1.1)--(0.5273333333333333,1.1)--cycle;
\fill[color=gray] (0.894,0)--(0.906,0)--(0.5726666666666667,1)--(0.5606666666666666,1)--cycle;
\fill[color=lightgray!50] (0.9773333333333334,-0.1)--(0.9893333333333334,-0.1)--(0.5893333333333334,1.1)--(0.5773333333333334,1.1)--cycle;
\fill[color=darkgray] (0.9,0.13200000000000012)--(0.9,0.16800000000000015)--(0.6560000000000001,0.9)--(0.6440000000000001,0.9)--cycle;
\fill[color=lightgray!50] (1.0273333333333334,-0.1)--(1.0393333333333334,-0.1)--(0.6393333333333333,1.1)--(0.6273333333333333,1.1)--cycle;
\fill[color=gray] (0.994,0)--(1,0)--(1,0.018000000000000016)--(0.6726666666666667,1)--(0.6606666666666667,1)--cycle;
\fill[color=lightgray!50] (1.0773333333333335,-0.1)--(1.0893333333333335,-0.1)--(0.6893333333333334,1.1)--(0.6773333333333333,1.1)--cycle;
\fill[color=darkgray] (0.9,0.43200000000000005)--(0.9,0.4680000000000001)--(0.756,0.9)--(0.744,0.9)--cycle;
\fill[color=lightgray!50] (1.1,-0.018000000000000016)--(1.1,0.018000000000000016)--(0.7393333333333334,1.1)--(0.7273333333333334,1.1)--cycle;
\fill[color=gray] (1,0.28200000000000025)--(1,0.3180000000000003)--(0.7726666666666668,1)--(0.7606666666666668,1)--cycle;
\fill[color=lightgray!50] (1.1,0.13200000000000012)--(1.1,0.16800000000000015)--(0.7893333333333334,1.1)--(0.7773333333333334,1.1)--cycle;
\fill[color=darkgray] (0.9,0.7320000000000003)--(0.9,0.7680000000000003)--(0.8560000000000001,0.9)--(0.8440000000000001,0.9)--cycle;
\fill[color=lightgray!50] (1.1,0.28200000000000025)--(1.1,0.3180000000000003)--(0.8393333333333335,1.1)--(0.8273333333333335,1.1)--cycle;
\fill[color=gray] (1,0.5820000000000005)--(1,0.6180000000000005)--(0.8726666666666669,1)--(0.8606666666666669,1)--cycle;
\fill[color=lightgray!50] (1.1,0.4319999999999997)--(1.1,0.46799999999999975)--(0.8893333333333333,1.1)--(0.8773333333333333,1.1)--cycle;
\fill[color=darkgray] (0.9,0.9)--(0.9,0.9)--(0.9,0.9)--(0.9,0.9)--cycle;
\fill[color=lightgray!50] (1.1,0.5819999999999999)--(1.1,0.6179999999999999)--(0.9393333333333334,1.1)--(0.9273333333333333,1.1)--cycle;
\fill[color=gray] (1,0.8820000000000001)--(1,0.9180000000000001)--(0.9726666666666668,1)--(0.9606666666666668,1)--cycle;
\fill[color=lightgray!50] (1.1,0.732)--(1.1,0.768)--(0.9893333333333334,1.1)--(0.9773333333333334,1.1)--cycle;
\fill[color=darkgray] (0.9,0.9)--(0.9,0.9)--(0.9,0.9)--(0.9,0.9)--cycle;
\fill[color=lightgray!50] (1.1,0.8820000000000001)--(1.1,0.9180000000000001)--(1.0393333333333334,1.1)--(1.0273333333333334,1.1)--cycle;
\fill[color=gray] (1,1)--(1,1)--(1,1)--(1,1)--cycle;
\fill[color=lightgray!50] (1.1,1.0320000000000003)--(1.1,1.0680000000000003)--(1.0893333333333335,1.1)--(1.0773333333333335,1.1)--cycle;
\fill[color=darkgray] (0.9,0.9)--(0.9,0.9)--(0.9,0.9)--(0.9,0.9)--cycle;
\fill[color=lightgray!50] (1.1,1.1)--(1.1,1.1)--(1.1,1.1)--(1.1,1.1)--cycle;
\fill[color=gray] (1,1)--(1,1)--(1,1)--(1,1)--cycle;
			\draw [very thick] (0,0) rectangle (1,1);
			\draw  (0.1,0.1) rectangle (.9,.9);

			\draw[very thin] (0,0) -- (0,-.06);
			\draw[very thin] (.1,0) -- (.1,-.06);
			\draw[very thin] (0,-0.03) -- (.1,-.03);
			\node[anchor=north] at (.05,-.06) {$n\varepsilon_1$};

			\draw[very thin] (.2+.1/3,-.1) -- (.2+.1/3,-.16);
			\draw[very thin] (.2+.1/3,-.13) -- (.25+.1/3,-.13);
			\draw[very thin] (.25+.1/3,-.1) -- (.25+.1/3,-.16);
			\node[anchor=north] at (.225+.1/3,-.16) {$\propto n\varepsilon_2$};

			\draw[very thin] (.55-.006+.1/3,-.1) -- (.55-.006+.1/3,-.16);
			\draw[very thin] (.55+.006+.1/3,-.1) -- (.55+.006+.1/3,-.16);
			\draw[very thin] (.55-.006+.1/3,-.13) -- (.55+.006+.1/3,-.13);
			\node[anchor=north] at (.55+.1/3,-.16) {$\propto n\varepsilon_3$};

			\draw[->] (.04,1.11) node[above] {$D_n$} -- (.04,1.01) ;

			\draw[->] (.24,1.11) node[above] {$D_n'$} -- (.24,.91) ;

			\draw (.44,1.11) node[above] {$D_n^0$} -- (.38333,.95) ;
			\draw (.44,1.11) -- (.48333,.95) ;
			\fill (.38333,.95) circle(1pt);
			\fill (.48333,.95) circle(1pt);

			\draw (.64,1.11) node[above] {$D_n^+$} -- (.56666,.85) ;
			\draw (.64,1.11)  -- (.66666,.85) ;
			\draw (.64,1.11)  -- (.76666,.85) ;
			\fill (.56666,.85) circle(1pt);
			\fill (.66666,.85) circle(1pt);
			\fill (.76666,.85) circle(1pt);

			\draw (.9,-.14) node[below] {$H_{n,k}$} -- (.81666,-.05) ;
			\draw (.9,-.14) -- (.96666,-.05) ;
			\draw (.9,-.14) -- (.86666,-.05) ;
			\draw (.9,-.14) -- (.91666,-.05) ;
			\fill (.81666,-.05) circle(1pt);
			\fill (.86666,-.05) circle(1pt);
			\fill (.96666,-.05) circle(1pt);
			\fill (.91666,-.05) circle(1pt);

			%

		\end{tikzpicture}
	\end{subfigure}\begin{subfigure}{0.5\textwidth}
		\centering
		\begin{tikzpicture}[x=5cm,y=5cm]
			\clip (-.1,-.3) rectangle (1.1,1.3);
			\fill[color=lightgray] (0.10600000000000001,0)--(0.194,0)--(0,0.5820000000000001)--(0,0.31800000000000006)--cycle;
\fill[color=darkgray] (0.11066666666666669,0.1)--(0.1226666666666667,0.1)--(0.1,0.16800000000000007)--(0.1,0.13200000000000003)--cycle;
\fill[color=lightgray] (0.20600000000000002,0)--(0.29400000000000004,0)--(0,0.8820000000000001)--(0,0.6180000000000001)--cycle;
\fill[color=darkgray] (0.21066666666666667,0.1)--(0.22266666666666668,0.1)--(0.1,0.46799999999999997)--(0.1,0.43199999999999994)--cycle;
\fill[color=lightgray] (0.30600000000000005,0)--(0.394,0)--(0.0606666666666667,1)--(0,1)--(0,0.9180000000000001)--cycle;
\fill[color=darkgray] (0.3106666666666667,0.1)--(0.3226666666666667,0.1)--(0.1,0.768)--(0.1,0.7320000000000001)--cycle;
\fill[color=lightgray] (0.406,0)--(0.494,0)--(0.16066666666666668,1)--(0.07266666666666671,1)--cycle;
\fill[color=darkgray] (0.4106666666666667,0.1)--(0.4226666666666667,0.1)--(0.15600000000000003,0.9)--(0.14400000000000002,0.9)--cycle;
\fill[color=lightgray] (0.506,0)--(0.5940000000000001,0)--(0.26066666666666677,1)--(0.1726666666666667,1)--cycle;
\fill[color=darkgray] (0.5106666666666667,0.1)--(0.5226666666666667,0.1)--(0.25600000000000006,0.9)--(0.24400000000000005,0.9)--cycle;
\fill[color=lightgray] (0.6060000000000001,0)--(0.6940000000000001,0)--(0.36066666666666675,1)--(0.2726666666666668,1)--cycle;
\fill[color=darkgray] (0.6106666666666667,0.1)--(0.6226666666666667,0.1)--(0.35600000000000004,0.9)--(0.34400000000000003,0.9)--cycle;
\fill[color=lightgray] (0.7060000000000001,0)--(0.794,0)--(0.4606666666666667,1)--(0.37266666666666676,1)--cycle;
\fill[color=darkgray] (0.7106666666666667,0.1)--(0.7226666666666667,0.1)--(0.456,0.9)--(0.444,0.9)--cycle;
\fill[color=lightgray] (0.806,0)--(0.894,0)--(0.5606666666666666,1)--(0.47266666666666673,1)--cycle;
\fill[color=darkgray] (0.8106666666666668,0.1)--(0.8226666666666668,0.1)--(0.556,0.9)--(0.544,0.9)--cycle;
\fill[color=lightgray] (0.906,0)--(0.994,0)--(0.6606666666666667,1)--(0.5726666666666667,1)--cycle;
\fill[color=darkgray] (0.9,0.13200000000000012)--(0.9,0.16800000000000015)--(0.6560000000000001,0.9)--(0.6440000000000001,0.9)--cycle;
\fill[color=lightgray] (1,0.018000000000000016)--(1,0.28200000000000025)--(0.7606666666666668,1)--(0.6726666666666667,1)--cycle;
\fill[color=darkgray] (0.9,0.43200000000000005)--(0.9,0.4680000000000001)--(0.756,0.9)--(0.744,0.9)--cycle;
\fill[color=lightgray] (1,0.3180000000000003)--(1,0.5820000000000005)--(0.8606666666666669,1)--(0.7726666666666668,1)--cycle;
\fill[color=darkgray] (0.9,0.7320000000000003)--(0.9,0.7680000000000003)--(0.8560000000000001,0.9)--(0.8440000000000001,0.9)--cycle;

			\draw [very thick] (0,0) rectangle (1,1);

			\draw (.64,1.11) node[above] {$\Lambda_{n,k}$} -- (.56666,.85) ;
			\draw (.64,1.11)  -- (.66666,.85) ;
			\draw (.64,1.11)  -- (.76666,.85) ;
			\fill (.56666,.85) circle(1pt);
			\fill (.66666,.85) circle(1pt);
			\fill (.76666,.85) circle(1pt);

			\draw (.36,-.11) node[below] {$\Delta_{n,k}$} -- (.23333,.05) ;
			\draw (.36,-.11) -- (.33333,.05) ;
			\draw (.36,-.11) -- (.43333,.05) ;
			\fill (.23333,.05) circle(1pt);
			\fill (.33333,.05) circle(1pt);
			\fill (.43333,.05) circle(1pt);

			\draw (-.2, 0)-- (1.2, .7);
			\fill (.5,.35) circle(1pt);
			\draw[->] (.62,-.13) node[below] {$z$} -- (0.5025,.34) ;
			\draw[->] (.81625,-.11) node[below] {$L^z$} -- (.68,.435);


		\end{tikzpicture}
	\end{subfigure}
\caption{Several constructions in Subsection~\ref{subsection_attachment_applications}}
\label{fig:washboard_empirical}
\end{figure}
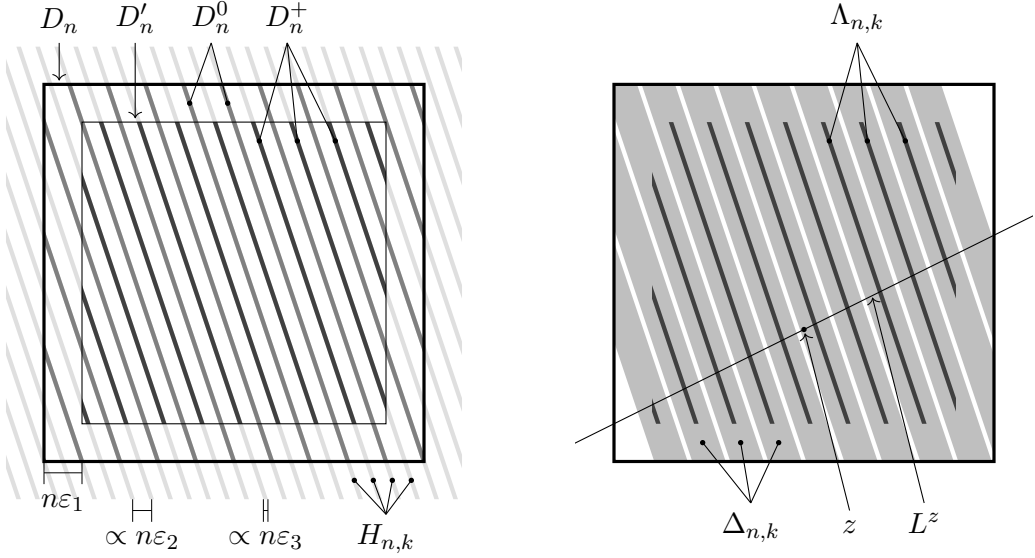

Set $t=1/2$,
and recall the definitions of $v$,
$p$, and $p_\alpha$ from the
proof of Lemma~\ref{lemma_PB_star_convex}
(Page~\pageref{washboarddefstart}).
Fix $\varepsilon_1$, $\varepsilon_2$, and $\varepsilon_3$
strictly positive
and consider $n\in\mathbb N$.
When taking limits
we shall take first $n\to\infty$,
then $\varepsilon_3\to 0$,
then $\varepsilon_2\to 0$,
and finally $\varepsilon_1\to 0$; it is again convenient to work
on different scales.
Define $D':=(\varepsilon_1,1-\varepsilon_1)^d\subset D$
and $D'_n:=nD'\cap\mathbb Z^d$.
Write $H_k$
for the affine hyperplane
$\{2v=k\varepsilon_2\}\subset\mathbb R^d$.
Note that the sets $(H_k)_{k\in\mathbb Z}$
correspond to the hyperplanes where the
gradient of $p_{\varepsilon_2}$ changes.
For $k$ even, $p_{\varepsilon_2}$
equals $u$ on $H_k$.
For $k$ odd, $p_{\varepsilon_2}$ equals
$u+\varepsilon_2/4$ on $H_k$.
Finally, write
\begin{align*}
  & H_{n,k}:=\{x\in\mathbb Z^d:d_2(x,nH_k)\leq n\varepsilon_3\},\\
  &
   D_n^0:=(\cup_{k\in 2\mathbb Z} H_{n,k})\cap
  D_n,\\&
   D_n^+:=
  (\cup_{k\in 2\mathbb Z+1} H_{n,k})\cap
  D_n'.
\end{align*}
See Figure~\ref{fig:washboard_empirical}
for an overview of this construction.

\begin{proposition}
  \label{propo_generic_nonsym}
  Assume the setting of Theorem~\ref{thm_beyond_moats}.
If $E=\mathbb Z$,
then there is a $\delta>0$ such that
\[
  n^{-d}
  \log
  \gamma_n(\text{$\phi_{ D_n^0}=\phi^u_{ D_n^0}$
  and $\phi_{ D_n^+}\geq \phi^u_{ D_n^+}+ n\delta \varepsilon_2$})
  =o(1)
\]
in the limit of $n$, $\varepsilon_3$, $\varepsilon_2$,
and $\varepsilon_1$.
If $E=\mathbb R$ and $\varepsilon>0$, then there is a $\delta>0$ such that
\[
  n^{-d}
  \log
  \gamma_n(\text{$|\phi_{ D_n^0}-\phi^u_{ D_n^0}|\leq \varepsilon$ and $\phi_{ D_n^+}\geq \phi^u_{ D_n^+}+n\delta \varepsilon_2$})
  =o(1)
\]
in the limit of $n$, $\varepsilon_3$, $\varepsilon_2$,
and $\varepsilon_1$.
\end{proposition}

\begin{proof}
  In fact, we shall demonstrate that any $\delta<1/4$ works.
  Write $f$ for the smallest $\|\cdot\|_{q_\varepsilon}$-Lipschitz function which satisfies $f\geq u$
  and which equals $p_{\varepsilon_2}$
  on $D'$.
  This function is well-defined
  and equals $u$ on $\mathbb R^d\smallsetminus D$
  for $\varepsilon_2$ sufficiently small (depending only on $\varepsilon_1$).

  The pressure $P_\Phi(D,b)$
   is equal to $\sigma(u)$
  because $\sigma$ is convex,
  $b=u|_{\partial D}$, and
   $\operatorname{Vol}(D)=1$.
   Moreover, $\int_D\sigma(\nabla f(x))dx$
   tends to $\sigma(u)$ in the limit of $\varepsilon_1$
   and $\varepsilon_2$,
   because $\sigma$ is affine on the line segment
   connecting $u_1$ and $u_2$,
   and because the gradient of $f$ equals $u_1$
   on roughly half of $D$,
   and $u_2$ on roughly the other half of $D$
   with respect to Lebesgue measure.
   Note that $\sigma(\nabla f)$ is bounded uniformly as $f$ is $\|\cdot\|_{q_\varepsilon}$-Lipschitz.
   This means that
   for any $\varepsilon'>0$,
   which is allowed to depend arbitrarily on $\varepsilon_1$ and $\varepsilon_2$,
   we have
   \[
      n^{-d}\log\gamma_n(\mathfrak G_n^{-1}( B_{\varepsilon'}^\infty(f)))=o(1)
   \]
   in the limit of $n$, $\varepsilon_2$,
   and $\varepsilon_1$.

   Note that for
   $\varepsilon_3$
   and $\varepsilon'$ sufficiently small
   depending on $\varepsilon_1$ and $\varepsilon_2$,
   all height functions
   $\phi\in \mathfrak G_n^{-1}(B_{\varepsilon'}^\infty(f))$
   satisfy $\phi_{ D_n^+}\geq \phi_{ D_n^+}^u+n\delta\varepsilon_2$
   (by virtue of the choice of $f$).
   Moreover, $\phi_{ D_n^0}$
   and $\phi^u_{ D_n^0}$
   must be close for such $\phi$.
   By repeating arguments of the proof of the lower bound
   on probabilities in the large deviations principle,
   it is straightforward to see that conditioning
   further on the exact values of $\phi_{ D_n^0}$ (up to $\varepsilon$ in the continuous case)
   does not decrease the value of the limit of the normalized probabilities.
   In particular, this implies the proposition.
\end{proof}

By interchanging the role of $u_1$ and $u_2$,
one obtains the same result as in Proposition~\ref{propo_generic_nonsym},
now with the inequality sign $\geq $
replaced by $\leq$,
and with $n\delta\varepsilon_2$
replaced by $-n\delta\varepsilon_2$.
By appealing to both the original proposition
and the version with replacements,
one deduces immediately the following proposition.

\begin{proposition}
  Assume the setting of Theorem~\ref{thm_beyond_moats}.
If $E=\mathbb Z$,
then there exists a $\delta>0$ such that
\[
  n^{-d}
  \log
  \gamma_n^2(\text{$(\phi_1-\phi_2)_{ D_n^0}=0$
  and
  $(\phi_1-\phi_2)_{ D_n^+}\geq n\delta \varepsilon_2$})
  =o(1)
\]
in the limit of $n$, $\varepsilon_3$, $\varepsilon_2$,
and $\varepsilon_1$.
If $E=\mathbb R$ and $\varepsilon>0$, then there is a $\delta>0$ such that
\[
  n^{-d}
  \log
  \gamma_n^2(\text{$|(\phi_1-\phi_2)_{ D_n^0}|\leq 2\varepsilon$
  and
  $(\phi_1-\phi_2)_{ D_n^+}\geq n\delta \varepsilon_2$})
  =o(1)
\]
in the limit of $n$, $\varepsilon_3$, $\varepsilon_2$,
and $\varepsilon_1$.
\end{proposition}

Recall Section~\ref{section:moats_heart} on moats;
we are now ready to apply the theory developed there.
If $k$ is odd with  $ H_{n,k}\cap D_n'$ nonempty,
then write $\Lambda_{n,k}:=H_{n,k}\cap D_n'$.
Note that $D_n^+=\cup_k\Lambda_{n,k}$.
Write also $\Delta_{n,k}$
for the connected component of $D_n\smallsetminus D_n^0$
containing $\Lambda_{n,k}$;
see Figure~\ref{fig:washboard_empirical}
for an example of the sets $\Lambda_{n,k}$
and $\Delta_{n,k}$.
Write $E_n^a(m)$
for the event that
each connected component $\Delta_{n,k}$
contains a sequence of $\lceil m\rceil$ nested $4K,4K+a$-moats
of $(\phi_1-\phi_2, \Lambda_{n,k})$.

\begin{lemma}
  \label{lemma_moats_application_in_product_setting}
  Assume the setting of Theorem~\ref{thm_beyond_moats}.
  For any $a\geq 4K$,
  there is a $\delta>0$ such that
  \[
    n^{-d}\log\gamma_n^2(E_n^a(n\delta\varepsilon_2))=o(1)
  \]
  in the limit of $n$, $\varepsilon_3$, $\varepsilon_2$,
  and $\varepsilon_1$.
\end{lemma}

\begin{proof}
  This follows immediately from the previous proposition
  and from Proposition~\ref{proposition_moats_effective_application}.
  Note that the prefactor which appears on the left
  in~\eqref{eq:thmgoal} is of order
  $
    n^{O(1/\varepsilon_2)\cdot O(n\delta\varepsilon_2)}$,
    because distances are bounded by $n$,
    there are at most $O(1/\varepsilon_2)$
    sets $\Lambda_{n,k}$,
    and because we enforce $n\delta\varepsilon_2$
    moats around each set $\Lambda_{n,k}$.
  In particular, keeping all constants
  other than $n$ fixed, the logarithm of this
  term is of order $O(n\log n)$,
  which disappears in the normalization
  because we normalize by $n^{-d}$ with $d\geq 2$.
\end{proof}

\begin{proof}[Proof of Theorem~\ref{thm_beyond_moats}]
Let us consider a configuration $(\phi_1,\phi_2)\in E_n^a(n\delta\varepsilon_2)$,
and focus on the collection of moats
of $f:=\phi_1-\phi_2$.
Fix $x\in\mathcal L$ with $u_1(x)-u_2(x)\neq 0$,
and define $L:=\mathbb Zx$ and $L_N:=\{-N,\dots,N\}x$.
Write $\bar L_N$ for a path through the square lattice of minimal length traversing all the
vertices in $L_N$.
Draw some vertex $z$ from $\mathcal L\cap D_n$ uniformly at random,
and write $L^z:=L+z$, $L_N^z:=L_N+z$, and $\bar L_N^z:=\bar L_N+z$.
We are interested in the line $L^z$,
and the way this line intersects the moats of $f$.
We make a series of important geometrical observations.
By saying that a quantity is \emph{uniformly positive}, we mean
that it has a strictly positive lower bound which is independent
of the four parameters, for $n$ sufficiently large
and for $\varepsilon_3$, $\varepsilon_2$, and $\varepsilon_1$
sufficiently small.

\begin{enumerate}
  \item If $a$ is at least $(4\vee 2m)K$,
  then the $d_1$-distance from the inside to the outside of a fixed
  climbing or descending $4K,4K+a$-moat
  is at least $\lfloor m\rfloor+1$,
  as $f$ is $2K$-Lipschitz
  (See Proposition~\ref{propo:collection},
  Statement~\ref{propo:collection:in_out_dist}).
  If
   $m\geq d_1(0,x)$
   and if $L^z_N$ intersects both the inside and outside of some moat,
   then $L^z_N$ must also intersect that moat.
  In particular, if $L^z$ intersects $\Lambda_{n,k}$,
  then $L^z$ must necessarily also intersect all moats surrounding $\Lambda_{n,k}$.
  In the sequel, we choose $a':=(4\vee 2 d_1(0,x))K$ and $a=3a'$.

  \item With uniformly positive probability,
  $z$ lies in $\Delta_{n,k}$
  with $L^z$ intersecting $\Lambda_{n,k}$,
  for some odd integer $k$.
  This is illustrated by Figure~\ref{fig:washboard_empirical};
  it is important here that $(u_1-u_2)(x)\neq 0$
  so that $x$ does not lie in the hyperplane
  $\{u_1-u_2=0\}$.
  Let us suppose that such an odd integer $k$ indeed exists.
  Write $m^\pm$ for the smallest and largest integer
  respectively such that
  $z+m^\pm x \in \Delta_{n,k}$.
  Then $m^+-m^-\leq O(n\varepsilon_2)$, where the constant is independent of all four parameters.
  But $\Delta_{n,k}$ contains a sequence of $\lceil n\delta\varepsilon_2\rceil$ nested
  $4K,4K+a$-moats
  of $\Lambda_{n,k}$; $L^z$ intersects each one of them.
  These moats thus have a uniformly positive density
  in the set $z+\{m^-,\dots,m^+\}x$.
  But $z$ was chosen uniformly random from $\mathcal L\cap D_n$
  and therefore we may rerandomize its position within $z+\{m^-,\dots,m^+\}x$.
  Since the moats are disjoint from one another and have a positive density within this set,
  we observe there exists a fixed constant $N\in\mathbb N$
  such that $L_N^z$ intersects at least five distinct nested moats
  with uniformly positive probability.
  In fact, each $4K,4K+a$-moat contains a
  $4K,4K+a'$-moat (Proposition~\ref{propo:collection}, Statement~\ref{moatcontainsmoat}), and $z$ is contained in such a moat with uniformly positive probability.
  Therefore,
  the event that $L_N^z$ intersects at least five distinct nested $4K,4K+a$-moats,
  and simultaneously $f(z)\in [4K,4K+a')$,
  has uniformly positive probability.

  \item
  Let us mention a first consequence of the event described above.
  Since $L_N^z$ intersects more than three distinct $4K,4K+a$-moats,
  it must intersect both the inside and outside of the middle moat.
  This moat contains both a $4K,4K+a'$-moat, as well as a
  $4K+2a',4K+3a'$-moat,
  which $L_N^z$ must both intersect.
  The value of $f$ differs by at least $a'\geq 4K$
  on these two moats. In particular, $\xi=\xi(\phi_1,\phi_2,U)$ cannot be constant on $L_N^z$,
  regardless of the value of $U$.
  Similarly, $\xi(\theta_z\phi_1,\theta_z\phi_2,U)$
  cannot be constant on $L_N$.

  \item
  Let us mention a second consequence.
  Since the set $L_N^z$ intersects
  five distinct nested $4K,4K+a$-moats,
  it must intersect both the inside and the outside  of the three middle moats.
  Fix $U\in [0,4K)$,
  and write $a'':=f(z)+U\in [4K,4K+2a')$.
  The set $\bar L_N^z$ must intersect three
  $a'',a''+4K$-moats: each of the three middle $4K,4K+a$-moats
  contains a $a'',a''+4K$-moats which $\bar L_N^z$ must also intersect.
  But these three moats correspond exactly
  to connected components
  of $\{\xi=0\}$
  for $\xi:=\xi(\theta_z\phi_1,\theta_z\phi_2,U)$,
  which are intersected by $\bar L_N$.
  We must however limit ourselves to local observations,
  as we always work in the topology of (weak) local convergence.
  Write therefore $\Sigma_m:=\{-m,\dots,m\}^d\subset\subset\mathbb Z^d$;
  we only consider $m$ so large that $\bar L_N\subset\Sigma_m$.
  The previous observation means that for any $m\in\mathbb N$,
  $\{\xi=0\}\cap \Sigma_m$ has three
  connected components which intersect both $\bar L_N$
  and $\partial^1\Sigma_m$,
  at least if $n$ is sufficiently large---this is because each moat must surround some set
  $\Lambda_{n,k}$, which grows large whenever $n$ is large.
\end{enumerate}

Let us summarize what we have done so far.
We proved that there exist constants $N\in\mathbb N$
and $\delta'>0$ with the following properties.
Choose $(\phi_1,\phi_2)\in E_n^a(n\delta\varepsilon_2)$,
and choose $z\in\mathcal L\cap D_n$ uniformly at random.
Then for fixed $m\in\mathbb N$,
the probability that for any $U\in [0,4K)$,
\begin{enumerate}
  \item $\xi:=\xi(\theta_z\phi_1,\theta_z\phi_2,U)$ is not constant on $L_N$,
  \item $\{\xi=0\}\cap \Sigma_m$ has three connected component which
  intersect both $\bar L_N$ and $\partial^1\Sigma_m$,
\end{enumerate}
is at least $\delta'$,
for $n$ sufficiently large depending on $m$,
and for $\varepsilon_3$, $\varepsilon_2$, and $\varepsilon_1$ small.

In the final part of the proof, we use this intermediate result,
as well as the large deviations principle
and compactness of the lower level sets $M_C$ of the
specific free energy, to construct the desired measure
for Theorem~\ref{thm_beyond_moats}.

Let us first consider the case $E=\mathbb Z$.
Consider $m\in\mathbb N$ so large that $\bar L_N\subset \Sigma_m$,
and write
$A_m\in\mathcal F^\nabla_{\Sigma_m}\times\mathcal F^\nabla_{\Sigma_m}$
for the event that
for any $U\in [0,4K)$,
the function
$\xi:=\xi(\phi_1,\phi_2,U)$ is not constant on $L_N$,
and that $\{\xi=0\}\cap\Sigma_m$
has three connected components intersecting both $\partial^1\Sigma_m$
and $\bar L_N$.
Write $B_m$ for the set of measures $\mu\in\mathcal P^2(\Omega,\mathcal F^\nabla)$
such that $\mu(A_m)>\delta'/2$.
Note that $B_m$ is in the basis for the tolopogy of weak local convergence
on the space of product measures $\mathcal P^2(\Omega,\mathcal F^\nabla)$.
Recall the definition of
$    \mathfrak L_n(\phi)
$
in Subsection~\ref{formal_LDP_topo},
and define, for the product setting,
\[
    \mathfrak L_n^2(\phi_1,\phi_2):=
    \int _D \delta_{(x,\theta_{[nx]_\mathcal L}\phi_1,\theta_{[nx]_\mathcal L}\phi_2)}dx\in\mathcal M_2^D,
\]
where by $\mathcal M_2^D$ we mean the set of measures
in $\mathcal M(D\times\Omega\times\Omega,\mathcal D\times\mathcal F^\nabla\times\mathcal F^\nabla)$
for which the first marginal
equals the Lebesgue measure on $D$.
By Lemma~\ref{lemma_moats_application_in_product_setting} and the intermediate result,
we know that
\[
  n^{-d}\log\gamma_n^2(\mathfrak L_n^2(D,\cdot)\in B_m)=o(1)
\]
as $n\to\infty$.
It therefore follows from the large deviations principle that
$\bar B_m$ contains a shift-invariant measure $\mu_m\in\mathcal P_\mathcal L^2(\Omega,\mathcal F^\nabla)$
with $S^2(\mu_m)=(u,u)$
and $\mathcal H^2(\mu_m|\Phi)\leq 2\sigma(u)$.
In particular, this means that $\mu_m(A_m)\geq \delta'/2$,
and in fact $\mu_m(A_{m'})\geq \delta'/2$ for all $m'\leq m$
because $A_m\subset A_{m'}$ for $m'\leq m$.
By compactness of the lower level sets of the specific free energy,
the sequence $(\mu_m)_{m\in\mathbb N}$
has a subsequential limit $\mu\in \mathcal P_\mathcal L^2(\Omega,\mathcal F^\nabla)$
in the topology of local convergence
which satisfies $S^2(\mu)=(u,u)$
and $\mathcal H^2(\mu|\Phi)\leq2\sigma(u)$.
In particular, $\mu(A_m)\geq \delta'/2$ for all $m$,
which means that $\mu$ satisfies all the requirements of Theorem~\ref{thm_beyond_moats};
the intersection $\cap_m A_m$
of the decreasing sequence $(A_m)_{m\in\mathbb N}$
is precisely the event that $\{\xi=0\}$
has three infinite level sets which intersect $\bar L_N$, regardless of the value of $U$.

In the case that $E=\mathbb R$, there is a slight complication.
If $E=\mathbb R$, then the indicator $1_{A_m}$ is not continuous with respect
to the topology of uniform convergence on $\Omega^2$,
and therefore the sets $B_m$ as defined above are not
in the basis of the topology of weak local convergence.
Introduce therefore the sequence of functions $(f_{m,k})_{k\in\mathbb N}$
where each function
$f_{m,k}:\Omega^2\to[0,1]$ is defined by
$f_{m,k}(\phi_1,\phi_2):=0\vee (1-kd_\infty(A_m,(\phi_1,\phi_2)))$;
here $d_\infty$ denotes the metric corresponding
to the norm $\|\cdot\|_\infty$ on $\Omega^2$.
Write $B_{m,k}$ for the set of product measures $\mu$
such that $\mu(f_{m,k})>\delta'/2$.
Then $B_{m,k}$ is in the basis of the topology of weak local convergence,
and we have
\[
  n^{-d}\log\gamma_n^2(\mathfrak L_n^2(D,\cdot)\in B_{m,k})=o(1)
\]
as $n\to\infty$.
Therefore $\bar B_{m,k}$ contains a measure $\mu_{m,k}$
with $S^2(\mu_{m,k})=(u,u)$
and $\mathcal H^2(\mu_{m,k}|\Phi)\leq 2\sigma(u)$.
Moreover, the sequence of measures $(\mu_{m,k})_{k\in\mathbb N}$
must
have a subsequential limit $\mu_{m}$ in the topology
of local convergence,
and this limit must satisfy $\mathcal H^2(\mu_{m}|\Phi)\leq 2\sigma(u)$,
 $S^2(\mu_{m})=(u,u)$,
 and $\mu_m(f_{m,k})\geq \delta'/2$
 for all $k$. The dominated convergence theorem says that
 $\mu_m(\bar A_m)=\mu_m(1_{\bar A_m})=\mu_m(\lim_k f_{m,k})\geq \delta'/2$.
 But $\mu_m(\partial A_m)=0$,
 since $\mu_m$ has finite specific free energy and is therefore locally
 absolutely continuous with respect to the Lebesgue measure.
 In particular, $\mu_m(A_m)\geq \delta'/2$.
 Take now a subsequential limit of the sequence $(\mu_m)_{m\in\mathbb N}$
 for the desired measure.
 For this last step, it is important that the topology of local convergence
 and the topology of weak local convergence coincide on the lower level sets of the specific free energy.
\end{proof}


\subsection{Application of the argument of Burton and Keane}

In this subsection we prove Theorem~\ref{thm_main_main},
which is equivalent to the conjunction of Theorem~\ref{theorem_strictly_convex_continuous}
and Theorem~\ref{theorem_strictly_convex_discrete}.
Recall the definition of $\rho$
and $\xi$ in the previous subsection.

\begin{lemma}
  \label{lemma_application_of_BK}
  Let $\Phi$ denote any potential in $\mathcal S_\mathcal L+\mathcal W_\mathcal L$,
  and consider a measure $\mu\in\mathcal P_\mathcal L^2(\Omega,\mathcal F^\nabla)$.
   Then one of the following properties must fail:
  \begin{enumerate}
    \item $\mu$ is ergodic and at least one of $S(\mu_1)$ and $S(\mu_2)$ lies in $U_\Phi$,
    \item $\mu$ is a minimizer in the sense that $\mathcal H^2(\mu|\Phi)=\sigma^2(S^2(\mu))<\infty$,
    \item With positive $\mu\times\rho$-probability,
    $\{\xi=0\}$ has at least three infinite components.
  \end{enumerate}
\end{lemma}

The proof uses a construction from the part in~\cite{lammers2019gen} on strict convexity.

\begin{proof}[Proof of Lemma~\ref{lemma_application_of_BK}]
  For a fixed configuration $(\phi_1,\phi_2,U)$,
  a \emph{trifurcation box}
  is a finite set $\Lambda\subset\subset\mathbb Z^d$
  such that for some $a\in\mathbb Z$,
  the set
  $\{\xi=a\}\smallsetminus\Lambda$
  has three infinite connected components, which are contained in a single connected component
  of $\{\xi=a\}$.
  If $\mu$ is shift-invariant then almost surely
  $\mu\times\rho$ has no trifurcation boxes,
  due to the argument of Burton and Keane~\cite{BURTON}.
  Note that it is important for this statement that the gradient of $\xi$
  is shift-invariant in $\mu\times\rho$.
  To arrive at the desired contradiction,
  we aim to prove that trifurcation boxes occur with positive probability
  for the measure $\mu$ described in the statement of the lemma.

  Write $\Omega_q^2$ for the set of pairs of $q$-Lipschitz
  height functions.
  The natural adaptation of Theorem~\ref{thm_main_minimizers_finite_energy}
  to the product setting
  asserts that
  \[
  \numberthis
  \label{eq:auiwdonduawbndaw}
    1_{\Omega_q^2}(\lambda^\Lambda\times\lambda^\Lambda\times \mu\pi_{\mathbb Z^d\smallsetminus \Lambda})\times\rho
    \ll
    \mu\times\rho
  \]
  for any $\Lambda\subset\subset\mathbb Z^d$,
  where by $\mu\pi_{\mathbb Z^d\smallsetminus \Lambda}$
  we mean the product measure $\mu$
  restricted to the vertices in the complement of $\Lambda$, as in the non-product setting.
  Therefore it suffices to demonstrate that
  trifurcation boxes occur with positive
  measure in the measure on the left in the display, for some $\Lambda\subset\subset\mathbb Z^d$.

  Suppose, without loss of generality,
  that $S(\mu_1)\in U_\Phi$.
  Write $\Sigma_n:=\{-n,\dots,n\}^d\subset\subset \mathbb Z^d$
  for $n\in\mathbb N$.
  Then for some fixed $n\in\mathbb N$,
  three infinite components of $\{\xi=0\}$
  intersect $\Sigma_n$ with positive $\mu\times\rho$-probability.
  Moreover, as $\mu$ is ergodic with $S(\mu_1)\in U_\Phi$,
  we observe that the two functions
  \[
    \numberthis
    \label{qwdqwdqdiwnqdwinoqwnq}
    (\phi_1\pm 8nK)|_{\Sigma_n}\phi_1|_{\mathbb Z^d\smallsetminus \Sigma_N}
  \]
  are $q$-Lipschitz with high $\mu$-probability as $N\to\infty$.
  This is due to Lemma~\ref{lemma_lipschitz_q_and_norm}, Theorem~\ref{thm:superergodresult} and because $S(\mu_1)\in U_\Phi$---recall for comparison
  the pyramid construction from the proof of Lemma~\ref{lemma:PBL_upper_bound}.
  In particular, for $N\geq n$ sufficiently large,
  the $\mu\times\rho$-probability that three infinite components of
  $\{\xi=0\}$ intersect $\Sigma_n$
  and simultaneously the two functions in~\eqref{lemma:PBL_upper_bound}
  are $q$-Lipschitz, is positive.
  Now choose $x\in\mathcal L$ such that $0\not\in\Sigma_N+x$,
  and write $\Sigma'_n:=\Sigma_n+x$ and $\Sigma_N':=\Sigma_N+x$.
  Due to shift-invariance, have now proven that with positive $\mu\times\rho$-probability,
  $\Sigma_n'$ intersects three connected components of $\{\xi=a\}$
  for some $a\in\mathbb Z$,
  and the two functions in~\eqref{qwdqwdqdiwnqdwinoqwnq}
  are $q$-Lipschitz for $\Sigma_n$ and $\Sigma_N$
  replaced by $\Sigma_n'$ and $\Sigma_N'$ respectively.
  Let us write $A$ for this event.

  Let us first discuss the discrete setting $E=\mathbb Z$.
  If $(\phi_1,\phi_2,U)\in A$,
  then there exists another $q$-Lipschitz function $\phi_1'\in\Omega$
  which equals $\phi_1$ on the complement of $\Sigma_N'$,
  and such that $\{\xi=a\}\cup\Sigma_n'\subset\{\xi'=a\}$
  where $\xi':=(\phi_1',\phi_2,U)$.
  In particular, this means that $\Sigma_N'$ is a trifurcation box for $\xi'$.
  For example, one can take $\phi_1'$ to be
  the smallest $q$-Lipschitz extension
  of $\phi_1|_{\mathbb Z^d\smallsetminus\Sigma_N'}$
  to $\mathbb Z^d$ which equals at least
  \[
    \phi_2+4Ka+U+(\phi_1(0)-\phi_2(0))
  \]
  on $\{\xi=a\}\cup\Sigma_n'$.
  This proves that the event that $\Sigma_N'$ is a trifurcation box
  has positive measure in the measure on the left in~\eqref{eq:auiwdonduawbndaw} if we choose $\Lambda=\Sigma_n'$.
  If $E=\mathbb R$, then we must show that not only such a $q$-Lipschitz
  function $\phi_1'$ exists, but also that the set of such functions $\phi'_1$
  has positive Lebesgue measure.
  The original measure $\mu$ has finite specific free energy
  and therefore almost surely the height functions $\phi_1$
  and $\phi_2$ are not taut, that is,
  for every $\Lambda\subset\subset\mathbb Z^d$
  there almost surely exists a positive constant
  $\varepsilon>0$ such that the restriction of $\phi_1$
  and $\phi_2$ to $\Lambda$ are $q_\varepsilon$-Lipschitz.
  Now choose $\Lambda$ so large that $\Sigma_N'\subset\Lambda^{-R}$,
  choose $\varepsilon$ at least so small that $S(\mu_1)\in U_{q_\varepsilon}$,
  and construct the initial height function $\phi_1'$
  such that it is also $q_\varepsilon$-Lipschitz.
  It is easy to see that one can employ the remaining flexibility
  granted by Proposition~\ref{propo_qqq}, Statement~\ref{propo_qqq:flex}
  to demonstrate that the set of of suitable height
  functions has positive Lebesgue measure.
\end{proof}

\begin{theorem}
  \label{theorem_strictly_convex_continuous}
  Let $\Phi$ denote a potential which is monotone
  and in $\mathcal S_\mathcal L+\mathcal W_\mathcal L$.
  If $E=\mathbb R$,
  then $\sigma$ is strictly convex on $U_\Phi$.
\end{theorem}

\begin{proof}
  Let $\mu$ denote the measure from Theorem~\ref{thm_beyond_moats},
  and write $w_\mu$ for its ergodic decomposition.
  The measure $\mu$ satisfies $\mathcal H^2(\mu|\Phi)=\sigma^2(S^2(\mu))<\infty$,
  and both $\mathcal H^2(\cdot|\Phi)$
  and $S^2(\cdot)$ are strongly affine.
  This implies that
   $w_\mu$-almost every measure $\nu$
  satisfies $\mathcal H^2(\nu|\Phi)=\sigma^2(S^2(\nu))<\infty$.
  Since $E=\mathbb R$, this implies also that
  $S(\nu_1),S(\nu_2)\in U_\Phi$.

  With positive $w_\mu$-probability,
  the $\nu\times\rho$-probability that
   $\{\xi=0\}$ has
   at least three distinct infinite
  connected components, is positive.
  We have now proven the existence of a measure which satisfies all criteria of
  Lemma~\ref{lemma_application_of_BK}.
  This is the desired contradiction.
\end{proof}

\begin{theorem} \label{theorem_strictly_convex_discrete}
  Let $\Phi$ denote a potential which is monotone
  and in $\mathcal S_\mathcal L+\mathcal W_\mathcal L$.
  Consider now the discrete case $E=\mathbb Z$.
  Suppose that $\sigma$ satisfies the following property:
  for any affine map $h:(\mathbb R^d)^*\to\mathbb R$
  such that $h\leq \sigma$,
  the set $\{h=\sigma\}\cap\partial U_\Phi$
  is convex.
  Then $\sigma$ is strictly convex on $U_\Phi$.
  In particular, $\sigma$ is strictly convex on $U_\Phi$
  if at least one of the following conditions is satisfied:
  \begin{enumerate}
  	\item $\sigma$ is affine on $\partial U_\Phi$,
  	but not on $\bar U_\Phi$,
  	\item $\sigma$ is not affine on $[u_1,u_2 ]$ for any distinct $u_1,u_2 \in \partial U_\Phi$ such that  $[u_1,u_2 ] \not\subset \partial U_\Phi$.
  \end{enumerate}
\end{theorem}

\begin{proof}
  Suppose that $\sigma$
  satisfies the property in the statement.
  Let $h:(\mathbb R^d)^*\to\mathbb R$
  denote an affine map such that $h\leq \sigma$,
  and such that the set $\{h=\sigma\}\cap U_\Phi$
  contains at least two slopes.
  We aim to derive a contradiction.

  Let us first cover the case that $\{h=\sigma\}\subset U_\Phi$.
  Let $\mu$ denote the measure from Theorem~\ref{thm_beyond_moats},
  with slope $S(\mu)=(u,u)$
  for some $u\in\{h=\sigma\}$.
  Write $w_\mu$ for the ergodic decomposition of $\mu$.
  Then $w_\mu$-almost surely $S(\nu_1),S(\nu_2)\in\{h=\sigma\}\subset U_\Phi$,
  and therefore the proof is the same as for the real case.

  Let us now discuss the case that $\{h=\sigma\}$
  intersects $\partial U_\Phi$.
  Recall Lemma~\ref{lemma_lipschitz_q_and_norm}.
  Since $\{h=\sigma\}\cap \partial U_\Phi$
  is convex, this intersection must be contained in the boundary
  of one of the half-spaces $H=H(p)$
  contributing to the intersection in Lemma~\ref{lemma_lipschitz_q_and_norm},
  where $p=(p_k)_{0\leq k\leq n}$
  is a path of finite length through $(\mathbb Z^d,\mathbb A)$
  with $p_n-p_0\in\mathcal L$.
  Set $y:=p_0$ and $x:=p_n-p_0$.
  If a shift-invariant measure in $\mathcal P_\mathcal L(\Omega,\mathcal F^\nabla)$
  has finite specific free energy and its slope in $\partial H(p)$,
  then the random function $\phi$ must satisfy \[
  \numberthis\label{eqiubqdwubqw}
  \phi(y+kx)-\phi(y)=kq(p):=k\sum_{k=1}^n q(p_{k-1},p_k)
  \]
  for any $k\in\mathbb Z$ almost surely.
  As $x$ is orthogonal to $\partial H(p)$,
  it is straightforward to
  find two distinct slopes $u_1,u_2\in \{h=\sigma\}\cap U_\Phi$
  such that $(u_1-u_2)(x)\neq 0$.

  Let $\mu$ denote the measure from Theorem~\ref{thm_beyond_moats},
  and write $w_\mu$ for its ergodic decomposition.
  The measure $\mu$ satisfies $\mathcal H^2(\mu|\Phi)=\sigma^2(S^2(\mu))$,
  and both $\mathcal H^2(\cdot|\Phi)$
  and $S^2(\cdot)$ are strongly affine.
  This implies that
   $w_\mu$-almost every measure $\nu$
  satisfies $\mathcal H^2(\nu|\Phi)=\sigma^2(S^2(\nu))<\infty$.
  We know that $w_\mu$-almost surely
  $S(\nu_1)$ and $S(
\nu_2
  )$ lie in $\{h=\sigma\}\subset \bar U_\Phi$,
  but it is not guaranteed that these slopes lie in
  $U_\Phi$.

  With positive $w_\mu$-probability,
  the $\nu\times\rho$-probability that
  $\xi$ is not constant on $y+\mathbb Zx$ and
  that
   $\{\xi=0\}$ has
   at least three distinct infinite
  connected components, is positive.
  But if $\xi$ is not constant on  $y+\mathbb Zx$,
  then~\eqref{eqiubqdwubqw} is false for $\phi$ having the distribution of either $\nu_1$
  or $\nu_2$, or both,
  and consequently at least one of $S(\nu_1)$ and $S(\nu_2)$
  does not lie in $\partial H(p)$.
  Conclude that with positive $w_\mu$-probability,
  at least one of $S(\nu_1)$ and $S(\nu_2)$
  lies in $U_\Phi$,
  and
  the $\nu\times\rho$-probability that $\{\xi=0\}$
  has three or more infinite connected components,
  is positive.
  We have now proven the existence of a measure which satisfies all criteria of
  Lemma~\ref{lemma_application_of_BK}.
  This is the desired contradiction.
\end{proof}

\section{Applications}

\subsection{The Holley criterion}

Each time we apply the theory,
we must verify that the specification associated
to the model of interest is monotone.
An interesting property of stochastic monotonicity is that it does not depends
on any formalism and can be checked through the Holley
criterion.
This criterion is usually stated in the context of the Ising model
or Fortuin-Kasteleyn percolation
 (see for example~\cite{grimmett2018probability}) but can be extended to
 random surfaces in a straightforward way.
 Throughout this section, we will use this criterion in combination with
  Theorem~\ref{thm_main_main}
 to prove the strict convexity of the surface tension for various interesting models.

\begin{theorem}[Holley criterion] \label{thm:Holley}
	The potential $\Phi \in \mathcal S_\mathcal L+\mathcal W_\mathcal L$ is monotone if  and only if for any two $q$-Lipschitz functions $\phi,\psi\in\Omega$
	with $\phi\leq \psi$ and for any $x\in\mathbb Z^d$,
	we have
	\[
	\gamma_{\{x\}}(\cdot,\phi)\preceq\gamma_{\{x\}}(\cdot,\psi).
	\]
\end{theorem}

\begin{proof}
	Choose $\phi$ and $\psi$ as in the statement of the theorem,
	and consider $\Lambda\subset\subset\mathbb Z^d$.
	We aim to demonstrate that
	\[
	\gamma_\Lambda(\cdot,\phi)\preceq\gamma_\Lambda(\cdot,\psi).
	\]
	Write $\kappa_\Lambda$ for the probability kernel associated with Glauber dynamics,
	that is,
	\[
			\kappa_\Lambda:=|\Lambda|^{-1}\sum_{x\in\Lambda}\gamma_{\{x\}}.
	\]
	It is clear under the assumption of the theorem that $\kappa_\Lambda$
	preserves the partial order $\preceq$ on $q$-Lipschitz measures.
	Claim now that
	\[
		\mu\kappa_\Lambda^n\to\mu\gamma_\Lambda
	\]
	in the strong topology as $n\to\infty$
	for any $q$-Lipschitz probability measure $\mu$;
	this would indeed imply the theorem.
	This is a standard fact in probability theory.
	The only detail requiring attention is that
	it is necessary
	for any $q$-Lipschitz function $\phi$,
	that
	$\gamma_\Lambda(\cdot,\phi)$-almost every height function $\psi$
	is accessible from $\phi$ by local moves,
	that is, by updating the value of $\phi$
	by one vertex in $\Lambda$ at a time,
	and such that all intermediate functions
	are also $q$-Lipschitz.
	This is straightforward to check from the definition of
	$q$---in particular, it is important that $q(x,y)+q(y,x)>0$
	for any $x,y\in\mathbb Z^d$ distinct.
\end{proof}


\subsection{Submodular potentials}

A potential $\Phi$ is said to be \emph{submodular} if for every $\Lambda\subset\subset\Z^d$,
$\Phi_{\Lambda}$ has the property that
\[
\Phi_{\Lambda} (\phi\wedge\psi) + \Phi_{\Lambda}(\phi\vee \psi)
\leq
\Phi_{\Lambda}(\phi) + \Phi_{\Lambda}(\psi).
\]
Sheffield proposes this family of potentials as a natural generalization of simply attractive potentials, and asks if similar results as the ones proved for simply attractive potentials in~\cite{S05} could be proved for finite-range submodular potentials.
It is easy to see that submodular potentials
generate monotone specifications.

\begin{lemma}
	A submodular potential is monotone.
\end{lemma}

\begin{proof}
	Let $\phi_1,\phi_2\in\Omega$
	denote $q$-Lipschitz functions
	with $\phi_1\leq\phi_2$.
	It suffices to check the Holley criterion (Theorem~\ref{thm:Holley}).
	Write $f_i$ for the Radon-Nikodym derivative of
	$
	\gamma_{\{x\}}(\cdot,\phi_i)\pi_{\{x\}}
	$ with respect to $\lambda$,
	for $i\in\{1,2\}$.
	It suffices to demonstrate that $f_1\lambda\preceq f_2\lambda$
	as measures on $(E,\mathcal E)$.
	Submodularity of $\Phi$ implies
	that $f_1(b)f_2(a)\leq f_1(a)f_2(b)$
	for $\lambda\times\lambda$-almost every $a,b\in E$
	with $a\leq b$.
	It is a simple exercise to see that this implies
	the desired stochastic domination.
\end{proof}

If $E=\mathbb R$ and $\Phi$ a submodular Lipschitz
potential fitting the framework of this article (which is a very mild requirement),
then we derive immediately from Theorem~\ref{thm_main_main} that the surface tension is strictly convex.

\begin{corollary}
	Suppose that $E=\mathbb R$
	and consider a submodular Lipschitz potential $\Phi\in\mathcal S_\mathcal L+\mathcal W_\mathcal L$.
	Then $\sigma$ is strictly convex on $U_\Phi$.
\end{corollary}

In the remainder of this section, we focus on the case $E=\mathbb Z$.
If $E=\mathbb Z$, then we cannot immediately conclude that
the surface tension is strictly convex,
because we must fulfill the additional condition in Theorem~\ref{thm_main_main}.
We demonstrate how to derive
this extra condition for many natural discrete models.
Let $(\mathbb A,q)$ denote the local Lipschitz constraint
associated with the potential of interest
and fix $R\in\mathbb N$ minimal subject to $d_1(x,y)\leq R$
for all $\{x,y\}\in\mathbb A$.

A measure $\mu\in\mathcal P_\mathcal L(\Omega,\mathcal F^\nabla)$
is called \emph{frozen}
if for any $\Lambda\subset\subset\mathbb Z^d$,
the values of the random function $\phi_\Lambda$ in $\mu$ depend
deterministically on the boundary values $\phi_{\partial^R\Lambda}$.
Call a local Lipschitz constraint \emph{freezing}
if any measure $\mu\in\mathcal P_\mathcal L(\Omega,\mathcal F^\nabla)$
which is supported on $q$-Lipschitz functions,
and which has $S(\mu)\in\partial U_\Phi$,
is frozen.
This condition on the local Lipschitz constraint implies that any such measure has zero specific entropy,
that is,
$
\mathcal H(\mu|\lambda)=0$.
Indeed, deterministic dependence implies that
\[
\mathcal H_{\mathcal F_{\Pi_n}^\nabla}(\mu|\lambda^{\Pi_n-1})
=
\mathcal H_{\mathcal F_{\partial^R\Pi_n}^\nabla}(\mu|\lambda^{\partial^R\Pi_n-1})
=O(n^{d-1})=o(n^d)
\]
as $n\to\infty$.

\begin{lemma}
	If the local Lipschitz constraint $(\mathbb A,q)$
	is invariant by the full lattice $\mathcal L=\mathbb Z^d$,
	then it is freezing.
	In particular, the local Lipschitz constraints corresponding to dimer models,
	the six-vertex model, and $Kd_1$-Lipschitz functions for $K\in\mathbb N$,
	are freezing.
\end{lemma}

\begin{proof}
	Fix $\mu\in\mathcal P_\mathcal L(\Omega,\mathcal F^\nabla)$
	with $S(\mu)\in\partial U_\Phi$ and supported on $q$-Lipschitz
	functions.
	As in the proof of Theorem~\ref{theorem_strictly_convex_discrete},
	there is a path $p=(p_k)_{0\leq k\leq n}$
	of finite length through $(\mathbb Z^d,\mathbb A)$
	with $x:=p_n-p_0\in\mathcal L\smallsetminus\{0\}$,
	such that
	\[
	\phi(p_0+y+kx)-\phi(p_0+y)
	\]
	is deterministic in $\mu$ for any $y\in\mathcal L$ and $k\in\mathbb Z$.
	Moreover, this path is a \emph{cycle lift}
	as defined in the proof of Lemma~\ref{lemma_lipschitz_q_and_norm}.
	Since $\mathcal L=\mathbb Z^d$,
	this means that
	$
	\phi(y+kx)-\phi(y)
	$
	is deterministic for any $y\in\mathbb Z^d$,
	and that $d_1(0,x)\leq R$.
	In particular, $\phi_\Lambda$
	depends deterministically on $\phi_{\partial^R\Lambda}$
	in $\mu$ for any $\Lambda\subset\subset\mathbb Z^d$.
\end{proof}

The final goal of this section is to prove the following theorem.

\begin{theorem}
	\label{thm_submodular_discrete_strictly_convex}
	Suppose that $E=\mathbb Z$,
	and that $\Phi\in\mathcal S_\mathcal L+\mathcal W_\mathcal L$
	is a submodular Lipschitz potential
	with a freezing local Lipschitz constraint.
	Then the associated surface tension $\sigma$ is strictly convex on $U_\Phi$.
\end{theorem}

We first prove two auxiliary lemmas.

\begin{lemma}
	If $E=\mathbb Z$ and $\Phi$ a submodular gradient potential,
	then
	\[
		\Phi_\Lambda(\lceil{\textstyle\frac{\phi_1+\phi_2}{2}}\rceil)
		+
		\Phi_\Lambda(\lfloor{\textstyle\frac{\phi_1+\phi_2}{2}}\rfloor)
		\leq
		\Phi_\Lambda(\phi_1)
		+
		\Phi_\Lambda(\phi_2)
	\]
	for any $\phi_1,\phi_2\in\Omega$
	and $\Lambda\subset\subset\mathbb Z^d$.
\end{lemma}

\begin{proof}
	Write $\xi^\pm:=\phi_1\pm\phi_2$,
	so that $\phi_1=(\xi^++\xi^-)/2$
	and $\phi_2=(\xi^+-\xi^-)/2$.
	Write
	\[
		F(\psi^+,\psi^-):=\Phi_\Lambda({\textstyle\frac{\psi^++\psi^-}{2}})+\Phi_\Lambda({\textstyle\frac{\psi^+-\psi^-}{2}})
	\]
	for any
	 $\psi^+,\psi^-\in \Omega$
	with $\psi^++\psi^-\equiv 0\mod 2$.
	For example, the right hand side of the display in the statement of the lemma
	equals $F(\xi^+,\xi^-)$,
	and the left hand side equals $F(\xi^+,p\circ\xi^-)$,
	where $p:\mathbb Z\to\{0,1\}$
	is the \emph{parity function}
	which maps even integers to $0$ and odd integers to $1$.
	Therefore it suffices to demonstrate that
	\[
		F(\psi^+,p\circ\psi^-)\leq F(\psi^+,\psi^-)
	\]
	for any $\psi^+,\psi^-\in \Omega$
 with $\psi^++\psi^-\equiv 0\mod 2$.

 Observe that $F$ has the following four properties:
 \begin{enumerate}
 	\item \emph{Translation invariance}: $F(\psi^++a_1,\psi^-+a_2)=F(\psi^+,\psi^-)$ for any $a_1,a_2\in\mathbb Z$
	with $a_1+a_2$ even, because $\Phi$ is a gradient specification,
	\item \emph{Inversion invariance}:
	$F(\psi^+,-\psi^-)=F(\psi^+,\psi^-)$; replacing $\psi^-$ by $-\psi^-$
	corresponds to interchanging the sum and difference of $\psi^+$ and $\psi^-$,
	\item \emph{Submodularity}:
	$F(\psi^+,|\psi^-|)\leq F(\psi^+,\psi^-)$;
	equivalent to submodularity of $\Phi$,
	\item \emph{Locally measurable}: $F(\psi^+,\psi^-)$ depends on
	$\psi^\pm_\Lambda$ only.
 \end{enumerate}
 By applying the three operations on the pair $(\psi^+,\psi^-)$ finitely many
 times, one can turn the original pair into a new pair $(\psi^+,\hat\psi^-)$,
 where $\hat\psi^-_\Lambda=(p\circ\psi^-)_\Lambda$.
 In particular, since each operation can only decrease the value of $F$,
 we have
 \[
		F(\psi^+,p\circ\psi^-)=F(\psi^+,\hat\psi^-)\leq F(\psi^+,\psi^-)
 \]
 as desired.
\end{proof}

\begin{corollary}
	\label{cor_adwbowaoboivnwioeafnrnierioen}
	Suppose that $E=\mathbb Z$
	and that $\Phi\in\mathcal S_\mathcal L+\mathcal W_\mathcal L$
	is submodular.
	If $\mu_1,\mu_2\in\mathcal P_\mathcal L(\Omega,\mathcal F^\nabla)$
	are ergodic, then there exists an ergodic measure $\nu\in\mathcal P_\mathcal L(\Omega,\mathcal F^\nabla)$
	with
	\[
	S(\nu)=\frac{S(\mu_1)+S(\mu_2)}2
	\qquad\text{and}\qquad
	\langle \nu|\Phi\rangle\leq \frac{\langle \mu_1|\Phi\rangle+\langle \mu_2|\Phi\rangle}2.
	\]
\end{corollary}

\begin{proof}
	Write $\hat\mu\in\mathcal P_\mathcal L(\Omega,\mathcal F^\nabla)$ for the following measure:
	to sample from $\hat\mu$, sample first
	a pair $(\phi_1,\phi_2)$ from $\mu_1\times\mu_2$,
	and sample $X$ from $\{0,1\}$
	independently and uniformly at random;
	the final sample $\psi$
	from $\hat\mu$ is now defined by
	\[
	\psi:=
		\begin{cases}
			\lceil{\textstyle\frac{\phi_1-\phi_1(0)+\phi_2-\phi_2(0)}{2}}\rceil
			&\text{if $X=0$,}\\
				\lfloor{\textstyle\frac{\phi_1-\phi_1(0)+\phi_2-\phi_2(0)}{2}}\rfloor
				&\text{if $X=1$.}
		\end{cases}
	\]
	Since $\phi_1-\phi_1(0)$
	and $\phi_2-\phi_2(0)$
	are asymptotically close to $S(\mu_1)$ and $S(\mu_2)$
	respectively in the measure $\mu_1\times\mu_2$ in the sense of Theorem~\ref{thm:superergodresult},
	it is clear that $\psi$ is asymptotically close to $(S(\mu_1)+S(\mu_2))/2$
	in $\hat\mu$ (see also the proof of Lemma~\ref{lemma_ergodic_approx_aux_1}).
	In particular, $S(\nu)=(S(\mu_1)+S(\mu_2))/2$
	for $w_{\hat\mu}$-almost every $\nu$  in the ergodic decomposition of $\hat\mu$.
	By the previous lemma,
	we have
	\[
		\langle \hat\mu|\Phi\rangle\leq \frac{\langle \mu_1|\Phi\rangle+\langle \mu_2|\Phi\rangle}2.
	\]
	As $\langle\cdot|\Phi\rangle$ is strongly affine,
	we have $\langle\nu|\Phi\rangle\leq \langle\hat\mu|\Phi\rangle$
	with positive $w_{\hat\mu}$-probability.
	This proves the existence of the desired measure $\nu$.
\end{proof}

\begin{lemma}\label{lem:frozen}
	Consider the case that $E=\mathbb Z$,
	$\Phi$ a potential in $\mathcal S_\mathcal L+\mathcal W_\mathcal L$,
	and $\mu$ an ergodic minimizer with $S(\mu)\in U_\Phi$.
	Then $\mathcal H(\mu|\lambda)<0$.
\end{lemma}

\begin{proof}
	Suppose that $\mu$ does have zero combinatorial entropy;
	we aim to derive a contradiction.
	Write $u:=S(\mu)$, and
	write $\hat\mu\in\mathcal P_\mathcal L^2(\Omega,\mathcal F^\nabla)$
	for the unique measure which has $\mu$
	as its first marginal, and in which $\phi_1$ and $\phi_2$
	are equal almost surely.
	Then
	$S^2(\hat\mu)=(u,u)$
	and
	$\mathcal H^2(\hat\mu|\Phi)=2\langle\mu|\Phi\rangle=2\mathcal H(\mu|\Phi)=\sigma^2(S^2(\hat\mu))<\infty$,
	that is, $\hat\mu$ is a minimizer in the product setting.
	The adaptation of Theorem~\ref{thm_main_minimizers_finite_energy}
	to the product setting implies that
	\[
	1_{\Omega_q^2}(\hat\mu\pi_{\mathbb Z^d\smallsetminus\Lambda}\times\lambda^\Lambda\times\lambda^\Lambda)
	\ll
	\hat\mu
	\]
	for any $\Lambda\subset\subset\mathbb Z^d$,
	where $\Omega_q^2$ is the set of pairs of $q$-Lipschitz height functions.
	However, since $\mu$ is ergodic with slope in $U_\Phi$,
	we can find some $\Lambda\subset\subset\mathbb Z^d$
	such that with positive $\hat\mu$-probability
	$\phi_1|_{\mathbb Z^d\smallsetminus\Lambda}$
	has more than a single $q$-Lipschitz extension to $\mathbb Z^d$.
	This contradicts that $\phi_1$ and $\phi_2$
	are almost surely equal in $\hat\mu$.
\end{proof}

We are now ready to prove the second main theorem of this section.

\begin{proof}[Proof of Theorem~\ref{thm_submodular_discrete_strictly_convex}]
Recall Theorem~\ref{thm_main_main}.
If $\sigma$ is not strictly convex,
then there is an affine map $h:(\mathbb R^d)^*\to\mathbb R$
with $h\leq \sigma$ and such that $\{h=\sigma\}\cap \partial U_\Phi$
is not convex.
Write $H$ for the exposed points of $\{h=\sigma\}\subset (\mathbb R^d)^*$
which are also in $\partial U_\Phi$.
Then the convex envelope of $H$ intersects $U_\Phi$.

Note that each slope in $H$ is also an exposed point of $\sigma$.
This means that for each slope in $H$, there is an ergodic minimizer $\mu$
of that slope.
Moreover, since $\mathcal H(\mu|\lambda)=0$
for any $\mu$ with $S(\mu)\in H\subset \partial U_\Phi$,
we must have $\langle\mu|\Phi\rangle=h(S(\mu))=\sigma(S(\mu))$
for any such measure $\mu$.
The fact that the convex envelope of $H$ intersects $U_\Phi$,
together with
Corollary~\ref{cor_adwbowaoboivnwioeafnrnierioen},
implies that there exists an ergodic measure $\mu\in\mathcal P_\mathcal L(\Omega,\mathcal F^\nabla)$
with $S(\mu)\in U_\Phi$
and $\langle \mu|\Phi\rangle\leq h(S(\mu))\leq \sigma(S(\mu))$.
But it is only possible that $\langle \mu|\Phi\rangle\leq \sigma(S(\mu))$
if $\langle \mu|\Phi\rangle=\sigma(S(\mu))$
and if $\mu$ is a minimizer with $\mathcal H(\mu|\lambda)=0$.
This contradicts Lemma~\ref{lem:frozen}.
\end{proof}

\subsection{Tree-valued graph homomorphisms}
\label{subsection:applications_tree}

The flexibility of the main theorem in this article can also be used to prove statements about the behavior of random functions taking values in target spaces other than $\Z$ and $\R$.
A noteworthy example is the model of tree-valued graph homomorphisms described in~\cite{MT16}.
Let $k\geq 2$ denote a fixed integer,
and let $\mathcal T_k$ denote the $k$-regular tree,
that is, a tree in which every vertex has exactly $k$ neighbors.
In this context, tree-valued graph homomorphisms are functions from $\Z^d$ to
the vertices of $\mathcal T_k$ which also map the edges of the square lattice to
the edges of the tree.
Regular trees are natural objects in
several fields of mathematics:
in group theory, for example, they arise as Cayley graphs
of free groups on finitely many generators.
As a significant result in~\cite{MT16}, the authors characterize
the surface tension for the model (there named \emph{entropy}) and show that
it is equivalent to the number of graph homomorphisms with nearly-linear boundary
conditions.
In this section we will confirm the conjecture from~\cite{MT16}, which states that
this entropy function is strictly convex.
We must first show how the model and the corresponding surface tension fit into the framework of this paper.
A tree-valued graph homomorphism can be represented
by an integer-valued graph homomorphism after introducing an infinite-range
potential to compensate for the ``loss of information''.

Let us first introduce some definitions.
Write $d_{\mathcal T_k}$ for the graph metric on $\mathcal T_k$.
Let $g$ denote a fixed \emph{bi-infinite geodesic} through $\mathcal T_k$,
that is, a $\mathbb Z$-indexed sequence of vertices
$g=(g_n)_{n\in\mathbb Z}\subset\mathcal T_k$
such that
$d_{\mathcal T_k}(g_n,g_m)=|m-n|$ for any $n,m\in\mathbb Z$.
Let $p:\mathcal T_k\to\mathbb Z$ denote the projection of
the tree onto $g$,
defined such that $p(x)$ minimizes $d_{\mathcal T_k}(x,g_{p(x)})$
for any $x\in\mathcal T_k$.
Write $h$ for the \emph{horocyclic height function} on $\mathcal T_k$;
this is the function $h:\mathcal T_k\to\mathbb Z$
defined by $h(x):=p(x)+d_{\mathcal T_k}(x,g_{p(x)})$ (see also~\cite{grimmett2018locality}).
In other words, if $x=g_n$ for some $n\in\mathbb Z$,
then $h(x)=n$,
and $h$ increases by one every time one moves away from the geodesic $g$.
The function $h$ can also be characterized as follows:
each vertex $x\in\mathcal T_k$ has a unique neighbor $y$ such that
$h(y)=h(x)-1$,
and $h(z)=h(x)+1$ for every other neighbor $z$ of $x$.

The graphs $\mathbb Z^d$, $\mathbb Z$, and $\mathcal T_k$
are bipartite, we shall call the two parts
the \emph{even vertices} and \emph{odd vertices}
respectively; the set of even vertices is the part containing
 $0$ if the graph is $\mathbb Z^d$
or $\mathbb Z$,
and the part containing $g_0$ if the graph is $\mathcal T_k$.
By a \emph{graph homomorphism} we mean a map from
$\mathbb Z^d$ to $\mathbb Z$ or $\mathcal T_k$
which preserves the parity of the vertices,
and which maps edges to edges.
Write $\Omega$ and $\tilde\Omega$
respectively for the
set of graph homomorphisms from $\mathbb Z^d$
to either $\mathbb Z$ or $\mathcal T_k$.
For fixed $\phi\in\Omega$ and $n\in\mathbb Z$,
we call some set $\Lambda\subset\subset\mathbb Z^d$
an \emph{$n$-upper level set} if $\Lambda$ is
a connected component of $\{\phi\geq n\}\subset\mathbb Z^d$ in the square lattice graph.
An $n$-upper level set is also called an \emph{$n$-level set}
or simply a \emph{level set}.


\begin{figure}
	\centering
  \newcommand{\fw}{.475\textwidth}
	\begin{subfigure}{\fw}
		\centering
    \includegraphics[width=\textwidth]{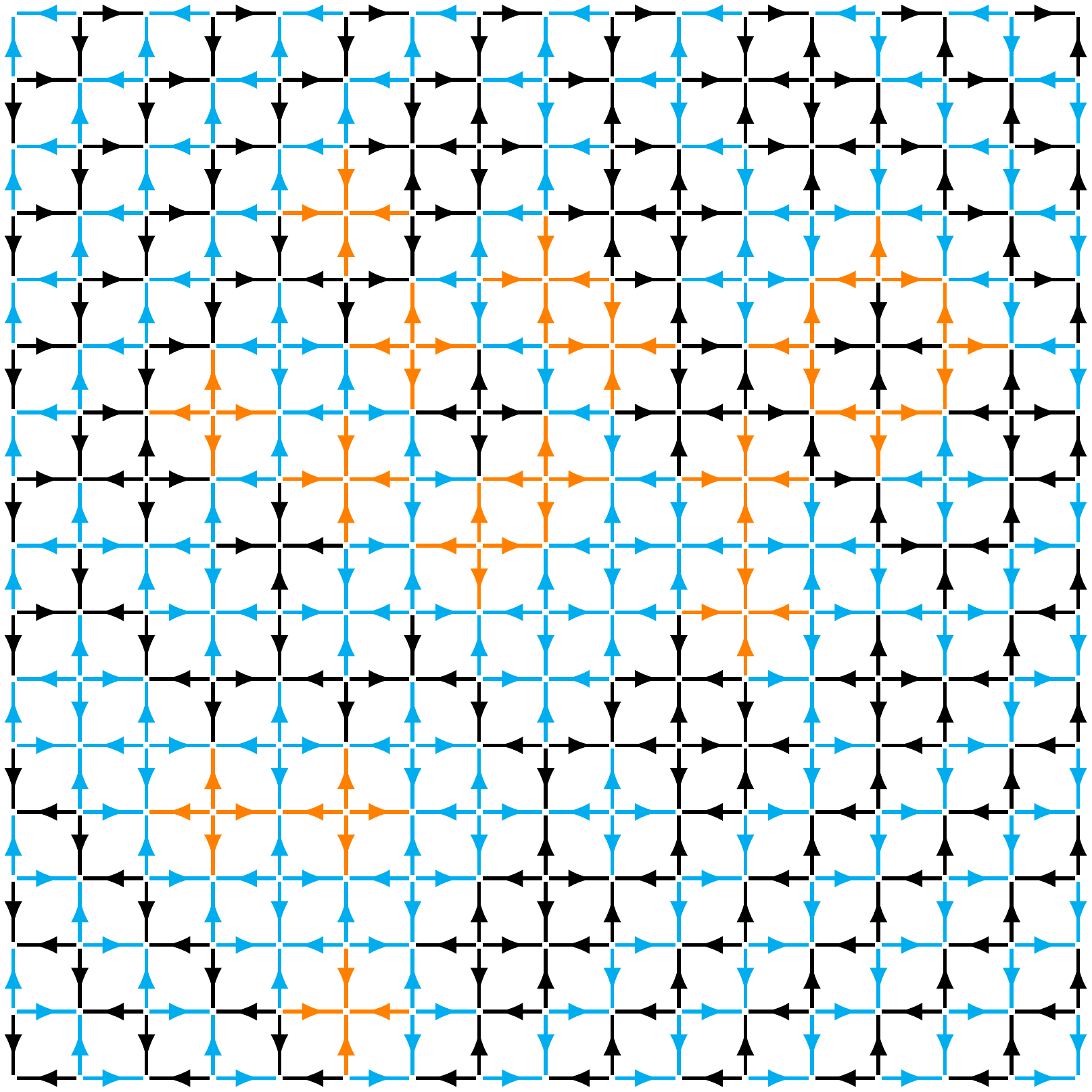}
		\caption*{The gradient of the graph homomorphism}
	\end{subfigure}
  \hfill
  \begin{subfigure}{\fw}
		\centering
    \includegraphics[width=\textwidth]{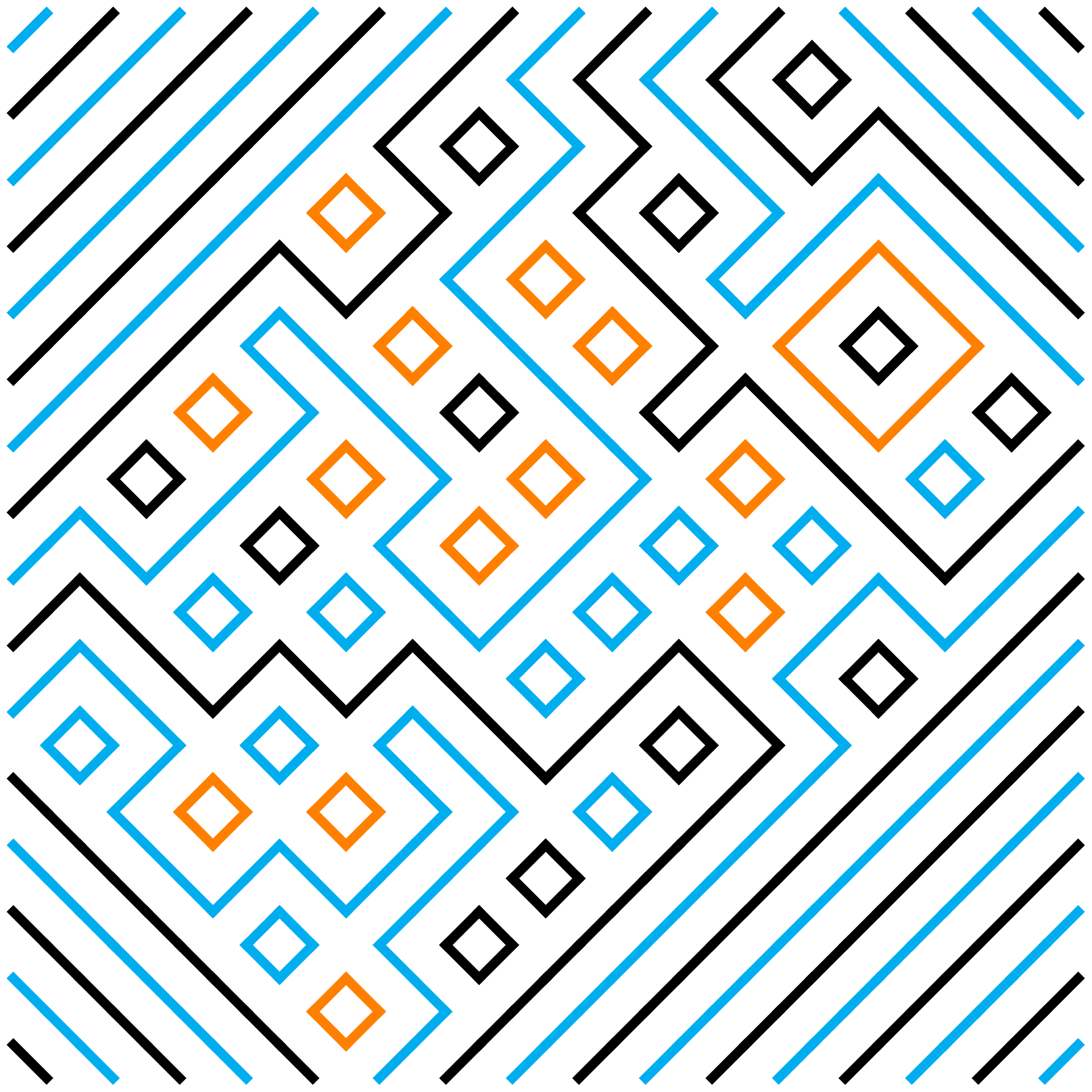}
		\caption*{The boundaries of the upper level sets}
	\end{subfigure}
	\caption{A random $\mathcal T_3$-valued graph homomorphism}
	\label{fig:tree_valued_small}
\end{figure}

Write $U$ for the set of slopes $u\in(\mathbb R^d)^*$
such that $|u(e_i)|<1$ for each element $e_i$
in the natural basis of $\mathbb R^d$.
For fixed $u\in\bar U$,
write $\phi^u\in\Omega$ for the graph homomorphism
defined by
\[
	\phi^u(x):=
	\lfloor u(x)\rfloor+
	\begin{cases}
		0&\text{if $d_1(0,x)\equiv \lfloor u(x)\rfloor \mod 2$,}\\
		1&\text{if $d_1(0,x)\equiv \lfloor u(x)\rfloor + 1 \mod 2$,}
	\end{cases}
\]
and write $\tilde\phi^u\in\tilde\Omega$
for the graph homomorphism defined by $\tilde\phi^u(x)=g_{\phi^u(x)}$.

 It is shown in Section~3 of~\cite{MT16} that the entropy function $\operatorname{Ent}:\bar U\to [-\log k,0]$ associated to the model of graph homomorphisms from $\Z^d$ to $\mathcal T_k$ can be estimated by counting for each slope
 $u\in\bar U$ the number of graph homomorphisms $\phi:\mathbb Z^d\to\mathcal T_k$
 which equal $\tilde\phi^u$ on the complement of $\Pi_n$.
More precisely, for $u\in\bar U$, we have
\[
\operatorname{Ent}(u)=\lim_{n \to \infty}-n^{-d}\log |\{\tilde\phi\in\tilde\Omega:\tilde\phi_{\mathbb Z^d\smallsetminus\Pi_n}=\tilde\phi_{\mathbb Z^d\smallsetminus\Pi_n}^u\}|.
\]
Notice that counting the number of functions in this
set is similar to considering the normalizing constant in
the definition of the specification,
as we frequently do in this paper.
Before proceeding, let us already remark that $\operatorname{Ent}(u)=0$
for $u\in\partial U$.
Indeed, for such $u$, the set in the display contains only a single element:
	the original function $\tilde\phi^u$.
	It is also easy to see that $\operatorname{Ent}$
	is not identically zero on $\bar U$.
	Consider, for example, the slope $u=0$,
	and consider the set of all graph homomorphisms $\tilde\phi$ which equal $\tilde\phi^u$
	on the complement of $\Pi_n$ and which
	map all the even
	vertices of the square lattice to $g_0\in\mathcal T_k$.
	Then this set contains
	at least $k^{\lfloor n^d/2\rfloor}$
	functions, proving that $\operatorname{Ent}(u)\leq -\frac12\log k<0$.

We now get to the heart of the case.
	Let us use the horocyclic height function to count the set in the previous display in a different way.
	Suppose that some graph homomorphism $\phi\in\Omega$
	equals $\phi^u$ on the complement of $\Pi_n$.
	How many graph homomorphisms $\tilde\phi\in\tilde\Omega$
	do there exist which satisfy $h\circ\tilde\phi=\phi$
	and equal $\tilde\phi^u$ on the complement
	of $\Pi_n$?
	It turns out that this number must be precisely
	$(k-1)^{F_{\Pi_n}(\phi)}$,
	where $F_\Lambda(\phi)$ denotes the number of level sets of $\phi$
	which are entirely contained in $\Lambda$,
	for any $\Lambda\subset\subset\mathbb Z^d$.
	Indeed, each time we see an $n$-level set
	of $\phi$,
	the function $\tilde\phi$ must be constant on the outer boundary
	of that $n$-level set---say with value $x\in\mathcal T_k$---and there
	are $k-1$ neighbors of $x$ which lead to an increase of the horocyclic height function
	by exactly one.
	In particular, we have
	\[
  \numberthis
  \label{eq:integer_def_for_Ent}
	\operatorname{Ent}(u)=\lim_{n \to \infty}-n^{-d}
	\log
	\sum_{\phi\in\Omega,\,
	\phi_{\mathbb Z^d\smallsetminus\Pi_n}=\phi_{\mathbb Z^d\smallsetminus\Pi_n}^u}
	(k-1)^{F_{\Pi_n}(\phi)}.
	\]
  See Figure~\ref{fig:tree_valued_small}
  for a sample of the model, with
  the gradient of the graph homomorphism on the left,
  and with the boundaries of the level sets of
  the horocyclic height function on the right.
	We have now reduced to a problem expressed entirely in terms of integer-valued
	functions.
	In fact, we do no longer require $k$ to be an integer,
	although we do require that $k\geq 2$.
	In the remainder of this section, we construct a potential $\Phi$
	which fits into our class $\mathcal S_\mathcal L+\mathcal W_\mathcal L$
	and which is monotone,
	and such that $U_\Phi=U$ and $\sigma=\operatorname{Ent}$.
	This proves that $\sigma$ and $\operatorname{Ent}$
	are strictly convex on $U_\Phi=U$.
	In fact, the specification induced by the potential that we construct is not perfectly monotone,
	but we shall demonstrate that it is sufficiently monotone
	for us to deduce that $\sigma$ is strictly convex.

Unfortunately, we cannot hope to use a potential that counts the level sets directly.
The reason is that there is no upper bound on the number of level sets containing a single point;
such a potential would always sum to infinity.
However, each finite level set has a uniquely defined outer boundary,
and each vertex is contained in only finitely many outer boundaries.
This means that counting outer boundaries of finite level sets is equivalent to counting
finite level sets, and the potential that does so is well-defined
and fits our framework, as we will show.
It is not possible through this method to count infinite level sets,
but we shall demonstrate how to work around this apparent difficulty.

We shall now describe how to characterize the
outer boundary of a finite level set.
This is not entirely straightforward due to
the connectivity properties of the square lattice.
By the $*$-graph on $\mathbb Z^d$, we mean the graph in which two vertices $x$ and $y$
are neighbors if and only if $\|x-y\|_\infty=1$.
For example, each vertex has $3^d-1$ distinct $*$-neighbors.
On every single occasion that we mention a graph-related notion, we mean the
usual square lattice graph, unless we explicitly mention the $*$-graph.
Due to the connectivity properties of the square lattice,
we have the following proposition.

\begin{proposition}
Suppose that $\Lambda\subset\subset\mathbb Z^d$
is finite and connected,
and that its complement $\Delta:=\mathbb Z^d\smallsetminus\Lambda$
is $*$-connected.
Define $\partial^*\Delta$ to be the set of vertices $x\in\mathbb Z^d$
such that:
\begin{enumerate}
	\item Either $x\in\Lambda=\mathbb Z^d\smallsetminus\Delta$ and $*$-adjacent to $\Delta$,
	\item Or $x\in\Delta$ and adjacent to $\Lambda=\mathbb Z^d\smallsetminus\Delta$.
\end{enumerate}
Then $\partial^*\Delta\cap\Lambda=\partial^*\Delta\cap(\mathbb Z^d\smallsetminus\Delta)$ is connected,
and so is $\partial^*\Delta$.
\end{proposition}

Consider a finite nonempty connected set $\Lambda\subset\subset\mathbb Z^d$.
Write $\Lambda^\infty$ for the \emph{outside} of $\Lambda$,
that is, the unique unbounded $*$-connected component of the complement of $\Lambda$.
Write also $\bar\Lambda$ for the complement of $\Lambda^\infty$:
this set is finite and connected,
and contains $\Lambda$.
The pair $(\bar\Lambda,\Lambda^\infty)$ will play the role of $(\Lambda,\Delta)$
in the previous proposition.
The set $\partial^*\Lambda^\infty$ can obviously be written as the disjoint union
of $\partial^*\Lambda^\infty\cap\Lambda^\infty$
and
$\partial^*\Lambda^\infty\cap\bar\Lambda$.
Claim that $\partial^*\Lambda^\infty\cap\bar\Lambda=\partial^*\Lambda^\infty\cap\Lambda$.
Indeed, if $x\in\partial^*\Lambda^\infty\cap\bar\Lambda$
is not in $\Lambda$,
then it should be in $\Lambda^\infty$ as it is $*$-adjacent to $\Lambda^\infty$;
this proves the claim.
This also means that all vertices in $\partial^*\Lambda^\infty\cap\Lambda^\infty$
are adjacent to $\Lambda$.

Suppose now that $\Lambda$ is also an $n$-level set of some graph homomorphism $\phi\in\Omega$.
Then $\phi$ must equal exactly $n-1$ on $\partial^*\Lambda^\infty\cap\Lambda^\infty$,
and $\phi$ must be at least $n$ on $\partial^*\Lambda^\infty\cap\bar\Lambda=\partial^*\Lambda^\infty\cap\Lambda$.
We have now proven the following lemma.

\begin{lemma}
	Suppose that $\Delta\subsetneq\mathbb Z^d$ is $*$-connected
	and cofinite, with its complement connected.
	Then
	\[
	\begin{split}
	&\{\phi\in\Omega:\text{$\Delta$ is the outside of a $n$-level set of $\phi$ for some $n\in\mathbb Z$}\}
	\\&\qquad=
	\{\phi\in\Omega:
		\text{$\phi_{\partial^*\Delta\cap\Delta}= n-1$
		and $\phi_{\partial^*\Delta\smallsetminus\Delta}\geq n$
		for some $n\in\mathbb Z$}
	\}
	\in\mathcal F_{\partial^*\Delta}^\nabla.
	\end{split}
	\]
	Moreover, no two level sets of $\phi$
  produce the same outside boundary $\partial^*\Delta$.
\end{lemma}

Define the potential
$\Xi=(\Xi_\Lambda)_{\Lambda\subset\subset\mathbb Z^d}$ by
\[
	\Xi_\Lambda(\phi)=-\log(k-1)
\]
if $\Lambda=\partial^*\Delta^\infty$ for some finite level set $\Delta$
of $\phi$,
and $\Xi_\Lambda(\phi)=0$ otherwise.
For fixed $x\in\mathbb Z^d$ and $\phi\in\Omega$,
there are at most $3^d$ finite level sets $\Delta$
of $\phi$ such that $x\in\partial^*\Delta^\infty$.
In particular, this means
that $\|\Xi\|\leq 3^d\log (k-1)$.
Moreover, since $\Xi_\Lambda\equiv 0$
whenever $\Lambda\subset\subset\mathbb Z^d$
is not connected,
it is clear that
$e^-(\Lambda)\leq |\partial\Lambda|\cdot \|\Xi\|$.
In particular, $e^-$ is an amenable function, which means that $\Xi\in\mathcal W_\mathcal L$.
Remark that
$
	H_\Lambda^\Xi(\phi)
$ equals $-\log(k-1)$ times the number of finite level sets $\Delta$ of $\phi$
for which $\partial^*\Delta^\infty$ intersects $\Lambda$.
Unfortunately, it is not possible to count infinite level sets with this construction;
this is a small inconvenience that we must circumvent.

Write $\Psi$ for the potential which forces graph homomorphisms,
that is,
$\Psi_{\Lambda}(\phi)=\infty$ if $\Lambda=\{x,y\}$
is an edge of the square lattice and $|\phi(y)-\phi(x)|\neq 1$,
and $\Psi_\Lambda(\phi)=0$ otherwise.
This potential belongs to $\mathcal S_\mathcal L$, modulo the detail explained
in Subsection~\ref{subsec_main_Lip_note},
which we shall simply ignore here.

\begin{lemma}
  For any integer $k\geq 2$,
  the surface tension $\sigma$ associated to the potential $\Phi:=\Psi+\Xi$
  equals the entropy function $\operatorname{Ent}$.
\end{lemma}

\begin{proof}
  We prove that $\sigma(u)=\operatorname{Ent}(u)$
  for $u\in U_\Phi$,
  the result extends to all $u\in \bar U_\Phi$
  because both $\sigma$ and $\operatorname{Ent}$
  are continuous on $\bar U_\Phi$.
   Due to Theorem~\ref{thm:sec_mr_ldp}, we know that
   \begin{align*}
     \sigma(u)=P_\Phi((0,1)^d,u|_{\partial (0,1)^d})&=
     \lim_{n\to\infty}
     -n^{-d}\log
     \int_{E^{\Pi_n}}
     e^{-H_\Lambda^\Phi(\psi\phi^u_{\mathbb Z^d\smallsetminus\Pi_n})}
     d\lambda^{\Pi_n}(\psi)
\\&=
\lim_{n\to\infty}
-n^{-d}\log
\sum_{\phi\in\Omega,\,
\phi_{\mathbb Z^d\smallsetminus\Pi_n}=\phi_{\mathbb Z^d\smallsetminus\Pi_n}^u}
e^{-H_{\Pi_n}^\Xi(\phi)}.
   \end{align*}
   But the logarithm of
   the ratio of $e^{-H_{\Pi_n}^\Xi(\phi)}$
   with $(k-1)^{F_{\Pi_n}(\phi)}$
   is of order $O(n^{d-1})=o(n^d)$
   uniformly over $\phi$ as $n\to\infty$,
   so that the equality $\sigma(u)=\operatorname{Ent}(u)$
   follows from~\eqref{eq:integer_def_for_Ent}.
\end{proof}

\begin{definition}
	Write $\Omega_-$ for the set of graph homomorphisms
	$\phi\in\Omega$ which have no infinite level sets.
\end{definition}

\begin{lemma}
	The specification induced by the potential $\Phi:=\Psi+\Xi$
	 is stochastically monotone over $\Omega_-$ for any $k\geq 2$.
\end{lemma}

\begin{proof}
  We use the Holley criterion (Theorem~\ref{thm:Holley})
  to prove that $\gamma_\Lambda$ preserves $\preceq$;
  we suppose that $\Lambda=\{0\}$ without loss of generality.
	Let $\phi_1,\phi_2\in\Omega_-$ denote graph homomorphisms
	without infinite level sets, and which satisfy
	 $\phi_1\leq \phi_2$.
	 Notice that the only case where the local Gibbs measure $\gamma_\Lambda(\cdot,\phi)$
	 is not a Dirac measure, is if there exist a $n \in \Z$ such that $\phi(x)=n$
	 for any neighbor $x$ of $0$.
	 If this is not the case for $\phi_1$ or $\phi_2$
	 then the proof is trivial;
	 we reduce to the case that $\phi_1(x)=\phi_2(x)=1$
	 for any neighbor $x$ of $0$ in $\mathbb Z^d$.
	 It remains to show that $\gamma_\Lambda(\cdot,\phi_1)\preceq\gamma_\Lambda(\cdot,\phi_2)$.
	 Without loss of generality, $\phi_1(0)=\phi_2(0)=0$.

	 Write $\psi$ for the random function in either local Gibbs measure.
	 Since $\phi_i(x)=1$ for any neighbor $x$ of $0$ and
	 for $i\in\{1,2\}$,
	 the function $\psi$ can only take two values with positive probability:
	 they are $0$ and $2$.
	 What we thus must show is that the quantity
	 \[
	 		a_i:=\frac{
				\gamma_\Lambda(\psi(0)=2,\phi_i)
				}
				{
					\gamma_\Lambda(\psi(0)=0,\phi_i)
				}
	 \]
	 satisfies $a_1\leq a_2$.
	 Claim that
	 $
	 	a_i=(k-1)^{2-X_i}
		$,
		where $X_i$ is the number of $1$-level sets of $\phi_i$ which are adjacent to $0$.
		If $\psi(0)=0$,
		then all $1$-level sets adjacent to $0$
		are counted separately,
		and $\{0\}$ is not a level set.
		If $\psi(0)=2$,
		then we count two level sets:
		the set $\{0\}$ is a $2$-level set,
		and all neighbors of $0$
		are contained in the same $1$-level set.
		All other level sets remain unaffected.
		This proves the claim.
		We must therefore prove that $X_1\geq X_2$.
		This is clear:
		increasing the values of $\phi$ can only
		increase the size of the $1$-level set
		containing a fixed vertex $x$,
		and potentially merge several $1$-level sets.
		In particular, it can only decrease the number of $1$-level
		sets adjacent to $0$.
\end{proof}

\begin{theorem}
	The surface tension $\sigma$
	associated to the potential $\Phi$
	defined above, is strictly convex on $U_\Phi$
	whenever $k\geq 2$.
\end{theorem}

\begin{proof}
	We must circumvent the problem that the specification $\gamma$
	induced by $\Phi$ is monotone only after restricting it to the set $\Omega_-$.
Remark that Theorem~\ref{thm:moats}
and Proposition~\ref{proposition_moats_effective_application}
remain true in this context if the measure $\mu$ in the statement
of Theorem~\ref{thm:moats} is supported on $\Omega_-$.
The only time that monotonicity is used in the proof for strict convexity of $\sigma$,
is in the application of these two results in Lemma~\ref{lemma_moats_application_in_product_setting}.
Recall that the local Gibbs measure $\gamma_n$ in the statement of Lemma~\ref{lemma_moats_application_in_product_setting}
was defined to be $\gamma_{\Pi_n}(\cdot,\phi^u)$;
this is now problematic because $\phi^u$
does have infinite level sets.
This can be easily solved by the following modification.
Define $\phi^u_n$ to be the smallest graph homomorphism which
equals $\phi^u$ on the set $\Pi_n\cup\partial\Pi_n$.
It is easy to check that $\{\phi^u_n\geq m\}$
is finite for any $n\in\mathbb N$ and $m\in\mathbb Z$;
in particular, $\phi_n^u\in\Omega_-$.
Moreover, the sequence $(\Pi_n,\phi^u_n)_{n\in\mathbb N}$
is as much an approximation of $((0,1)^d,u|_{\partial (0,1)^d})$
as the original sequence $(\Pi_n,\phi^u)_{n\in\mathbb N}$.
In particular, all of the same arguments apply if we simply replace each local Gibbs measure
$\gamma_n=\gamma_{\Pi_n}(\cdot,\phi^u)$
	by
	$\gamma_{\Pi_n}(\cdot,\phi^u_n)$.
  We had already seen that $\sigma=0$ on $\partial U_\Phi$
  and $\sigma(0)<0$, which proves that $\sigma$ is strictly convex.
\end{proof}


\subsection{Stochastic monotonicity in the six-vertex model}


\begin{figure}
	\centering
  \newcommand{\fw}{.15\textwidth}
	\newcommand{\da}[4]{
	  \draw[very thick](#1,#2)--(#3,#4);
		\draw[very thick,-Latex](#1,#2)--(#1+0.7*#3-0.7*#1,#2+0.7*#4-0.7*#2);
	}
	\begin{subfigure}{\fw}
		\centering
		\begin{tikzpicture}
			\da{-1}{0}{0}{0}
			\da{0}{0}{0}{1}
			\da{0}{-1}{0}{0}
			\da{0}{0}{1}{0}
		\end{tikzpicture}
		\caption*{$a_+$}
	\end{subfigure}
	\begin{subfigure}{\fw}
		\centering
		\begin{tikzpicture}
			\da{0}{0}{-1}{0}
			\da{0}{1}{0}{0}
			\da{0}{0}{0}{-1}
			\da{1}{0}{0}{0}
		\end{tikzpicture}
		\caption*{$a_-$}
	\end{subfigure}
	\begin{subfigure}{\fw}
		\centering
		\begin{tikzpicture}
			\da{-1}{0}{0}{0}
			\da{0}{1}{0}{0}
			\da{0}{0}{0}{-1}
			\da{0}{0}{1}{0}
		\end{tikzpicture}
		\caption*{$b_+$}
	\end{subfigure}
	\begin{subfigure}{\fw}
		\centering
		\begin{tikzpicture}
			\da{0}{0}{-1}{0}
			\da{0}{0}{0}{1}
			\da{0}{-1}{0}{0}
			\da{1}{0}{0}{0}
		\end{tikzpicture}
		\caption*{$b_-$}
	\end{subfigure}
	\begin{subfigure}{\fw}
		\centering
		\begin{tikzpicture}
			\da{-1}{0}{0}{0}
			\da{0}{0}{0}{1}
			\da{0}{0}{0}{-1}
			\da{1}{0}{0}{0}
		\end{tikzpicture}
		\caption*{$c_+$}
	\end{subfigure}
	\begin{subfigure}{\fw}
		\centering
		\begin{tikzpicture}
			\da{0}{0}{-1}{0}
			\da{0}{1}{0}{0}
			\da{0}{-1}{0}{0}
			\da{0}{0}{1}{0}
		\end{tikzpicture}
		\caption*{$c_-$}
	\end{subfigure}
	\caption{The six types of arrow configurations and their weights}
	\label{fig:six_vertex}
\end{figure}
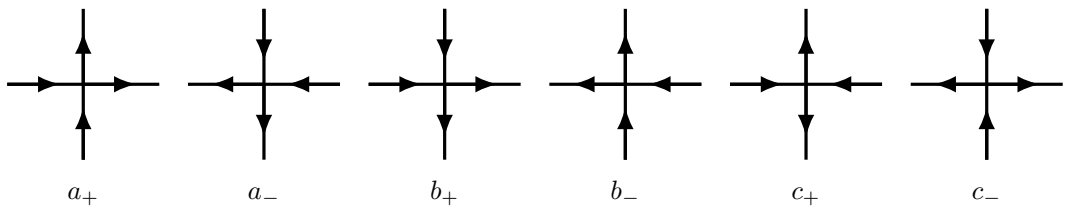

Consider the two-dimensional square lattice.
An \emph{arrow configuration}
is an orientation of
each edge of the square lattice, in such a way that each vertex
has exactly two incoming edges and two outgoing edges.
This means that there are six configurations
for the four edges incident to a fixed vertex;
see Figure~\ref{fig:six_vertex}.
Each of these six types receives a weight,
and one studies the probability measure where the probability of observing
an arrow configuration is proportional to the product of the weights over the vertices in that configuration.
This is the six-vertex model, which is the subject of an extensive literature.
Each arrow configuration has an associated height function,
which assigns integers to the faces of the square lattice,
and is defined as follows:
the height of the face to the right of an arrow is always exactly one
more than the height of the face to the left of it,
and the height of a fixed reference face is set to zero.
It is straightforward to see that this uniquely defines the height functions
associated to an arrow configuration.
The six-vertex model can thus be considered a Lipschitz random surface.
Our main theorem asserts that the surface tension of this random surface model is strictly
convex, if the specification is monotone.
It is a straightforward exercise to demonstrate that the specification
is monotone if and only if
\[
c_+c_- \geq \max\{a_+a_-,b_+b_-\};
\]
this is verified through checking the Holley criterion (Theorem~\ref{thm:Holley}).
Informally,
this means that the specification is monotone if the model
prefers vertices for which the four values of the adjacent faces
are as close to each other as possible.
Finally, we should mention that from the perspective of the specification,
there is some gauge equivalence in the choice of the six weights;
for details we refer to the work of Sridhar~\cite[Section~2.2]{sridhar2016limit}.

\begin{theorem}
  \label{thm:six-v_statement}
  The potential $\Phi\in\mathcal S_\mathcal L$
  corresponding to the six-vertex model is monotone if and only if
   $c_+c_- \geq \max\{a_+a_-,b_+b_-\}$,
   in which case $\sigma$ is strictly convex on
   $U_\Phi$.
\end{theorem}

Although it is not directly stated in \emph{Random Surfaces}~\cite{S05},
the potential $\Phi$ can be written as a simply attractive potential
whenever $c_+c_- \geq \max\{a_+a_-,b_+b_-\}$.
Therefore this theorem should be considered an alternative proof
rather than a novel result.

\renewcommand\optionalindent{}
\section*{Acknowledgment}
\addcontentsline{toc}{section}{Acknowledgment}

The authors would like to thank
Nathana\"el
Berestycki,
Georg Menz,
James Norris,
Scott Sheffield,
Fabio Toninelli,
and
Peter Winkler
for many useful discussions,
as well as
Richard Kenyon
for providing the authors with simulations of the $5$-vertex model.
The authors are especially grateful to Nathana\"el
Berestycki for enabling them to collaborate on this project.

The first author was supported by the Department of
Pure Mathematics and Mathematical Statistics, University of Cambridge, the UK
Engineering and Physical Sciences Research Council grant EP/L016516/1,
and the Shapiro Visitor Program of the Department of Mathematics, Dartmouth College.
The second author was supported by
the UK
Engineering and Physical Sciences Research Council grant EP/L018896/1
and the Peter Whittle Fund.

\bibliographystyle{amsalpha}
\bibliography{bib,bib_clean,bib_ldp}
\addcontentsline{toc}{section}{References}


\end{document}